\documentclass[letterpaper]{report}

\usepackage{amsfonts,amsmath,amssymb,amsthm}

\newtheorem{thm}{Theorem}[section]
\newtheorem{prop}[thm]{Proposition}
\newtheorem{lem}[thm]{Lemma}
\newtheorem{lemA}{Lemma}[chapter] 							
\newtheorem{cor}[thm]{Corollary}

\newtheorem{asm}{Assumption}
\newtheorem{op}[thm]{Question}

\theoremstyle{remark}

\theoremstyle{definition}
\newtheorem{defn}[thm]{Definition}

\newcommand{\ra}{\rightarrow}
\newcommand{\Ra}{\Rightarrow}

\newcommand{\N}{\mathbb N}     
\newcommand{\Q}{\mathbb Q}     
\newcommand{\R}{\mathbb R}     
\newcommand{\Z}{\mathbb Z}     

\renewcommand{\a}{\alpha}
\renewcommand{\b}{\beta}

\renewcommand{\d}{\delta}

\newcommand{\e}{\varepsilon}
\newcommand{\del}{\partial}
\newcommand{\w}{\omega}

\renewcommand{\l}{\lambda}

\newcommand{\bL}{\overline{\Lambda}}
\newcommand{\s}{\sigma}

\renewcommand{\P}{\mathbb{P}}   
\newcommand{\bP}{\overline{\mathbb{P}}}
\newcommand{\E}{\mathbb{E}}   
\newcommand{\bE}{\overline{\mathbb{E}}}
\renewcommand{\S}{\bar{S}}
\newcommand{\bw}{\bar{\w}}  

\newcommand{\nt}{\lfloor{nt}\rfloor}
\newcommand{\limn}{\underset{n\rightarrow\infty}{\longrightarrow}} 
\newcommand{\limN}{\underset{N\rightarrow\infty}{\longrightarrow}} 
 
\newcommand{\limdw}{\overset{\mathcal{D}_\w}{\longrightarrow}}
\newcommand{\bigo}{\mathcal{O}}

\author{Jonathon Peterson}
\title{Limiting Distributions and Large Deviations for Random Walks in Random Environments}
\date{June 17, 2008}

\begin{document}









\thispagestyle{empty}
\begin{center}
 \Huge
Limiting Distributions and Large Deviations for Random Walks in Random Environments

\normalsize 
\vfill
A DISSERTATION\\ SUBMITTED TO THE FACULTY OF THE GRADUATE SCHOOL\\ OF THE UNIVERSITY OF MINNESOTA\\ BY

\vspace{1in}
\LARGE
Jonathon Robert Peterson
\normalsize

\vspace{1in}
IN PARTIAL FULFILLMENT OF THE REQUIREMENTS\\ FOR THE DEGREE OF\\ DOCTOR OF PHILOSOPHY

\vfill
Adviser: Ofer Zeitouni

\vspace{.5in}
July, 2008

\end{center}

\newpage
\thispagestyle{empty}
\hfill
\vspace{7in}
\begin{center}
 \copyright Jonathon Robert Peterson 2008
\end{center}

\newpage
\pagenumbering{roman}
\begin{center}
 \textbf{Acknowledgments}
\end{center}

First of all, I would like to thank God.
I dare not, in pride, pretend that any accomplishment I have achieved is of my own doing. 
My life has been blessed in so many ways by things that are beyond my control.
I know that God has guided my life until this point, and I trust that He will guide me safely home. 
Until that time, my intent is to honor Him by making the most of the abilities that He has given me.

\vspace{.1in}

I would like to thank my wife, Jana, for keeping my life balanced and for bringing me back to reality when I "lose track of time."  
As I finish this milestone and begin on a new journey, it is so wonderful to know that I will have your support, encouragement, and understanding along the way. 

\vspace{.1in}

I would also like to thank my adviser, Ofer Zeitouni, for teaching me so much over the past few years. I only hope that I have absorbed a small portion of your knowledge of probability theory. Thank you for guiding me to such a great thesis topic, and for all your help and advice throughout the past few years.

\newpage
\begin{center}
 \textbf{Dedication}
\end{center}

\noindent This thesis is dedicated to the teachers and professors who challenged me to do more:
\vspace{.2in}

\noindent \underline{Mr. Dan Halberg} -- high school teacher and math team coach. \\
He gave me my first math research problem (which I later realized was "discovering" the multinomial coeffiecients), and he helped me see that it was okay to think math was fun.
\vspace{.2in}

\noindent \underline{Dr. Stephen Ratliff} -- undergraduate physics and differential equations professor. \\
The first person to encourage me to consider going to graduate school.
\vspace{.2in}

\noindent \underline{Prof. Don Corliss} -- undergraduate abstract algebra professor. \\
He shocked me by writing me a letter that said he had no doubt I could obtain a Ph.D. in mathematics.  Without his strong encouragement and vote of confidence, I may never have achieved so much.



\tableofcontents


\begin{chapter}{Introduction}\label{Thesis_Introduction} \pagenumbering{arabic}

\begin{section}{RWRE: Notation and Terminology}
A simple random walk $\{X_n\}$ in $\Z^d$ is most easily described as the sum of i.i.d. $\Z^d$-valued random variables, $\xi_i$, where $X_0=0$ and $X_n=\xi_1+\cdots + \xi_n$. Alternatively, it can be described as a time-homogeneous Markov chain on $\Z^d$ with transition probabilities given by $P(X_{n+1}=x | X_n=y) = P(\xi_1 = x-y)$. 
While random walks have long been studied, a more recent area of research is random walks in random environments (RWRE). A RWRE consists of two parts: choosing an \emph{environment} according to a specified distribution, and then performing a random walk on that environment.

Specifically, let $\mathcal{M}(\Z^d)$ be the collection of all probability distributions on $\Z^d$. Then, we define an environment to be an element $\w = \{ \w(x,x+\cdot) \}_{x\in\Z^d} \in\mathcal{M}(\Z^d)^{\Z^d} =: \Omega$. 
$\mathcal{M}(\Z^d)$ with the weak topology is a Polish space, and thus $\Omega$ is a Polish space as well (since it is the countable product of Polish spaces). Let $P$ be a probability distribution on $(\Omega, \mathcal{F})$, where $\mathcal{F}$ is the $\s-$field generated by the cylinder sets of $\Omega$.
Given an environment $\w\in \Omega$, one can define a random walk in the environment $\w$ to be a time-homogeneous Markov chain on $\Z^d$ with transition probabilities given by
\[
P(X_{n+1}=x | X_n=y) = \w(y,x).
\]
Let $P_\w^x$ be the law of a random walk in environment $\w$ started at the point $X_0 = x$. For each $\w$, $P_\w^x$ is a probability distribution on the space of paths $\left((\Z^d)^\N,\mathcal{G}\right)$, where $\mathcal{G}$ is the $\s$-field generated by the cylinder sets of $(\Z^d)^{\N}$. Now, given any $A\in \mathcal{G}$, each $P_\w^x(A):(\Omega, \mathcal{F})\ra [0,1]$ is a measurable function of $\w$. Thus, we can define a probability measure $\P^x:= P\otimes P^x_\w$ on $(\Omega\times(\Z^d)^\N,\mathcal{F}\times\mathcal{G})$ by the formula
\[
\P^x(F\times G) := \int_F P_\w^x(G) P(d\w),\quad F\in\mathcal{F},G\in\mathcal{G}.
\]

Generally, the events that we are interested in concern only the path of the RWRE and not the specific environment chosen (i.e., events of the form $ \Omega\times G$). Thus, with a slight abuse of notation, $\P^x$ can also be used to denote the marginal on $(\Z^d)^\N$. Expectations under $P_\w^x$ and $\P^x$ will be denoted $E_\w^x$ and $\P^x$, respectively. Also, since generally the RWRE starts at the origin, $P_\w, E_\w, \P$ and $ \E$ will be understood to mean $P_\w^0, E_\w^0, \P^0$ and $ \E^0$, respectively. 

It is important to understand the different probability measures and the differences between them. Thus we give a quick review:
\begin{itemize}
\item $P$ is a probability measure on the space of environments.
\item For a \textit{fixed} environment $\w$, $P_\w^x$ is a
probability distribution of a random walk. However, for fixed $A\in
\mathcal{G}$, $P_\w^x(A)$ is a random variable. Statements
involving $P_\w^x$ are called \textit{quenched}, and since
$P_\w^x(A)$ is a random variable, a statement such as
$P_\w^x(A)=0$ is only true $P-a.s.$ 
\item $\P^x$ is the
probability of observing an event in the RWRE \textit{without}
first observing the environment. For $A\in\mathcal{G}$, $\P^x(A)$ is deterministic and
not a random variable. Probabilistic statements involving $\P^x$
are called \textit{annealed}. 
\item The random walk $\{X_n\}$ is a
Markov chain under the measure $P_\w^x$, but not under $\P^x$, and
it is stationary (in space) under $\P^x$ but not under $P_\w^x$.
\item The relationship between $P_\w^x$ and $\P^x$ is given by
$\P^x(A)=E_P(P_\w^x(A))$ for $A\in \mathcal{G}$. 
\end{itemize}
We end this section with a few further definitions of types of commonly studied RWRE. 
\begin{enumerate}
\item \textbf{Nearest neighbor:} A nearest neighbor RWRE is such that $\w(x,y)=0$ whenever $\|x-y\|_1 \neq 1$. 
\item \textbf{i.i.d. environment:} The collection of vectors $\w(x,x+\cdot)$ are independent and identically distributed under the distribution $P$. 
This assumption generally simplifies the analysis of RWRE because the independence of disjoint portions of the environment makes random walks restricted to disjoint subsets of $\Z^d$ independent.  
\item \textbf{Elliptic / uniformly elliptic:} A nearest neighbor RWRE is called \emph{elliptic} if $P(\w(0,e) > 0) = 1$ for all $\|e\|_1=1$, and \emph{uniformly elliptic} if there exists a $\kappa>0$ such that $P(\w(0,e) > \kappa ) = 1$ for all $\|e\|_1=1$. 
\end{enumerate}
\end{section}

\begin{section}{Structure of the Thesis}\label{Thesis_structure}
The thesis is divided into two major parts:

\noindent\textbf{Part I:} Chapters \ref{1dlimitingdist}-\ref{Thesis_AppendixBallistic} --- Limiting Distributions for RWRE on $\Z$.

Chapter \ref{1dlimitingdist} begins with a review of some of the standard results for RWRE on $\Z$, such as criteria for recurrence/transience and a law of large numbers. 
This review affords us the opportunity to introduce some of the notation and methods that will be used in later chapters. In particular, formulas for hitting probabilities and formulas for the expectation and variance of hitting times are all provided in Section \ref{1dprelim}. 

Section \ref{AnnealedLimits} is a review of known annealed limiting distribution results for transient RWRE on $\Z$. In contrast with random walks in constant environments, random walks in random environments do not always satisfy a central limit theorem. Theorem \ref{Thesis_annealedstable} is a classical result of Kesten, Kozlov, and Spitzer \cite{kksStable}, which classifies the annealed limiting distribution of a transient RWRE according to a parameter $s$ of the distribution $P$ on environments. If $s>2$, then a central limit theorem holds, but if $s<2$, the limiting distributions are related to a stable distribution of index $s$. In Section \ref{AnnealedLimits}, we give a brief overview of the different approaches used in proving variations of Theorem \ref{Thesis_annealedstable}. We give particular attention to the approach used by Enriquez, Sabot, and Zindy \cite{eszStable} in providing a new proof of Theorem \ref{Thesis_annealedstable} when $s<1$, since, in Chapters \ref{Thesis_AppendixZeroSpeed} and \ref{Thesis_AppendixBallistic}, we use similar methods to analyze the quenched limiting distributions. 

The main results of the first part of the thesis, concerning quenched limiting distributions for transient RWRE, are stated in Section \ref{mainresults}. When $s>2$, we obtain a quenched functional central limit theorem with a random (depending on the environment) centering. When $s<2$, however, there is no quenched limiting distribution for the RWRE. In fact, with probability one, there exist two different sequences (depending on the environment) along which different limiting distributions hold. In Section \ref{mainresults} we provide a sketch of these results on quenched limiting distributions, but the full proofs are given in Chapters \ref{Thesis_AppendixQCLT}-\ref{Thesis_AppendixBallistic}.

In Chapter \ref{Thesis_AppendixQCLT}, we give the full proof of the quenched functional central limit theorem when $s>2$. We first prove a quenched functional central limit for the hitting times of the random walk using the Lindberg-Feller condition for triangular arrays of random variables. Then, we transfer this result to a quenched functional central limit theorem for the random walk. The main difficulty in Chapter \ref{Thesis_AppendixQCLT} is to obtain a centering term for the random walk which only depends only on the environment. 

Chapters \ref{Thesis_AppendixZeroSpeed} and \ref{Thesis_AppendixBallistic} consist of two recent articles which contain the proofs of the quenched results for $s<2$ that were stated in Chapter \ref{1dlimitingdist}. In order to keep Chapters \ref{Thesis_AppendixZeroSpeed} and \ref{Thesis_AppendixBallistic} consistent and self-contained, these articles are left relatively unchanged from their original format. Thus, the introductory sections of Chapters \ref{Thesis_AppendixZeroSpeed} and \ref{Thesis_AppendixBallistic} repeat some of the material from Chapter \ref{1dlimitingdist}. 

Chapter \ref{Thesis_AppendixZeroSpeed} concerns the case $s<1$, which is the zero-speed regime (i.e., $\lim_{n\ra\infty} \frac{X_n}{n} = 0$). 
Our main result for $s<1$ is that, with probability one, there exist two different sequences $t_k$ and $t_k'$ (depending on the environment) along which the quenched limiting distributions of the random walk are different. 
Along the sequence $t_k$, the random walk is localized in an interval of size $(\log t_k)^2$, and along the sequence $t_k'$ the random walk has scaling of order $(t_k')^{s}$ (which is the annealed scaling in Theorem \ref{Thesis_annealedstable} when $s<1$).
 
In Chapter \ref{Thesis_AppendixBallistic}, we consider the case $s\in(1,2)$. In this regime, the random walk is \emph{ballistic}:  That is, $\lim_{n\ra\infty} \frac{X_n}{n} =: v_P > 0$. As in the case $s<1$, our main result in Chapter \ref{Thesis_AppendixBallistic} is that there exist two different sequences $t_k$ and $t_k'$ (depending on the environment) along which the quenched limiting distributions of the random walk are different. 
However, when $s\in(1,2)$, the existence of a positive speed for the random walk allows for a more precise description of the quenched limiting distributions along the sequences $t_k$ and $t_k'$. 
Along the sequence $t_k$, the limiting distribution is the negative of a centered exponential distribution, and along the sequence $t_k'$ the limiting distribution is Gaussian.

\noindent\textbf{Part II:} Chapter \ref{mdLDP} --- Large Deviations for RWRE on $\Z^d$.

After reviewing some of the basics of multidimensional RWRE in Section \ref{mdprelims}, in Section \ref{RWRELDP} we review the known large deviation results for RWRE. In particular, Theorems \ref{annealedLDP} and \ref{Thesis_zeroset} are large deviation results of Varadhan for multidimensional RWRE, but these results provide much less information about the quenched and annealed rate functions than is known for the rate functions of one-dimensional RWRE. In Section \ref{LDPnewresults}, we study properties of the annealed rate function $H(v)$. 
Our main result is that, when the distribution on environments $P$ is \emph{non-nestling}, the rate function is analytic in a neighborhood of the limiting velocity $v_P := \lim_{n\ra\infty} \frac{X_n}{n}$.  
Our strategy is to first define a function $\bar{J}(v)$ as a possible alternative formulation of $H(v)$. Then, we show that $\bar{J}(v)$ is analytic in a neighborhood of $v_P$ and that $\bar{J}(v)=H(v)$ in a neighborhood of $v_P$. 
We end Section \ref{LDPnewresults} by showing that when $d=1$, $H(v)=\bar{J}(v)$ wherever $\bar{J}(v)$ is defined. 

\end{section}

\end{chapter}

\begin{chapter}{Limiting Distributions for Transient RWRE on $\Z$}\label{1dlimitingdist}
\begin{section}{Preliminaries for RWRE on $\Z$}\label{1dprelim}
In this section, we will review some of the standard results for nearest neighbor RWRE on $\Z$. This will also serve as an introduction to some of the notation and techniques that will be used in proving our main results. In particular, the main results depend heavily on a few explicit formulas that we will derive in this section. 

For a nearest neighbor RWRE on $\Z$, $\w(x,x-1) = 1-\w(x,x+1)$, and so we can define an environment by only specifying the probability of moving to the right at each location. For ease of notation, let $\w_x:=\w(x,x+1)$ so that $1-\w_x=\w(x,x-1)$. Unless we specifically state that the environments are i.i.d., we will only be assuming that the distribution $P$ on $[0,1]^{\Z}$ is ergodic with respect to the spatial shift $(\theta \w)_n := \w_{n+1}$.

\begin{subsection}{Hitting Probabilities and Recurrence / Transience}
The feature of RWRE in one dimension that makes them much easier to analyze than in higher dimensions is the fact that, for any elliptic environment (i.e., for environments with $\w_i \in (0,1)$ for all $i\in \Z$), the random walk is a reversible Markov chain. In fact, any irreducible Markov chain on a tree is reversible. The fact that the quenched law of the RWRE is reversible allows us to represent certain quenched probabilities and expectations with explicit formulas in terms of the environment.
To make these formulas more compact, we introduce the following notation:
\begin{equation}
\rho_i := \frac{1-\w_i}{\w_i}, \qquad
\Pi_{i,j} := \prod_{k=i}^j \rho_k\,, \label{Thesis_rhodef}
\end{equation}
\begin{equation}
W_{i,j} := \sum_{k=i}^j \Pi_{k,j}, \qquad W_j := \sum_{k\leq j} \Pi_{k,j}\,, \label{Thesis_Wdef}
\end{equation}
and
\begin{equation}
R_{i,k} := \sum_{j=i}^k \Pi_{i,j}, \qquad R_i:=
\sum_{j=i}^\infty \Pi_{i,j}. \label{Thesis_Rdef}
\end{equation}
Using this notation, we have for any $i \leq x \leq j$ that 
\begin{equation} \label{Thesis_hittingprob}
P_\w^x( T_j < T_i ) = \frac{ R_{i,x-1} }{R_{i,j-1}}, \quad\text{and}\quad P_\w^x( T_i < T_j ) = \frac{ \Pi_{i,x-1}R_{x,j-1} }{R_{i,j-1}} \, ,
\end{equation}
where $T_j := \inf\{ n\geq 0: X_n = j \}$ is the hitting time of site $j$. 
These formulas also appear in \cite[formula (2.1.4)]{zRWRE}, but with different notation.
To see that \eqref{Thesis_hittingprob} holds, note that for any fixed $i<j$, letting $h(x) := P_\w^x( T_j < T_i )$, we have that $h(i)=0$, $h(j)=1$ and $h(x) = \w_x h(x+1) + (1-\w_x) h(x-1)$ for $i<x<j$. It is easy to check that the first formula in \eqref{Thesis_hittingprob} satisfies these relations and that this solution is unique (since any such $h(x)$ is a discrete harmonic function with prescribed boundary values). 

The following criterion for recurrence/transience follows from \eqref{Thesis_hittingprob}:
\begin{thm}[Solomon \cite{sRWRE}]\label{1drt}
$E_P(\log \rho_0)$ determines the recurrence/transience of
the RWRE:
\begin{enumerate}
\item $E_P(\log \rho_0)<0 \quad \Ra \quad \lim_{n\ra\infty} X_n =
+\infty, \quad \P-a.s. $ \item $E_P(\log \rho_0)>0 \quad \Ra \quad
\lim_{n\ra\infty} X_n = -\infty, \quad \P-a.s. $ \item $E_P(\log
\rho_0)=0 \quad \Ra \quad \liminf_{n\ra\infty} X_n =
-\infty, \: \limsup_{n\ra\infty} X_n =+\infty, \quad \P-a.s.$
\end{enumerate}
\end{thm}
\end{subsection}

\begin{subsection}{Recursions for Hitting Times and a Law of Large Numbers.}
For each $i\geq 1$, define
\[
 \tau_i := T_i-T_{i-1}
\]
to be the amount of time it takes for the random walk to reach $i$ after first reaching $i-1$. In this section, we will show how simple recursions allow us to compute an explicit formula (depending on the environment) for the quenched mean $E_\w \tau_i$. To this end, note that 
\begin{equation}
\tau_1 = 1 + \mathbf{1}_{X_1=-1}( \tau_0'+\tau_1'), \label{Thesis_taudecomposition}
\end{equation}
where $\tau_0'$ is the time it takes to reach $0$ after first hitting $-1$, and $\tau_1'$ is the time it takes to go from $0$ to $1$ after first hitting $-1$. Taking quenched expectations of both sides in \eqref{Thesis_taudecomposition} and using the strong Markov property, we have that 
\[
 E_\w \tau_1 = 1 + (1-\w_0)\left( E_{\w}^{-1} T_0 + E_\w^0 T_1 \right) = 1 + (1-\w_0)\left( E_{\theta^{-1} \w} \tau_1 + E_\w \tau_1 \right).
\]
Assuming for the moment that the environment is elliptic (i.e., $\w_i \in (0,1)$ for all $i$) and that $E_\w \tau_1 < \infty$ (which by ellipticity implies that $E_{\theta^{-1}\w} \tau_1 < \infty$ as well), we can solve the above equation for $E_\w \tau_1$ to get
\[
 E_\w \tau_1 = \frac{1}{\w_0} + \rho_0 E_{\theta^{-1}\w} \tau_1.
\]
Iterating this equation, we get that for any $m\geq1$,
\begin{equation}
 E_\w \tau_1 = \frac{1}{\w_0} + \frac{1}{\w_{-1}} \rho_0 +  \cdots + \frac{1}{\w_{-m}} \Pi_{-m+1,0} + \Pi_{-m,0} E_{\theta^{-m-1}\w} \tau_1. \label{Thesis_Etauiteration}
\end{equation}
From this it is not hard to see that 
\begin{equation}
E_\w \tau_1 = \bar{S}(\w):=\frac{1}{\w_0} + \sum_{i=1}^{\infty}
\frac{1}{\w_{-i}}\Pi_{-i+1,0} = 1+ 2W_0,  \label{Thesis_QET}
\end{equation}
where $W_0$ is defined in \eqref{Thesis_Wdef}. 
In fact, it can be shown that \eqref{Thesis_QET} holds even if $E_\w \tau_1 = \infty$ or if the environment is allowed to have $\w_i =1$ for some $i\leq 0$ (in which case the last term in \eqref{Thesis_Etauiteration} is eventually zero). 
We will omit the details of this argument since they can be found in \cite{zRWRE}, and since the details of a similar argument are provided in the computation of the quenched variance of $\tau_1$ in Appendix \ref{Thesis_qvarformula}.

If $\limsup_{n\ra\infty} X_n = +\infty$,
the ergodicity of the law $P$ on the environments implies that the sequence $\{\tau_i\}_{i=1}^\infty$ is ergodic under $\P$ (see \cite{sRWRE} or \cite[Lemma 2.1.10]{zRWRE}).
Then, Birkoff's ergodic theorem yields
\begin{equation}
\frac{T_n}{n}=\frac{1}{n}\sum_{i=1}^n \tau_i \limn \E \tau_1 =
E_P( \bar{S}(\w)) .  \label{Thesis_barSdef}
\end{equation}
Moreover, a standard argument changing the index from space to time shows that the 
convergence $\frac{T_n}{n}\limn \frac{1}{v}$ implies $ \frac{X_n}{n}\limn v$. Therefore, one obtains the
following theorem:
\begin{thm}[Solomon \cite{sRWRE}] Assume that $E_P \log \rho_0 < 0$. Then
\[ 
E_P(\bar{S}(\w))< \infty \Longrightarrow \lim_{n\ra\infty}
\frac{X_n}{n} = \frac{1}{E_P(\bar{S}(\w))},
\] 
and
\[
E_P(\bar{S}(\w)) = \infty \Longrightarrow \lim_{n\ra\infty}
\frac{X_n}{n} = 0.
\]
\end{thm}
For general ergodic distributions on environments, $E_P \S(\w)$ is difficult to calculate. However, if the environment is i.i.d, recalling the definition of $W_0$ in \eqref{Thesis_Wdef}, we have that
\[
E_P(\bar{S}(\w))= 1 + 2 E_P W_0 = 1 + 2 \sum_{k\leq 0} E_P \Pi_{k,0} =  1 + 2 \sum_{k=1}^\infty (E_P \rho_0)^k .
\]
Thus, if $P$ is i.i.d., the condition $E_P(\bar{S}(\w))<\infty$ is equivalent to $E_P \rho_0 < 1$.
We therefore obtain the following corollary:
\begin{cor}[Solomon \cite{sRWRE}]\label{Thesis_LLNiid} If $P$ is an i.i.d. product measure on $\Omega$ and $E_P \log \rho_0 < 0$, then 
\begin{align*}
&(a)\quad E_P(\rho_0) < 1  &\Longrightarrow\quad &\lim_{n\ra\infty} \frac{X_n}{n} = \frac {1-E_P(\rho_0)}{1+E_P(\rho_0)} > 0,&\quad \P-a.s.\\
&(b)\quad E_P(\rho_0) \geq  1  &\Longrightarrow\quad &\lim_{n\ra\infty} \frac{X_n}{n} = 0, &\quad \P-a.s.
\end{align*}
\end{cor}
For the remainder of the thesis we will denote $v_P:=\lim_n \frac{X_n}{n}$ whenever the limit exists and is constant $\P-a.s$.

Variances under the law $P_\w$ will be denoted by $Var_\w$. That is, $Var_\w \tau_1 := E_\w( \tau_1 - E_\w \tau_1 )^2$. 
In a manner similar to the derivation \eqref{Thesis_Etauiteration} of $E_\w \tau_1$, one obtains
\begin{align}
Var_\w\tau_1  = \S(\w)^2 - \S(\w) + 2 \sum_{n=1}^\infty \Pi_{-n+1,0}\S(\theta^{-n}\w)^2 = 4(W_{0}+ W_{0}^2)+ 8 \sum_{i<0}
\Pi_{i+1,0}(W_{i}+W_{i}^2).   \label{Thesis_qvar}
\end{align}
This formula also appears (with different notation) in \cite{aRWRE} and \cite{gQCLT}, but for completeness, we will provide a proof in Appendix \ref{Thesis_qvarformula}. 

\end{subsection}

\end{section}

\begin{section}{Review of Annealed Limit Laws for Transient RWRE on $\Z$} \label{AnnealedLimits}
In this section, we review known results on annealed limiting distributions for transient RWRE. This will also serve as an introduction to some of the techniques we will use later in deriving quenched limiting distributions. 
%
The following theorem of Kesten, Kozlov, and Spitzer was the first result on annealed limiting distributions of transient RWRE in $\Z$. 
\begin{thm}[Kesten, Kozlov, and Spitzer \cite{kksStable}] \label{Thesis_annealedstable}
Let $X_n$ be a nearest neighbor, one-dimensional RWRE with an i.i.d. measure $P$ on environments
such that $E_P \log \rho_0 < 0$. Further, assume that there exists an $s>0$ such that $E_P \rho_0^s =1$ and $E_P \rho_0^s \log \rho_0 < \infty$ and that the distribution of $\log \rho_0$ is non-lattice
(i.e., the support of $\log \rho_0$ is not contained in $\a + \b\Z$ for any $\a,\b\in \R$).
Then, there exists a constant $b>0$ such that
\begin{align}
\begin{array}{lllcl}
(a)\quad s\in(0,1) & \Ra & \lim_{n\ra\infty} \P\left( \frac{X_n}{n^{s}} \leq x \right) &= &1 - L_{s,b}(x^{-1/s}) \\
(b)\quad s\in(1,2) & \Ra &  \lim_{n\ra\infty} \P\left( \frac{ X_n - n v_P }{n^{1/s}} \leq x \right) &= &1 - L_{s,b}(-x) \\
(c)\quad s>2 & \Ra & \lim_{n\ra\infty} \P\left( \frac{ X_n - n v_P }{ b \sqrt{n} } \leq x \right) &= &\Phi(x) ,
\end{array} \label{Thesis_stablelimits}
\end{align}
where $L_{s,b}$ is the distribution function for the stable law of index $s$ with characteristic function
\[
\int e^{i t x} L_{s,b}(dx) = \exp\left\{ -b |t|^s \left( 1 - i \frac{t}{|t|} \tan(\pi s/2) \right)   \right\},
\]
and $\Phi(x)$ is the cumulative distribution function for a standard Gaussian distribution. 
\end{thm}
\noindent\textbf{Remarks:} \\
\textbf{1. } Annealed limiting distributions were also obtained in \cite{kksStable} for the borderline cases $s=1$ and $s=2$. For simplicity, we will not discuss those results since we will only obtain quenched results when 
$s\in (0,1)\cup (1,2)\cup (2,\infty)$. \\
\textbf{2. } The significance of the parameter $s$ is that $\E T_1^\gamma$ and $E_P(E_\w T_1)^\gamma$ are finite if $\gamma < s$. The fact that $E_P(E_\w T_1)^\gamma < \infty$ follows from the explicit formula for $E_\w T_1$ given in \eqref{Thesis_QET} and the fact that $E_P \rho^\gamma < 1$ for $\gamma < s$. The proof that $\E T_1^\gamma < \infty$ is more difficult and is based on a representation of $T_1$ as a branching process in a random environment (see \cite[Lemma 2.4]{dpzTE1D} for details). Also, note that $E_P \rho_0 < 1$ if and only if $s>1$. Therefore, from Corollary \ref{Thesis_LLNiid}, we have that $s\leq 1$ implies that $v_P = 0$ (the zero-speed regime) and $s>1$ implies $v_P > 0$ (the ballistic regime).

The approach of Kesten, Kozlov, and Spitzer was to first obtain annealed stable limit laws for the hitting times $T_n$ and then to transfer the results to $X_n$. 
For instance, the first line in \eqref{Thesis_stablelimits} follows from 
\[
 \quad s\in(0,1) \quad \Ra \quad \lim_{n\ra\infty} \P\left( \frac{T_n}{n^{1/s}} \leq x \right) =  L_{s,b}(x).
\]
The approach used in \cite{kksStable} to derive stable limit laws for $T_n$ was to relate $T_n$ to a branching process in a random environment, and then to prove stable limit laws for the related branching process. 
The same approach was used in \cite{mrzStable} to extend Theorem \ref{Thesis_annealedstable} to certain mixing environments that are generated by a Markov chain. 

Recently, Enriquez, Sabot, and Zindy \cite{eszStable} provided a new proof of part $(a)$ of Theorem \ref{Thesis_annealedstable} which allowed for a probabilistic representation of the constant $b$ (and in fact an exact calculation for $b$ when the environment is i.i.d. with Dirichlet distribution).
We will provide here a brief discussion of their techniques since we will use similar methods in analyzing the quenched distributions later\footnote{With the exception of the method for analyzing the quenched Laplace transform of the crossing time of a large block, our work and \cite{eszStable} were developed independently.}.
 Their approach differs from that of \cite{kksStable} in that they prove the annealed stable limit laws for $T_n$ by analyzing the \emph{potential} $V(x)$ of the environment as it was defined by Sinai in his analysis of recurrent RWRE \cite{sRRWRE}. That is,
\begin{equation}
V(x):= 
\begin{cases}
\sum_{i=0}^{x-1} \log \rho_i & \text{if } x \geq 1,\\
0                            & \text{if }  x = 0,\\
-\sum_{i=x}^{-1} \log \rho_i & \text{if } x \leq -1.
\end{cases} 
\end{equation}
Since $E_P \log \rho < 0$, $V(x)$ is decreasing ``on average''. However, there are sections of the environment (traps) where the potential is increasing (which means the random walk is more likely to move left than right). 
It turns out that the key to analyzing the hitting times $T_n$ is understanding the amount of time it takes to cross the longest sections of the environment where the potential is increasing. To this end, define the ``ladder locations'' $\nu_i$ of the environment by
\begin{align}
\nu_0 = 0, \quad\text{and}\quad \nu_i =
\inf\{n > \nu_{i-1}: V(n) < V(\nu_{i-1}) \} \quad\text{for all }  i \geq 1.
\label{Thesis_nudef1}
\end{align}
We will refer to the sections
of the environment between $\nu_{i-1}$ and $\nu_i - 1$ as the ``blocks'' of the environment.
The exponential height of a block is given by
\begin{equation}
M_k:=\max\{\Pi_{\nu_{k-1}, j} : \nu_{k-1}\leq j < \nu_k \} = \max\left\{ e^{V(j)-V(\nu_{k-1})}: \nu_{k-1} < j \leq \nu_k \right\}. \label{Thesis_Mdef}
\end{equation}
Note that the $M_k$ are i.i.d. since the environment is $i.i.d.$ 
Since $P$ is i.i.d. and $E_P \log \rho_0 < 0$, the potential $V$ is a random walk with negative drift. Thus, a result of Iglehart on excursions of random walks with negative drift \cite[Theorem 1]{iEV} implies that there exists a constant $C_5>0$ such that
\begin{equation}
Q(M_1 > x)\sim C_5 x^{-s}, \qquad \text{as } x\ra\infty, \label{Thesis_Mtail}
\end{equation}
where, as usual, $f(x)\sim g(x)$ as $x\ra\infty$ means that $\lim_{x\ra\infty} f(x)/g(x) = 1$. 
Therefore, the largest exponential height amongst the first $n$ blocks will be roughly of order $n^{1/s}$. 
Enriquez, Sabot, and Zindy show in \cite{eszStable} that in analyzing $T_n$, only the crossing times of the blocks in $[0,n]$ with $M_k \geq \frac{1-\e}{s}$ are relevant (the sum of the crossing times of all ``small blocks'' is $o(n)$). The limiting distribution for $T_n$ is then obtained by analyzing the annealed Laplace transform of the time to cross a ``large block.'' The analysis of the latter is accomplished in two steps: first by showing that the quenched Laplace transform is approximately the Laplace transform of an exponential random variable with a random (depending on the environment) parameter, and then by analyzing the tails of this random parameter. 
We will use this analysis of the crossing time of a large block later in our analysis of the quenched distribution of $T_n$. Corollary \ref{Thesis_explimit} contains a precise statement of our approximation of the quenched Laplace transform. 


There have been a number other approaches to proving an annealed central limit theorem (i.e., part (c) of Theorem \ref{Thesis_annealedstable}) under different assumptions, such as non-i.i.d. environments. Zeitouni \cite[Theorem 2.2.1]{zRWRE} gives an annealed central limit theorem for certain non-i.i.d. environments. 
Following an approach of Kozlov \cite{kARW} and Molchanov \cite{mLRM}, Zeitouni uses homogenization, i.e., the point of view of the particle, to first derive a quenched central limit theorem for the martingale $\overline{M}_n := X_n - n v_P + h(X_n, \w)$ where 
\[
h(x,\w)= 
\begin{cases}
\sum_{j=0}^{x-1} (v_P \S(\theta^j \w) - 1 ) & \text{if } x > 0 ,\\
0 & \text{if } x=0 ,\\
-\sum_{j=x}^{-1} (v_P \S(\theta^j \w) - 1 ) & \text{if } x < 0 .
\end{cases}
\]
An annealed CLT is then obtained by analyzing the fluctuations of the harmonic correction $h(X_n,\w)$. In particular, writing 
\begin{equation}
Z_n=Z_n(\w):= \sum_{j=1}^{\lfloor n v_P \rfloor} ( v_P \S(\theta^j \w) - 1 ), \label{Thesis_ZnDef}
\end{equation}
he shows that $\frac{1}{\sqrt{n}}( Z_n - h(X_n,\w))$ tends to zero in $\P$-probability, and that $Z_n$ satisfies a central limit theorem. Since $Z_n$ depends only on the environment, this can be combined with the quenched central limit theorem for the martingale $\overline{M}_n$ to derive an annealed central limit theorem for $X_n$ with deterministic centering $n v_P$. 

The argument in \cite{zRWRE} gives a quenched CLT for $\overline{M}_n$ in which $X_n$ is centered by a function of both the environment and the position of the random walk. One would like to replace $h(X_n,\w)$ by $Z_n(\w)$ to get a quenched CLT with random centering depending only on the environment. However, the argument of the proof in \cite{zRWRE} only shows that 
\begin{equation}
P_\w\left( \frac{X_n-nv_P +
Z_n}{\s_{P,1}\sqrt{n}} > x  \right) \limn \Phi(-x), \quad\text{in } P-\text{probability.} \label{Thesis_qCLTip}
\end{equation}
A second approach to proving an annealed CLT was given by Alili in \cite{aRWRE}. Alili's approach was to first use the Lindberg-Feller condition for triangular arrays to prove a quenched CLT for the hitting times $T_n$. However, in order to translate this result to $X_n$ Alili needed to make restrictive assumptions which essentially forced $Z_n(\w)$ to be bounded (which can only happen for certain non-i.i.d. environments) so that the quenched (and therefore annealed) central limit theorem holds with deterministic centering $n v_P$. 
In Section \ref{Thesis_quenchedCLT}, we extend this approach to prove a quenched central limit theorem (with random centering) for i.i.d. and strongly mixing environments. That is, we show that the convergence in \eqref{Thesis_qCLTip} holds for $P-$almost every environment $\w$. 

It should be noted that the random centering necessary for a quenched central limit theorem is unique to one-dimensional RWRE. 
Recent results by Berger and Zeitouni \cite{bzQCLT} and Rassoul-Agha and Sepp\"al\"ainen \cite{rsQCLT} show that, for RWRE in i.i.d. environments on $\Z^d$ with $d\geq 2$, if the random walk has non-zero limiting velocity (i.e. $v_P \neq 0$) and an annealed central limit theorem holds (and some other mild assumptions), a quenched central limit theorem also holds with the same (deterministic) centering. 

Limiting distributions have also been studied for RWRE on a strip $\Z \times \{1,2,\ldots, m\}$ which is a generalization of RWRE on $\Z$ with bounded jump size (identify elements $(x,i) \in \Z \times \{1,2,\ldots, m\}$ with $x m + i \in \Z$). Roitershtein \cite{rTRWS} has used homogenization methods to give sufficient conditions for an annealed central limit theorem for transient RWRE on the strip for environments with certain mixing properties. A recent result of Bolthausen and Goldsheid \cite{bgLRW} shows that recurrent RWRE on $\Z$ with bounded jump size either has scaling of order $(\log t)^2$ (as was shown by Sinai in the nearest neighbor case \cite{sRRWRE}) or satisfies a central limit theorem. The latter is shown to hold if and only if the random walk is a martingale (i.e. the environment has zero drift at each location). Also, Goldsheid \cite{gQCLT,gSCLT} has given quenched central limit theorems for RWRE on $\Z$ and on a strip\footnote{Goldsheid's results were obtained independently from ours below. However, while Goldsheid is able to prove a quenched central limit theorem for ergodic environments, we are restricted to strongly mixing environments but are able to prove a functional central limit theorem.}. 
\end{section}

\begin{section}{Quenched Limits for Transient RWRE on $\Z$} \label{mainresults}
In this section, we consider the quenched limiting distributions of transient RWRE on $\Z$. As the previous section showed, there are many results for annealed limiting distributions of transient RWRE on $\Z$. Until now, however, there have been very few results on quenched limiting distributions. 
Alili \cite{aRWRE} and Rassoul-Agha and Sepp\"al\"ainen \cite{rsBFD} have obtained quenched central limit theorems, but under assumptions on the environment which do not include the case of nearest neighbor RWRE on $\Z$ in i.i.d. environments. 
In this section, we will state our main results on quenched limits for transient nearest neighbor RWRE on $\Z$, and we will also give brief sketches of the proofs. The full proofs of the main results are contained in the Chapters \ref{Thesis_AppendixQCLT}-\ref{Thesis_AppendixBallistic}. 
Previously, no quenched limiting distribution results were known for nearest neighbor RWRE in i.i.d. environments\footnote{As mentioned above, Goldsheid \cite{gQCLT} has obtained a quenched central limit theorem similar to ours below, but this work was done independently and at the same time as our results.}.

Our analysis of the quenched limits for transient, nearest neighbor RWRE is divided into the three different cases that appear in Theorem \ref{Thesis_annealedstable}, depending on the value of the parameter $s$. These are respectively, the Gaussian regime ($s>2$), the ballistic, sub-Gaussian regime ($s\in (1,2)$), and the zero-speed regime ($s\in(0,1)$). The case $s>2$ is handled in Subsection \ref{Thesis_quenchedCLT}, while the cases $s\in(0,1)$ and $s\in(1,2)$ are handled in Subsection \ref{Thesis_quenchedsl2}.

\begin{subsection}{$s>2$: Quenched Central Limit Theorem} \label{Thesis_quenchedCLT}
In this section we will give an outline of the proof of a quenched functional central limit theorem for certain nearest neighbor, one-dimensional RWRE. The full proof is contained in Chapter \ref{Thesis_AppendixQCLT}.  
To prove a functional CLT for the RWRE we will make the following
assumptions:
\begin{asm} \label{Thesis_uelliptic}
The environment is uniformly elliptic. That is, $\exists \kappa >0$
such that $\w \in [\kappa, 1-\kappa]^\Z$, P-a.s.
\end{asm}
\begin{asm}
$E_P \log \rho_0 < 0$. That is, the RWRE is transient to the right. 
\end{asm}
\begin{asm}\label{Thesis_mixing} $P$ is $\a$-mixing, with $\a(n) = e^{-n\log(n)^{1+\eta}}$ for some $\eta>0$.
That is, for any $l$-separated measurable functions $f_1,
f_2 \in \{ f: \|f\|_\infty \leq 1 \}$, 
\[
E_P(f_1(\w)f_2(\w))\leq E_P(f_1(\w))E_P(f_2(\w)) + \a(l). 
\]
\end{asm}
Assumption \ref{Thesis_mixing} was also made in \cite[Section 2.4]{zRWRE} in the context of studying certain large deviations of one-dimensional RWRE. 
As noted in \cite[Section 2.4]{zRWRE}, the above assumptions imply that $\frac{1}{n} \sum_{i=0}^n \log \rho_i$ satisfies a large deviation principle with a rate function $J(x)$ (see \cite{bdLDSM}). 
%

For our final assumption, we wish to restrict our attention to the regime where there is an annealed CLT. When the environment is i.i.d., Theorem \ref{Thesis_annealedstable} shows that this is the case when $s>2$, where $s$ is the unique positive solution to $E_P \rho_0^s = 1$. Since we are not assuming i.i.d. environments, we need to define the parameter $s$ differently. 
\begin{asm} \label{Thesis_sg2}
$J(0)>0$ and $s:=\min_{y > 0} \frac{1}{y}J(y) > 2$, where $J(x)$ is the large deviation rate function for $\frac{1}{n} \sum_{i=0}^{n-1} \log \rho_i$.
\end{asm}
Note that Varadhan's Lemma \cite[Theorem 4.3.1]{dzLDTA} implies that the parameter $s$ defined in Assumption \ref{Thesis_sg2} is also the smallest non-negative solution of $\lim_{n\ra \infty} \frac{1}{n} \log E_P \Pi_{0,n-1}^s = 0$. Therefore, the above definition of $s$ is consistent with the previous definition of $s$ in the case of i.i.d. environments. 
Assumption \ref{Thesis_sg2} is the crucial assumption that we need for a central limit theorem, since it implies that
$\E\tau_1^\gamma < \infty$ for some $\gamma>2$. In fact, $\E\tau_1^\gamma<\infty $ for all $\gamma< s$ (see
\cite[Lemma 2.4.16]{zRWRE}). 

Let $D[0,\infty)$ be the space of real valued functions on $[0,\infty)$ which are right continuous and which have limits from the left, equipped with the Skorohod topology. 
Our main result in this Subsection is the following theorem:
\begin{thm}
 Assume that Assumptions \ref{Thesis_uelliptic}-\ref{Thesis_sg2} hold, and let 
\[
 B_t^{n} := \frac{X_{\lfloor nt \rfloor}-nt v_P+Z_{nt}(\w)}{v_P^{3/2}\s\sqrt{n}},
\]
where $\s^2 = \E \tau_1^2 - E_P \S(\w)^2 < \infty$ and $Z_{nt}$ is defined as in \eqref{Thesis_ZnDef}. Then, for $P-a.e.$ environment $\w$, the random variables $B^n_\cdot \in D[0,\infty)$ converge in quenched distribution as $n\ra\infty$ to a standard Brownian motion. 
\end{thm}
\begin{proof}[Sketch of proof]
Since the hitting times are the sum of independent (quenched) random variables, we can use the Lindberg-Feller condition to prove a quenched functional CLT for the hitting times. In particular, as elements of $D[0,\infty)$,
\[
\frac{1}{\s\sqrt{n}}\sum_{k=1}^{\nt} (\tau_k-E_\w\tau_k) \limdw W_t,
\]
where $\s^2 = E_P Var_\w \tau_1 = \E\tau_1^2 - E_P(\bar{S}(\w)^2)$, $W_t$ is standard Brownian motion, and $\limdw$ signifies convergence in distribution (in the space $D[0,\infty)$) as $n\ra\infty$ of the quenched law for $P-a.e.$ environment $\w$. (Note: although $\limdw$ signifies convergence in distributin of random functions of $t$, in the above and subsequent uses of $\limdw$ we will keep the index $t$ of the functions for clarity).
To transfer the CLT to the random walk, we first introduce the random variable $X^*_t:= \max_{k\leq t} X_k$ which is the farthest to the right the random walker has gone by time $t$. The mixing properties of $P$ and the fact that $X_n \ra +\infty$ with positive speed $v_P$, are enough to show that $X^*_n$ is very close to $X_n$ (in particular, eventually $X^*_n-X_n < \log^2(n)$ for all $n$ large enough). Then, a standard random time change argument ($t\mapsto \frac{X^*_{nt}}{n}$) implies that $\frac{1}{\s\sqrt{n}}\sum_{k=1}^{X^*_{nt}} (\tau_k-E_\w\tau_k)\limdw W_{v_P\cdot t}$. Next, the definition of $X^*_t$ implies that 
\[
\frac{1}{\s\sqrt{n}}\sum_{k=1}^{X^*_{nt}} (\tau_k-E_\w\tau_k) 
\leq \frac{1}{\s\sqrt{n}}\left(nt - \sum_{k=1}^{X^*_{nt}} E_\w\tau_k\right)  
\leq \frac{1}{\s\sqrt{n}}\sum_{k=1}^{X^*_{nt}} (\tau_k-E_\w\tau_k) + \frac{\tau_{X^*_{nt}+1}}{\s\sqrt{n}} \, .
\]
Therefore, since we can prove that $\frac{\tau_{X^*_{nt}+1}}{\sqrt{n}}$ is negligible, we obtain that
\begin{equation}
\frac{-1}{\s\sqrt{n}}\left(nt - \sum_{k=1}^{X^*_{nt}} E_\w\tau_k\right) 
= \frac{1}{v_P\s\sqrt{n}}\left(X^*_{nt} - ntv_P + \sum_{k=1}^{X^*_{nt}}(v_P E_\w\tau_k - 1)\right) 
\limdw W_{v_P\cdot t} \, .\label{Thesis_nuCLT}
\end{equation}
All that remains in order to obtain a quenched CLT for $X_n$ is to replace $X^*_{nt}$ by $X_{nt}$ (which is easy since the difference is of order $\log^2(n)$) and then replace the centering $\sum_{k=1}^{X^*_{nt}}(v_P E_\w\tau_k - 1)$ by $\sum_{k=1}^{ntv_P}(v_P E_\w\tau_k - 1) = Z_{nt}(\w)$ which depends only on the environment. (This is the same $Z_n$ as defined above in \eqref{Thesis_ZnDef}.) This replacement is the hardest part of the proof, and is accomplished by first proving that for $\a<1$,
\[
\max_{j,k\in[1,n];|k-j|<n^\a} \left|\frac{1}{\sqrt{n}} \sum_{i=j}^k (\tau_i - \frac{1}{v_P}) \right|\limn 0,\quad \P-a.s.
\]
and then, using this, to showing for any $\frac{1}{2}<\a$,  
\[
P_\w\left( \sup_{0\leq t \leq 1} |X^*_{nt} - ntv_P| \geq  n^\a\right) \limn 0, \quad P-a.s. 
\]
Finally, since
\begin{align*}
P_\w\left(\sup_{0\leq t\leq 1} \frac{1}{\sqrt{n}}\left| \sum_{k=1}^{X^*_{nt}}( E_\w\tau_k - \frac{1}{v_P})  \right.\right.- & \left.\left.\sum_{k=1}^{ntv_P}( E_\w\tau_k - \frac{1}{v_P}) \right| \geq \d\right) \\
& \leq P_\w\left( \sup_{0\leq t\leq 1} |X^*_{nt} - ntv_P| \geq n^\a\right) \\
&\quad\quad + P_\w\left(\max_{j,k\in[1,n];|k-j|<n^\a} \left|\frac{1}{\sqrt{n}} \sum_{i=j}^k (\tau_i - \frac{1}{v_P}) \right| \geq \frac{\d}{2} \right) \\
&\quad\quad + P_\w\left(\max_{j,k\in[1,n];|k-j|<n^\a} \left|\frac{1}{\sqrt{n}} \sum_{i=j}^k (\tau_i - E_\w\tau_i) \right| \geq \frac{\d}{2} \right) 
\end{align*}
for any $\a \in(\frac{1}{2},1)$, the above estimates imply that the first two terms on the right go to zero, and the quenched functional CLT for hitting times implies that the third term goes to zero also. Thus, all the replacements in \eqref{Thesis_nuCLT} discussed above are valid, and we get the quenched functional CLT for the random walk:
\[
\frac{X_{\lfloor nt \rfloor }-ntv_P + Z_{nt}}{v_P^{3/2}\s\sqrt{n}} \limdw W_t, \quad P-a.s.
\]
\end{proof}
\end{subsection}

\begin{subsection}{Quenched Limits when $s<2$}\label{Thesis_quenchedsl2}
In this section we will give an outline of our results on the quenched limiting distributions of transient nearest neighbor RWRE on $\Z$ in the annealed sub-Gaussian regime (i.e., $s<2$). The full proofs are contained in Chapters \ref{Thesis_AppendixZeroSpeed} and \ref{Thesis_AppendixBallistic}. 

For our main results in this section, we will
make the following assumptions:
\begin{asm} \label{Thesis_essentialasm}
$P$ is an i.i.d. product measure on $\Omega$ such that
\begin{equation}
E_P \log\rho < 0 \quad\text{and}\quad E_P \rho^s = 1 \text{ for
some } s>0 . \label{Thesis_zerospeedregime}
\end{equation}
\end{asm}
\begin{asm}
The distribution of $\log \rho$ is non-lattice under
$P$ and $E_P \rho^s \log\rho<\infty$.  \label{Thesis_techasm}
\end{asm}
Assumption \ref{Thesis_essentialasm}
contains the essential assumption necessary for the walk to be transient. 
Note that $E_P \rho^\gamma$ is a convex function of $\gamma$, and thus $E_P \rho^\gamma < 1$ for all $\gamma \in (0,s)$. Corollary \ref{Thesis_LLNiid} then gives that $s\leq 1$ if and only if $ v_P = 0$. 
Assumption \ref{Thesis_techasm} is a technical condition that was also invoked in \cite{kksStable} for the proof of the annealed limit laws and is used here to give that certain random variables have regularly varying tails. Our main results, however, seem to depend only on much rougher tail asymptotics. Thus, we suspect that in fact Assumption \ref{Thesis_techasm} is not needed for Theorems 
\ref{Thesis_local} - \ref{Thesis_qEXP}. 
However, Assumption \ref{Thesis_techasm} is probably necessary for Theorem \ref{Thesis_qETVarStable} which is interesting in its own right and which greatly simplifies the proofs of Theorems \ref{Thesis_local} - \ref{Thesis_qEXP}.

As was shown above, when $s>2$, the limiting distribution for $X_n$ is Gaussian in both the annealed and quenched cases (with a random centering in the quenched case). Therefore, when $s<2$, one could possibly expect the quenched limiting distributions to also be of the same type as in Theorem \ref{Thesis_annealedstable}. Somewhat surprisingly, this turns out not to be the case. In fact, when $s<2$, there are no quenched limiting distributions for $X_n$ (or for its hitting times $T_n$). Moreover, we are able to prove that for almost any environment $\w$ there exist two different random (depending on the environment) sequences along which different quenched limiting distributions hold. We divide our analysis of the quenched limiting distributions when $s<2$ into two subcases: $s\in (0,1)$ and $s\in(1,2)$. 

When $s\in(0,1)$ our main results are the following:
\begin{thm}\label{Thesis_local}
Let Assumptions \ref{Thesis_essentialasm} and \ref{Thesis_techasm} hold, and let $s<1$. Then,
$P$-a.s., there exist random subsequences $t_m=t_m(\w)$ and
$u_m=u_m(\w)$ such that, 
\[
\lim_{m\ra\infty} \frac{X_{t_m} - u_m}{(\log t_m)^2} = 0 , \qquad \mbox{in $P_\w$-probability.}
\]
\end{thm}
\begin{thm} \label{Thesis_nonlocal}
Let Assumptions \ref{Thesis_essentialasm} and \ref{Thesis_techasm} hold, and let $s<1$. Then,
$P$-a.s., there exists a random subsequence $n_{k_m}=n_{k_m}(\w)$
of $n_k=2^{2^k}$ and a random sequence $t_m=t_m(\w)$ such that,
\begin{equation}
\lim_{m\ra\infty} \frac{\log t_m}{\log n_{k_m} }= \frac{1}{s} \label{Thesis_scalingfactor}
\end{equation}
and
\[
 \lim_{m\ra\infty}
P_\w\left(\frac{X_{t_m}}{n_{k_m}} \leq x \right) =
\begin{cases}
0&\quad \mbox{if } x \leq 0, \\
\frac{1}{2}&\quad \mbox{if } 0<x<\infty.
\end{cases}
\]
\end{thm}
\noindent\textbf{Remarks:}\\
\textbf{1. } Theorem \ref{Thesis_local} is a strong localization result. Recall the definition of the ladder locations $\nu_i$ in \eqref{Thesis_nudef1}.  In the proof of Theorem \ref{Thesis_local}, we prove that, with probability tending to $1$, the distribution of the random walk at time $t_m$ is concentrated near a single block. Since the block lengths $\nu_{i}-\nu_{i-1}$ are i.i.d. with exponential tails, the longest of the first $n$ blocks is on the order of $\log n$.\\
\textbf{2. } Theorem \ref{Thesis_nonlocal} shows that the strong localization in Theorem \ref{Thesis_local} does not always occur. Note that \eqref{Thesis_scalingfactor} implies that the scaling is roughly of order $t_m^s$, which is what the annealed scaling is when $s\in(0,1)$ in Theorem \ref{Thesis_annealedstable}. 

We now state our main results in the case where $s\in(1,2)$. 
When $s\in(1,2)$, the existence of a positive speed for the random walk allows us to get a more straightforward description of two different limiting distributions along different random sequences. 
Let $\Phi(x)$ and $\Psi(x)$ be the distribution functions for a Gaussian and exponential random variable, respectively. That is, 
\[
\Phi(x):= \int_{-\infty}^x \frac{1}{\sqrt{2\pi}} e^{-y^2/2} dy \quad\text{and}\quad \Psi(x):= \begin{cases} 0 & x < 0, \\ 1-e^{-x} & x\geq 0. \end{cases} 
\]
\begin{thm}\label{Thesis_qCLT}
Let Assumptions \ref{Thesis_essentialasm} and \ref{Thesis_techasm} hold, and let $s\in(1,2)$. Then, $P-a.s.$, there exists a random subsequence $n_{k_m}=n_{k_m}(\w)$ of $n_k=2^{2^k}$ and non-deterministic random variables $v_{k_m,\w}$ such that 
\[
\lim_{m\ra\infty} P_\w\left( \frac{ T_{n_{k_m}} - E_\w T_{n_{k_m}} }{ \sqrt{v_{k_m,\w}} } \leq x \right) = \Phi(x), \qquad \forall x\in\R,
\]
and 
\[
\lim_{m\ra\infty} P_\w\left( \frac{ X_{t_m} - n_{k_m} }{v_P \sqrt{v_{k_m,\w}} } \leq x \right) = \Phi(x), \qquad \forall x\in\R,
\]
where $t_m=t_m(\w):= \left\lfloor E_\w T_{n_{k_m}} \right\rfloor$.
\end{thm}
\begin{thm}\label{Thesis_qEXP}
Let Assumptions \ref{Thesis_essentialasm} and \ref{Thesis_techasm} hold, and let $s\in(1,2)$. Then, $P-a.s.$, there exists a random subsequence $n_{k_m}=n_{k_m}(\w)$ of $n_k=2^{2^k}$ and non-deterministic random variables $v_{k_m,\w}$ such that 
\[
\lim_{m\ra\infty} P_\w\left( \frac{ T_{n_{k_m}} - E_\w T_{n_{k_m}} }{ \sqrt{v_{k_m,\w}} } \leq x \right) = \Psi(x+1), \qquad \forall x\in\R,
\]
and 
\[
\lim_{m\ra\infty} P_\w \left( \frac{X_{t_m} - n_{k_m}}{v_P \sqrt{v_{k_m,\w}} } \leq x \right) = 1-\Psi(-x+1), \qquad \forall x\in\R,
\] 
where $t_m=t_m(\w):= \left\lfloor E_\w T_{n_{k_m}} \right\rfloor$.
\end{thm}
\noindent\textbf{Remarks:}\\
\textbf{1.} The choice of Gaussian and exponential distributions in Theorems \ref{Thesis_qCLT} and \ref{Thesis_qEXP} represent the two extremes of the quenched limiting distributions that can be found along random subsequences. In fact, it will be shown in Corollary \ref{Thesis_explimit} that $T_n$ is approximately the sum of a finite number of exponential random variables with random (depending on the environment) parameters. 
The exponential limits in Theorem \ref{Thesis_qEXP} are obtained when one of the exponential random variables has a much larger parameter than all the others. 
The Gaussian limits in Theorem \ref{Thesis_qCLT} are obtained when the exponential random variables with the largest parameters all have roughly the same size. 
We expect, in fact, that any distribution which is the sum of (or limit of sums of) exponential random variables can be obtained as a quenched limiting distribution of $T_n$ along a random subsequence. \\
\textbf{2.} The sequence $n_k=2^{2^k}$ in Theorems \ref{Thesis_qCLT} and \ref{Thesis_qEXP} is chosen only for convenience. In fact, for any sequence $n_k$ growing sufficiently fast, $P-a.s.$ there will be a random subsequence $n_{k_m}(\w)$ such that the conclusions of Theorems \ref{Thesis_qCLT} and \ref{Thesis_qEXP} hold. \\
\textbf{3.} The definition of $v_{k_m,\w}$ is given below in \eqref{Thesis_dkvkdef}, 
and it can be shown in a manner similar to the proof of Theorem \ref{Thesis_qETVarStable} below that 
$\lim_{n\ra\infty} P\left( n_k^{-2/s} v_{k,\w} \leq x \right) = L_{\frac{s}{2},b}(x)$ for some $b>0$. 
Also, from \eqref{Thesis_barSdef} we have that $t_m \sim \E T_1 n_{k_m}$. Thus, the scaling in Theorems \ref{Thesis_qCLT} and \ref{Thesis_qEXP} is of the same order as the annealed scaling, but cannot be replaced by a deterministic scaling. 

Before turning to the proofs of Theorems \ref{Thesis_local} - \ref{Thesis_qEXP}, we need to introduce some notation and state some preliminary results that will be used the the proofs of Theorems \ref{Thesis_local} - \ref{Thesis_qEXP}. 
As was the case when $s>2$, we study the quenched distributions of the location of the random walk by first studying the quenched distributions for the hitting times. The hitting times are then studied by examining the crossing times of the blocks of the environment $T_{\nu_i} - T_{\nu_{i-1}}$, where the ladder locations $\nu_i$ are defined in \eqref{Thesis_nudef1}. We now introduce some more notation that will help us deal with a couple of difficulties that arise in the analysis of $T_{\nu_n}$. 

A major difficulty in analyzing $T_{\nu_n}$ is that the crossing time from $\nu_{i-1}$ to $\nu_i$ depends on the entire environment to the left of $\nu_i$. Thus $E_\w^{\nu_{i-1}} T_{\nu_i}$ and $E_\w^{\nu_{j-1}} T_{\nu_j}$ (and similarly $Var_\w (T_{\nu_i} - T_{\nu_{i-1}})$ and $Var_\w (T_{\nu_j} - T_{\nu_{j-1}})$) are not independent even if $|i-j|$ is large. However, it can be shown that  the RWRE generally will not backtrack too far (in fact, Lemma \ref{Thesis_nu} implies that $X_n^* - X_n = o(\log^2 n), \; \P-a.s.$). Thus, the dependence of $E_\w^{\nu_{i-1}} T_{\nu_i}$ and $E_\w^{\nu_{j-1}} T_{\nu_j}$ is quite weak when $|i-j|$ is large. (The explicit formulas for the quenched mean and variance of hitting times \eqref{Thesis_QET} and \eqref{Thesis_qvar} make this dependence precise.)
Thus, with minimal probabilistic cost, we can modify the environment of the RWRE to make crossing times of blocks that are far apart independent. For
$n=1,2,\ldots$, let $b_n:= \lfloor \log^2(n) \rfloor$.
Let $\bar{X}_t^{(n)}$ be the random walk that is the same as $X_t$
with the added condition that after reaching $\nu_k$ the
environment is modified by setting $\w_{\nu_{k-b_n}} = 1 $ , i.e.,
never allow the walk to backtrack more than $\log^2(n)$ ladder
times (that is, we deal with a dynamically changing environment). 
We couple $\bar{X}_t^{(n)}$ with the random walk $X_t$ in such a way that $\bar{X}_t^{(n)} \geq X_t$, with equality holding until
the first time $t$ when the walk $\bar{X}_t^{(n)}$ reaches a modified environment location. 
Denote by $\bar{T}_{x}^{(n)}$ the corresponding hitting
times for the walk $\bar{X}_t^{(n)}$.
It will be shown below in Chapter \ref{Thesis_AppendixZeroSpeed} that $\lim_{n\ra\infty} P_\w( T_{\nu_n} \neq \bar{T}_{\nu_n}^{(n)} ) = 0$, $P-a.s.$, so that the added reflections don't affect the quenched limiting distribution. 

A second difficulty is that, under $P$, the environment is not stationary under shifts by the ladder locations. However, if we define a new measure on environments by $Q(\cdot\:) = P(\cdot \: | V(x) > 0 , \: \forall x < 0)$, then under $Q$ the environment is stationary under those shifts. In particular, $\{ E_\w^{\nu_{i-1}} T_{\nu_i} \}_{i=1}^\infty$ and $\{ Var_\w (T_{\nu_i} - T_{\nu_{i-1}}) \}_{i=1}^\infty$ are stationary under $Q$. It should be noted that events only depending on the environment to the right of the origin have the same probability under $Q$ and $P$. In particular, if we let
\[
\mu_{i,n,\w} := E_\w^{\nu_{i-1}} \bar{T}^{(n)}_{\nu_i}, \quad\text{and}\quad \s_{i,n,\w}^2 := Var_\w \left( \bar{T}^{(n)}_{\nu_i} - \bar{T}^{(n)}_{\nu_{i-1}} \right),
\]
then $\mu_{i,n,\w}$ and $\s_{i,n,\w}^2$ have the same distribution under $P$ and $Q$ when $i > \log^2 n$.  

One of the main preliminary results that we obtain is the following annealed stable limit law:
\begin{thm}\label{Thesis_qETVarStable}
If $s<1$, there exists a constant $b'>0$ such that 
\[
\lim_{n\ra\infty} Q\left( \frac{ E_\w T_{\nu_n}}{n^{1/s}} \leq x \right) = L_{s,b'}(x).
\]
If $s<2$, then there exists a constant $b''>0$ such that 
\[
\lim_{n\ra\infty} Q\left( \frac{ Var_\w T_{\nu_n}}{n^{2/s}} \leq x \right) = L_{s/2,b''}(x).
\]
\end{thm}
\begin{proof}[Sketch of proof:]
We first derive the tail asymptotics of $E_\w T_\nu$ and $Var_\w T_\nu$ under $Q$. In particular, we prove that there exists a constant $K_\infty\in (0,\infty)$ such that 
\begin{equation}
\lim_{x\ra\infty} x^s Q( E_\w T_\nu > x ) = K_\infty,
\quad\text{and}\quad
\lim_{x\ra\infty} x^{s/2} Q( Var_\w T_\nu > x) = K_\infty. \label{Thesis_Qtailasym}
\end{equation}
The proof of the tail asymptotics of $E_\w T_\nu$ is similar to the proof of tail asymptotics in \cite{kksStable} and is based on the explicit formula $E_\w T_\nu = \nu + 2 \sum_{j=0}^{\nu-1} W_j$ from \eqref{Thesis_QET} and a result of Kesten \cite{kRDE} stating that there exists a constant $K$ such that $P(W_i > x) = P(R_i > x) \sim K x^{-s}$. The tail asymptotics of $Var_\w T_\nu$ are then derived by using the explicit formulas in \eqref{Thesis_QET} and \eqref{Thesis_qvar} to compare $Var_\w T_\nu$ to $(E_\w T_\nu)^2$. 

Now, if $\{E_\w^{\nu_{i-1}} T_{\nu_i} \}_{i=1}^\infty$ and $\{ Var_\w T_{\nu_i} - T_{\nu_{i-1}} \}_{i=1}^\infty$ were i.i.d. sequences, then \eqref{Thesis_Qtailasym} would be enough to prove the stable limit laws. Instead, we introduce some independence by adding reflections and restricting ourselves to large blocks. 
Recall the definition of $M_i$ in \eqref{Thesis_Mdef}. 
Then, for any $\e>0$, we may re-write
\begin{equation}
\frac{1}{n^{1/s}} E_\w T_{\nu_n} = \frac{1}{n^{1/s}} E_\w( T_{\nu_n} - \bar{T}_{\nu_n}^{(n)} ) + \frac{1}{n^{1/s}} \sum_{i=1}^n \mu_{i,n,\w}\mathbf{1}_{M_i \leq n^{(1-\e)/s}} + \frac{1}{n^{1/s}} \sum_{i=1}^n \mu_{i,n,\w}\mathbf{1}_{M_i > n^{(1-\e)/s}}. \label{Thesis_divideETnun}
\end{equation}
The explicit representation of $E_\w T_1$ in \eqref{Thesis_QET} can be used to show that $n^{-1/s}  E_\w( T_{\nu_n} - \bar{T}_{\nu_n}^{(n)} ) $ converges to zero in $Q$-probability. Also, \eqref{Thesis_QET} can be used to show that $\mu_{i,n,\w}$ cannot be too much larger than $M_i$, and then, since the $M_i$ are i.i.d., \eqref{Thesis_Mtail} can be used to approximate the number of $i\leq n$ with $M_i \in (n^{(1-\e')/s}, n^{(1-\e'')/s}]$ for any $\e',\e''>0$. Therefore, the second term on the right in \eqref{Thesis_divideETnun} also converges to zero in $Q$-probability. Finally, it can be shown that the tails of $E_\w T_\nu$ are not affected by the added reflections and restrictions to ``large blocks'' with $M_i > n^{(1-\e)/s}$. That is, we can show $\lim_{n\ra\infty} n Q( \mu_{i,n,\w} > x n^{1/s}, \; M_i \geq n^{(1-\e)/s} ) = K_\infty x^{-s}$. 
Then, \eqref{Thesis_Mtail} implies that the ``large blocks'' are far enough apart so that $\{ \mu_{i,n,\w}\mathbf{1}_{M_i > n^{(1-\e)/s}} \}_{i=1}^\infty$ is mixing enough to be able to apply a result of Kobus \cite{kGPD} to prove a stable limit law for the last term in \eqref{Thesis_divideETnun}.
\end{proof}

We now turn to a brief discussion of the proofs of our main results, Theorems \ref{Thesis_local} - \ref{Thesis_qEXP}. 

\begin{subsubsection}{Proofs of the Main Results when $s<1$}

\begin{proof}[\textbf{Sketch of proof of Theorem \ref{Thesis_local}:}]$\left.\right.$\\
The idea of the proof of Theorem \ref{Thesis_local} is to find a subsequence of the ladder locations $\nu_{j_m}$ such that the expected time to cross from $\nu_{j_m -1}$ to $\nu_{j_m}$ is much larger than the expected time to first reach $\nu_{j_m-1}$. From this, we can then find a sequence of times $t_m$ such that, with probability tending to one, $X_{t_m} \in[ \nu_{j_m-1} , \nu_{j_m} )$. The main result needed to find this subsequence is given by the following lemma:
\begin{lem}\label{Thesis_QBB}
Assume $s<1$. Then, for any $C>1$,
\[
\liminf_{n\ra\infty} Q\left( \exists k \in[1, n/2]: M_k \geq C \!\!\!\!\!\! \sum_{j\in[1,n]\backslash \{k\}} \!\!\!\!\!\! E_\w^{\nu_{j-1}} \bar{T}_{\nu_j}^{(n)}\right) > 0\,.
\]
\end{lem}
\begin{proof}[Sketch of proof:]
Note that, since $C>1$ and $E_\w^{\nu_{k-1}} \bar{T}^{(n)}_{\nu_k} \geq M_k$, there can only be at most one $k\leq n$ with $M_k \geq C \sum_{k\neq j \leq n} E_\w^{\nu_{j-1}} \bar{T}_{\nu_j}^{(n)}$. Therefore,
\begin{equation}
Q\left( \exists k\in[1,n/2]: M_k \geq  C \!\!\!\!\!\! \sum_{j\in[1,n]\backslash \{k\}} \!\!\!\!\!\! E_\w^{\nu_{j-1}} \bar{T}^{(n)}_{\nu_j} \right) = \sum_{k=1}^{n/2} Q\left( M_k \geq C \!\!\!\!\!\! \sum_{j\in[1,n]\backslash \{k\}} \!\!\!\!\!\! E_\w^{\nu_{j-1}} \bar{T}_{\nu_j}^{(n)}\right) . \label{Thesis_onebigblock}
\end{equation}
Now, $E_\w^{\nu_{j-1}} \bar{T}_{\nu_j}^{(n)}$ depends on the environment between $\nu_{k-1}$ and $\nu_k$ for $k < j \leq k+b_n$. However, it can be shown that $\sum_{j=k+1}^{k+b_n} E_\w^{\nu_{j-1}} \bar{T}_{\nu_j}^{(n)} = o(n^{-1/s})$ with $Q-$probability tending to one. Then, since $E_\w^{\nu_{j-1}} \bar{T}_{\nu_j}^{(n)}$ is independent of $M_k$ for all $j<k$ or $j> k+b_n$, 
\begin{align}
&Q\left( M_k \geq C \!\!\!\!\!\! \sum_{j\in[1,n]\backslash \{k\}} \!\!\!\!\!\! E_\w^{\nu_{j-1}} \bar{T}_{\nu_j}^{(n)}\right) \nonumber \\
&\qquad \geq \left( Q(M_k > C n^{1/s}) + o(1/n) \right) \left( Q\left(E_\w \bar{T}_{\nu_{k-1}}^{(n)} + E_\w^{\nu_{k+b_n}} \bar{T}_{\nu_n} \leq (1+o(1))n^{1/s}\right) + o(1) \right) \nonumber \\
&\qquad \geq \left( Q(M_k > C n^{1/s}) + o(1/n) \right) \left( Q\left(E_\w \bar{T}_{\nu_{n}}^{(n)}  \leq (1+o(1))n^{1/s}\right) + o(1) \right). \label{uniflbolb}
\end{align}
Now, \eqref{Thesis_Mtail} implies that $Q(M_k > C n^{1/s}) = P(M_1 > C n^{1/s}) \sim C_5 C^{-s} \frac{1}{n}$ as $n\ra\infty$, and Theorem \ref{Thesis_qETVarStable} implies that 
$\lim_{n\ra\infty} Q\left(E_\w \bar{T}_{\nu_{n}}^{(n)}  \leq (1+o(1))n^{1/s}\right) = L_{s,b'}(1)$. Therefore, \eqref{Thesis_onebigblock} and \eqref{uniflbolb} imply that
\[
\liminf_{n\ra\infty} Q\left( \exists k \in[1, n/2]: M_k \geq C \!\!\!\!\!\! \sum_{j\in[1,n]\backslash \{k\}} \!\!\!\!\!\! E_\w^{\nu_{j-1}} \bar{T}_{\nu_j}^{(n)}\right) \geq 
\frac{1}{2} C_5 C^{-s} L_{s,b'}(1) > 0. 
\]
\end{proof}
Lemma \ref{Thesis_QBB} can then be used to prove that, for $P-$almost every environment $\w$, there exists a sequence $j_m=j_m(\w)$ such that
$M_{j_m} \geq m^2 \mu_{j_m, j_m, \w}$.
Now, let $t_m=t_m(\w):= \frac{1}{m} M_{j_m}$ and $u_m=u_m(\w):= \nu_{j_m-1}$. Since $X_t^* - X_t = o(\log^2 t)$ and $\max_{i\leq t} \nu_i - \nu_{i-1} = o(\log^2 t)$, it is enough to show that 
\[
 \lim_{m\ra\infty} P_\w\left( X_{t_m}^* \in [ \nu_{j_m-1}, \nu_{j_m} ) \right) = 1. 
\]
However, 
\[
 P_\w\left( X_{t_m}^*  < \nu_{j_m-1} \right) = P_\w\left( T_{\nu_{j_m-1}} > t_m
\right) \leq P_\w\left( T_{\nu_{j_m-1}} \neq
\bar{T}_{\nu_{j_m-1}}^{(j_m)} \right) +
P_\w\left( \bar{T}_{\nu_{j_m-1}}^{(j_m)} > t_m \right) 
\,.
\]
The first term on the right tends to zero, and, by Chebychev's inequality and the definition of $t_m$, the second term is bounded above by 
\[
 \frac{1}{t_m} \mu_{j_m,j_m,\w} = \frac{m \mu_{j_m,j_m,\w}}{M_{j_m}} \leq \frac{1}{m} . 
\]
On the other hand
\begin{align*}
P_\w\left( X_{t_m}^* < \nu_{j_m} \right)
= P_\w( T_{\nu_{j_m}} > t_m )
\geq P_\w^{\nu_{j_m -1}}\left( T_{\nu_{j_m}} > \frac{1}{m} M_{j_m} \right)
\geq P_\w^{\nu_{j_m -1}} \left( T_{\nu_{j_m-1}}^+ <
T_{\nu_{j_m}} \right) ^{\frac{1}{m} M_{j_m} } ,
\end{align*}
where $T_x^+:= \min\{n > 0:  X_n =x \}$ is the first return time to $x$. Then, \eqref{Thesis_hittingprob} can be used to show that this last term is larger than
$\left(1-1/M_{j_m} \right)^{\frac{1}{m} M_{j_m}}$ which tends to $1$ as $m\ra\infty$. 
\end{proof}

\begin{proof}[\textbf{Sketch of proof of Theorem \ref{Thesis_nonlocal}:}]$\left. \right.$\\
Theorem \ref{Thesis_nonlocal} represents the opposite extreme of Theorem \ref{Thesis_local}. Therefore, in contrast to the proof of Theorem \ref{Thesis_local}, the key to proving Theorem \ref{Thesis_nonlocal} is to find sections of the environment where none of crossing times of a block is much larger than all the others. 

To this end, let
\[
\mathcal{S}_{\d,n,a} := \bigcup_{ \substack{ I\subset [1,\d n] \\ \#(I) = 2a } } \!\! \left( \bigcap_{i\in I} \left\{  \mu_{i,n,\w}^2 \in[n^{2/s},2n^{2/s}) \right\} \bigcap_{j\in[1,\d n]\backslash I} \left\{ \mu_{j,n,\w}^2 < n^{2/s} \right\} \right) 
\, ,
\]
and 
\begin{equation}
 U_{\d,n,c} := \left\{ \sum_{i=\d n +1}^{cn} \mu_{i,n,\w} \leq 2 n^{1/s} \right\}.  \label{Thesis_Udncdef}
\end{equation}
On the event $\mathcal{S}_{\d,n,a}$, $2a$ of the first $\d n$ blocks have roughly the same size crossing times and the rest are all smaller. On the event $\mathcal{S}_{\d,n,a} \cap U_{\d,n,c}$, we have additionally that the total expected crossing time from $\nu_{\d n}$ to $\nu_{cn}$ is smaller than the large expected crossing times in the first $\d n$ blocks. By Theorem \ref{Thesis_qETVarStable}, $U_{\d,n,c}$ is a typical event in the sense that $Q(U_{\d,n,c})$ tends to a non-zero constant as $n\ra\infty$. If the $\mu_{i,n,\w}$ were independent, an easy lower bound for $Q(\mathcal{S}_{\d,n,a})$ would be
\[
 \binom{\d n}{ 2 a } Q\left( \mu_{1,n,\w}^2 \in [n^{2/s},2n^{2/s}) \right)^{2a} Q\left( E_\w T_{\nu_{\d n}} \leq n^{1/s} \right) .
\]
(In Chapter \ref{Thesis_AppendixZeroSpeed}, we account for the dependence of the $\mu_{i,n,\w}$ to get a slightly different lower bound. However, the difference between the true lower bound and the lower bound given above is negligible for the purposes of our argument here.) 
Now $Q\left( E_\w T_{\nu_{\d n}} \leq n^{1/s} \right)$ has a non-zero limit as $n\ra\infty$ by Theorem \ref{Thesis_qETVarStable} and for $\d$ small and $a$ and $n$ large, we have that $\binom{\d n}{ 2 a } Q\left( \mu_{1,n,\w}^2 \in [n^{2/s},2n^{2/s}) \right)^{2a}$ is approximately 
\[
 \frac{(\d n)^{2a}}{ (2a)!} (K_\infty n^{-1} )^{2a} = \frac{(\d K_\infty)^{2a}}{(2a)!}. 
\]
This lower bound is good enough to ensure that events like $\mathcal{S}_{\d,n,a} \cap U_{\d,n,c}$ happen infinitely often along a sparse enough subsequence. 
The definitions of $\mathcal{S}_{\d,n,a}$ and $ U_{\d,n,c}$ imply that, along this subsequence, there are many large blocks whose expected crossing times are approximately the same, and all the other blocks have smaller expected crossing times. 
We then apply the Lindberg-Feller condition for triangular arrays to show that the limiting distribution of hitting times along this subsequence is Gaussian. In particular, let $n_k:=2^{2^k}$ and $d_k:=n_k-n_{k-1}$. Then, for $P-$almost every environment $\w$, there exists a random sequence $n_{k_m}=n_{k_m}(\w)$ and $\a_m < \b_m < \gamma_m$ with $\a_m = n_{k_m-1}$, $\b_m = o(n_{k_m})$ and $\lim_{m\ra\infty} \gamma_m / n_{k_m} = \infty$, such that for any $x_m \in[\nu_{\b_m}, \nu_{\gamma_m}]$, 
\[
\lim_{m\ra\infty} P_\w \left( \frac{T_{x_m} - E_\w T_{x_m}}{\sqrt{v_{m,\w}}} \leq y  \right) = \Phi(y), \quad\text{where}\quad v_{m,\w} := \sum_{i=\a_m + 1}^{\b_m} \mu_{i,d_{k_m},\w}^2 .
\]
Moreover, the subsequence is chosen so that $\lim_{m\ra\infty} v_{m,\w} / d_{k_m}^{2/s} = \infty$ and $\lim_{m\ra\infty} E_\w^{\nu_{\b_m}} T_{\nu_{\gamma_m}} / d_{k_m}^{1/s} \leq 2$. Finally, letting $t_m = t_m(\w) := E_\w T_{n_{k_m}}$, we have for any $x>0$, 
\begin{align*}
P_\w\left( \frac{X_{t_m}^*}{n_{k_m}} < x \right) 
= P_\w\left( T_{x n_{k_m}} > t_m \right)
= P_\w\left( \frac{T_{x n_{k_m}} - E_\w T_{x
n_{k_m}}}{\sqrt{v_{m,\w}}}
> \frac{ E_\w T_{n_{k_m}}-E_\w T_{x n_{k_m}}}{\sqrt{v_{m,\w}}} \right)
\,.\end{align*}
Then, since for all $m$ large enough $\nu_{\b_m} < n_{k_m} < x n_{k_m} < \nu_{\gamma_m}$,
\[
\frac{T_{x n_{k_m}} - E_\w T_{x n_{k_m}}}{\sqrt{v_{m,\w}}} \limdw Z\sim N(0,1) \quad\text{and}\quad \frac{ E_\w T_{n_{k_m}}-E_\w T_{x n_{k_m}}}{\sqrt{v_{m,\w}}} \leq \frac{E_\w^{\nu_{\b_m}} T_{\nu_{\gamma_m}}}{\sqrt{v_{m,\w}}} \underset{m\ra\infty}{\longrightarrow} 0.
\] 
Therefore, $\lim_{m\ra\infty} P_\w\left( \frac{X_{t_m}^*}{n_{k_m}} < x \right) = \frac{1}{2}$ for any $x>0$.
The proof of Theorem \ref{Thesis_nonlocal} is then finished by first showing that
$\lim_{m\ra\infty} \frac{\log t_m}{\log n_{k_m}} = \lim_{m\ra\infty} \frac{\log E_\w T_{n_{k_m}} }{\log n_{k_m}} = \frac{1}{s}$, 
and then recalling that $X_t^* - X_t = o(\log^2 t)$.
\end{proof}

\end{subsubsection}

\begin{subsubsection}{Proofs of the Main Results when $s\in(1,2)$}

It turns out to be much easier to transfer limiting distributions from $T_n$ to $X_n$ when $s>1$ than it was when $s<1$. 
This is due to the fact that, first, the walk moves with a linear speed $n v_P$, and, second, the fluctuations of the variance are of order $n^{1/s} = o(n)$. A key to proving Theorems \ref{Thesis_qCLT} and \ref{Thesis_qEXP} is the following proposition:
\begin{prop} \label{Thesis_generalprop}
Let Assumptions \ref{Thesis_essentialasm} and \ref{Thesis_techasm} hold, and let $s\in(1,2)$. Also, let $n_k$ be a sequence of integers growing fast enough so that $\lim_{k\ra\infty} \frac{n_k}{n_{k-1}^{1+\d}} = \infty$ for some $\d>0$, and let
\begin{equation}
d_k:= n_k-n_{k-1},\quad\text{and}\quad 
v_{k,\w} := 
\sum_{i=n_{k-1}+1}^{n_k} \s_{i,d_k,\w}^2 = 
Var_\w \left( \bar{T}^{(d_{k})}_{\nu_{n_k}} - \bar{T}^{(d_{k})}_{\nu_{n_{k-1}}} \right) \, . \label{Thesis_dkvkdef}
\end{equation}
Assume that $F$ is a continuous distribution function for which $P-a.s.$ there exists a subsequence $n_{k_m}= n_{k_m}(\w)$ such that, for $\a_m:= n_{k_m-1}$,
\[
\lim_{m\ra\infty} P_\w^{\nu_{\a_m}}\left( \frac{\bar{T}^{(d_{k_m})}_{x_m} - E_\w^{\nu_{\a_m}}\bar{T}^{(d_{k_m})}_{x_m} }{\sqrt{v_{k_m,\w}}} \leq y \right) = F(y), \quad \forall y\in \R,
\]
for any sequence $x_m \sim n_{k_m}$. Then, $P-a.s.$, for all $y\in \R$,
\begin{equation}
\lim_{m\ra\infty} P_\w\left( \frac{T_{x_m} - E_\w T_{x_m} }{\sqrt{v_{k_m,\w}}} \leq y \right) = F(y), \label{Thesis_Tlim}
\end{equation}
for any $x_m\sim n_{k_m}$, and 
\begin{equation}
\lim_{m\ra\infty} P_\w\left( \frac{X_{t_m} - n_{k_m} }{ v_P\sqrt{v_{k_m,\w}}} \leq y \right) = 1-F(-y),   \label{Thesis_Xlim}
\end{equation}
where $t_m:= \left\lfloor E_\w T_{n_{k_m}} \right\rfloor$. 
\end{prop}
\begin{proof}[Sketch of proof]
As mentioned previously,
there is not much difference between the distributions of $ \bar{T}_{x_m}^{(d_{k_m})}$ and $T_{x_m}$. 
In particular, we can show that 
\[
 \lim_{m\ra\infty} P_\w^{\nu_{\a_m}}\left( T_{x_m} \neq \bar{T}_{x_m}^{(d_{k_m})} \right) = 0 \quad\text{and}\quad \lim_{m\ra\infty} E_\w^{\nu_{\a_m}}\left( T_{x_m} - \bar{T}_{x_m}^{(d_{k_m})} \right) = 0, \quad P-a.s.
\]
Thus, to prove \eqref{Thesis_Tlim}, it is enough to show that 
\begin{equation}
 \lim_{m\ra\infty} P_\w \left( \left| \frac{ T_{\nu_{\a_m}} - E_\w T_{\nu_{\a_m}}}{\sqrt{v_{k_m,\w}}} \right| \geq \e \right) = 0, \quad P-a.s. \label{Thesis_startatam}
\end{equation}
However, $T_{\nu_{\a_m}} - E_\w T_{\nu_{\a_m}}$ is roughly of the order $\a_m^{1/s} \approx (E_P \nu_1) n_{k_m-1}^{1/s}$, whereas $\sqrt{v_{k_m,\w}}$ is roughly of the order $n_{k_m}^{1/s}$. The conditions on the rate of growth of $n_k$ are enough to show that \eqref{Thesis_startatam} holds. Also, note that the convergence in \eqref{Thesis_Tlim} must be uniform in $y$ since $F$ is continuous. 

Since $X_t^* - X_t = o(\log^2 t)$, it is enough to prove \eqref{Thesis_Xlim} for $X_{t_m}^*$ in place of $X_{t_m}$. For any $y \in \R$, let $ x_m(y):= \left \lceil n_{k_m}+ y \, v_P \sqrt{v_{k_m,\w}} \right \rceil$. 
Then,
\begin{align}
P_\w \left( \frac{X_{t_m}^* - n_{k_m}}{v_P \sqrt{v_{k_m,\w}} } < y \right) &= P_\w \left( X_{t_m}^* < x_m(y) \right) \nonumber = P_\w \left( T_{ x_m(y) } > t_m \right) \nonumber \\
&= P_\w \left( \frac{ T_{ x_m(y) } - E_\w T_{ x_m(y) } }{ \sqrt{v_{k_m,\w}} } > \frac{t_m - E_\w T_{  x_m(y)  } }{\sqrt{v_{k_m,\w}}} \right)  \label{Thesis_timespace}
\end{align}
Since the scaling $\sqrt{v_{k_m,\w}}$ is roughly of the order $n_{k_m}^{1/s} = o(n_{k_m})$, we have that $x_m(y) \sim n_{k_m}$. Therefore, recalling that the convergence in \eqref{Thesis_Tlim} is uniform in $y$, it is enough to show that 
\begin{equation}
 \lim_{m\ra\infty} \frac{t_m - E_\w T_{  x_m(y)  } }{\sqrt{v_{k_m,\w}}} = -y. \label{Thesis_avgofET}
\end{equation}
Assuming that $y>0$ (a similar argument works for $y<0$), we may re-write 
\[
 \frac{t_m - E_\w T_{  x_m(y)  } }{\sqrt{v_{k_m,\w}}} = \frac{-1}{\sqrt{v_{k_m,\w}}} \sum_{i=n_{k_m}+1}^{ n_{k_m}+ y \, v_P \sqrt{v_{k_m,\w}} } E_\w^{i-1} T_i.
\]
Since $s>1$, this should be close to $-y v_P \E T_1 = -y$. In fact, it can be shown that the sequence $n_k$ grows fast enough to ensure that \eqref{Thesis_avgofET} holds for any $y\in\R$. 
\end{proof}

\begin{proof}[\textbf{Sketch of proof of Theorem \ref{Thesis_qCLT}:}]$\left.\right.$\\
As in the proof of Theorem \ref{Thesis_nonlocal}, the key is to first find a random sequence along which the hitting times have Gaussian limiting distribution. The sequence can be chosen in such a way so that Proposition \ref{Thesis_generalprop} can be used to give Gaussian limits for the random walk along a random subsequence. The proof of the existence of Gaussian limits for hitting times is almost identical to its analogue in the proof of Theorem \ref{Thesis_nonlocal}. The main difference is that, instead of using the set $U_{\d,n,c}$ from \eqref{Thesis_Udncdef}, we instead use
\[
 U_{\d,n} := \left\{ \sum_{i= \lfloor \eta n \rfloor + 1}^n \s_{i,n,\w}^2 < 2 n^{2/s} \right\}. 
\]
\end{proof}

\begin{proof}[\textbf{Sketch of proof of Theorem \ref{Thesis_qEXP}:}]$\left.\right.$\\
First, we need to show that the crossing time of a large block is approximately exponentially distributed. 
As mentioned above, we follow an idea from \cite{eszStable} in computing the quenched Laplace transform of $\bar{T}^{(n)}_\nu$. The strategy is to decompose $\bar{T}^{(n)}_\nu$ into a series of excursions away from 0. An excursion is considered a ``failure'' if the random walk returns to zero before hitting $\nu$ (i.e., if $T_\nu > T_0^+:= \min\{ k > 0: X_k = 0 \}$) and a ``success'' if the random walk reaches $\nu$ before returning to zero. Let $p_\w:=P_\w ( T_\nu<T^+_0)$, and let $N$ be a geometric random variable with parameter $p_\w$ (i.e., $P(N=k) = p_\w (1-p_\w)^k$ for $k\in \N$). Also, let $\{F_i\}_{i=1}^\infty$ be an i.i.d. sequence (also independent of $N$), with $F_1$ having the same distribution as $\bar{T}_\nu^{(n)}$ conditioned on $\left\{ \bar{T}^{(n)}_\nu > T_0^+ \right\}$, and let $S$ be a random variable with the same distribution as $T_\nu$ conditioned on $\left\{ \bar{T}^{(n)}_\nu < T_0^+ \right\}$ and independent of everything else. Thus, 
\begin{equation}
\bar{T}^{(n)}_\nu \stackrel{Law}{=} S + \sum_{i=1}^N F_i \qquad\text{(quenched).} \label{Thesis_Tdec}
\end{equation}
Consequently,
\begin{align*}
E_\w e^{-\l \bar{T}^{(n)}_\nu} = E_\w e^{-\l S} E_\w\left[ \left(E_\w e^{-\l F_1}\right)^N \right] &= E_\w e^{-\l S} \frac{p_\w}{1-(1-p_\w)\left(E_\w e^{-\l F_1}\right)} \\
&= E_\w e^{-\l S} \frac{1}{1+E_\w N (1- E_\w e^{-\l F_1} )} , \quad \forall \l\geq 0,
\end{align*}
where, in the last equality, we used $E_w N = \frac{1-p_\w}{p_\w}$. Therefore, since $1-x \leq e^{-x} \leq 1 \wedge (1-x+\frac{x^2}{2})$ for any $x\geq 0$,
\[
 \frac{1-\l E_\w S}{1+\l (E_\w N)(E_\w F_1)} \leq E_\w e^{-\l \bar{T}^{(n)}_\nu} \leq \frac{1}{1 + \l (E_\w N)( E_\w F_1)  + \frac{\l^2}{2} (E_\w N)(E_\w F_1^2 )}.
\]
Then, replacing $\l$ by $\frac{\l}{\mu_{1,n,\w}}$ and noting that $\mu_{1,n,\w} = (E_\w N)(E_\w F_1) + E_\w S$, 
\[
 \frac{1-\l \frac{ E_\w S}{\mu_{1,n,\w} }}{1+\l \left( 1 - \frac{E_\w S}{\mu_{1,n,\w}} \right)} \leq E_\w e^{-\l \frac{\bar{T}^{(n)}_\nu}{\mu_{1,n,\w} }} \leq \frac{1}{1 + \l \left( 1 - \frac{E_\w S}{\mu_{1,n,\w}} \right)  + \frac{\l^2}{2} (E_\w N)(E_\w F_1^2 )}. 
\]
Now, the failures and excursions $F_1$ and $S$ can be represented as random walks in certain modified environments, and therefore we can use the formulas \eqref{Thesis_QET} and \eqref{Thesis_qvar} (which hold for \emph{any} environment) to get bounds on $E_\w S$ and $E_\w F_1^2$ when $M_1$ is large. 
Thus, we can show that, with probability close to one, $E_\w e^{-\l \frac{\bar{T}^{(n)}_\nu}{\mu_{1,n,\w} }}$ is approximately $\frac{1}{1+\l}$ when $M_1$ is large. In particular, letting $\phi_{i,n}(\l):= E_\w^{\nu_{i-1}} \exp \left\{ -\l \frac{\bar{T}_{\nu_i}^{(n)}}{\mu_{i,n,\w}} \right\}$ be the scaled, quenched Laplace transforms, we are able to show:
\begin{lem} \label{Thesis_mgflem}
Assume $\e<\frac{1}{8}$, and let $\e':= \frac{1-8\e}{5} > 0$. Then,
\[
Q\left( \exists \l \geq 0: \phi_{1,n}(\l) \notin \left[ \frac{1-\l n^{-\e /s}}{1+\l} , \frac{1}{1+\l-\left(\l+\frac{3 \l^2}{2}\right)n^{-\e/s} } \right],\: M_1 > n^{(1-\e)/s} \right) = o\left( n^{-1-\e'} \right). 
\]
\end{lem}
\begin{cor} \label{Thesis_explimit}
Assume $\e<\frac{1}{8}$, and let $n_k:=2^{2^k}$. Then, $P-a.s.$, for any sequence $i_k=i_k(\w)$ such that $i_k\in (n_{k-1},n_k]$ and $M_{i_k} > d_k^{(1-\e)/s}$, we have
\begin{equation}
\lim_{k\ra\infty} \phi_{i_k,d_k}(\l) = \frac{1}{1+\l},\quad \forall \l\geq 0, \label{Thesis_mgfbounds}
\end{equation}
and thus
\begin{equation}
\lim_{k\ra\infty} P_\w^{\nu_{i_k-1}} \left( \bar{T}^{(d_k)}_{\nu_{i_k}} > x \mu_{i_k,d_k,\w} \right) =  \Psi(x), \quad \forall x\in \R. \label{Thesis_expl}
\end{equation}
\end{cor}

Assuming Corollary \ref{Thesis_explimit}, we can then complete the proof of Theorem \ref{Thesis_qEXP}. In a manner similar to the proof of Theorem \ref{Thesis_local}, we find random subsequences $n_{k_m}=n_{k_m}(\w)$ and $i_m = i_m(\w) \in( n_{k_m-1}, n_{k_m}]$, such that the time to cross the first $n_{k_m}$ blocks is dominated by $\bar{T}^{(d_{k_m})}_{\nu_{i_m}} - \bar{T}^{(d_{k_m})}_{\nu_{i_m-1}}$, which by Corollary \ref{Thesis_explimit} is approximately exponentially distributed. The proof of Theorem \ref{Thesis_qEXP} is then completed by an application of Proposition \ref{Thesis_generalprop}. 

\end{proof}

\end{subsubsection}
\end{subsection}

\end{section}


\end{chapter}

\begin{chapter}{Quenched Functional CLT} \label{Thesis_AppendixQCLT}
In this chapter, we provide a full proof of the quenched functional central limit theorem (CLT) stated in Chapter \ref{1dlimitingdist}. 
To keep the chapter self-contained, we repeat the assumptions that were stated in Subsection \ref{Thesis_quenchedCLT}:
\begin{asm} \label{uelliptic}
The environment is uniformly elliptic. That is, $\exists \kappa >0$
such that $\w \in [\kappa, 1-\kappa]^\Z$, P-a.s.
\end{asm}
\begin{asm}
$E_P(\bar{S}(\w))<\infty$. Thus, the random walk is transient to
the right with positive speed $v_P:= \lim_{n\ra\infty}
\frac{X_n}{n} = \frac{1}{E_P(\bar{S}(\w))} > 0$.
\end{asm}
\begin{asm}\label{mixing} $P$ is $\a$-mixing, with $\a(n) = e^{-n\log(n)^{1+\eta}}$ for some $\eta>0$.
That is, for any $l$-separated measurable functions $f_1,
f_2 \in \{ f: \|f\|_\infty \leq 1 \}$, 
\[
E_P(f_1(\w)f_2(\w))\leq E_P(f_1(\w))E_P(f_2(\w)) + \a(l). 
\]
\end{asm}
As noted in \cite[Section 2.4]{zRWRE}, the above assumptions imply that $\frac{1}{n} \sum_{i=0}^n \log \rho_i$ satisfies a large deviation principle with a good rate function $J(x)$. 
(Recall that a non-negative function $J(x)$ is a \emph{good rate function} if $J(x)$ is lower semi-continuous and 
$\{x:J(x)\leq M\}$ is compact for all $M$.)
The final critical assumption is then
\begin{asm} \label{sg2}
$J(0)>0$ and $s:=\min_{y > 0} \frac{1}{y}J(y) > 2$, where $J(x)$ is the large deviation rate function for $\frac{1}{n} \sum_{i=0}^{n-1} \log \rho_i$.
\end{asm}
Recall that when $P$ is i.i.d., the parameter $s$ can also be
defined as the smallest positive solution to $E_P\rho_0^s = 1$ (as in Theorem \ref{Thesis_annealedstable}). 
Assumption \ref{sg2} is the crucial assumption needed for a central limit theorem, since it implies that
$\E\tau_1^\gamma < \infty$ for some $\gamma>2$. 
In fact,
that $\E\tau_1^\gamma<\infty $ for all $\gamma< s$ (see
\cite[Lemma 2.4.16]{zRWRE}). Since we will use this repeatedly, we
fix such a $\gamma \in (2,s)$ for the remainder of the Chapter.

\begin{section}{Quenched CLT for Hitting Times}
The first step in proving a quenched functional CLT for the RWRE is to prove
a quenched functional CLT for the hitting times. 
Recall that $D[0,\infty)$ is the space of real valued functions on $[0,\infty)$ which are right continuous and which have limits from the left, equipped with the Skorohod topology. 
For any environment $\w$, let $Z^n_\cdot \in D[0,\infty)$ be defined by 
\[
 Z_t^n:=\frac{1}{\s\sqrt{n}}\sum_{i=1}^{\nt} (\tau_i-E_\w \tau_i)=\frac{T_{\nt}-E_\w T_{\nt}}{\s\sqrt{n}},
\]
where $\s^2=\E(\tau_1^2)-E_P\left(\bar{S}(\w)^2\right)$.
\begin{thm}\label{Thesis_HTCLT}
The hitting times $T_n$ satisfy a quenched functional CLT. That is, for $P-a.e.$ environment $\w$, the random variables $Z^n_\cdot \in D[0,\infty)$ converge in quenched distribution as $n\ra\infty$ to a standard Brownian motion. 
\end{thm}
\begin{proof}
Alili proves a quenched CLT for the hitting times $T_n$ in \cite[Theorem 5.1]{aRWRE}. The proof here is a minor modification of Alili's proof that implies a functional CLT. 
First, note that by the remarks after Assumption \ref{sg2}, $\s^2 < \infty$. 
Then, a version of the Lindberg-Feller condition for triangular arrays of random functions \cite[Theorem 18.2]{bCOPM} implies that it is enough to show the following:
\begin{equation}\label{Thesis_CLT1}
\lim_{n\ra\infty} \sup_{0\leq t \leq T} \left( \frac{1}{n} \sum_{k=1}^{\nt} E_\w\left(\tau_k-E_\w\tau_k\right)^2 - \s^2 t \right) = 0,\quad \forall T<\infty,\quad  P-a.s.,
\end{equation}
and
\begin{equation}\label{Thesis_CLT2}
\lim_{n\ra\infty} \frac{1}{n} \sum_{k=1}^{\lfloor n T \rfloor} E_\w\left((\tau_k-E_\w\tau_k)^2\mathbf{1}_{\{|\tau_k-E_\w\tau_k|>\e\sqrt{n}\}}\right)=0,\quad \forall T<\infty,\quad P-a.s.
\end{equation}
The proof of \eqref{Thesis_CLT2} can be found in the proof of Theorem 5.1 in \cite{aRWRE} and depends on the ergodic theorem and the fact that $E_P\left[ E_\w (\tau_1 - E_\w \tau_1)^2 \right] = \s^2 < \infty$.
To prove \eqref{Thesis_CLT1} we re-write
$E_\w\left(\tau_k-E_\w(\tau_k)\right)^2=E_\w\tau_k^2-(E_\w \tau_k)^2  = E_{\theta^{k-1}\w}\tau_1^2-(E_{\theta^{k-1}\w} \tau_1)^2 $.
Then, since $P$ is an ergodic distribution on environments, we have that for any $t$,
\begin{align*}
\lim_{n\ra\infty} \frac{1}{n} \sum_{k=1}^{\nt} \left(E_{\theta^{k-1}\w}\tau_1^2-(E_{\theta^{k-1}\w} \tau_1)^2 \right)
= E_P(E_\w\tau_1^2 - (E_\w\tau_1)^2) t 
=\s^2 t,\quad P-a.s.
\end{align*}
Thus, $\frac{1}{n}\sum_{k=1}^{\nt} E_\w(\tau_k-E_\w\tau_k)^2$ converges pointwise to $\s^2 t$. However, since both functions are monotone in $t$ and the limit function is continuous, convergence is therefore uniform on compact intervals. 
Thus, we have finished the proof of \eqref{Thesis_CLT1} and, therefore, the proof of the theorem.
\end{proof}
\end{section}

\begin{section}{A Random Time Change}
In this section we will use a random time change argument to convert the
quenched CLT for the hitting times into one for the position of
the RWRE. We begin with a few definitions, with $\s$ defined as in Theorem
\ref{Thesis_HTCLT}:
\begin{align*}
&X^*_t:=\max \{X_n:n\leq t\}=\max\left\{l:\sum_{i=1}^l \tau_i \leq t\right\}\\
&Y_t^n:=\frac{1}{\s\sqrt{n}}\sum_{i=1}^{X^*_{tn}}(\tau_i-E_\w\tau_i)\\
&R_t^n:=\frac{1}{\s\sqrt{n}} \left(nt-\sum_{i=1}^{X^*_{nt}} E_\w \tau_i\right)
\end{align*}
The following lemma shows that we do not lose much by working with
$X^*_n$ instead of $X_n$:
\begin{lem} \label{Thesis_nu}
For all $\d>0$, $\P\left( \sup_{0\leq t \leq 1} X^*_{nt}-X_{\lfloor n t \rfloor}  \geq \d \log^2(n)\quad\text{i.o.} \right)=0$.
\end{lem}
\begin{proof}
First, note that the formulas for hitting times \eqref{Thesis_hittingprob} imply that
\[
 P_\w(T_{-M}< \infty)
= \frac{\Pi_{-M,-1}R_0}{R_{-M}} = \frac{ \Pi_{-M+1,0}(1+R_1)}{1+R_{-M+1}} \leq \Pi_{-M+1,0} (1+ R_1) = \sum_{j=0}^\infty \Pi_{-M+1,j}.
\]
Therefore, by the shift invariance of $P$,
\begin{equation}
\P(T_{-M}< \infty) \leq \sum_{k=M-1}^\infty E_P(\Pi_{0,k}) .\label{productpi}
\end{equation}
Now, since $\rho_k$ is bounded (by Assumption \ref{uelliptic}), and since $J$ is a good rate function we may apply Varadhan's
Lemma \cite[Lemma 4.3.6]{dzLDTA} to get that 
\[
 \lim_{k\ra\infty}
\frac{1}{k} \log E_P \Pi_{0,k-1} = \lim_{k\ra\infty}
\frac{1}{k} \log E_P e^{k \left(\frac{1}{k} \sum_{i=0}^{k-1} \log \rho_i  \right)} = \sup_x (x - J(x))<0,
\]
where the last inequality is due to Assumption \ref{sg2} and the fact that $J(x)$ is non-negative and lower semi-continuous. 
Thus, there exists an $n_0$ such that $ E_P(\Pi_{0,k-1})\leq e^{\frac{k}{2}\sup_x (x-J(x))}$, for all  $n\geq n_0$.
Then, \eqref{productpi} implies that there exists a constant $\d_1>0$ such that $\P(T_{-M} < \infty) < e^{-\d_1 M }$ for all $M$ large enough. 
Therefore, for all $n$ large enough,
\begin{align*}
\P\left( \sup_{0\leq t \leq 1} X^*_{nt}-X_{\lfloor n t \rfloor} \geq \d \log^2(n)\right) & \leq \sum_{x=0}^{n-1} \P^x \left( X_k \leq x-\d\log^2(n), \quad \text{for some } k\leq n \right)\\
&\leq n \P(T_{- \lceil \d \log^2(n) \rceil } < \infty)\\
&\leq n e^{-\d_1 \d \log^2(n)}.
\end{align*}
This last term is summable, and thus the lemma holds by the Borel-Cantelli Lemma.
\end{proof}
An immediate consequence of this last lemma is that $\lim_{n\ra\infty} \frac{X^*_{n}}{n}= \lim_{n\ra\infty}\frac{X_n}{n} =  v_P,\quad \P-a.s$. Letting $\phi^n(t):= \frac{X^*_{nt}}{n}$ and $\phi(t):= t\cdot v_P$ for $t\geq 0$, this implies that $\phi^n(t)$ converges to $\phi(t)$ pointwise. However, since each $\phi^n$ is monotone in $t$ and $\phi$ is monotone and continuous, the convergence is uniform on compact subsets.

\begin{lem}\label{Thesis_R}
For $P-a.e.$ environment $\w$, the random variables $R^n_\cdot \in D[0,\infty)$ converge in quenched distribution as $n\ra\infty$ to $W_{v_P \cdot}$, where $W_\cdot$ is a standard Brownian motion.
\end{lem}
\begin{proof}
For any $T\in(0,\infty)$, let $D[0,T]$ be the space of all real valued funtions on $[0,T]$ which are right continous and which have limits from the left, equipped with the Skorohod topology. Then, it is enough to show that $R^n_\cdot \in D[0,T]$ converges in quenched distribution to $W_{v_P \cdot}$ in the space $D[0,T]$ for all $T<\infty$. 

For the remainder of the chapter, we will use $\eta_n \limdw \eta$ to mean that $\eta_n$ converges in quenched distribution to $\eta$ as $n\ra\infty$. 
Note that the remarks preceeding the theorem imply that $\phi^n \limdw \phi$ in $D[0,T]$ for any $T<\infty$. Also, recall that
Theorem \ref{Thesis_HTCLT} implies that $Z^n \limdw W$.
Also, note that $Y^n=Z^n\circ \phi^n$ by definition. Therefore, by \cite[lemma on p. 151]{bCOPM},
$Y^n \in D[0,T]$ converges in distribution to $W\circ \phi$. 
(This is just a consequence of the continuous mapping theorem for Polish spaces and the fact that the mapping $(x,\psi)\mapsto x\circ \psi$ is a continuous mapping from $D[0,T]\times D_0$ to $D[0,T]$, where $D_0\subset D[0,T]$ is the subset non-decreasing functions with values between 0 and 1.)

It follows from the definition of $X^*_{nt}$, that
$\sum_{i=1}^{X^*_{nt}} \tau_i \leq nt < \sum_{i=1}^{X^*_{nt}+1}
\tau_i$. Thus, 
\[
Y_t^n \leq R_t^n<
Y_t^n+\frac{1}{\s\sqrt{n}}\tau_{X^*_{nt}+1}.
\] 
For any $\e>0$, Chebychev's inequality implies that $\P(\tau_k \geq \e \sqrt{k} ) \leq \e^{-\gamma} \E \tau_1^\gamma k^{-\gamma/2}$. Since $\gamma/2 > 1$, the Borel-Cantelli Lemma implies that $ \lim_{k\ra\infty} \frac{\tau_k}{\sqrt{k}} = 0$, $\P-a.s$.
This can be used to show that $\max_{i\leq n} \frac{\tau_i}{\sqrt{n}}$
converges almost surely to 0, and thus
$\frac{1}{\s\sqrt{n}}\tau_{X^*_{nt}+1}$ converges uniformly to 0
for $t\in[0,1]$ as $n\ra\infty$. 
Thus, $R_\cdot^n$ is
squeezed between two sequences of functions that both converge in
distribution to $W_{v_P\cdot }$.
\end{proof}
While it may not be immediately apparent, Lemma \ref{Thesis_R} is not far from a quenched functional CLT for the random walk. To see this, note that 
\[
R_t^n = \frac{-1}{\s\sqrt{n}}\left(nt - \sum_{k=1}^{X^*_{nt}} E_\w\tau_k\right) 
= \frac{1}{v_P\s\sqrt{n}}\left(X^*_{nt} - ntv_P + \sum_{k=1}^{X^*_{nt}}(v_P E_\w\tau_k - 1)\right).
\]
By Lemma \ref{Thesis_nu}, we may replace $X_{nt}^*$ above by $X_{\lfloor nt \rfloor}$ without changing the limiting distribution. Thus, to obtain a quenched functional CLT for the random walk, we only need to replace $\sum_{k=1}^{X^*_{nt}}(v_P E_\w\tau_k - 1)$ by something that only depends on the environment. In order to accomplish this, we first need to make a few technical estimates. 

\end{section}

\begin{section}{A Few Technical Estimates}
For the following Lemmas we will need to define a few additional random variables in order to take advantage of the mixing properties of the environment. Consider a RWRE modified by never allowing it to backtrack a distance of $\log^2(n)$ from its farthest excursion to the right. That is, after first hitting $i$ the environment is changed so that $\w_{i-\lceil \log^2 n \rceil} = 1$.  Let $T_i^{(n)}$ be the hitting time of the point $i$ for such a walk, and then let $\tau_i^{(n)}:=T_i^{(n)}-T_{i-1}^{(n)}$. Also let $\E\tau_1^{(n)}=:\frac{1}{v_P^{(n)}}$.
Note, the argument given in Lemma \ref{Thesis_nu} shows that $\P(\tau_1^{(n)} \neq \tau_1) \leq \P(T_{-\lceil \log^2(n) \rceil} < \infty) \leq e^{-\d_1\log^2(n)}=n^{-\d_1\log(n)}$ for all $n$ large enough. Using this and the Cauchy-Schwartz inequality, it follows that that
\begin{align*}
\E(\tau_1-\tau_1^{(n)}) \leq \left( \E(\tau_1-\tau_1^{(n)})^2\P(\tau_1\neq \tau_1^{(n)}) \right)^{1/2} \leq \sqrt{\E \tau_1^2 }n^{\frac{-\d_1}{2}\log(n)}.
\end{align*}
Thus, there exist positive constants $A$ and $B$ depending \emph{only} on the law of the environment $P$, such that $\E(\tau_1 - \tau_1^{(n)}) \leq A n^{-B\log(n)}$ for all $n$. These constants $A$ and $B$ appear in the statement of the following lemma, which provides a crucial estimate:
\begin{lem}\label{Thesis_block}
For any $x>0$ and any integers $k$ and $n$,
\begin{align*}
&\P\left( \max_{1\leq j \leq k} \left| \sum_{i=1}^j
(\tau_i-\frac{1}{v_P}) \right| \geq x \right) \\
&\qquad \leq \frac{3k^2}{x}
A n^{-B\log{n}} +D_\gamma \frac{K_n^{1+\gamma}}{x^\gamma}
\left\lceil \frac{k}{K_n} \right\rceil^{\gamma/2}
+(k+2K_n)\left(\frac{\E \tau_1^\gamma}{n^\gamma} +
n\a(K_n)\right)+\mathbf{1}_{\{A n^{-B\log(n)}\geq \frac{x}{3k}\}} ,
\end{align*}
where $K_n:=\lceil \log^2(n) \rceil$, $A$ and $B$ are positive constants depending only on the distribution $P$, and $D_\gamma$ is a positive constant depending only on $P$ and $\gamma$.
\end{lem}

\begin{proof}
First, note that the probability in the statement of the lemma is less than
\begin{align} \label{Thesis_split}
\P\left( \max_{1\leq j \leq k} \left| \sum_{i=1}^j (\tau_i-\tau_i^{(n)}) \right|  \geq \frac{x}{3}\right)
&+\P\left( \max_{1\leq j \leq k} \left| \sum_{i=1}^j (\tau_i^{(n)}-\frac{1}{v_P^{(n)}}) \right| \geq \frac{x}{3}\right)\\
&\quad +\P\left( \max_{1\leq j \leq k} \left| \sum_{i=1}^j (\frac{1}{v_P}-\frac{1}{v_P^{(n)}}) \right| \geq \frac{x}{3} \right). \nonumber
\end{align}
By Chebychev's inequality, the first probability in \eqref{Thesis_split} is less than
\[
k \P\left( \tau_1-\tau_1^{(n)} \geq \frac{x}{3k} \right)
\leq \frac{3k^2}{x} \E|\tau_1-\tau_1^{(n)}|
\leq \frac{3k^2}{x} A n^{-B\log(n)}.
\]
The third probability in \eqref{Thesis_split} is either 0 or 1, since it involves no random variables. Also, $\tau_i^{(n)}\leq \tau_i$ for any $n$, and so $\frac{1}{v_P}\leq \frac{1}{v_P^{(n)}}$. Thus, the maximum in the third term is obtained when $j=k$. Since,
\begin{align*}
\sum_{i=1}^k (\frac{1}{v_P}-\frac{1}{v_P^{(n)}}) \geq \frac{x}{3}
\Longrightarrow \E \left( \sum_{i=1}^k \tau_i - \tau_i^{(n)} \right) \geq  \frac{x}{3}
\Longrightarrow k \E(\tau_1-\tau_1^{(n)}) \geq \frac{x}{3}
 \Longrightarrow  A n^{-B\log(n)} \geq \frac{x}{3k},
\end{align*}
it follows that $\P\left( \left| \sum_{i=1}^k (\frac{1}{v_P}-\frac{1}{v_P^{(n)}}) \right| \geq \frac{x}{3} \right) \leq \mathbf{1}_{\{ A n^{-B\log(n)} \geq \frac{x}{3k} \}}$.

To get an upper bound on the second probability in \eqref{Thesis_split},
we will break the sum inside the probability into ``blocks'' of
exponentially mixing random variables. Let $K_n:=\lceil \log^2(n)
\rceil$. Now, $\tau_i^{(n)}$ and $\tau_j^{(n)}$ are $K_n$-separated
if $|i-j| > 2\log^2(n)$. We will break the set of integers into
$2K_n = 2\lceil \log^2(n) \rceil$ blocks: $B_0 = \{\ldots, 0,
2K_n, 4K_n,\ldots  \}$, $B_1 = \{\ldots,1,1+2K_n,1+ 4K_n,\ldots
\}$, $B_2=\{\ldots,2,2+2K_n, 2+4K_n,\ldots \}$, and so on. Then,
\begin{align}
\P\left( \max_{1\leq j \leq k} \left| \sum_{i=1}^j (\tau_i^{(n)}-\frac{1}{v_P^{(n)}}) \right| \geq \frac{x}{3} \right)
&\leq \P\left( \max_{1\leq j \leq k}\sum_{m=1}^{2K_n} \left| \sum_{i\in B_m \cap [1,j]} ( \tau_i^{(n)} - \frac{1}{v_P^{(n)}} ) \right|  \geq \frac{x}{3} \right) \nonumber \\
&\leq 2 K_n \P\left( \max_{1\leq j \leq k} \left| \sum_{i\in B_1\cap[1,j]} ( \tau_i^{(n)}-\frac{1}{v_P^{(n)}} )  \right| \geq \frac{x}{6K_n}  \right). \label{Thesis_blk}
\end{align}
Now, let $\bar\tau_i^{(n)}$ be i.i.d. random variables that are independent
of $\tau_i^{(n)}$, but with the same distribution. Then, the mixing
properties of Assumption \ref{mixing} allow us to substitute
$\bar\tau_i^{(n)}$ for $\tau_i^{(n)}$ with a small probabilistic cost. In
particular:
\begin{align}
&\P\left( \max_{1\leq j \leq k} \left| \sum_{i\in B_1\cap[1,j]} (
\tau_i^{(n)}-\frac{1}{v_P^{(n)}} )  \right| \geq x \right) \nonumber \\
&\qquad \leq
\left\lceil \frac{k}{2K_n} \right\rceil \left(\frac{\E
\tau_1^\gamma}{n^\gamma} + n\a(K_n)\right) + \P\left( \max_{1\leq
j \leq k} \left| \sum_{i\in B_1\cap[1,j]} (
\bar\tau_i^{(n)}-\frac{1}{v_P^{(n)}} ) \right| \geq x
\right). \label{Thesis_switch}
\end{align}
To see this, we first substitute in $\bar\tau_1^{(n)}$ for $\tau_1^{(n)}$. For ease of notation, let $\xi_i:=\tau_i^{(n)}-\frac{1}{v_P}$ and $\bar\xi_i := \bar\tau_i^{(n)}-\frac{1}{v_P}$. Then,
\begin{align*}
&\P\left( \max_{1\leq j \leq k} \left| \sum_{i\in B_1\cap[1,j]} \xi_i  \right| \geq x \right)
=E_P\left[ P_\w\left( \max_{1\leq j \leq k} \left| \sum_{i\in B_1\cap[1,j]}  \xi_i  \right| \geq x \right) \right]\\
&\qquad=E_P\left[\sum_{m=1}^\infty P_\w(\tau_1^{(n)} = m)P_\w\left( \max_{1\leq j \leq k} \left| (m-\frac{1}{v_P}) + \sum_{i\in B_1\cap[2K_n+1,j]}  \xi_i   \right| \geq x \right)  \right]\\
&\qquad\leq \sum_{m\leq n} E_P\left[ P_\w(\tau_1^{(n)} = m)P_\w\left( \max_{1\leq j \leq k} \left| (m-\frac{1}{v_P}) + \sum_{i\in B_1\cap[2K_n+1,j]}  \xi_i   \right| \geq x \right) \right]
 + \P(\tau_1^{(n)}> n) \\
&\qquad\leq \sum_{m\leq n} \P(\bar\tau_1^{(n)} = m)\P\left( \max_{1\leq j \leq k} \left| (m-\frac{1}{v_P}) + \sum_{i\in B_1\cap[2K_n+1,j]}  \xi_i   \right| \geq x \right) 
+ n\a(K_n) + \P(\tau_1^{(n)}> n)\\
&\qquad\leq \P\left(\max_{1\leq j\leq k} \left|\bar\xi_1 + \sum_{i\in
B_1\cap[2K_n+1,j]} \xi_i \right|\geq x \right)+
n\a(K_n)+\frac{\E\tau_1^\gamma}{n^\gamma} .
\end{align*}
Iterating this argument proves \eqref{Thesis_switch}. Then, \eqref{Thesis_switch} and \eqref{Thesis_blk} imply
\begin{align}
&\P\left( \max_{1\leq j \leq k} \left| \sum_{i=1}^j (\tau_i^{(n)}-\frac{1}{v_P^{(n)}}) \right| \geq \frac{x}{3} \right) \nonumber \\
&\qquad \leq (k+2K_n)\left(\frac{\E \tau_1^\gamma}{n^\gamma} + n\a(K_n)\right) + 2K_n\P\left( \max_{1\leq j \leq k} \left| \sum_{i\in B_1\cap[1,j]} ( \bar\tau_i^{(n)}-\frac{1}{v_P^{(n)}} )  \right| \geq \frac{x}{6K_n} \right)\nonumber \\
&\qquad \leq  (k+2K_n)\left(\frac{\E \tau_1^\gamma}{n^\gamma} +
n\a(K_n)\right) + \frac{2\cdot 6^\gamma
K_n^{1+\gamma}}{x^\gamma}\E\left|\sum_{i\in
B_1\cap[1,k]}(\bar\tau_i^{(n)}-\frac{1}{v_P^{(n)}})\right|^\gamma .
\label{Thesis_blk2}
\end{align}
The second inequality above follows from the Kolmogorov inequality for
martingales, since the random variables $(\bar\tau_i^{(n)}-\frac{1}{v_P^{(n)}})$ are i.i.d. and have zero mean. 

The Zygmund-Marcinkiewicz inequality \cite[Theorem 2]{cPTIIM} says that for any $p\geq 1$, there exists a universal constant $C_p$ such that $E|\sum_{i=1}^k \xi_i|^p < C_p E|\sum_{i=1}^k
\xi_i^2|^{p/2}$, for any independent, zero-mean random variables
$\xi_i$. If, in addition, $p\geq 2$, then by Jensen's
inequality $|\sum_{i=1}^k \xi_i^2|^{p/2}\leq
k^{p/2-1}\sum_{i=1}^k |\xi_i|^p$, which implies $E|\sum_{i=1}^k
\xi_i|^p \leq C_p k^{p/2-1}E\left( \sum_{i=1}^k |\xi_i|^p\right)$.
Furthermore, if the $\xi_i$ are also identically distributed then
this last term equals $C_p k^{p/2} E|\xi_1|^p$. Thus, since the
random variables $\bar\tau_i^{(n)} - \frac{1}{v_P^{(n)}}$ are
i.i.d., we can apply the Zygmund-Marcinkiewicz inequality to obtain
\begin{align*}
\E\left|\sum_{i\in
B_1\cap[1,k]}(\bar\tau_i^{(n)}-\frac{1}{v_P^{(n)}})\right|^\gamma
\leq C_\gamma \left\lceil \frac{k}{2K_n} \right\rceil^{\gamma/2}
\E\left| \bar\tau_1^{(n)} - \frac{1}{v_P^{(n)}}  \right|^\gamma 
\leq C_\gamma \left\lceil \frac{k}{2K_n} \right\rceil^{\gamma/2}
2^{\gamma-1}\E\left(|\tau_1|^\gamma +\frac{1}{v_P^\gamma}\right).
\end{align*}
Combining this with \eqref{Thesis_blk2} gives
\[
\P\left( \max_{1\leq j \leq k} \left| \sum_{i=1}^j
(\tau_i^{(n)}-\frac{1}{v_P^{(n)}}) \right| \geq \frac{x}{3}
\right) \leq (k+2K_n)\left(\frac{\E \tau_1^\gamma}{n^\gamma} +
n\a(K_n)\right)\\ + \frac{K_n^{1+\gamma}}{x^\gamma} D_\gamma
\left\lceil \frac{k}{2K_n} \right\rceil^{\gamma/2},
\]
where $D_\gamma$ is a constant depending only on $P$ and $\gamma$. 
\end{proof}

The following lemma is the essential step in proving a quenched
CLT:
\begin{lem}\label{Thesis_sum}
For any $\a<\beta < \gamma$, 
\[
\max_{ j,k\in[1,n^\b];\, |k-j|<n^\a } \left|\frac{1}{n^{\b/2}}
\sum_{i=j}^k (\tau_i - \frac{1}{v_P}) \right| \limn 0, \quad
\P-a.s.
\]
\end{lem}
\begin{proof}
By dividing the interval $[1,n^\b]$ into blocks of length $n^\a$,
we get that for any $\d>0$,
\[
\P\left(\max_{j,k\in[1,n^\b];\, |k-j|<n^\a}
\left|\frac{1}{n^{\b/2}} \sum_{i=j}^k (\tau_i - \frac{1}{v_P})
\right| \geq \d \right) \leq \lceil n^{\b-\a} \rceil \P\left(
\max_{1\leq k < n^\a} \left|\frac{1}{n^{\b/2}} \sum_{i=1}^k
(\tau_i-\frac{1}{v_P}) \right| \geq \frac{\d}{3} \right).
\]
Now, choose an integer $m$ large enough so that
$\min\{(\gamma-\beta),(\frac{\gamma}{2}-1)(\b-\a)\}>\frac{1}{m}$.
Then, letting $N^m$ take the place of $n$ above and applying Lemma
\ref{Thesis_block} (with $k=N^{m\a}$, $x=\frac{\d}{3}N^{m\beta/2}$ and $n
= N^m$),
\begin{align*}
\P \Biggl( \max_{j,k\in[1,N^{m\b}];\, |k-j|<N^{m\a}}  &
\left|\frac{1}{N^{m\b/2}} \sum_{i=j}^k (\tau_i - \frac{1}{v_P})
\right| \geq \d \Biggr)  \\
&\leq \lceil N^{m(\b-\a)} \rceil
\P\left( \max_{1\leq k < N^{m\a}} \left|\sum_{i=1}^k
(\tau_i-\frac{1}{v_P})
\right| \geq \frac{\d}{3}N^{m\b/2} \right)\\
&=\lceil N^{m(\b-\a)} \rceil \left( \bigo\left(N^{m\gamma(\a-\beta)/2}\log(N)^{2+\gamma}\right)+\bigo\left(N^{m(\a-\gamma)}\right) \right) \\
&=\bigo\left(N^{m(\frac{\gamma}{2}-1)(\a-\beta)}\log(N)^{2+\gamma}\right)
+ \bigo\left( N^{m(\beta-\gamma)} \right).
\end{align*}
Our choice of $m$ makes both of the exponents of $N$ in the last
line less than $-1$ so that the last line is summable. Thus, the
Borel-Cantelli Lemma implies that
\begin{equation}
\max_{j,k\in[1,N^{m\b}];\, |k-j|<N^{m\a}}
\left|\frac{1}{N^{m\b/2}} \sum_{i=j}^k (\tau_i - \frac{1}{v_P})
\right| \limN 0 ,\quad \P-a.s. \label{Thesis_subseq}
\end{equation}
This essentially says that the limit in the statement of the lemma
converges to 0 along the subsequence $n^m$. It turns out this
subsequence is dense enough to get convergence of the original
sequence. We re-write the original sequence to be able to apply
\eqref{Thesis_subseq}:
\begin{align*}
&\max_{j,k\in[1,n^\b];\, |k-j|<n^\a} \left|\frac{1}{n^{\b/2}}
\sum_{i=j}^k (\tau_i - \frac{1}{v_P}) \right| \\
&\qquad \leq
\max_{j,k\in[1,\lceil n^{1/m} \rceil^{m\b}];\, |k-j|<\lceil
n^{1/m} \rceil^{m\a}} \left|\frac{1}{n^{\b/2}} \sum_{i=j}^k
(\tau_i - \frac{1}{v_P})\right|\\
&\qquad =\frac{\lceil
n^{1/m}\rceil^{m\b/2}}{n^{\b/2}}\max_{j,k\in[1,\lceil n^{1/m}
\rceil^{m\b}];\, |k-j|<\lceil n^{1/m} \rceil^{m\a}}
\left|\frac{1}{\lceil n^{1/m}\rceil^{m\b/2}} \sum_{i=j}^k (\tau_i
- \frac{1}{v_P})\right|.
\end{align*}
Since $\lceil n^{1/m}\rceil^{m\b /2}\sim n^{\b/2}$, we may apply
\eqref{Thesis_subseq}, with $N=\lceil n^{1/m} \rceil$, to finish the proof of the lemma.
\end{proof}

\begin{cor}\label{Thesis_nu2}
For any $\b> 1$ and any $\d>0$, 
\[\lim_{n\ra\infty} P_\w\left( \sup_{0\leq t \leq 1} |X^*_{nt}-ntv_P| \geq \d n^{\b/2} \right) = 0,\quad P-a.s.\]
\end{cor}
\begin{proof}
We may assume without loss of generality that $\b<2$. It follows that
\begin{align*}
& P_\w \left( \sup_{0\leq t\leq 1} |X^*_{nt}-ntv_P| \geq \d n^{\b/2} \right) \\
&\qquad \leq P_\w\left(\exists t\in[0,1]: X^*_{nt}>ntv_P + \d n^{\b/2} \right) + P_\w \left( \exists t\in[0,1]: X^*_{nt}<ntv_P- \d n^{\b/2} \right)\\
&\qquad \leq P_\w \left(\exists t\in[0,1]:  \sum_{i=1}^{\lceil ntv_P +\d n^{\b/2} \rceil } \tau_i \leq nt \right)
      + P_\w \left(\exists t\in[0,1]:  \sum_{i=1}^{\lceil ntv_P - \d n^{\b/2} \rceil } \tau_i > nt \right)\\
&\qquad \leq P_\w \left( \inf_{0\leq t \leq 1} \sum_{i=1}^{\lceil ntv_P +\d n^{\b/2} \rceil } (\tau_i - \frac{1}{v_P}) < \frac{-\d n^{\b/2}}{v_P} \right) \\
&\qquad \qquad + P_\w \left( \sup_{0\leq t \leq 1} \sum_{i=1}^{\lceil ntv_P -\d n^{\b/2} \rceil } (\tau_i - \frac{1}{v_P}) > \frac{\d n^{\b/2} - 1}{v_P} \right)\\
&\qquad \leq 2 P_\w \left( \max_{1\leq k\leq n^\b} \left| \frac{1}{n^{\b/2}}\sum_{i=1}^k (\tau_i-\frac{1}{v_P})  \right| \geq \frac{\d}{2 v_P} \right), 
\end{align*}
for all $n$ sufficiently large. 
Then, Lemma \ref{Thesis_sum} implies that the last line tends to zero as $n\ra\infty$.
\end{proof}

\begin{cor}
For any $\b>1$ and $\d>0$,
\[\lim_{n\ra\infty} P_\w \left(\sup_{0\leq t \leq 1} |X_{nt} - ntv_p| \geq \d n^{\b/2}\right) = 0, \quad  P-a.s.\]
\end{cor}
\begin{proof}
First, note that
\begin{align*}
&P_\w \left(\sup_{0\leq t \leq 1} |X_{nt} - ntv_p| \geq \d n^{\b/2} \right) \\
&\qquad \leq P_\w \left( \sup_{0\leq t \leq 1} |X_{nt} - X^*_{nt}| \geq \frac{\d n^{\b/2}}{2}\right) + P_\w \left( \sup_{0\leq t \leq 1} |X^*_{nt} - ntv_p| \geq \frac{\d n^{\b/2}}{2}\right).
\end{align*}
Then, the proof of the corollary follows from Lemma \ref{Thesis_nu} and Corollary \ref{Thesis_nu2}.
\end{proof}
\end{section}

\begin{section}{Quenched CLT for the Random Walk}
For $t>0$, let
\[
Z_{nt}(\w):= \sum_{i=1}^{\lfloor nt v_P \rfloor } (v_P E_\w \tau_i - 1).
\]
$Z_{nt}$ will be the random centering that appears in the quenched
CLT for the random walk. The following lemma is a consequence of the technical estimates of the last section:
\begin{lem}\label{Thesis_Z}
For any $\d>0$ and any $t$,
\[
\lim_{n\ra\infty} P_\w \left(\sup_{0\leq t\leq 1}
\frac{1}{\sqrt{n}}\left|\sum_{i=1}^{X^*_{nt}} (E_\w \tau_i
-\frac{1}{v_P}) - \frac{1}{v_P}Z_{nt} \right| \geq \d \right) =
0, \quad P-a.s.
\]
\end{lem}
\begin{proof}
Let $\frac{1}{2} < \a < 1$. Then, 
\begin{align}
&P_\w \left(\sup_{0\leq t\leq 1}
\frac{1}{\sqrt{n}}\left|\sum_{i=1}^{X^*_{nt}} (E_\w \tau_i
-\frac{1}{v_P}) - \frac{Z_{nt}}{v_P} \right| \geq \d \right) \nonumber \\
&\qquad = P_\w \left( \sup_{0\leq t\leq 1} \frac{1}{\sqrt{n}}\left|
\sum_{i=1}^{X^*_{nt}} (E_\w \tau_i -\frac{1}{v_P}) -
\sum_{i=1}^{ntv_P} (E_\w \tau_i -\frac{1}{v_P}) \right| \geq \d
\right) \nonumber \\ 
&\qquad \leq P_\w\left( \sup_{0\leq t\leq 1} |X^*_{nt} - ntv_P| \geq n^{\a}  \right) + P_\w\left(\max_{j,k\in[1,n];\,|j-k|<n^\a} \left|
\frac{1}{\sqrt{n}} \sum_{i=j}^{k} (E_\w \tau_i
-\frac{1}{v_P}) \right| \geq \d \right). \label{replacecentering}
\end{align}
By Corollary \ref{Thesis_nu}, the first term in \eqref{replacecentering} tends
to 0 as $n\ra\infty,\; P-a.s$. The second term in \eqref{replacecentering} is bounded above by
\begin{align*}
& P_\w\left(\max_{j,k\in[1,n];\,|j-k|<n^\a} \left|
\frac{1}{\sqrt{n}} \sum_{i=j}^{k} (\tau_i
-\frac{1}{v_P}) \right| \geq \frac{\d}{2} \right) \\
&\qquad + P_\w\left(\max_{j,k\in[1,n];\,|j-k|<n^\a} \left|
\frac{1}{\sqrt{n}} \sum_{i=j}^{k} (\tau_i - E_\w \tau_i
) \right| \geq \frac{\d}{2} \right).
\end{align*}
Since $\a<1$, Lemma \ref{Thesis_sum} shows that, $P-a.s.$, the first term above goes to 0 as
$n\ra\infty$. Also, the quenched functional CLT for hitting times, Theorem \ref{Thesis_HTCLT}, shows that, $P-a.s.$, the
second term above goes to 0 as $n\ra\infty$. Therefore, $P-a.s.$, the second term in \eqref{replacecentering} tends to zero as $n\ra 0$. 
\end{proof}
We can now prove a quenched functional CLT for the random walk.
\begin{thm}
 Assume that Assumptions \ref{uelliptic}-\ref{sg2} hold, and let 
\[
 B_t^{n} := \frac{X_{\lfloor nt \rfloor}-nt v_P+Z_{nt}(\w)}{v_P^{3/2}\s\sqrt{n}},
\]
where $\s$ is defined in Theorem \ref{Thesis_HTCLT}.
Then, for $P-a.e.$ environment $\w$, the random variables $B^n_\cdot \in D[0,\infty)$ converge in quenched distribution as $n\ra\infty$ to a standard Brownian motion. 
\end{thm}
\begin{proof}
As noted in the proof of Lemma \ref{Thesis_R}, it is enough to prove convergence in quenched distribution in the space $D[0,T]$ for all $T<\infty$. We will handle the case when $T=1$ since the proof is the same for any $T<\infty$. For the remainder of the proof, when denoting convergence in distribution of random functions in $D[0,1]$, we will keep the index $t$ for clarity. That is, we will write $Z_t^n\limdw W_{v_P t}$ instead of $Z^n_\cdot \limdw W_{v_P \cdot}$. 

Recall that Lemma \ref{Thesis_R} implies
\begin{align}
R^n_t = \frac{nt-\sum_{i=1}^{X^*_{nt}}E_\w \tau_i}{\s\sqrt{n}} =\frac{nt-\frac{X^*_{nt}}{v_P}-\sum_{i=1}^{X^*_{nt}}(E_\w \tau_i-
\frac{1}{v_P})}{\s\sqrt{n}} \limdw W_{t v_P} . \label{Thesis_l1}
\end{align}
Also, Lemma \ref{Thesis_Z} shows that, as elements of $D[0,1]$,
\begin{equation}
\frac{\sum_{i=1}^{X^*_{nt}}(E_\w \tau_i-
\frac{1}{v_P})-\frac{1}{v_P}Z_{nt}(\w)}{\sqrt{n}} \limdw 0.
\label{Thesis_l2}
\end{equation}
Combining \eqref{Thesis_l1} and \eqref{Thesis_l2},
\[
\frac{nt-\frac{X^*_{nt}}{v_P}-\frac{1}{v_P}Z_{nt}(\w)}{\s\sqrt{n}}
\limdw W_{t v_P},
\]
or equivalently (since $W_t$ is symmetric),
\[
\frac{X^*_{nt}-ntv_P+Z_{nt}(\w)}{v_P^{3/2}\s\sqrt{n}} \limdw W_t,
\]
in the space $D[0,1]$.
Finally, Lemma \ref{Thesis_nu} implies that
$\frac{X^*_{nt}-X_{nt}}{\sqrt{n}}\limdw 0$. So,
\[
\frac{X_{nt}-ntv_P+Z_{nt}(\w)}{v_P^{3/2}\s\sqrt{n}} \limdw W_t,
\]
in the space $D[0,1]$.
\end{proof}
\end{section}

\end{chapter}

\begin{chapter}{Quenched Limits: Zero Speed Regime}\label{Thesis_AppendixZeroSpeed}

This chapter consists of the article \emph{Quenched Limits for Transient, Zero Speed One-Dimensional Random Walk in Random Environment}, by Jonathon Peterson and Ofer Zeitouni, which was recently accepted for publication by the Annals of Probability.
This article contains the full proofs of Theorems \ref{Thesis_local} and \ref{Thesis_nonlocal} and the first part of Theorem \ref{Thesis_qETVarStable} (sketches of these proofs were provided in Chapter \ref{1dlimitingdist}). 

In order to keep this chapter self-contained, the above mentioned article has been left relatively unchanged. Therefore, much of the introductory material in Section \ref{sl1_Introduction} has already appeared in Chapters \ref{Thesis_Introduction} and \ref{1dlimitingdist}. 
The notation used in this chapter is consistent with the notation in Chapters \ref{Thesis_Introduction} and \ref{1dlimitingdist}. 

While the main results of this chapter are for the case when the parameter $s\in(0,1)$, many of the preliminary results are true in greater generality. Since some of these preliminary results will be referenced in Chapter \ref{Thesis_AppendixBallistic}, which concerns the case $s\in(1,2)$, if no mention is made of bounds on $s$, then it is to be understood that the statement holds for all $s>0$.

\newpage

\begin{section}{Introduction and Statement of Main Results} \label{sl1_Introduction}
Let $\Omega = [0,1]^\Z$ and let $\mathcal{F}$ be the Borel $\s-$algebra on $\Omega$. A random environment is an $\Omega$-valued random variable $\w = \{\w_i\}_{i\in\Z}$ with distribution $P$. We will assume that the $\w_i$ are i.i.d.
The \emph{quenched} law $P_\w^x$ for a random walk $X_n$ in the environment $\w$ is defined by
\[
P_\w^x( X_0 = x ) = 1 \quad \text{and} \quad
P_\w^x\left( X_{n+1} = j | X_n = i \right) =
\begin{cases}
\w_i &\quad \text{if } j=i+1, \\
1-\w_i &\quad \text{if } j=i-1.
\end{cases}
\]
$\Z^\N$ is the space for the paths of the random walk $\{X_n\}_{n\in\N}$,
and $\mathcal{G}$ denotes the $\s-$algebra generated by the cylinder sets.
Note that for each $\w \in \Omega$, $P_\w$ is a probability measure
on $\mathcal{G}$, and for each $G\in \mathcal{G}$,
$P_\w^x(G):(\Omega, \mathcal{F}) \ra [0,1]$ is a measurable
function of $\w$.  Expectations under the law $P_\w^x$ are denoted $E_\w^x$.
The \emph{annealed} law for the random walk in random
environment $X_n$ is defined by
\[
\P^x(F\times G) = \int_F P_\w^x(G)P(d\w),
\quad F\in \mathcal{F},  G\in \mathcal{G}\!.
\]
For ease of notation, we will use $P_\w$ and $\P$ in place
of $P_\w^0$ and $\P^0$ respectively. We will also use $\P^x$ to
refer to the marginal on the space of paths, i.e., $\P^x(G)=
\P^x(\Omega\times G) = E_P\left[ P^x_\w(G) \right]$ for
$G\in \mathcal{G}$. Expectations under the law $\P$ will be written $\E$.

A simple criterion for recurrence and a formula for the speed of
transience was given by Solomon in \cite{sRWRE}. For any integers
$i\leq j$, let
\begin{equation}
\rho_i := \frac{1-\w_i}{\w_i}, \quad \text{and}\quad
\Pi_{i,j} := \prod_{k=i}^j \rho_k\,, \label{rhodef}
\end{equation}
and for $x\in \Z$, define the hitting times
\[
T_x:= \min\{n \geq 0:X_n=x\}\,.
\]
Then, $X_n$ is transient to the right (resp. to the left)
if $E_P(\log \rho_0) < 0$ (resp. $E_P \log \rho_0 > 0$) and recurrent
if $E_P (\log \rho_0) = 0$. (henceforth we will write $\rho$ instead of
$\rho_0$ in expectations involving only $\rho_0$.) In the case where
$E_P \log\rho < 0$ (transience to the right),
Solomon established the following law of large numbers
\[
v_P:= \lim_{n\ra\infty} \frac{X_n}{n} =
\lim_{n\ra\infty} \frac{n}{T_n} = \frac{1}{\E T_1}, \quad \P-a.s.
\]
For any integers $i<j$, let
\begin{equation}
W_{i,j} := \sum_{k=i}^j \Pi_{k,j}, \quad \text{and}
\quad W_j := \sum_{k\leq j} \Pi_{k,j}\,. \label{Wdef}
\end{equation}
When $E_P \log \rho< 0$, it was shown in \cite{sRWRE},\cite[remark following Lemma 2.1.12]{zRWRE}
that
\begin{equation}
E_\w^j T_{j+1} = 1+2W_j < \infty, \quad P-a.s., \label{QET}
\end{equation}
and thus $v_P =
1/(1+2E_P W_0)$. Since $P$ is a product measure, $E_P W_0 =
\sum_{k=1}^\infty \left(E_P \rho\right)^k$. In particular, $v_P =
0$ if
$E_P \rho \geq 1$.

Kesten, Kozlov, and Spitzer \cite{kksStable} determined the
annealed limiting distribution of a RWRE with $E_P \log \rho < 0$, i.e.,
transient to the right. They derived the limiting distributions for the walk by first establishing a stable
limit law of index $s$ for $T_n$, where $s$ is defined by the equation
\[
E_P\rho^s = 1\,.
\]
In particular, they showed that when $s<1$, there exists a $b>0$ such that
$$\lim_{n\ra\infty} \P\left( \frac{T_n}{n^{1/s}} \leq x \right) =
L_{s,b}(x)\,,$$
and
\begin{equation}
\lim_{n\ra\infty} \P\left( \frac{X_n}{n^s} \leq x
\right) = 1-L_{s,b}(x^{-1/s}), \label{annealedstable}
\end{equation}
where $L_{s,b}$ is the distribution function for a stable random
variable with characteristic function
\begin{equation}
\hat{L}_{s,b}(t)= \exp\left\{ -b|t|^s \left(
1-i\frac{t}{|t|}\tan(\pi s/2)  \right) \right\}. \label{char}
\end{equation}
The value of $b$ was recently identified \cite{eszStable}.
While the annealed limiting distributions for transient
one-dimensional RWRE have been known for quite a while, the corresponding
quenched limiting distributions have remained largely unstudied
until recently. 
In Chapter \ref{Thesis_AppendixQCLT} we proved that when $s>2$ a quenched CLT holds with a
random (depending on the environment)
centering.
A similar result was given by Rassoul-Agha and Sepp\"al\"ainen in \cite{rsBFD} under different assumptions on the environment.
Previously, in \cite{kmCLT} and \cite{zRWRE},
it was shown that the limiting statement for
the quenched CLT with random centering holds in probability rather
than almost surely. No other results of quenched limiting
distributions are known
when $s\leq 2$.

In this chapter, we analyze the quenched limiting distributions of a
one-dimensional transient RWRE in the case $s<1$. One could expect
that the quenched limiting distributions are of the same type as
the annealed limiting distributions since annealed probabilities
are averages of quenched probabilities. However, this turns out not to be the
case. In fact, a consequence of our main results,
Theorems \ref{refstable}, \ref{local}, and \ref{nonlocal} below, is that
the annealed stable behavior of $T_n$ comes from
fluctuations in the environment.

Throughout the chapter, we will
make the following assumptions:
\begin{asm} \label{essentialasm}
$P$ is an i.i.d. product measure on $\Omega$ such that
\begin{equation}
E_P \log\rho < 0 \quad\text{and}\quad E_P \rho^s = 1 \text{ for
some } s>0 . \label{zerospeedregime}
\end{equation}
\end{asm}
\begin{asm}
The distribution of $\log \rho$ is non-lattice under
$P$ and $E_P \rho^s \log\rho<\infty$.  \label{techasm}
\end{asm}
\noindent\textbf{Note:} Since $E_P \rho^\gamma$ is a convex function of
$\gamma$, the two statements in \eqref{zerospeedregime} give that
$E_P \rho^\gamma < 1$ for all
$\gamma<s$ and $E_P \rho^\gamma > 1$ for all $\gamma > s$.
Assumption \ref{essentialasm}
contains the essential assumption necessary for the walk to be transient. The main results of this chapter are for $s<1$ (the zero-speed regime), but many statements hold for $s\in(0,2)$ or even $s\in(0,\infty)$. If no mention is made of bounds on $s$, then it is assumed that the statement holds for all $s>0$. We recall that the technical conditions contained in
Assumption \ref{techasm} were also invoked in \cite{kksStable}.

Define the ``ladder locations'' $\nu_i$ of the environment by
\begin{align}
\nu_0 = 0, \quad\text{and}\quad \nu_i =
\begin{cases}
\inf\{n > \nu_{i-1}: \Pi_{\nu_{i-1},n-1} < 1\}, &\quad  i \geq 1,\\
\sup \{j < \nu_{i+1}: \Pi_{k,j-1}<1,\quad \forall k<j \}, &\quad  i \leq -1
\,.\end{cases}
\label{nudef}
\end{align}
Throughout the remainder of the chapter, we will let $\nu=\nu_1$. We
will sometimes refer to sections of the environment between
$\nu_{i-1}$ and $\nu_i -1$ as ``blocks'' of the environment. Note
that the block between $\nu_{-1}$ and $\nu_0 -1$ is different from
all the other blocks between consecutive ladder locations. Define
the measure $Q$ on environments by $Q(\cdot)=P(\cdot|\mathcal{R})$, where the
event
\[
\mathcal{R}:=\{ \w\in\Omega: \Pi_{-k,-1} < 1,\quad\forall k \geq 1\}.
\]
Note that $P(\mathcal{R}) > 0$ since $E_P \log \rho < 0$. 
$Q$ is defined so that the blocks of the environment between ladder locations are i.i.d.
under $Q$, all with distribution the same as that of the block
from $0$ to $\nu -1$ under $P$. In Section \ref{stablecrossing}, we
prove the following annealed theorem:
\begin{thm}\label{refstable}
Let Assumptions \ref{essentialasm} and \ref{techasm} hold, and let $s<1$. Then,
there exists a $b'>0$ such that
\[
\lim_{n\ra\infty} Q\left( \frac{ E_\w T_{\nu_n} }{n^{1/s}} \leq x
\right) = L_{s,b'}(x).
\]
\end{thm}
We then use Theorem \ref{refstable} to prove the following two theorems
which show that $P-a.s.$ there exist two different random sequences of
times (depending on the environment) where the random walk has
different limiting behavior. These are the main results of the chapter.
\begin{thm}\label{local}
Let Assumptions \ref{essentialasm} and \ref{techasm} hold, and let $s<1$. Then,
$P$-a.s., there exist random subsequences $t_m=t_m(\w)$ and
$u_m=u_m(\w)$ such that, for any $\d> 0$,
\[
\lim_{m\ra\infty} P_\w\left( \frac{X_{t_m} - u_m}{(\log t_m)^2}
\in [-\d, \d] \right) = 1.
\]
\end{thm}
\begin{thm} \label{nonlocal}
Let Assumptions \ref{essentialasm} and \ref{techasm} hold, and let $s<1$. Then,
$P$-a.s., there exists a random subsequence $n_{k_m}=n_{k_m}(\w)$
of $n_k=2^{2^k}$ and a random sequence $t_m=t_m(\w)$ such that
\[
\lim_{m\ra\infty} \frac{\log t_m}{\log n_{k_m} }= \frac{1}{s}
\]
and
\[
 \lim_{m\ra\infty}
P_\w\left(\frac{X_{t_m}}{n_{k_m}} \leq x \right) =
\begin{cases}
0&\quad \text{if } x \leq 0,\\
\frac{1}{2}&\quad \text{if } 0<x<\infty.
\end{cases}
\]
\end{thm}
Note that Theorems \ref{local} and \ref{nonlocal}
preclude the possibility of a quenched analogue of the annealed
statement \eqref{annealedstable}.
It should be noted that in \cite{gsMVSS}, Gantert and Shi prove
that when $s\leq 1$, there exists a random sequence of times $t_m$ at which
the local time of the random walk at a single site
is a positive fraction of $t_m$. This is related to the
statement of Theorem \ref{local}, but we do not see a simple
argument which directly implies Theorem \ref{local} from the
results of \cite{gsMVSS}.

As in \cite{kksStable}, limiting distributions for $X_n$ arise from first
studying limiting distributions for $T_n$. Thus,
to prove Theorem \ref{nonlocal}, we first prove that there
exists random subsequences $x_m=x_m(\w)$ and $v_{m,\w}$ in which
\[
\lim_{m\ra\infty} P_\w \left( \frac{T_{x_m} - E_\w T_{x_m}}{\sqrt{v_{m,\w}}}
\leq y  \right) = \int_{-\infty}^y
\frac{1}{\sqrt{2\pi}} e^{-t^2/2} dt =: \Phi(y)\,.
\]
We actually prove a stronger statement than this in Theorem \ref{gaussianT}
below,
where we prove that all $x_m$ ``near'' a subsequence $n_{k_m}$
of $n_k=2^{2^k}$ have the same Gaussian behavior (What we mean by ``near''
the subsequence $n_{k_m}$ is made precise in the statement of the theorem.)

The structure of the chapter is as follows:
In Section \ref{introlemmas} we prove some introductory lemmas
which will be used throughout the chapter. Section \ref{stablecrossing} is
devoted to proving Theorem \ref{refstable}. In Section \ref{localization},
we use the latter to prove Theorem \ref{local}. In Section \ref{gaussian},
we prove the existence of random subsequences $\{n_k\}$
where $T_{n_k}$ is approximately
Gaussian, and use this fact to prove Theorem \ref{nonlocal}.
Section \ref{tailofTnu} contains the proof of the following technical
theorem which is used throughout the chapter:
\begin{thm}\label{Tnutail}
Let Assumptions \ref{essentialasm} and \ref{techasm} hold. Then, there exists a constant $K_\infty\in (0,\infty)$ such that
\[
Q(E_\w T_\nu > x ) \sim K_\infty x^{-s}.
\]
\end{thm}
\noindent
The proof of Theorem \ref{Tnutail} is based on
results from \cite{kRDE} and mimics the proof of tail asymptotics
in \cite{kksStable}.

\end{section}


\begin{section}{Introductory Lemmas}\label{introlemmas}
Before proceeding with the proofs of the main theorems we mention a few easy lemmas which will be used throughout the rest of the chapter. Recall the definitions of $\Pi_{1,k}$ and $W_i$ in \eqref{rhodef} and \eqref{Wdef}.
\begin{lem}\label{nutail}
For any $c< -E_P \log\rho$, there exist $\d_c, A_c>0$ such that
\begin{equation}
 P(\Pi_{1,k} > e^{-c k })  = P \left( \frac{1}{k}\sum_{i=1}^k \log \rho_i > -c  \right) \leq A_ce^{-\d_c k}. \label{LDPrho}
\end{equation}
Also, there exist constant $C_1,C_2>0$ such that $P(\nu > x) \leq C_1
e^{-C_2 x}$ for all $x\geq 0$.
\end{lem}
\begin{proof}
First, note that due to Assumption \ref{essentialasm}, $\log \rho$
has negative mean and finite exponential moments in a
neighborhood of zero.
If
 $c < - E_P \log \rho $,
Cram\'{e}r's
Theorem \cite[Theorem 2.2.3]{dzLDTA} then yields
\eqref{LDPrho}.
By the definition of $\nu$ we have $P(\nu>x) \leq P(\Pi_{0,\lfloor
x \rfloor -1} \geq 1)$, which together with
\eqref{LDPrho} completes the proof of the lemma.
\end{proof}
\noindent
From  \cite[Theorem 5]{kRDE}, there exist constants
$K,K_1>0$ such that for all $i$
\begin{equation}
P(W_i > x) \sim K x^{-s},\quad\text{and}\quad P(W_i > x) \leq K_1
x^{-s}\,. \label{PWtail}
\end{equation}
The tails of $W_{-1}$, however,  are different (under the measure $Q$),
as the following lemma shows.
\begin{lem}\label{Wtail}
There exist constants $C_3,C_4>0$ such that $Q(W_{-1} > x) \leq
C_3 e^{-C_4  x}$ for all $x\geq 0$.
\end{lem}
\begin{proof}
Since $\Pi_{i,-1} < 1,\quad Q-a.s.$ we have $W_{-1} < k + \sum_{i
< -k} \Pi_{i,-1}$ for any $k>0$. Also, note that from
\eqref{LDPrho} we have $Q(\Pi_{-k,-1} > e^{-c k} ) \leq
{A_c} e^{-\d_c k}/
P(\mathcal{R})$. Thus,
\begin{align*}
Q(W_{-1} > x) &\leq Q\left( \frac{x}{2} +
\sum_{k=\frac{x}{2}}^\infty e^{-ck} > x \right) +
Q\left( \Pi_{-k,-1}> e^{-ck} \text{, for some } k\geq \frac{x}{2} \right)\\
& \leq \mathbf{1}_{\frac{x}{2} + \frac{1}{1-e^{-c}} > x} +
\sum_{k=\frac{x}{2}}^\infty Q(\Pi_{-k,-1}>e^{-ck})
\leq \mathbf{1}_{\frac{1}{1-e^{-c}} > \frac{x}{2}} +
\bigo\left(e^{-\d_c x/2}\right)\,.
\end{align*}
\end{proof}
\noindent
We also need a few more definitions
that will be used throughout the chapter. For any $i\leq k$,
\begin{equation}
R_{i,k} := \sum_{j=i}^k \Pi_{i,j},\quad\text{and}\quad R_i:=
\sum_{j=i}^\infty \Pi_{i,j}. \label{Rdef}
\end{equation}
Note that since $P$ is a product measure,
$R_{i,k}$ and $R_i$ have the same distributions as $W_{i,k}$ and
$W_i$ respectively. In particular with $K,K_1$ the same as
in \eqref{PWtail},
\begin{equation}
P(R_i > x) \sim K x^{-s},\quad\text{and}\quad P(R_i > x) \leq K_1
x^{-s}\,. \label{Rtail}
\end{equation}
\end{section}


\begin{section}{Stable Behavior of Expected Crossing Time}\label{stablecrossing}
Recall from Theorem \ref{Tnutail} that there exists
$K_\infty>0$ such that $Q(E_\w T_\nu > x) \sim K_\infty x^{-s}$.
Thus $E_\w T_\nu$ is in the domain of attraction of a
stable distribution. Also, from the comments after the definition of
$Q$ in the introduction it is evident that under $Q$, the environment
$\w$ is stationary under shifts of the ladder times
$\nu_i$. Thus, under $Q$, $\{ E_\w^{\nu_{i-1}} T_{\nu_i} \}_{i\in \Z}$
is a stationary sequence of random variables. Therefore, it is reasonable
to expect that $n^{-1/s}E_\w T_{\nu_n} = n^{-1/s}
\sum_{i=1}^n E_\w^{\nu_{i-1}} T_{\nu_i}$ converge in distribution
to a stable distribution of index $s$. The main obstacle to proving
this is that the random variables $E_\w^{\nu_{i-1}} T_{\nu_i}$ are
not independent. This dependence, however, is rather weak. The strategy
of the proof of Theorem \ref{refstable} is to first show that we need
only consider the blocks where the expected crossing time
$E_\w^{\nu_{i-1}} T_{\nu_i}$ is relatively large. These blocks will
then be separated enough to make the expected crossing times
essentially independent.

For every $k\in \Z$, define
\begin{equation}
	M_k:=\max
\{\Pi_{\nu_{k-1}, j} : \nu_{k-1}\leq j < \nu_k \}. \label{Mdef}
\end{equation}
Theorem 1 in
\cite{iEV} gives that there exists a constant $C_5>0$ such that
\begin{equation}
Q(M_1 > x)\sim C_5 x^{-s}.\label{Mtail}
\end{equation}
Thus $M_1$ and $E_\w T_\nu$ have similar tails under $Q$. We will now show
that $E_\w T_\nu$ cannot be too much larger than $M_1$.
From \eqref{QET} we have that
\begin{equation}
E_\w T_\nu = \nu + 2 \sum_{j=0}^{\nu-1} W_j = \nu+2
W_{-1}R_{0,\nu-1} + 2 \sum_{i=0}^{\nu-1} R_{i,\nu-1}.
\label{ETnuexpand}
\end{equation}
From the definitions of $\nu$ and $M_1$ we have that $R_{i,\nu-1}
\leq (\nu - i) M_1 \leq \nu M_1$ for any $0\leq i < \nu$.
Therefore, $E_\w T_\nu \leq \nu + 2 W_{-1}\nu M_1 + 2 \nu^2 M_1$.
Thus, given any $0<\a<\b$ and $\d>0$ we have
\begin{align}
\label{TbigMsmall}
Q(E_\w T_\nu > \d n^{\b}, M_1 \leq n^{\a})
&\leq Q(\nu +2 W_{-1}\nu n^{\a}  + 2\nu^2 n^{\a} > \d n^{\b})  \\
&\leq Q(W_{-1} > n^{(\b-\a)/2}) + Q\left(\nu^2 > n^{(\b-\a)/2}\right)
= o\left(e^{-n^{(\b-\a)/5}}\right),
\nonumber
\end{align}
where the second inequality holds for all $n$ large enough and the
last equality is a result of Lemmas \ref{nutail} and \ref{Wtail}.
We now show that only the ladder times with $M_k>n^{(1-\e)/s}$
contribute to the limiting distribution of $n^{-1/s} E_\w
T_{\nu_n}$.
\begin{lem}\label{smallblocks}
Assume $s<1$. Then for any $\e>0$ and any $\d>0$ there exists an $\eta> 0$ such that
\[
\lim_{n\ra\infty} Q\left( \sum_{i=1}^n (E_\w^{\nu_{i-1}} T_{\nu_i}
) \mathbf{1}_{M_i \leq n^{(1-\e)/s}}
> \d n^{1/s} \right) = o(n^{-\eta})
\,.\]
\end{lem}
\begin{proof}
First note that
\begin{align*}
Q \left( \sum_{i=1}^n (E_\w^{\nu_{i-1}} T_{\nu_i} ) \mathbf{1}_{
M_i \leq n^{(1-\e)/s}} > \d n^{1/s} \right) &\leq Q\left(
\sum_{i=1}^n (E_\w^{\nu_{i-1}} T_{\nu_i} )
\mathbf{1}_{E_\w^{\nu_{i-1}} T_{\nu_i}
 \leq n^{(1-\frac{\e}{2} )/s}} > \d n^{1/s}  \right) \\
&\quad\quad + n Q\left( E_\w T_{\nu}
 > n^{(1-\frac{\e}{2})/s}, M_1 \leq n^{(1-\e)/s} \right)
\,.
\end{align*}
By \eqref{TbigMsmall}, the last term above decreases
faster than any power of $n$. Thus it is enough to prove that for
any $\d,\e>0$ there exists an $\eta>0$ such that
\[
Q\left( \sum_{i=1}^n (E_\w^{\nu_{i-1}} T_{\nu_i} )
\mathbf{1}_{E_\w^{\nu_{i-1}} T_{\nu_i}
 \leq n^{(1-\e )/s}} > \d n^{1/s}  \right) = o( n^{-\eta} )\,.
\]
Next, pick $C\in \left(1, \frac{1}{s} \right)$ and let
$J_{C,\e,k,n}:= \left\{ i\leq n : n^{(1-C^k \e)/s} <
E_\w^{\nu_{i-1}} T_{\nu_i} \leq n^{(1-C^{k-1}\e)/s} \right\}$. Let
$k_0=k_0(C,\e)$ be the smallest integer such that $(1-C^k\e) \leq
0$. Then for any $k < k_0$ we have
\begin{align*}
Q \left( \sum_{i\in J_{C,\e,k,n} }  E_\w^{\nu_{i-1}} T_{\nu_i}  >
\d n^{1/s} \right)
 &\leq Q \left( \# J_{C,\e,k,n} > \d n^{1/s-(1-C^{k-1}\e)/s} \right)\\
&\leq \frac{n Q( E_\w T_{\nu}  > n^{(1-C^{k}\e)/s})}{\d
n^{C^{k-1}\e/s}} \sim \frac{K_\infty}{\d}
n^{-C^{k-1}\e(\frac{1}{s}-C)}
\,,\end{align*}
where the asymptotics in the last line above is from Theorem
\ref{Tnutail}. Letting $\eta=
\frac{\e}{2}\left(\frac{1}{s}-C\right)$ we have for any $k < k_0$
that
\begin{equation}
Q \left( \sum_{i\in J_{C,\e,k,n} } E_\w^{\nu_{i-1}} T_{\nu_i}  >
\d n^{1/s} \right) = o(n^{-\eta}). \label{klessk0}
\end{equation}
Finally, note that
\begin{equation}
Q\left( \sum_{i=1}^n (E_\w^{\nu_{i-1}} T_{\nu_i})
\mathbf{1}_{E_\w^{\nu_{i-1}} T_{\nu_i} \leq n^{(1-C^{k_0-1}\e)/s}}
\geq \d n^{1/s} \right) \leq
\mathbf{1}_{n^{1+(1-C^{k_0-1}\e)/s} \geq \d n^{1/s}}. \label{k0}
\end{equation}
However, since $C^{k_0} \e \geq 1 > Cs$ we have $C^{k_0-1}\e
> s $, which
implies that
the right side of \eqref{k0} vanishes for all $n$ large enough.
Therefore, combining \eqref{klessk0} and \eqref{k0} we have
\begin{align*}
Q\left( \sum_{i=1}^n (E_\w^{\nu_{i-1}} T_{\nu_i})
\mathbf{1}_{E_\w^{\nu_{i-1}} T_{\nu_i}
 \leq n^{(1-\e )/s}} > \d n^{1/s} \right)
&\leq  \sum_{k=1}^{k_0-1} Q \left( \sum_{i\in J_{C,\e,k,n} }
E_\w^{\nu_{i-1}} T_{\nu_i}  > \frac{\d}{k_0}
n^{1/s} \right) \\
&
\!\!\!\!\!
\!\!\!\!\!
\!\!\!\!\!
\!\!\!\!\!
\!\!\!\!\!
\!\!\!\!\!
\!\!\!\!\!
\!\!\!\!\!
\!\!\!\!\!
\!\!\!\!\!
+ Q\left( \sum_{i=1}^n (E_\w^{\nu_{i-1}} T_{\nu_i})
\mathbf{1}_{E_\w^{\nu_{i-1}} T_{\nu_i} \leq n^{(1-C^{k_0-1}\e)/s}}
\geq \frac{\d}{k_0}
n^{1/s} \right)
= o(n^{-\eta}).
\end{align*}
\end{proof}
In order to make the crossing times of the significant blocks essentially
independent, we introduce some reflections to the RWRE. For
$n=1,2,\ldots$, define
\begin{equation}
b_n:= \lfloor \log^2(n) \rfloor. \label{bdef}
\end{equation}
Let $\bar{X}_t^{(n)}$ be the random walk that is the same as $X_t$
with the added condition that after reaching $\nu_k$ the
environment is modified by setting $\w_{\nu_{k-b_n}} = 1 $ , i.e.
never allow the walk to backtrack more than $\log^2(n)$ ladder
times. 
We couple $\bar{X}_t^{(n)}$ with the random walk $X_t$ in such a way that $\bar{X}_t^{(n)} \geq X_t$ with equality holding until
the first time $t$ when the walk $\bar{X}_t^{(n)}$ reaches a modified environment location.
%
Denote by $\bar{T}_{x}^{(n)}$ the corresponding hitting
times for the walk $\bar{X}_t^{(n)}$. The following lemmas show that we can add reflections to
the random walk without changing the expected crossing time by
very much.
\begin{lem} \label{ETdiff}
There exist $B,\d' > 0$ such that for any $x>0$
\[
Q\left( E_\w T_\nu - E_\w \bar{T}_\nu^{(n)} > x \right) \leq B
(x^{-s}\vee 1) e^{-\d' b_n}.
\]
\end{lem}
\begin{proof}
First, note that for any $n$ the formula for $E_\w
\bar{T}_\nu^{(n)}$ is the same as for $E_\w T_\nu$ in
\eqref{ETnuexpand} except with $\rho_{\nu_{-b_n}} = 0$. Thus $E_\w
T_\nu$ can be written as
\begin{equation}
E_\w T_\nu = E_\w \bar{T}_\nu^{(n)} + 2 (1+W_{\nu_{-b_n}-1})
\Pi_{\nu_{-b_n } , -1}R_{0,\nu-1}. \label{reflectexpand}
\end{equation}
Now, since $\nu_{-b_n} \leq -b_n$ we have
\[
Q\left(
\Pi_{\nu_{-b_n},-1}
> e^{-c b_n} \right) \leq \sum_{k=b_n}^\infty Q\left( \Pi_{-k,-1}
> e^{-c k} \right) \leq \sum_{k=b_n}^\infty
\frac{1}{P(\mathcal{R})} P\left( \Pi_{-k,-1} > e^{-c k} \right).
\]
Applying \eqref{LDPrho},
we have that for any $0<c<-E_P \log \rho$
there exist $A',\d_c > 0$ such that 
\[
 Q\left( \Pi_{\nu_{-b_n},-1}
> e^{-c b_n} \right) \leq A' e^{-\d_c b_n} .
\]
Therefore, for any
$x>0$,
\begin{align}
Q\left(  E_\w T_{\nu} - E_\w \bar{T}_{\nu}^{(n)} > x \right) &\leq
Q\left( 2(1+W_{\nu_{-b_n}-1}) \Pi_{\nu_{-b_n},-1} R_{0,\nu-1} > x
\right) \nonumber \\
& \leq Q\left( 2(1+W_{\nu_{-b_n}-1}) R_{0,\nu-1} >
x e^{c b_n} \right) + A' e^{-\d_c b_n} \nonumber \\
&=  Q\left( 2(1+W_{-1}) R_{0,\nu-1} > x e^{c b_n} \right) + A'
e^{-\d_c b_n},  \label{ETd1}
\end{align}
where the equality in the second line is due to the fact that the
blocks of the environment are i.i.d under $Q$. Also, from
\eqref{ETnuexpand} and Theorem \ref{Tnutail} we have
\begin{equation}
Q\left( 2(1+W_{-1}) R_{0,\nu-1} > x e^{c b_n} \right) \leq Q\left(
E_\w T_\nu > x e^{c b_n} \right) \sim K_\infty x^{-s} e^{-c s
b_n}. \label{ETd2}
\end{equation}
Combining \eqref{ETd1} and \eqref{ETd2} finishes the proof.
\end{proof}
\begin{lem}\label{reftail}
For any $x>0$ and $\e>0$ we have that 
\begin{equation}
\lim_{n\ra\infty} n Q\left(E_\w \bar{T}_\nu^{(n)} > x n^{1/s}, M_1 > n^{(1-\e)/s} \right)
= K_\infty x^{-s}. \label{refTnutail}
\end{equation}
\end{lem}
\begin{proof}
Since adding reflections only decreases the crossing times, we can
get an upper bound using Theorem \ref{Tnutail},
that is
\begin{align}
\limsup_{n\ra\infty} n Q\left(E_\w \bar{T}_\nu^{(n)} > x n^{1/s}, M_1 > n^{(1-\e)/s} \right)
&\leq \limsup_{n\ra\infty} n Q(E_\w T_\nu > x n^{1/s}) = K_\infty x^{-s}. \label{refTnub}
\end{align}
To get a lower bound we first note that for any $\d>0$,
\begin{align}
Q\left( E_\w T_\nu > (1+\d)x n^{1/s} \right) &\leq  Q\left( E_\w
\bar{T}_\nu^{(n)} > x n^{1/s}, M_1 > n^{(1-\e)/s} \right)
 + Q\left( E_\w T_\nu - E_\w \bar{T}_\nu^{(n)} > \d x
n^{1/s} \right) \nonumber \\
&\quad + Q\left( E_\w T_\nu > (1+\d)x n^{1/s}, M_1 \leq
n^{(1-\e)/s} \right) \nonumber \\
&\leq Q\left( E_\w \bar{T}_\nu^{(n)} > x n^{1/s}, M_1 >
n^{(1-\e)/s} \right) + o(1/n), \label{r1}
\end{align}
where the second inequality is from \eqref{TbigMsmall} and Lemma
\ref{ETdiff}. Again using Theorem \ref{Tnutail} we have
\begin{align}
\liminf_{n\ra\infty} n Q\left(E_\w \bar{T}_\nu^{(n)} > x n^{1/s}, M_1 > n^{(1-\e)/s} \right) 
&\geq \liminf_{n\ra\infty} n Q\left( E_\w T_\nu > (1+\d)x n^{1/s} \right) - o(1) \nonumber \\
&= K_\infty (1+\d)^{-s}x^{-s}. \label{refTnulb}
\end{align}
Thus, by applying \eqref{refTnub}
and \eqref{refTnulb} and then letting $\d\ra 0$ we get \eqref{refTnutail}.
\end{proof}

\noindent
Our general strategy is
to show that the partial sums
\[
\frac{1}{n^{1/s}} \sum_{k=1}^n E_\w^{\nu_{k-1}} \bar{T}_{\nu_k}^{(n)} \mathbf{1}_{M_k>n^{(1-\e)/s} }
\]
converge in distribution to a stable law of parameter $s$.
To establish this,
we will need bounds on the mixing properties of the sequence $E_\w^{\nu_{k-1}}
\bar{T}_{\nu_k}^{(n)}\mathbf{1}_{M_k>n^{(1-\e)/s} }$. As in
\cite{kGPD}, we say that an array $\{ \xi_{n,k}: k\in \Z, n\in
\N\}$ which is stationary in rows is $\a-$mixing if
$\lim_{k\ra\infty}\limsup_{n\ra\infty} \a_n(k) = 0$, where
\[
\a_n(k) := \sup \left\{ |P(A\cap B)-P(A)P(B)|: A\in
\s\left(\ldots, \xi_{n,-1}, \xi_{n,0} \right),
B\in  \s \left(\xi_{n,k},\xi_{n,k+1}, \ldots  \right) \right\}.
\]
\begin{lem} \label{alphamixing}
For any $0<\e< \frac{1}{2}$, under the measure $Q$, the array of
random variables \\$\{ E_\w^{\nu_{k-1}} \bar{T}_{\nu_k}^{(n)}
\mathbf{1}_{M_k>n^{(1-\e)/s} } \}_{k\in\Z, n\in\N}$ is $\a$-mixing,
with
\[
\sup_{k\in [1,\log^2 n]}\alpha_n(k)=o(n^{-1+2\epsilon}),
\quad \alpha_n(k)=0,\quad \forall k>\log^2 n.
\]
\end{lem}
\begin{proof}
Fix $\e \in (0,\frac{1}{2})$. For ease of notation, define
$\xi_{n,k}:= E_\w^{\nu_{k-1}} \bar{T}_{\nu_k}^{(n)}
\mathbf{1}_{M_k>n^{(1-\e)/s} }$. As we mentioned before, under $Q$
the environment is stationary
under shifts of the sequence of ladder locations and
thus $\xi_{n,k}$ is stationary in rows under $Q$.

If $k>\log^2(n)$, then because of the reflections,
$\s\left(\ldots, \xi_{n,-1}, \xi_{n,0} \right)$ and
$\s \left(\xi_{n,k},\xi_{n,k+1}, \ldots  \right)$ are independent and
so $\a_n(k)= 0$.
To handle the case when $k\leq \log^2(n)$,
fix $ A\in
\s\left(\ldots, \xi_{n,-1}, \xi_{n,0} \right)$ and $B\in  \s
\left(\xi_{n,k},\xi_{n,k+1}, \ldots \right) $, and
define the event
\[
C_{n,\e}:=\{M_j \leq n^{(1-\e)/s}, \text{for } 1\leq j \leq b_n \}
= \{\xi_{n,j} = 0, \text{for } 1\leq j \leq b_n \}.
\]
For any $j> b_n$, we have that $\xi_{n,j}$ only depends on the
environment to the right of zero. Thus,
\[
Q(A\cap B \cap C_{n,\e})= Q(A)Q(B\cap C_{n,\e})
\]
since $B \cap C_{n,\e} \in \s(\w_0,\w_1,\ldots)$. Also, note that
by \eqref{Mtail} we have $Q(C_{n,\e}^c ) \leq b_n Q(M_1 >
n^{(1-\e)/s}) = o(n^{-1+2\e})$. Therefore,
\begin{align*}
|Q(A\cap B)-Q(A)Q(B)| & \leq |Q(A\cap B) - Q(A\cap B\cap C_{n,\e})|\\
&\quad\quad  + |Q(A\cap B\cap C_{n,\e}) - Q(A)Q(B\cap C_{n,\e})| \\
&\quad\quad + Q(A)|Q(B\cap C_{n,\e})-Q(B)|
 \leq 2 Q(C_{n,\e}^c) = o(n^{-1+2\e})
\end{align*}
\end{proof}
\begin{proof}[\textbf{Proof of Theorem \ref{refstable}}]$\left.\right.$ \\
First, we show that the partial sums
\[
\frac{1}{n^{1/s}} \sum_{k=1}^n  E_\w^{\nu_{k-1}} \bar{T}_{\nu_k}^{(n)} \mathbf{1}_{M_k>n^{(1-\e)/s} }
\]
converge in distribution to a stable random variable of parameter $s$.
To this end, we will
apply \cite[Theorem 5.1(III)]{kGPD}. We now verify
the conditions of that theorem.
The first condition that needs to be satisfied is:
\[
\lim_{n\ra\infty} n Q\left( n^{-1/s} E_\w \bar{T}_{\nu}^{(n)} \mathbf{1}_{M_1>n^{(1-\e)/s} } > x \right) = K_\infty x^{-s}.
\]
However, this is exactly the content of Lemma \ref{reftail}.\\
Secondly, we need a sequence $m_n$ such that $m_n\ra\infty$,
$m_n=o(n)$ and $n\a_n(m_n) \ra 0$ and such that for any $\d>0$,
\begin{equation}
\lim_{n\ra\infty} \sum_{k=1}^{m_n} n Q\left( E_\w \bar{T}_{\nu}^{(n)}
\mathbf{1}_{M_1>n^{(1-\e)/s} } > \d n^{1/s},
E_\w^{\nu_{k}} \bar{T}_{\nu_{k+1}}^{(n)}
\mathbf{1}_{M_{k+1}>n^{(1-\e)/s} } > \d n^{1/s} \right) = 0. \label{mncond}
\end{equation}
However, by the independence of $M_1$ and $M_{k+1}$ for any
$k\geq 1$, the probability inside the sum is less
than $Q(M_1> n^{(1-\e)/s})^2$. By \eqref{Mtail} this last
expression is $ \sim C_5 n^{-2+2\e} $.
Thus letting $m_n = n^{1/2- \e}$ yields \eqref{mncond}. (Note that by
Lemma \ref{alphamixing}, $n\a_n(m_n) = 0$ for all $n$ large enough.) \\
Finally, we need to show that
\begin{equation}
\lim_{\d\ra 0}\limsup_{n\ra\infty} n E_Q \left[
n^{-1/s} E_\w \bar{T}_{\nu}^{(n)} \mathbf{1}_{M_1>n^{(1-\e)/s} }
\mathbf{1}_{E_\w \bar{T}_{\nu}^{(n)}  \leq \d} \right] = 0 \,.
\label{truncexp}
\end{equation}
Now, by \eqref{refTnub} there exists a constant $C_6>0$ such that
for any $x > 0$,
\[
Q\left( E_\w \bar{T}_{\nu}^{(n)} > x
n^{1/s} , M_1
> n^{(1-\e)/s} \right) \leq C_6 x^{-s}\frac{1}{n}.
\]
Then using this we have
\begin{align*}
n E_Q \left[ n^{-1/s} E_\w \bar{T}_{\nu}^{(n)} \mathbf{1}_{M_1>n^{(1-\e)/s} }
\mathbf{1}_{E_\w \bar{T}_{\nu}^{(n)}  \leq \d} \right] &=
n \int_0^\d Q\left( E_\w \bar{T}_{\nu}^{(n)} > x n^{1/s} ,
M_1 > n^{(1-\e)/s} \right)  dx \\
&\leq C_6 \int_0^\d x^{-s} dx = \frac{C_6 \d^{1-s}}{1-s}\,,
\end{align*}
where the last
integral is finite since $s<1$. \eqref{truncexp} follows.

Having checked all its hypotheses,
\cite[Theorem 5.1(III)]{kGPD} applies and yields
that there exists a $b'>0$ such that
\begin{equation}
Q\left( \frac{1}{n^{1/s}} \sum_{k=1}^n E_\w^{\nu_{k-1}}
\bar{T}_{\nu_k}^{(n)} \mathbf{1}_{M_k>n^{(1-\e)/s} } \leq x \right) =
L_{s,b'}(x)\,, \label{refdist}
\end{equation}
where the characteristic function for the distribution $L_{s,b'}$
is given in \eqref{char}. To get the limiting distribution of
$\frac{1}{n^{1/s}} E_\w T_{\nu_n}$ we use \eqref{reflectexpand}
and re-write this as
\begin{align}
\frac{1}{n^{1/s}} E_\w T_{\nu_n} &= \frac{1}{n^{1/s}}
\sum_{k=1}^n E_\w^{\nu_{k-1}} \bar{T}_{\nu_k}^{(n)}
\mathbf{1}_{M_k > n^{(1-\e)/s}} \label{bb}\\
&\quad\quad + \frac{1}{n^{1/s}} \sum_{k=1}^n
E_\w^{\nu_{k-1}} \bar{T}_{\nu_k}^{(n)}
\mathbf{1}_{M_k \leq n^{(1-\e)/s}}  \label{sb}\\
&\quad\quad + \frac{1}{n^{1/s}} \left( E_\w T_{\nu_n} - E_\w
\bar{T}_{\nu_n}^{(n)} \right)\,.
 \label{rb}
\end{align}
Lemma \ref{smallblocks} gives that \eqref{sb} converges in
distribution (under Q) to 0. Also, we can use Lemma \ref{ETdiff}
to show that \eqref{rb} converges in distribution to 0 as well.
Indeed, for any $\d>0$
\begin{align*}
Q\left( E_\w T_{\nu_n} - E_\w \bar{T}_{\nu_n}^{(n)} > \d n^{1/s}
\right) &\leq n Q\left( E_\w T_{\nu} - E_\w \bar{T}_{\nu}^{(n)}
> \d n^{1/s-1} \right) = \bigo\left( n^s e^{-\d' b_n} \right)
\,.\end{align*}
Therefore $n^{-1/s}E_\w T_{\nu_n}$ has the same limiting
distribution (under $Q$) as the right side of \eqref{bb}, which by
\eqref{refdist} is an $s$-stable distribution with distribution
function $L_{s,b'}$.
\end{proof}
\end{section}

\begin{section}{Localization along a subsequence}
\label{localization} The goal of this section is to show
when $s<1$ that $P$-a.s. there exists a subsequence $t_m= t_m(\w)$
of times such that the RWRE is essentially located in a section of
the environment of length $\log^2(t_m)$. This will essentially be
done by finding a ladder time whose crossing time is \emph{much}
larger than all the other ladder times before it. As a first step
in this direction we prove that with strictly positive probability
this happens in the first $n$ ladder locations. Recall the definition
of $M_k$, c.f. \eqref{Mdef}.
\begin{lem}\label{QBB}
Assume $s<1$. Then for any $C>1$ we have
\[
\liminf_{n\ra\infty} Q\left( \exists k \in[1, n/2]: M_k \geq C \!\!\!\!\!\! \sum_{j\in[1,n]\backslash \{k\}} \!\!\!\!\!\! E_\w^{\nu_{j-1}} \bar{T}_{\nu_j}^{(n)}\right) > 0\,.
\]
\end{lem}
\begin{proof}
Recall that $\bar{T}^{(n)}_x$ is the hitting time of $x$ by the RWRE modified so that it never backtracks $b_n=\lfloor \log^2(n) \rfloor$ ladder locations. \\
To prove the lemma, first note that since $C>1$ and $E_\w^{\nu_{k-1}} \bar{T}^{(n)}_{\nu_k} \geq M_k$ there can only be at most one $k\leq n$ with $M_k \geq C \sum_{k\neq j \leq n} E_\w^{\nu_{j-1}} \bar{T}_{\nu_j}^{(n)}$. Therefore
\begin{equation}
Q\left( \exists k\in[1,n/2]: M_k \geq  C \!\!\!\!\!\! \sum_{j\in[1,n]\backslash \{k\}} \!\!\!\!\!\! E_\w^{\nu_{j-1}} \bar{T}^{(n)}_{\nu_j} \right) = \sum_{k=1}^{n/2} Q\left( M_k \geq C \!\!\!\!\!\! \sum_{j\in[1,n]\backslash \{k\}} \!\!\!\!\!\! E_\w^{\nu_{j-1}} \bar{T}_{\nu_j}^{(n)}\right) \label{onebigblock}
\end{equation}
Now, define the events
\begin{equation}
F_{n}:=\{\nu_j-\nu_{j-1} \leq b_n ,\quad \forall j\in(-b_n,n] \},
\quad G_{k,n,\e}:= \{ M_j \leq n^{(1-\e)/s},\quad \forall j\in(k,k+b_n] \}. \label{FGevents}
\end{equation}
$F_n$ and $G_{k,n,\e}$ are both \emph{typical} events. Indeed,
from Lemma \ref{nutail} $Q(F_{n}^c) \leq (b_n+n) Q(\nu > b_n) =
\bigo(n e^{-C_2 b_n})$, and from \eqref{Mtail} we have
$Q(G_{k,n,\e}^c) \leq b_n Q(M_1 > n^{(1-\e)/s}) = o(n^{-1+2\e})$.
Now, from \eqref{QET} adjusted for reflections we have for any $j\in[1,n]$
that
\begin{align*}
E_\w^{\nu_{j-1}} \bar{T}_{\nu_j}^{(n)} &=(\nu_j - \nu_{j-1}) +
2\sum_{l=\nu_{j-1}}^{\nu_j-1} W_{\nu_{j-1-b_n},l} \\
&= (\nu_j - \nu_{j-1}) + 2\sum_{ \nu_{j-1} \leq i\leq l < \nu_j} \Pi_{i,l} + 2\sum_{\nu_{j-1-b_n}< i < \nu_{j-1} \leq l < \nu_j} \Pi_{i,\nu_{j-1}-1}\Pi_{\nu_{j-1}, l} \\
&\leq (\nu_j - \nu_{j-1}) + 2\left(\nu_j-\nu_{j-1}\right)^2 M_j +
2(\nu_j-\nu_{j-1})(\nu_{j-1}-\nu_{j-1-b_n})M_j,
\end{align*}
where in the last inequality we used the facts that $\Pi_{\nu_{j-1}, i-1} \geq 1$ for $\nu_{j-1} < i < \nu_j$ and $\Pi_{i, \nu_{j-1}-1} < 1$ for all $i<
\nu_{j-1}$. Then, on the event $F_n\cap
G_{k,n,\e}$ we have for $k+1\leq j\leq k+b_n$ that
\begin{align*}
E_\w^{\nu_{j-1}} \bar{T}_{\nu_j}^{(n)} \leq b_n + 2 b_n^2
n^{(1-\e)/s} + 2 b_n^3 n^{(1-\e)/s} \leq 5 b_n^3 n^{(1-\e)/s},
\end{align*}
where for the first inequality we used that on the event $F_n\cap
G_{k,n,\e}$ we have $\nu_j-\nu_{j-1} \leq b_n$ and $M_1 \leq
n^{(1-\e)/s}$. Then, using this we get
\begin{align*}
Q\left( M_k \geq C \!\!\!\!\!\! \sum_{j\in[1,n]\backslash \{k\}} \!\!\!\!\!\! E_\w^{\nu_{j-1}} \bar{T}_{\nu_j}^{(n)}\right)
& \geq Q\left( M_k \geq C \left( E_\w
\bar{T}_{\nu_{k-1}}^{(n)} +  5b_n^4 n^{(1-\e)/s} +
E_\w^{\nu_{k+b_n}}\bar{T}_{\nu_n}^{(n)} \right), F_n, G_{k,n,\e} \right) \\
& \geq Q\left( M_k \geq C n^{1/s}, \quad \nu_{k}-\nu_{k-1} \leq b_n \right)\\
&\quad \times Q\left( E_\w \bar{T}_{\nu_{k-1}}^{(n)} +  5b_n^4
n^{(1-\e)/s} + E_\w^{\nu_{k+b_n}}\bar{T}_{\nu_n}^{(n)} \leq
n^{1/s},  \tilde{F}_n, G_{k,n,\e} \right),
\end{align*}
where $\tilde{F}_n := \{\nu_j-\nu_{j-1} \leq b_n, \quad \forall j\in(-b_n, n]\backslash\{k\} \} \supset F_n $.
In the last inequality we used the fact that
$E_\w^{\nu_{j-1}} \bar{T}_{\nu_j}^{(n)}$ is independent of
$M_k$ for $j<k$ or $j>k+b_n$. Note that we can replace
$\tilde{F}_n$ by $F_n$ in the last line above because it will only
make the probability smaller. Then, using the above and the fact that
$E_\w \bar{T}_{\nu_{k-1}}^{(n)} + E_\w^{\nu_{k+b_n}}\bar{T}_{\nu_n}^{(n)}
\leq E_\w T_{\nu_n}$ we have
\begin{align*}
Q&\left( M_k \geq C \!\!\!\!\!\! \sum_{j\in[1,n]\backslash \{k\}}
\!\!\!\!\!\! E_\w^{\nu_{j-1}} \bar{T}_{\nu_j}^{(n)}\right)\\
&\quad\geq Q\left( M_k \geq C n^{1/s}, \quad
\nu_{k}-\nu_{k-1} \leq b_n \right) Q\left(E_\w T_{\nu_n}
\leq n^{1/s}-5 b_n^4 n^{(1-\e)/s}, F_n, G_{k,n,\e}\right)\\
&\quad\geq \left( Q(M_1\geq Cn^{1/s}) -
Q(\nu > b_n)  \right)  \left( Q( E_\w T_{\nu_n}
\leq n^{1/s}(1-5b_n^4 n^{-\e/s}) ) - Q(F_n^c) - Q(G_{k,n,\e}^c)   \right)\\
&\quad \sim C_5 C^{-s} L_{s,b'}(1) \frac{1}{n}\,,
\end{align*}
where the asymptotics in the last line are from
\eqref{Mtail} and Theorem \ref{refstable}. Combining the last
display and \eqref{onebigblock} proves the lemma.
\end{proof}
In Section \ref{stablecrossing}, we showed that the proper scaling
for $E_\w T_{\nu_n}$ (or $E_\w \bar{T}^{(n)}_{\nu_n}$) was
$n^{-1/s}$. The following lemma gives a bound on the moderate
deviations, under the measure $P$.
\begin{lem}\label{mdevu}
Assume $s\leq 1$. Then for any $ \d>0$,
\[
P\left( E_\w T_{\nu_n} \geq n^{1/s + \d} \right) = o(n^{-\d s/2})
\,.\]
\end{lem}
\begin{proof}
First, note that
\begin{equation}
P(E_\w T_{\nu_n} \geq n^{1/s+ \d}) \leq P(E_\w T_{2\bar\nu n} \geq n^{1/s+\d}) + P(\nu_n \geq 2 \bar\nu n ) \label{md1}
\,,\end{equation}
where $\bar\nu := E_P \nu$. To handle the second term on the right
hand side of \eqref{md1} we note that $\nu_n$ is the sum of
$n$ i.i.d. copies of $\nu$, and that $\nu$ has exponential tails (by Lemma \ref{nutail}).
Therefore, Cram\'er's Theorem \cite[Theorem 2.2.3]{dzLDTA}
gives that $P({\nu_n}/{n} \geq
2\bar\nu ) = \bigo( e^{-\d' n} )$
for some $\d' > 0$.\\
To handle the first term on the right hand side of \eqref{md1} we
note that for any $\gamma < s$ we have $E_P(E_\w T_1)^\gamma <
\infty$ This follows from the fact that $P(E_\w T_1
> x) = P(1 + 2W_0 > x) \sim K 2^s x^{-s}$ by
\eqref{QET} and \eqref{PWtail}. Then, by Chebychev's inequality and the fact
that $\gamma < s \leq 1$ we have
\begin{equation}
P\left( E_\w T_{2\bar\nu n} \geq n^{1/s+\d} \right) \leq
\frac{E_P\left( \sum_{k=1}^{2 \bar\nu n} E_\w^{k-1} T_k
\right)^\gamma}{n^{\gamma(1/s+\d)}} \leq \frac{2 \bar\nu n
E_P(E_\w T_1)^\gamma }{n^{\gamma(1/s+\d)}}\,. \label{Tnbig}
\end{equation}
Then, choosing $\gamma$ arbitrarily close to $s$ we can have that this last term is $o(n^{-\d s/2})$.
\end{proof}
Throughout the remainder of the chapter we will use the following subsequences of integers:
\begin{equation}
  n_k := 2^{2^k},\qquad d_k := n_k-n_{k-1} \label{nkdkdef}
\end{equation}
Note that $n_{k-1} = \sqrt{n_k}$ and so $d_k \sim n_k$ as $k\ra\infty$.
\begin{cor}\label{subseq1}
For any $k$ define
\[
\mu_k := \max \left\{ E_\w^{\nu_{j-1}} \bar{T}_{\nu_j}^{(d_k)} : n_{k-1} < j \leq n_k  \right\}.
\]
If $s<1$, then 
\[
\lim_{k\ra\infty} \frac{ E_\w^{\nu_{n_{k-1}}} \bar{T}_{\nu_{n_{k}}}^{(d_k)} - \mu_k }{E_\w \bar{T}_{\nu_{n_k}}^{(d_k)} - \mu_k} = 1, \quad P-a.s.
\]
\end{cor}
\begin{proof}
Let $\e>0$. Then,
\begin{align}
\label{shortlong}
P\left( \frac{ E_\w^{\nu_{n_{k-1}}} \bar{T}_{\nu_{n_{k}}}^{(d_k)} - \mu_k }{E_\w \bar{T}_{\nu_{n_k}}^{(d_k)} - \mu_k} \leq 1-\e \right)
&= P\left( \frac{E_\w \bar{T}_{\nu_{n_{k-1}}}^{(d_k)}}{E_\w 
\bar{T}_{\nu_{n_k}}^{(d_k)} - \mu_k} \geq \e \right)\\
&\leq P\left( E_\w \bar{T}_{\nu_{n_{k-1}}}^{(d_k)} 
\geq n_{k-1}^{1/s+\d} \right) + 
P\left( E_\w \bar{T}_{\nu_{n_k}}^{(d_k)} -
\mu_k \leq \e^{-1} n_{k-1}^{1/s+\d} \right) 
\,.\nonumber\end{align}
Lemma \ref{mdevu} gives that $P\left( E_\w \bar{T}_{\nu_{n_{k-1}}}^{(d_k)} \geq n_{k-1}^{1/s+\d} \right) \leq  P\left( E_\w T_{\nu_{n_{k-1}}} \geq n_{k-1}^{1/s+\d} \right) = o(n_{k-1}^{-\d s/2})$.
To handle the second term in the right side of
\eqref{shortlong},
note that if $\d<\frac{1}{3s}$, then the subsequence 
$n_k$ grows fast enough such that for all $k$ large 
enough $n_k^{1/s-\d} \geq \e^{-1} n_{k-1}^{1/s+\d}$.  
Therefore, for $k$ sufficiently large and $\d< \frac{1}{3s}$ we have
\[
P\left( E_\w \bar{T}_{\nu_{n_k}}^{(d_k)} - \mu_k \leq 
\e^{-1} n_{k-1}^{1/s+\d} \right) \leq P\left( 
E_\w \bar{T}_{\nu_{n_k}}^{(d_k)} - \mu_k \leq  n_{k}^{1/s-\d} \right).
\]
However, $ E_\w \bar{T}^{(d_k)}_{\nu_{n_k}} - \mu_k \leq  n_{k}^{1/s-\d} $ 
implies that $M_j < E_\w^{\nu_{j-1}} \bar{T}_{\nu_j}^{(d_k)}
\leq n_k^{1/s-\d}$ for at least $n_k-1$ of the $j\leq n_k$. Thus,
since $P(M_1> n_k^{1/s-\d})\sim C_5 n_k^{-1+\d s}$, we have that
\begin{align}
P\left( E_\w \bar{T}_{\nu_{n_k}}^{(d_k)} - \mu_k \leq \e^{-1}
n_{k-1}^{1/s+\d} \right)  \leq n_k \left( 1 -
P\left( M_1 > n_k^{1/s-\d} \right) \right)^{n_k-1}
 = o(e^{-n_k^{\d s/2}})\,. \label{longsmall}
\end{align}
Therefore, for any $\e>0$ and $\d<\frac{1}{3s}$ we have that
\[
P\left( \frac{ E_\w^{\nu_{n_{k-1}}} \bar{T}_{\nu_{n_{k}}}^{(d_k)}
- \mu_k }{E_\w \bar{T}_{\nu_{n_k}}^{(d_k)} - \mu_k}
\leq 1-\e \right) = o\left(n_{k-1}^{-\d s/2}\right).
\]
By our choice of $n_k$, the sequence $n_{k-1}^{-\d s/2}$ is
summable in $k$. Applying the Borel-Cantelli lemma completes the proof.
\end{proof}
\begin{cor}\label{ssdom}
Assume $s<1$. Then $P-$a.s. there exists a random subsequence
$j_m= j_m(\w)$ such that
\[
M_{j_m} \geq m^2 E_\w \bar{T}_{\nu_{j_m-1}}^{(j_m)}
\,.\]
\end{cor}
\begin{proof}
Recall the definitions of $n_k$ and $d_k$ in \eqref{nkdkdef}. Then for any $C>1$, define the event
\[
D_{k,C}:=\left\{\exists j\in (n_{k-1}, n_{k-1}+ d_k/2]: M_j \geq C\left(E_\w^{\nu_{n_{k-1}}}\bar{T}_{\nu_{j-1}}^{(d_k)} + E_\w^{\nu_j}\bar{T}_{\nu_{n_k}}^{(d_k)} \right) \right\}.
\]
Note that due to the reflections, the event $D_{k,C}$ depends only on the environment from $\nu_{n_{k-1}-b_{d_k}}$ to $\nu_{n_{k}}-1$. Then, since $n_{k-1} - b_{d_k} > n_{k-2}$ for all $k \geq 4$,
we have that the events $\{ D_{2k,C} \}_{k=2}^{\infty}$
are all independent. Also, since the events do not involve the
environment to the left of $0$ they have the same probability under
$Q$ as under $P$. Then since $Q$ is stationary under shifts of
$\nu_i$ we have that for $k\geq 4$,
\begin{align*}
P(D_{k,C}) = Q(D_{k,C}) = Q\left(
\exists j\in [1,d_k/2]: M_j \geq C\left(E_\w \bar{T}_{\nu_{j-1}}^{(d_k)}
+ E_\w^{\nu_j}\bar{T}_{\nu_{d_k}}^{(d_k)} \right)  \right).
\end{align*}
Thus for any $C>1$, we have by Lemma \ref{QBB} that
$\liminf_{k\ra\infty} P(D_{k,C}) > 0$.
This combined with the fact that the events $\{D_{2k,C}\}_{k=2}^\infty$
are independent gives that for any $C>1$ infinitely many of
the events $D_{2k,C}$ occur $P-a.s.$ Therefore, there exists a
subsequence $k_m$ of integers such that for each $m$,
there exists $j_m \in (n_{k_m-1} , n_{k_m-1}+d_{k_m}/2 ]$ such that
\[
M_{j_m} \geq 2m^2 \left( E_{\w}^{\nu_{n_{k_m-1}}}
\bar{T}_{\nu_{j_m-1}}^{(d_{k_m})} + E_\w^{\nu_{j_m}}
\bar{T}_{\nu_{n_{k_m}}}^{(d_{k_m})} \right) =
2m^2\left(E_\w^{\nu_{n_{k_m-1}}}
\bar{T}_{\nu_{n_{k_m}}}^{(d_{k_m})} - \mu_{k_m}\right)  ,
\]
where the second equality holds due to our choice of $j_m$, which
implies that
 $\mu_{k_m} = E_\w^{\nu_{j_m -1}}
 \bar{T}_{\nu_{j_m}}^{(d_{k_m})}$. Then,
 by Corollary \ref{subseq1} we have that for all $m$ large enough,
\[
M_{j_m}\geq 2m^2 \left( E_\w^{\nu_{k_m -1}}
\bar{T}_{\nu_{n_{k_m}}}^{(d_{k_m})} - \mu_{k_m} \right)
\geq m^2 \left( E_\w \bar{T}_{\nu_{n_{k_m}}}^{(d_{k_m})} -
\mu_{k_m} \right) \geq m^2 E_\w
\bar{T}_{\nu_{j_m-1}}^{(d_{k_m})},
\]
where the last inequality is because
$\mu_{k_m} = E_\w^{\nu_{j_m -1}} \bar{T}_{\nu_{j_m}}^{(d_{k_m})}$.
Now, for all $k$ large enough we have $n_{k-1} + d_k/2 < d_k$.
Thus, we may assume (by possibly choosing a further subsequence)
that $j_m < d_{k_m}$ as well, and since allowing less
backtracking only decreases the crossing time we have
\[
M_{j_m} \geq m^2 E_\w \bar{T}_{\nu_{j_m-1}}^{(d_{k_m})}
\geq m^2 E_\w \bar{T}_{\nu_{j_m-1}}^{(j_m)}
\,.\]
\end{proof}
The following lemma shows that the reflections that we have
been using this whole time really do not affect the random walk.
Recall the coupling of $X_t$ and $\bar{X}_t^{(n)}$ 
introduced after \eqref{bdef}.
\begin{lem} \label{bt}
\[
\lim_{n\ra\infty} P_\w\left( T_{\nu_{n-1}} \neq \bar{T}_{\nu_{n-1}}^{(n)} \right) = 0, \quad P-a.s.
\]
\end{lem}
\begin{proof}
Let $\e>0$. By Chebychev's inequality,
\[
 P\left( P_\w\left( T_{\nu_{n-1}} \neq \bar{T}_{\nu_{n-1}}^{(n)} \right)
> \e  \right) \leq \e^{-1} \P\left( T_{\nu_{n-1}} \neq \bar{T}_{\nu_{n-1}}^{(n)} 
\right) .
\]
 Thus by the Borel-Cantelli lemma it is
enough to prove that $\P\left( T_{\nu_{n-1}} \neq \bar{T}_{\nu_{n-1}}^{(n)}  
\right)$ is summable. Now, the event 
$T_{\nu_{n-1}} \neq \bar{T}_{\nu_{n-1}}^{(n)} $ implies that there is 
an $i < \nu_{n-1}$ such that after reaching $i$ for the first time,
the random walk then backtracks a distance of $b_{n}$. 
Thus, again letting $\bar\nu = E_P \nu$ we have
\begin{align*}
\P\left( T_{\nu_{n-1}} \neq \bar{T}_{\nu_{n-1}}^{(n)}  
\right) 
&\leq P(\nu_{n-1} \geq 2\bar\nu (n-1)) + \sum_{i=0}^{2\bar\nu (n-1)} \P^i(T_{i-b_{n}} 
< \infty ) \\
&=  P(\nu_{n-1} \geq 2\bar\nu (n-1)) + 2\bar\nu (n-1) \P(T_{-b_{n}} < \infty )
\end{align*}
As noted in Lemma \ref{mdevu}, $P(\nu_{n-1} \geq 2\bar\nu (n-1)) =
\bigo( e^{-\d' n})$, so we need only to show that $n\P(T_{-b_{n}} < \infty)$
is summable. However, \cite[Lemma 3.3]{gsMVSS} gives that 
there exists a constant $C_7$ such that for any $k\geq 1$\,,
\begin{equation}
\P( T_{-k} < \infty ) \leq e^{-C_7 k}\,. \label{backtracktail}
\end{equation}
Thus $n\P(T_{-b_{n}} < \infty) \leq n
e^{-C_7 b_{n}}$ which is summable by the definition of $b_n$.
\end{proof}
We define the random variable $N_t:= \max \{k: \exists n\leq t, 
X_n = \nu_k\}$ to be the maximum number of ladder 
locations crossed by the random walk by time $t$.
\begin{lem}\label{seperation}
\[
\lim_{t\ra\infty} \frac{\nu_{N_t}-X_t}{\log^2(t)} = 0, \quad \P-a.s.
\]
\end{lem}
\begin{proof}
Let $\d> 0$. If we can show that $\sum_{t=1}^\infty \P(|N_t-X_t|\geq \d\log^2 t) < \infty$, then by the Borel-Cantelli lemma we will be done.
Now, the only way that $N_t$ and $X_t$ can differ by more than $\d \log^2 t$ is if either one of the gaps between the first $t$ ladder times is larger than $\d\log^2 t$ or if for some $i<t$ the random walk backtracks $\d\log^2 t$ steps after first reaching $i$. Thus,
\begin{equation}
\P(|N_t-X_t|\geq \d\log^2 t) \leq P\left(\exists j\in[1,t+1]: \nu_j-\nu_{j-1} > \d\log^2 t  \right) + t \P( T_{-\lceil \d\log^2 t \rceil } < T_1) \label{backtrack}
\end{equation}
So we need only to show that the two terms on the right hand side are summable. For the first term we use Lemma \ref{nutail} we note that
$$P\left(\exists j\in[1,t+1]: \nu_j-\nu_{j-1} > \d\log^2 t  \right)
\leq (t+1)P(\nu> \d\log^2 t) 
\leq (t+1)C_1e^{-C_2\d\log^2 t}\,,$$
which is summable in $t$. By \eqref{backtracktail} the second 
term on the right side of \eqref{backtrack} is also summable.
\end{proof}
\begin{proof}[\textbf{Proof of Theorem \ref{local}:}]$\left.\right.$\\
By Corollary \ref{ssdom}, $P$-a.s there exists a subsequence $j_m(\w)$ such that $
M_{j_m} \geq m^2 E_\w \bar{T}_{\nu_{j_m-1}}^{(j_m)}
$.  Define
$
t_m = t_m(\w) = \frac{1}{m} M_{j_m}
$ and $u_m=u_m(\w)= \nu_{j_m-1}$.
Then,
\[
 P_\w\left( \frac{X_{t_m} - u_m}{\log^2 t_m} \notin [-\d, \d] \right) \leq P_\w(N_{t_m}\neq j_m-1 ) + P_\w(|\nu_{N_{t_m}}-X_{t_m}| > \d \log^2 t_m )
\,.
\]
From Lemma \ref{seperation} the second term goes to zero as $m\ra\infty$.
Thus, we only need to show that
\begin{equation}
\lim_{m\ra\infty} P_\w( N_{t_m} = j_m -1 ) = 1. \label{loc}
\end{equation}
To see this first note that
\[
P_\w\left( N_{t_m} < j_m -1 \right) = P_\w\left( T_{\nu_{j_m-1}} > t_m
\right) \leq P_\w\left( T_{\nu_{j_m-1}} \neq
\bar{T}_{\nu_{j_m-1}}^{(j_m)} \right) +
P_\w\left( \bar{T}_{\nu_{j_m-1}}^{(j_m)} > t_m \right)
\,.
\]
By Lemma \ref{bt},
$P_\w\left( T_{\nu_{j_m-1}} \neq \bar{T}_{\nu_{j_m-1}}^{(j_m)} \right)
\ra 0$ as $m\ra\infty$, $P-a.s.$ Also, by our definition of $t_m$ and
our choice of the subsequence $j_m$ we have
\[
P_\w\left( \bar{T}_{\nu_{j_m-1}}^{(j_m)} > t_m \right)
\leq \frac{E_\w \bar{T}_{\nu_{j_m-1}}^{(j_m)}}{t_m} =
\frac{m E_\w \bar{T}_{\nu_{j_m-1}}^{(j_m)}}{M_{j_m}}\leq
\frac{1}{m} \underset{m\rightarrow\infty}{\longrightarrow} 0.
\]
It still remains to show $\lim_{m\ra\infty} P_\w \left(
N_{t_m} < j_m \right) = 1$. To prove this, first define the
stopping times $T_x^+:= \min\{n > 0:  X_n =x \} $. Then,
\begin{align*}
P_\w\left( N_{t_m} < j_m \right)
= P_\w( T_{\nu_{j_m}} > t_m )
\geq P_\w^{\nu_{j_m -1}}\left( T_{\nu_{j_m}} > \frac{1}{m} M_{j_m} \right)
\geq P_\w^{\nu_{j_m -1}} \left( T_{\nu_{j_m-1}}^+ <
T_{\nu_{j_m}} \right) ^{\frac{1}{m} M_{j_m} } .
\end{align*}
Then, using the hitting time calculations given in
\cite[(2.1.4)]{zRWRE}, we have that
\begin{align*}
P_\w^{\nu_{j_m -1}} \left( T_{\nu_{j_m-1}}^+ < T_{\nu_{j_m}} \right)
= 1-\frac{1-\w_{\nu_{j_m-1}}}{ R_{\nu_{j_m-1},\nu_{j_m}-1} }.
\end{align*}
Therefore, since $M_{j_m} \leq R_{\nu_{j_m-1},\nu_{j_m}-1}$ we have
\[
P_\w\left( N_{t_m} < j_m \right) \geq
\left( 1-\frac{1-\w_{\nu_{j_m-1}}}{ R_{\nu_{j_m-1},\nu_{j_m}-1} }
\right)^{\frac{1}{m} M_{j_m}} \geq \left( 1-\frac{1}{M_{j_m}}
\right)^{\frac{1}{m} M_{j_m}} \underset{m\rightarrow\infty}{\longrightarrow} 1,
\]
thus proving \eqref{loc} and therefore the theorem.
\end{proof}
\end{section}


\begin{section}{Non-local behavior on a Random Subsequence} \label{gaussian}
There are two main goals of this section. The first is to prove the
existence of random subsequences $x_m$ where the hitting times $T_{x_m}$
are approximately gaussian random variables. This result is then
used to prove the existence of random times $t_m(\w)$ in which the
scaling for the random walk is of the order $t_m^s$ instead
of $\log^2 t_m$ as in Theorem \ref{local}. However, before we can begin
proving a quenched CLT for the hitting times $T_n$ (at least
along a random subsequence), we first need to understand the
tail asymptotics of
$Var_\w T_\nu:= E_\w((T_\nu-E_w T_\nu)^2)$,
the quenched variance of $T_\nu$.


\begin{subsection}{Tail Asymptotics of $Q( Var_\w T_\nu > x)$}
The goal of this subsection is to prove the following theorem:
\begin{thm}\label{qVartail}
Let Assumptions \ref{essentialasm} and \ref{techasm} hold. Then with $K_\infty>0$ the same as in Theorem \ref{Tnutail}, we have
\begin{equation}
Q\left( Var_\w T_\nu > x \right) \sim Q\left( (E_\w T_\nu)^2 > x
\right) \sim K_\infty x^{-s/2}\quad \text{ as } x\ra\infty,
\label{vartail}
\end{equation}
and for any $\e > 0$ and $x>0$, 
\begin{equation}
Q\left( Var_\w \bar{T}_\nu^{(n)} > x n^{2/s} ,\quad M_1 >
n^{(1-\e)/s} \right)  \sim K_\infty x^{-s/2} \frac{1}{n} \quad
\text{ as } n\ra\infty. \label{rvartail}
\end{equation}
Consequently,
\begin{equation}
Q\left( Var_\w T_\nu > \d n^{1/s}, M_1\leq n^{(1-\e)/s} \right) =
o(n^{-1})\,. \label{VarbigMsmall}
\end{equation}
\end{thm}
A formula for the quenched variance of crossing times is given in \cite[(2.2)]{gQCLT}. Translating to our notation and simplifying we have the formula
\begin{equation}\label{qvar}
Var_\w T_1 := E_\w (T_1 - E_\w T_1)^2 = 4(W_{0}+ W_{0}^2)+ 8 \sum_{i<0}
\Pi_{i+1,0}(W_{i}+W_{i}^2)\,.
\end{equation}
Now, given the environment the crossing times $T_j-T_{j-1}$ are independent.
Thus we get the formula
\begin{align} 
Var_\w T_\nu &= 4\sum_{j=0}^{\nu-1}(W_{j}+W_{j}^2) + 8\sum_{j=0}^{\nu-1}\sum_{i< j} \Pi_{i+1,j}(W_{i}+W_{i}^2)  \nonumber \\
&= 4\sum_{j=0}^{\nu-1}(W_{j}+W_{j}^2) + 8 R_{0,\nu-1}\left(W_{-1}+W_{-1}^2+ \sum_{i < -1} \Pi_{i+1,-1}(W_{i}+W_{i}^2)\right) \label{VarTnuexpand}\\
&\qquad + 8\sum_{0\leq i < j < \nu}\Pi_{i+1,j}(W_{i}+W_{i}^2). \nonumber
\end{align}
We want to analyze the
tails of $Var_\w T_\nu$ by comparison with $(E_\w T_\nu)^2$. Using
\eqref{ETnuexpand} we have
\[
(E_\w T_\nu)^2  = \left( \nu + 2 \sum_{j=0}^{\nu-1} W_{j} \right)^2
= \nu^2 + 4 \nu\sum_{j=0}^{\nu-1} W_j + 4\sum_{j=0}^{\nu-1}W_j^2 + 8\sum_{0\leq i<j<\nu} W_i W_j.
\]
Thus, we have
\begin{align}
(E_\w T_\nu)^2 - Var_\w T_\nu &= \nu^2 +4(\nu-1)\sum_{j=0}^{\nu-1} W_{j}
+ 8\sum_{0\leq i<j<\nu} W_{i}\left( W_j - \Pi_{i+1,j}-
\Pi_{i+1,j}W_i \right)  \label{above} \\
&\quad\quad - 8 R_{0,\nu-1}\left(W_{-1}+W_{-1}^2+
\sum_{i < -1} \Pi_{i+1,-1}(W_{i}+W_{i}^2)\right) \label{below}\\
&=: D^+(\w) - 8 R_{0,\nu-1} D^-(\w)\,. \label{Ddef}
\end{align}
Note that $D^-(\w)$ and $D^+(\w)$ are non-negative random variables.
The next few lemmas show that the tails of $D^+(\w)$ and $R_{0,\nu-1} D^-(\w)$ are much smaller than the tails of $(E_\w T_\nu)^2$.
\begin{lem}\label{abound}
For any $\e>0$, we have $Q\left( D^+(\w) > x \right) = o(x^{-s+\e})$.
\end{lem}
\begin{proof}
Notice first that from \eqref{ETnuexpand} we have
$\nu^2+4(\nu-1)\sum_{j=0}^{\nu-1} W_j \leq 2\nu E_\w T_\nu$. Also
we can re-write $ W_j - \Pi_{i+1,j}-\Pi_{i+1,j}W_i = W_{i+2,j}$
when $i < j-1$ (this term is zero when $i=j-1$). Therefore,
\[
Q\left( D^+(\w) > x \right) \leq Q( 2\nu E_\w T_\nu > x/2 ) +
Q\left( 8 \sum_{i=0}^{\nu-3} \sum_{j=i+2}^{\nu-1} W_i W_{i+2,j} > x/2 \right)\,.
\]
Lemma \ref{nutail} and Theorem \ref{Tnutail} give that, for any $\e>0$,
\[
 Q\left(
2\nu E_\w T_\nu >  x \right) \leq Q(2\nu > \log^2(x)) + Q\left(
E_\w T_\nu > \frac{ x}{\log^2(x)} \right) = o(x^{-s+\e}).
\]
Thus we need only prove that $Q\left( \sum_{i=0}^{\nu-3}
\sum_{j=i+2}^{\nu-1} W_i W_{i+2,j} > x \right) = o(x^{-s+\e})$ for
any $\e>0$.  Note that for $i<\nu$ we have  $W_i= W_{0,i} +
\Pi_{0,i}W_{-1} \leq \Pi_{0,i}(i+1+W_{-1})$, thus
\begin{align}
Q\left( \sum_{i=0}^{\nu-3} \sum_{j=i+2}^{\nu-1} W_i W_{i+2,j} > x \right)
&\leq Q\left((\nu+W_{-1}) \sum_{i=0}^{\nu-3} \sum_{j=i+2}^{\nu-1}\Pi_{0,i} W_{i+2,j} > x \right) \nonumber \\
&\leq Q(\nu> \log^2(x)/2) + Q(W_{-1}> \log^2(x)/2) \label{nuplusW}\\
&\quad\quad + \sum_{i=0}^{\log^2(x)-3} \; \sum_{j=i+2}^{\log^2(x)-1} P\left( \Pi_{0,i} W_{i+2,j} > \frac{x}{\log^6 (x)} \right) \label{pitimesW},
\end{align}
where we were able to switch to $P$ instead of $Q$ in the last
line because the event inside the probability only concerns the
environment to the right of $0$. Now, Lemmas \ref{nutail} and
\ref{Wtail} give that \eqref{nuplusW} is $o(x^{-s+\e})$ for any
$\e>0$, so we need only to consider \eqref{pitimesW}. Under
the measure $P$ we have that $\Pi_{0,i}$ and $W_{i+2,j}$ are
independent, and by \eqref{PWtail} we have $P(W_{i+2,j}> x)\leq
P(W_{j}>x) \leq K_1 x^{-s}$. Thus,
\[
P\left( \Pi_{0,i} W_{i+2,j} > \frac{x}{\log^6 (x)} \right) =
E_P\left[ P \left( W_{i+2,j} > \frac{x}{\log^6(x) \Pi_{0,i}}
\Biggl| \Pi_{0,i} \right)  \right] \leq K_1 \log^{6s}(x)
x^{-s} E_P[\Pi_{0,i}^s]\,.
\]
Then because $E_P \Pi_{0,i}^s = ( E_P \rho^s )^{i+1} = 1$ by
Assumption \ref{essentialasm}, we have
\[
\sum_{i=0}^{\log^2(x)-3} \; \sum_{j=i+2}^{\log^2(x)-1}
P\left( \Pi_{0,i} W_{i+2,j} > \frac{x}{\log^6 (x)} \right)
\leq K_1 \log^{4+6s}(x) x^{-s} = o(x^{-s+\e})\,.
\]
\end{proof}
\begin{lem}\label{piW2}
For any $\e>0$,
\begin{equation}
Q\left( D^-(\w) > x \right) = o(x^{-s+\e}), \label{piW2tail}
\end{equation}
and thus for any $\gamma < s$,
\begin{equation}
E_Q  D^-(\w)^\gamma < \infty. \label{EpiW2}
\end{equation}
\end{lem}
\begin{proof}
It is obvious that \eqref{piW2tail} implies \eqref{EpiW2} and so
we will only prove the former. For any $i$ we may expand $W_i+W_i^2$ as
\begin{align*}
 W_i+W_i^2 = \sum_{k\leq i} \Pi_{k,i} + \left( \sum_{k\leq i} \Pi_{k,i} \right)^2 
&= \sum_{k\leq i} \Pi_{k,i} + \sum_{k\leq i} \Pi_{k,i}^2 + 2\sum_{k\leq i}\sum_{l < k} \Pi_{k,i}\Pi_{l,i} \\
&= \sum_{k\leq i} \Pi_{k,i} \left( 1 + \Pi_{k,i} + 2 \sum_{l<k} \Pi_{l,i} \right).
\end{align*}
Therefore, we may re-write
\begin{equation}
D^-(\w)= W_{-1}+W_{-1}^2+ \sum_{i<-1}\Pi_{i+1,-1}(W_i+W_i^2) =
\sum_{i\leq -1}\sum_{k\leq i} \Pi_{k,-1} \left( 1 + \Pi_{k,i} +
2 \sum_{l < k} \Pi_{l,i} \right). \label{piW2expand}
\end{equation}
Next, for any $c>0$ and $n\in \N$ define the event
\[
E_{c,n} := \left\{ \Pi_{j,i} \leq e^{-c(i-j+1)},\quad \forall -n\leq i \leq -1 ,\forall j \leq i-n \right\} = \bigcap_{-n\leq i\leq -1} \bigcap_{j \leq i-n} \{\Pi_{j,i} \leq e^{-c(i-j+1)} \}
.\]
Now, under the measure $Q$ we have that $\Pi_{k,-1} < 1$ for all
$k\leq -1$, and thus on the event $E_{c,n}$ we have using the representation in \eqref{piW2expand} that
\begin{align}
 D^-(\w) &= \sum_{ i\leq -1}\sum_{k\leq i} \Pi_{k,-1} \left( 1 + \Pi_{k,i} 
       + 2 \sum_{l < k} \Pi_{l,i} \right)  \nonumber \\
 & \leq \sum_{-n\leq i\leq -1} \left( \sum_{k\leq i} \Pi_{k,i}(\Pi_{i+1,-1}+\Pi_{k,-1}) + 2 \sum_{i-n<k\leq i} \: \sum_{l<k} \Pi_{l,i} + 2 \sum_{l<k \leq i-n} e^{ck} \Pi_{l,i} \right)  \nonumber \\
 &\qquad + \sum_{i<-n} \left(\sum_{k\leq i } e^{ck} + \sum_{k\leq i } e^{ck}\Pi_{k,i} + 2 \sum_{l<k\leq i } e^{ck} \Pi_{l,i} \right)  \nonumber \\
 &\leq \sum_{-n\leq i\leq -1} \left( (2+n)W_i + 2 \sum_{l<k \leq i-n} \!\!\!\! e^{ck}e^{-c(i-l+1)} \right) \nonumber \\
 &\qquad + \sum_{i<-n} \left( \frac{e^{c(i+1)}}{e^c-1} + e^{ci} W_i +  \frac{2 e^{c(i+1)}}{e^c-1} \sum_{l<i} \Pi_{l,i} \right)  \nonumber \\
 &\leq (2+n)\sum_{-n\leq i\leq -1} W_i + \frac{2 e^{-c(2n-1)}}{(e^c-1)^3(e^c+1)} 
  + \frac{e^{-c(n-1)}}{(e^c-1)^2} + \sum_{i<-n} e^{ci} W_i \left(1+ \frac{2 e^c}{e^c-1}\right) \nonumber \\
 & \leq (2+n)\sum_{-n\leq i\leq -1} W_i + \frac{ e^c(1+e^{2c}) }{(e^c-1)^3(e^c+1)}
    + \frac{3 e^c - 1}{e^c-1} \sum_{i<-n} e^{ci} W_i   \label{piW2ub}
\end{align}
Then, using \eqref{piW2ub} with $n$ replaced by $\lfloor \log^2 x \rfloor=b_x$ we have
\begin{align}
Q\left( D^-(\w)  > x \right)
\label{onE} &\leq Q\left( E_{c,b_x}^c \right) + \mathbf{1}_{\{
\frac{ e^c(1+e^{2c}) }{(e^c-1)^3(e^c+1)} > x/3\} } + Q\left( \sum_{-b_x \leq i\leq -1} W_i > \frac{x}{3(2+ b_x)} \right)  \\
&\qquad + Q\left( \sum_{i<-1} e^{ci} W_i > \frac{(e^c-1)x}{3(3e^c-1)} \right).
\nonumber
\end{align}
Now, for any $0<c<-E_P \log \rho$ Lemma \ref{nutail} gives that $Q(\Pi_{i,j}
> e^{-c (j-i+1)}) \leq \frac{A_c}{P(\mathcal{R})} e^{-\d_c (j-i+1)}$ for some $\d_c,A_c>0$. Therefore, 
\begin{equation}
Q(E_{c,n}^c) \leq  \sum_{-n\leq i\leq -1} \sum_{j \leq i-n}
Q(\Pi_{j,i} > e^{-c(i-j+1)}) \leq
\frac{n A_c e^{-\d_c n}}{P(\mathcal{R})(e^{\d_c}-1)} =
o(e^{-\d_c n/2}). \label{QEc}
\end{equation}
Thus, for any $0<c<- E_P \log\rho$ we have that
the first two terms on the right hand side of \eqref{onE} are decreasing in $x$ of order
$o(e^{-\d_c b_x /2}) = o(x^{-s+\e})$. To handle last two terms in the
right side of \eqref{onE},
note first that from \eqref{PWtail},
$Q\left( W_i > x \right) \leq \frac{1}{P(\mathcal{R})} P( W_i > x) \leq
\frac{K_1}{P(\mathcal{R})} x^{-s}$ for any $x>0$ and any $i$. Thus,
\[
Q\left( \sum_{-b_x \leq i\leq -1} W_i > \frac{x}{3(2+ b_x)} \right) \leq \sum_{-b_x\leq i\leq -1} Q\left( W_i > \frac{x}{3(2+ b_x)b_x} \right)  = o(x^{-s+\e}),
\]
and since $\sum_{i=1}^\infty e^{-c i/2} = (e^{c/2} -1)^{-1}$, we
have
\begin{align*}
Q\left( \sum_{i<-1} e^{ci} W_i > \frac{(e^c-1)x}{9e^c-3} \right) &\leq Q\left(
\sum_{i=1}^\infty e^{-c i }W_{-i} > \frac{(e^c-1)x}{9e^c-3}(e^{c/2} - 1)
\sum_{i=1}^{\infty} e^{-c i / 2} \right)\\
& \leq \sum_{i=1}^\infty Q\left( W_{-i} > \frac{(e^c-1)(e^{c/2} -
1)}{9e^c-3} x e^{c i/2} \right) \\
&\leq \frac{K_1 (9e^c-3)^s }{P(\mathcal{R})(e^c-1)^s(e^{c/2}-1)^s} x^{-s}
\sum_{i=1}^\infty  e^{-c s i /2} = \bigo(x^{-s})
\,.\end{align*}
\end{proof}
\begin{cor}\label{bbound}
For any $\e>0$, $Q\left(R_{0,\nu-1} D^-(\w) > x \right) = o(x^{-s + \e}).$
\end{cor}
\begin{proof}
From \eqref{Rtail} it is easy to see that for any $\gamma < s$ there exists a $K_\gamma > 0$ such that
$P(R_{0,\nu-1} > x) \leq P(R_{0} > x) \leq K_\gamma x^{-\gamma}$. Then, letting $\mathcal{F}_{-1} = \s(\ldots,\w_{-2}, \w_{-1})$ we have that
\begin{align*}
Q&\left(R_{0,\nu-1}D^-(\w) > x \right) = E_Q\left[ Q\left(R_{0,\nu-1} > \frac{x}{D^-(\w)}
\biggl|
\mathcal{F}_{-1} \right) \right] \leq K_\gamma x^{-\gamma} E_Q\left( D^-(\w)
\right)^{\gamma}\,.
\end{align*}
Since $\gamma < s$, the expectation in the last expression
is finite by \eqref{EpiW2}. Choosing $\gamma = s-\frac{\e}{2}$
finishes the proof.
\end{proof}

\begin{proof}[\textbf{Proof of Theorem \ref{qVartail}:}]$\left.\right.$\\
Recall from \eqref{Ddef} that
\begin{equation}
\left( E_\w T_\nu \right)^2 - D^+(\w) \leq  Var_\w T_\nu  \leq
\left( E_\w T_\nu \right)^2 + 8R_{0,\nu-1}D^-(\w)\,. \label{compare}
\end{equation}
The lower bound in \eqref{compare} gives that for any $\d>0$,
\[
Q(Var_\w T_\nu > x) \geq Q\left( (E_\w T_\nu)^2 > (1+\d)x \right)
- Q\left(D^+(\w) > \d x \right).
\]
Thus, from Lemma \ref{abound} and Theorem \ref{Tnutail} we have
that
\begin{align}
\liminf_{x\ra\infty} x^{s/2} Q(Var_\w T_\nu > x) \geq K_\infty
(1+\d)^{-s/2}\,. \label{comparel}
\end{align}
Similarly, the upper bound in \eqref{compare}  and Corollary
\ref{bbound} give that for any $\d> 0$,
\begin{align}
Q(Var_\w T_\nu > x) &\leq Q\left( (E_\w T_\nu)^2 > (1-\d)x \right)
+ Q\left( 8 R_{0,\nu-1} D^-(\w) > \d x \right)\,, \nonumber
\end{align}
and then Corollary \ref{bbound} and Theorem \ref{Tnutail} give
\begin{equation}
\limsup_{x\ra\infty} x^{s/2} Q(Var_\w T_\nu > x) \leq K_\infty
(1-\d)^{-s/2} \,. \label{compareu}
\end{equation}
Letting $\d\ra 0$ in \eqref{comparel} and \eqref{compareu} finishes the proof of \eqref{vartail}.
\\
Essentially the same proof works for \eqref{rvartail}. The
difference is that when evaluating the difference $(E_\w
\bar{T}_\nu^{(n)})^2 - Var_\w \bar{T}_\nu^{(n)}$ the upper and
lower bounds in \eqref{above} and \eqref{below} are smaller in
absolute value. This is because every instance of $W_i$ is
replaced by $W_{\nu_{-b_n}+1,i} \leq W_i$ and the sum in
\eqref{below} is taken only over $\nu_{-b_n} < i <-1$. Therefore,
the following bounds still hold:
\begin{equation}
\left( E_\w \bar{T}^{(n)}_\nu \right)^2 - D^+(\w) \leq  Var_\w
\bar{T}^{(n)}_\nu \leq \left( E_\w \bar{T}^{(n)}_\nu \right)^2 + 8
R_{0,\nu-1} D^-(\w)\,. \label{comparer}
\end{equation}
The rest of the proof then follows in the same manner, noting that
from Lemma \ref{reftail} we have $Q\left( \left(E_\w
\bar{T}_\nu^{(n)}\right)^2 > x n^{2/s} ,\quad M_1 > n^{(1-\e)/s}
\right)\sim K_\infty x^{-s/2}\frac{1}{n}$, as $n\ra\infty$.
\end{proof}
\end{subsection}


\begin{subsection}{Existence of Random Subsequence of Non-localized Behavior}
Introduce the notation:
\begin{equation}
\mu_{i,n,\w} := E_\w^{\nu_{i-1}} \bar{T}_{\nu_i}^{(n)}, \quad
\s_{i,n,\w}^2 := E_\w^{\nu_{i-1}} \left( \bar{T}_{\nu_i}^{(n)} -
\mu_{i,n,\w}  \right)^2 = Var_\w \left(\bar{T}_{\nu_i}^{(n)} - \bar{T}_{\nu_{i-1}}^{(n)}\right). \label{def}
\end{equation}
It is obvious (from the coupling of $\bar{X}_t^{(n)}$ and $X_t$) that $\mu_{i,n,\w} \nearrow E_\w^{\nu_{i-1}} T_{\nu_i}$ as $n\ra\infty$. It is also true, although not as obvious, that $\s_{i,n,\w}^2$ is increasing in $n$ to $Var_\w \left(T_{\nu_i} - T_{\nu_{i-1}} \right)$. Therefore, we will use the notation $\mu_{i,\infty,\w}:= E_\w^{\nu_{i-1}} T_{\nu_i}$ and $\s_{i,\infty,\w}^2 := Var_\w \left(T_{\nu_i} - T_{\nu_{i-1}} \right)$. 
To see that $\s_{i,n,\w}^2$ is increasing in $n$, note that the expansion for $Var_\w \bar{T}^{(n)}_\nu$ 
is the same as the expansion for $Var_\w T_\nu$ given in \eqref{VarTnuexpand} but with each
$W_i$ replaced by $W_{\nu_{-b_n}+1, i}$ and with the final sum in
the second line restricted to $\nu_{-b_n} < i < -1$.

The first goal of this subsection is to prove a CLT (along random
subsequences) for the hitting times $T_n$. We begin by showing
that for any $\e>0$ only the crossing times of ladder times with
$M_k > n^{(1-\e)/s}$ are relevant in the limiting distribution, at
least along a sparse enough subsequence. 

\begin{lem}\label{Vsmall}
Assume $s<2$. Then for any $\e,\d>0$ there exists an $\eta>0$ and a sequence $c_n = o(n^{-\eta})$ such that for any $m \leq \infty$
\begin{align*}
Q \left( \sum_{i=1}^n \s_{i,m,\w}^2 \mathbf{1}_{ M_i \leq n^{(1-\e)/s}} > \d n^{2/s} \right) \leq c_n.
\end{align*}
\end{lem}
\begin{proof}
Since $\s_{i,m,\w}^2 \leq \s_{i,\infty,\w}^2$ it is enough to consider only the case $m=\infty$ (that is, the walk without reflections).  
First, we need a bound on the probability of $\s_{i,\infty,\w} ^2= Var_\w (T_{\nu_i} - T_{\nu_{i-1}})$ being much larger than $M_i^2$. Note that from
\eqref{compare} we have $Var_\w T_\nu \leq (E_\w
T_\nu)^2 + 8 R_{0,\nu-1}D^-(\w)$. Then, since
$R_{0,\nu-1} \leq \nu M_1$ we have for any $\a,\b>0$ that
\begin{align*}
Q\left( Var_\w T_\nu >  n^{2\b}, M_1 \leq n^{\a}
\right) & \leq Q\left( E_\w T_\nu >
\frac{n^\b}{\sqrt{2}}, M_1 \leq n^{\a} \right) + Q\left( 8 \nu
D^-(\w) > \frac{n^{2\b-\a}}{2} \right).
\end{align*}
By \eqref{TbigMsmall}, the first term on the right is
$o(e^{-n^{(\b-\a)/5}})$. To bound the second term on the right we
use Lemma \ref{nutail} and Lemma \ref{piW2} to get that for any $\a<
\b$
\[
Q\left( 8 \nu D^-(\w) > \frac{n^{2\b-\a}}{2} \right) \leq Q(\nu >
\log^2 n) + Q\left(  D^-(\w) > \frac{ n^{2\b - \a}}{16\log^2 n}
\right) = o(n^{-\frac{s}{2}(3\b-\a)})\,.
\]
Therefore, similarly to \eqref{TbigMsmall} we have the bound
\begin{align}
Q\left( Var_\w T_\nu >  n^{2\b}, M_1 \leq n^{\a}
\right) = o(n^{-\frac{s}{2}(3\b-\a)}) \,.\label{VbigMsmall}
\end{align}
The rest of the proof is similar to the proof of Lemma
\ref{smallblocks}. First, from \eqref{VbigMsmall},
\begin{align*}
Q \left( \sum_{i=1}^n \s_{i,\infty,\w} ^2 \mathbf{1}_{ M_i \leq n^{(1-\e)/s}} > \d n^{2/s} \right) & \leq Q\left( \sum_{i=1}^n \s_{i,\infty,\w} ^2 \mathbf{1}_{\s_{i,\infty,\w} ^2 \leq n^{2(1-\frac{\e}{4} )/s}} > \d n^{2/s}  \right) \\
& \quad\quad + n Q\left( Var_\w T_\nu > n^{2(1-\frac{\e}{4})/s}, M_1 \leq n^{(1-\e)/s} \right)\\
& = Q\left( \sum_{i=1}^n \s_{i,\infty,\w} ^2 \mathbf{1}_{\s_{i,\infty,\w} ^2 \leq
n^{2(1-\frac{\e}{4} )/s}} > \d n^{2/s}  \right) + o(n^{-\e/8})\,.
\end{align*}
Therefore, it is enough to prove that for any $\d,\e>0$ there
exists $\eta>0$ such that
\[
Q\left( \sum_{i=1}^n \s_{i,\infty,\w} ^2 \mathbf{1}_{\s_{i,\infty,\w} ^2 \leq
n^{2(1-\frac{\e}{4} )/s}} > \d n^{2/s}  \right) = o( n^{-\eta} )
\,.\]
We prove the above statement by choosing
$C\in (1,\frac{2}{s})$, since $s>2$, and then using Theorem \ref{qVartail}
to get bounds on the size of the set
$\left\{i\leq n: Var_\w \left(T_{\nu_i} - T_{\nu_{i-1}} \right) \in \left( n^{2(1-\e C^k)/s},
n^{2(1-\e C^{k-1})/s} \right] \right\}$ for all $k$ small enough
so that $\e C^k < 1$. This portion of the proof is similar to that of Lemma
\ref{smallblocks} and thus will be omitted.
\end{proof}
\begin{cor}\label{Vsdiff}
Assume $s<2$. 
Then for any $\d>0$ there exists an $\eta'>0$ and a sequence $c_n'=o(n^{-\eta'})$ such that for any $m\leq \infty$
\[
Q\left( \left| \sum_{i=1}^n \left( \s_{i,m,\w}^2-\mu_{i,m,\w}^2
\right) \right| \geq \d n^{2/s} \right) \leq c_n'
\, .\]
\end{cor}
\begin{proof}
For any $\e > 0$
\begin{align}
Q\left( \left| \sum_{i=1}^n \left( \s_{i,m,\w}^2-\mu_{i,m,\w}^2
\right) \right|
\geq \d n^{2/s} \right) &\leq Q\left( \sum_{i=1}^n \s_{i,m,\w}^2\mathbf{1}_{M_i \leq n^{(1-\e)/s}} \geq \frac{\d}{3}n^{2/s} \right) \label{smM1}\\
&\quad + Q\left( \sum_{i=1}^n \mu_{i,m,\w}^2\mathbf{1}_{M_i \leq n^{(1-\e)/s}} \geq \frac{\d}{3}n^{2/s} \right)  \label{smM2}\\
&\quad + Q\left( \sum_{i=1}^n \left| \s_{i,m,\w}^2 -
\mu_{i,m,\w}^2  \right| \mathbf{1}_{M_i > n^{(1-\e)/s}} \geq
\frac{\d}{3}n^{2/s} \right). \label{lgM}
\end{align}
Lemma \ref{Vsmall} gives that \eqref{smM1} decreases
polynomially in $n$ (with a bound not depending on $m$). 
Also, essentially the same proof as in Lemmas \ref{Vsmall} and \ref{smallblocks} can be used to show that \eqref{smM2} also decreases polynomially in $n$ (again with a bound not depending on $m$). 
Finally \eqref{lgM} is bounded above by
\[
Q\left( \# \left\{ i\leq n: M_i > n^{(1-\e)/s} \right\} > n^{2\e}
\right) + n Q\left( \left|Var_\w \bar{T}_\nu^{(m)} - (E_\w
\bar{T}_\nu^{(m)})^2 \right| \geq \frac{\d}{3} n^{2/s - 2\e}
\right) ,
\]
and since by \eqref{Mtail}, $Q\left( \# \left\{ i\leq n: M_i >
n^{(1-\e)/s} \right\} > n^{2\e} \right) \leq \frac{ n Q(M_1 >
n^{(1-\e)/s}) }{n^{2\e}} \sim C_5 n^{-\e}$ we need only show that
for some $\e>0$ the second term above is decreasing faster than a power of $n$.
However, from \eqref{comparer} we have $\left|Var_\w
\bar{T}_\nu^{(m)} - (E_\w \bar{T}_\nu^{(m)})^2 \right| \leq
D^+(\w) + 8 R_{0,\nu-1} D^-(\w)$. Thus
\[
 n Q\left( \left|Var_\w \bar{T}_\nu^{(m)} - (E_\w
 \bar{T}_\nu^{(m)})^2 \right| \geq \frac{\d}{3} n^{2/s - 2\e}
 \right) \leq n Q\left( D^+(\w) + 8 R_{0,\nu-1} D^-(\w)> \frac{\d}{3} n^{2/s - 2\e} \right),
\]
and for any $\e<\frac{1}{2s}$ Lemma \ref{abound} and
Corollary \ref{bbound} give that the last term above decreases faster than some power of $n$.
%
%
%
%
\end{proof}
Since $T_{\nu_n} = \sum_{i=1}^n (T_{\nu_i} - T_{\nu_{i-1}})$
is the sum of independent (quenched) random variables, in order to
prove a CLT we cannot have any of the first $n$ crossing times of
blocks dominating all the others (note this is exactly what
happens in the localization behavior we saw in Section
\ref{localization}). Thus, we look for a random subsequence where
none of the crossing times of blocks are dominant. 
Now, for any $\d\in(0,1]$ and any positive integer $a < n/2$
define the event
\[
\mathcal{S}_{\d,n,a} := \left\{ \#\left\{ i\leq \d n:
\mu_{i,n,\w}^2 \in [n^{2/s}, 2n^{2/s}) \right\} = 2a, \quad
\mu_{j,n,\w}^2 < 2 n^{2/s} \quad \forall j\leq \d n \right\}
\,.\]
On the event $\mathcal{S}_{\d, n, a}$, $2a$ of the first $\d
n$ crossings times from $\nu_{i-1}$ to $\nu_i$ have roughly the
same size expected crossing times $\mu_{i,n,\w}$, and the rest are
all smaller (we work with $\mu_{i,n,\w}^2$ instead of
$\mu_{i,n,\w}$ so that comparisons with $\s_{i,n,\w}^2$ are slightly
easier). We want a lower bound on the probability of
$\mathcal{S}_{\d,n,a}$. The difficulty in getting a lower bound is
that the $\mu^2_{i,n,\w}$ are not independent. However, we can
force all the large crossing times to be independent by forcing
them to be separated by at least $b_n$ ladder locations.

Let $\mathcal{I}_{\d,n,a}$ be the collection of all subsets $I$ of
$[1,\d n]\cap \Z$ of size $2 a$ with the property that any two
distinct points in $I$ are separated by at least $2b_n$. Also,
define the event
\[
A_{i,n}:= \left\{ \mu_{i,n,\w}^2 \in \left[ n^{2/s},2 n^{2/s}
\right) \right\}.
\]
Then, we begin with a simple lower bound.
\begin{align}
Q( \mathcal{S}_{\d,n,a} ) &\geq Q\left( \bigcup_{ I\in
\mathcal{I}_{\d,n,a}} \left( \bigcap_{i\in I} A_{i,n}
\bigcap_{j\in[1,\d n]\backslash I} \left\{ \mu_{j,n,\w}^2 <
n^{2/s} \right\} \right) \right) \nonumber \\
& = \sum_{I\in \mathcal{I}_{\d,n,a}} Q\left( \bigcap_{i\in I}
A_{i,n} \bigcap_{j\in[1,\d n]\backslash I} \left\{ \mu_{j,n,\w}^2
<
 n^{2/s} \right\} \right)\,. \label{simlb}
\end{align}
Now, recall the definition of the event $G_{i,n,\e}$ from
\eqref{FGevents}, and define the event
\[
H_{i,n,\e} := \left\{ M_j \leq n^{(1-\e)/s} \text{ for all } j\in
[i-b_n,i) \right\}\,.
\]
Also, for any $I\subset \Z$ let $d(j,I):= \min \{ |j-i| : i\in I
\}$ be the minimum distance from $j$ to the set $I$. Then, with
minimal cost, we can assume that for any $I\in
\mathcal{I}_{\d,n,a}$ and any $\e>0$ that all $j\notin I$ such
that $d(j,I)\leq b_n$ have $M_j \leq n^{(1-\e)/s}$. Indeed,
\begin{align}
Q\left( \bigcap_{i\in I} A_{i,n} \right.&\left. \bigcap_{j\in[1,\d
n]\backslash
I} \left\{ \mu_{j,n,\w}^2 <  n^{2/s} \right\} \right) \nonumber \\
& \geq Q\left( \bigcap_{i\in I} \left( A_{i,n} \cap G_{i,n,\e} \cap
H_{i,n,\e} \right) \bigcap_{j\in[1,\d n]:d(j,I)>b_n} \left\{
\mu_{j,n,\w}^2 <
 n^{2/s} \right\} \right) \nonumber \\
&\quad - Q\left(\bigcup_{j\notin{I}, d(j,I)\leq b_n} \left\{
\mu_{j,n,\w}^2
\geq  n^{2/s}  , M_j \leq n^{(1-\e)/s} \right\} \right) \nonumber \\
&\geq \prod_{i\in I} Q(A_{i,n} \cap H_{i,n,\e} ) Q\left(
\bigcap_{i\in I} G_{i,n,\e} \bigcap_{j\in[1,\d n]:d(j,I)>b_n} \left\{
\mu_{j,n,\w}^2 <  n^{2/s} \right\} \right) \nonumber \\
& \quad - 4 a b_n Q\left( E_\w T_\nu \geq n^{1/s}, M_1 \leq
n^{(1-\e)/s} \right)\,. \label{Ilb}
\end{align}
From Theorem  \ref{Tnutail} and Lemma \ref{reftail} we have
$Q(A_{i,n}) \sim K_\infty (1-2^{-s/2}) n^{-1}$. We wish to show
the same asymptotics are true for $Q(A_{i,n} \cap H_{i,n,\e})$
as well. From \eqref{Mtail} we have $Q(H_{i,n,\e}^c) \leq b_n
Q(M_1
> n^{(1-\e)/s}) = o(n^{-1+2\e})$. Applying this, along with
\eqref{Mtail} and \eqref{TbigMsmall}, gives that for $\e>0$,
\begin{align*}
Q(A_{i,n}) &\leq Q( A_{i,n} \cap H_{i,n,\e} ) + Q\left(M_1 >
n^{(1-\e)/s}\right)Q(H_{i,n,\e}^c) + Q\left( E_\w T_\nu > n^{1/s}
, M_1 \leq n^{(1-\e)/s} \right)  \\
& = Q( A_{i,n} \cap H_{i,n,\e}
 ) + o(n^{-2+3\e}) + o(e^{-n^{\e/(5s)}})\,.
\end{align*}
Thus, for any $\e <\frac{1}{3}$ there exists a $C_\e> 0$ such that
\begin{equation}
Q(A_{i,n}\cap H_{i,n,\e}) \geq C_\e n^{-1}. \label{AHlb}
\end{equation}
To handle the next probability in \eqref{Ilb}, note that
\begin{align}
Q\left( \bigcap_{i\in I} G_{i,n,\e} \bigcap_{j\in[1,\d
n]:d(j,I)>b_n} \left\{ \mu_{j,n,\w}^2 <  n^{2/s} \right\} \right)
& \geq Q\left(\bigcap_{j\in[1,\d n]} \left\{ \mu_{j,n,\w}^2 <
n^{2/s} \right\} \right) - Q\left(
\bigcup_{i\in I} G_{i,n,\e}^c \right) \nonumber \\
& \geq Q\left( E_\w T_{\nu_n} < n^{1/s} \right) - 2 a Q(
G_{i,n,\e}^c ) \nonumber \\
&= Q\left( E_\w T_{\nu_n} < n^{1/s} \right) - a o(n^{-1+2\e})\,.
\label{Glb}
\end{align}
Finally, from \eqref{TbigMsmall} we have $ 4 a b_n Q\left( E_\w
T_\nu \geq n^{1/s}, M_1 \leq n^{(1-\e)/s} \right) = a o\left(
e^{-n^{\e/(6 s)}} \right) $. This, along with \eqref{AHlb} and
\eqref{Glb} applied to \eqref{simlb} gives
\begin{align*}
Q\left( \mathcal{S}_{\d,n,a} \right) \geq \#(\mathcal{I}_{\d,n,a})
\left[\left( C_\e n^{-1} \right)^{2a}\left( Q\left( E_\w T_{\nu_n}
< n^{1/s} \right) - a o(n^{-1+2\e}) \right) -  a o\left(
e^{-n^{\e/(6 s)}} \right) \right]\,.
\end{align*}
An obvious upper bound for $\#(\mathcal{I}_{\d,n,a})$ is
$\binom{\d n}{2a} \leq \frac{(\d n)^{2a}}{(2a)!}$. To get a lower
bound on $\#(\mathcal{I}_{\d,n,a})$ we note that any set
$I\in\mathcal{I}_{\d,n,a}$ can be chosen in the following way:
first choose an integer $i_1 \in [1,\d n]$ ($\d n$ ways to do
this). Then, choose an integer $i_2 \in [1,\d n] \backslash \{
j\in \Z: |j-i_1| \leq 2b_n \}$ (at least $\d n - 1 - 4b_n$ ways to
do this). Continue this process until $2a$ integers have been
chosen. When choosing $i_j$, there will be at least $\d n -
(j-1)(1+4b_n)$ integers available. Then, since there are $(2a)!$
orders in which to choose each set if $2a$ integers we have
\[
\frac{(\d n)^{2a}}{(2a)!} \geq \#(\mathcal{I}_{\d,n,a}) \geq
\frac{1}{(2a)!} \prod_{j=1}^{2a} \left(\d n - (j-1)(1+4b_n)\right)
\geq \frac{(\d n)^{2a}}{(2a)!} \left( 1 - \frac{(2a-1)(1+4b_n)}{\d
n} \right)^{2a}
\,.\]
Therefore, applying the upper and lower bounds on
$\#(\mathcal{I}_{\d,n,a})$ we get
\begin{align*}
Q\left( \mathcal{S}_{\d,n,a} \right) & \geq  \frac{(\d
C_\e)^{2a}}{(2a)!} \left( 1 - \frac{(2a-1)(1+4b_n)}{\d n}
\right)^{2a} \left( Q\left( E_\w T_{\nu_n} < n^{1/s} \right) - a
o(n^{-1+2\e}) \right)  \\
&\qquad - \frac{(\d n)^{2a}}{(2a)!} a o\left( e^{-n^{\e/(6 s)}}
\right)
\,.\end{align*}
Recall the definitions of $d_k$ in \eqref{nkdkdef} and define
\begin{equation}
a_k:=\lfloor\log\log k\rfloor \vee 1, \quad\text{and}\quad
\d_k:=a_k^{-1}. \label{adeltadef}
\end{equation}
Now, replacing $\d,n$ and $a$ in the above by $\d_k, d_k$ and
$a_k$ respectively we have
\begin{align}
Q\left( \mathcal{S}_{\d_k,d_k,a_k} \right) & \geq  \frac{(\d_k
C_\e)^{2a_k}}{(2a_k)!} \left( 1 - \frac{(2a_k-1)(1+4b_{d_k})}{\d_k
d_k} \right)^{2a_k} \left( Q\left( E_\w T_{\nu_{d_k}} < d_k^{1/s}
\right) - a_k
o(d_k^{-1+2\e}) \right)  \nonumber \\
&\qquad - \frac{(\d_k d_k)^{2a_k}}{(2a_k)!} a_k o\left(
e^{-d_k^{\e/(6 s)}} \right) \nonumber \\
& \geq  \frac{(\d_k C_\e)^{2a_k}}{(2a_k)!} \left( 1+ o(1)\right)
\left( L_{s,b'}(1) - o(1) \right) - o(1/k). \label{sublb}
\end{align}
The last inequality is a result of the definitions of $\d_k, a_k,$
and $d_k$ (it's enough to recall that $d_k \geq 2^{2^{k-1}}$, $a_k
\sim \log\log k$, and $\d_k \sim \frac{1}{\log\log k} $), as well
as Theorem \ref{refstable}. Also, since $\d_k = a_k^{-1}$ we get
from Sterling's formula that $ \frac{(\d_k C_\e)^{2a_k}}{(2a_k)!}
\sim \frac{ ( C_\e e / 2)^{2a_k}}{\sqrt{2\pi a_k}}$. Thus since
$a_k\sim \log\log k$, we have that $\frac{1}{k} = o\left(
\frac{(\d_k C_\e)^{2a_k}}{(2a_k)!} \right)$. This, along with
\eqref{sublb}, gives that $Q\left( \mathcal{S}_{\d_k,d_k,a_k}
\right) > \frac{1}{k}$ for all $k$ large enough.\\
We now have a good lower bound on the probability of not having
any of the crossing times of the first $\d_k d_k$ blocks dominating
all the others.  However for the purpose of proving Theorem
\ref{nonlocal} we need a little bit more. We also need that none
of the crossing times of succeeding blocks are too large either.
Thus, for any $0<\d<c$ and $n\in \N$ define the events
\[
U_{\d,n,c}:=\left\{ \sum_{i=\d n + 1}^{cn} \mu_{i,n,\w} \leq 2
n^{1/s} \right\}, \quad \tilde{U}_{\d,n,c}:=\left\{ \sum_{i=\d n +
b_n+1}^{cn} \mu_{i,n,\w} \leq n^{1/s} \right\}
\,.\]
\begin{lem} \label{smallblocklemma}
Assume $s<1$. Then there exists a sequence $c_k \ra \infty$, $c_k = o( \log a_k )$
such that
\[
\sum_{k=1}^{\infty}  Q\left( \mathcal{S}_{\d_k,d_k,a_k} \cap U_{\d_k,d_k,c_k} \right) = \infty
\,.\]
\end{lem}
\begin{proof}
For any $\d<c$ and $a<n/2$ we have
\begin{align}
Q\left( \mathcal{S}_{\d,n,a} \cap U_{\d,n,c} \right) & \geq
Q\left( \mathcal{S}_{\d,n,a}\right) Q\left( \tilde{U}_{\d,n,c}
\right) - Q\left( \sum_{i=1}^{b_n} \mu_{i,n,\w}
> n^{1/s}  \right) \nonumber \\
& \geq Q\left( \mathcal{S}_{\d,n,a}\right) Q\left( E_\w
T_{\nu_{cn}} \leq n^{1/s}\right) - b_n Q\left( E_\w T_\nu >
\frac{n^{1/s}}{ b_n }\right) \nonumber \\
&\geq Q\left( \mathcal{S}_{\d,n,a}\right) Q\left( E_\w
T_{\nu_{cn}} \leq n^{1/s} \right) - o(n^{-1/2}), \label{SU}
\end{align}
where the last inequality is from Theorem \ref{Tnutail}. Now,
define $c_1=1$ and for $k>1$ let
\[
c_k' := \max \left\{ c \in \N : Q\left( E_\w T_{\nu_{c d_k}} \leq
d_k^{1/s} \right) \geq \frac{1}{\log k} \right\} \vee 1
\,.\]
Note that by Theorem \ref{refstable} we have that $c_k'\ra\infty$,
and so we can define $c_k = c_k' \wedge \log\log (a_k)$. Then
applying \eqref{SU} with this choice of $c_k$ we have
\[
\sum_{k=1}^\infty Q\left( \mathcal{S}_{\d_k,d_k,a_k} \cap
U_{\d_k,d_k,c_k} \right) \geq \sum_{k=1}^\infty \left[Q\left(
\mathcal{S}_{\d_k,d_k,a_k}\right) Q\left( E_\w T_{\nu_{c_k d_k}}
\leq d_k^{1/s} \right) - o(d_k^{-1/2})\right] = \infty,
\]
and
the last
sum is infinite because $d_k^{-1/2}$ is summable and for
all $k$ large enough we have
\[
Q\left( \mathcal{S}_{\d_k,d_k,a_k}\right) Q\left( E_\w T_{\nu_{c_k
d_k}} \leq d_k^{1/s} \right) \geq \frac{1}{k\log k}.
\]
\end{proof}
\begin{cor}\label{smallblocks2}
Assume $s<1$, and let $c_k$ be as in Lemma
\ref{smallblocklemma}. Then,
$P$-a.s. there exists a random subsequence $n_{k_m}=n_{k_m}(\w)$
of $n_k=2^{2^k}$ such that for the sequences $\a_m, \b_m,$ and
$\gamma_m$ defined by
\begin{equation}
\a_m := n_{k_m-1},\qquad \b_m := n_{k_m-1} +\d_{k_m} d_{k_m} ,
\qquad \gamma_m:= n_{k_m-1} + c_{k_m} d_{k_m}, \label{abgdef}
\end{equation}
we have that for all $m$
\begin{equation}
\max_{ i\in(\a_m, \b_m] } \mu_{i,d_{k_m},\w}^2  \leq
2d_{k_m}^{2/s} \leq \frac{1}{a_{k_m}} \sum_{i=\a_m+1}^{\b_m}
\mu_{i,d_{k_m},\w}^2, \label{smbk}
\end{equation}
and
\[
 \sum_{\b_m + 1}^{\gamma_m} \mu_{i,d_{k_m},\w} \leq 2 d_{k_m}^{1/s}\,.
\]
\end{cor}
\begin{proof}
Define the events
\begin{align*}
\mathcal{S}_k' &:= \left\{ \#\left\{ i \in (n_{k-1}, n_{k-1}+\d_k d_k] :
\mu_{i,d_k,\w}^2 \in [d_k^{2/s}, 2d_k^{2/s}) \right\} = 2a_k \right\} \\
&\qquad \cap \left\{ \mu_{j,d_k,\w}^2 < 2 d_k^{2/s} \quad \forall
j\in (n_{k-1}, n_{k-1}+\d_k d_k] \right\}\,,
\end{align*}
\[
U_k' := \left\{ \sum_{n_{k-1}+ \d_k d_k + 1}^{n_{k-1}+c_k d_k}
\mu_{i,d_{k},\w} \leq 2d_{k}^{1/s} \right\}\,.
\]
Note that due to the reflections of the random walk, the event
$\mathcal{S}_k' \cap U_k' $ depends on the environment between
ladder locations $n_{k-1}-b_{d_k}$ and $n_{k-1}+c_kd_k$. Thus, for
$k_0$ large enough $\{\mathcal{S}_{2k}' \cap U_{2k}'
\}_{k=k_0}^{\infty}$ is an independent sequence of events.
Similarly, for $k$ large enough $\mathcal{S}_k'\cap U_k'$ does not
depend on the environment to left of the origin. Thus
\[
P(\mathcal{S}_k'\cap U_k') =Q(\mathcal{S}_k'\cap U_k')=Q\left( \mathcal{S}_{\d_k,d_k,a_k} \cap U_{\d_k,d_k,c_k}
\right)
\]
for all $k$ large enough. Lemma \ref{smallblocklemma} then gives
that $\sum_{k=1}^\infty P(\mathcal{S}_{2k}'\cap U_{2k}') =
\infty$, and the Borel-Cantelli lemma then implies that infinitely
many of the events $\mathcal{S}_{2k}' \cap U_{2k}'$ occur $P-a.s$.
Finally, note that $\mathcal{S}_{k_m}'$ implies the event in
\eqref{smbk}.
\end{proof}
Before proving a quenched CLT (along a subsequence) for the hitting times
$T_n$, we need one more lemma that gives us some control on the
quenched tails of crossing times of blocks. We can get this from
an application of Kac's moment formula. Let $\bar{T}_y$ be the
hitting time of $y$ when we add a reflection at the starting point
of the random walk. Then Kac's moment formula  \cite[(6)]{fpKMF} 
and the Markov property give that 
$E_\w^x (\bar{T}_y)^j \leq j! \left( E_\w^x
\bar{T}_y \right)^j$ (note that because
of the reflection at $x$, $E_\omega^x(\bar{T}_y)
\geq E_\omega^{x'}(\bar{T}_y)$ for any $x'\in (x,y)$). Thus,
\begin{equation}
E_\w^{\nu_{i-1}} (\bar{T}_{\nu_i}^{(n)})^j \leq
E_\w^{\nu_{i-1-b_n}} (\bar{T}_{\nu_i} )^j \leq j! \left(
E_\w^{\nu_{i-1-b_n}} \bar{T}_{\nu_i} \right)^j \leq j! \left(
E_\w^{\nu_{i-1-b_n}} \bar{T}_{\nu_{i-1}} + \mu_{i,n,\w} \right)^j.
\label{Kac}
\end{equation}
\begin{lem} \label{momentbound}
For any $\e< \frac{1}{3}$, there exists an $\eta>0$ such that
\[
Q\left( \exists i\leq n,\quad j\in\N: M_i > n^{(1-\e)/s},\quad
E_\w^{\nu_{i-1}} ( \bar{T}_{\nu_i}^{(n)} )^j
> j! 2^j \mu_{i,n,\w}^j \right) = o(n^{-\eta})
\,.\]
\end{lem}
\begin{proof}
We use \eqref{Kac} to get
\begin{align*}
Q\left( \exists i\leq n,\quad j\in\N: M_i > n^{(1-\e)/s}, \right.
& \left. \quad E_\w^{\nu_{i-1}} ( \bar{T}_{\nu_i}^{(n)} )^j
> j! 2^j \mu_{i,n,\w}^j \right) \\
& \leq Q\left( \exists i\leq n : M_i > n^{(1-\e)/s}, \quad
E_\w^{\nu_{i-1-b_n}}
\bar{T}_{\nu_{i-1}} > \mu_{i,n,\w} \right)\\
& \leq n Q\left( M_1 > n^{(1-\e)/s} , \quad E_\w^{\nu_{-b_n}} T_0
> n^{(1-\e)/s} \right)\\
& = n Q\left( M_1 > n^{(1-\e)/s} \right) Q\left( E_\w^{\nu_{-b_n}}
T_0 > n^{(1-\e)/s}\right)\,,
\end{align*}
where the second inequality is due to a union bound and the fact
that $\mu_{i,n,\w} > M_i$. Now, by \eqref{Mtail} we have $n
Q\left( M_1 > n^{(1-\e)/s} \right) \sim C_5 n^{\e}$, and by
Theorem \ref{Tnutail}
\[
Q\left( E_\w^{\nu_{-b_n}} T_0 > n^{(1-\e)/s}\right) \leq b_n
Q\left( E_\w T_\nu > \frac{n^{(1-\e)/s}}{b_n} \right) \sim
K_\infty b_n^{1+s} n^{-1+\e}\,.
\]
Therefore, $Q\left( \exists i\leq n,\quad j\in\N: M_i >
n^{(1-\e)/s},\quad E_\w^{\nu_{i-1}} ( \bar{T}_{\nu_i}^{(n)} )^j
> j! 2^j \mu_{i,n,\w}^j \right) = o(n^{-1+3\e})$.
\end{proof}
\begin{thm} \label{gaussianT}
Let Assumptions \ref{essentialasm} and \ref{techasm} hold, and let $s<1$. Then $P-a.s.$ there exists a random subsequence $n_{k_m}=n_{k_m}(\w)$
of $n_k=2^{2^k}$ such that for $\a_m$, $\b_m$ and $\gamma_m$ as in
\eqref{abgdef} and any sequence $x_m \in [\nu_{\b_m},
\nu_{\gamma_m} ]$, we have
\begin{equation}
\lim_{m\ra\infty} P_\w \left( \frac{ T_{x_m} - E_\w T_{x_m} } {
\sqrt{v_{m,\w}} } \leq y \right) = \Phi(y)\,,
\end{equation}
where
\[
v_{m,\w} := \sum_{i=\a_m+1}^{\b_m} \mu_{i,d_{k_m}, \w}^2 .
\]
\end{thm}
\begin{proof}
Let $n_{k_m}(\w)$ be the random subsequence specified in Corollary
\ref{smallblocks2}. For ease of notation, set
$\tilde{a}_m = a_{k_m}$ and
$ \tilde{d}_m = d_{k_m}\,.$
We have
\[\max_{i\in(\a_m, \b_m]} \mu_{i,\tilde{d}_m,\w}^2 \leq
2 \tilde{d}_m^{2/s} \leq
\frac{1}{\tilde{a}_m}\sum_{i=\a_m+1}^{\b_m}
\mu_{i,\tilde{d}_m,\w}^2 = \frac{v_{m,\w}}{\tilde{a}_m},\quad
\text{and} \quad \sum_{i=\b_m+1}^{\gamma_m} \mu_{i,\tilde{d}_m,\w}
\leq 2 \tilde{d}_m^{1/s}.
\]
Now, let $\{ x_m \}_{m=1}^\infty$ be any sequence of integers
(even depending on $\w$) such that $x_m \in [\nu_{\b_m},
\nu_{\gamma_m}]$. Then, since $(T_{x_m} - E_\w T_{x_m}) =
(T_{\nu_{\a_m}} - E_\w T_{\nu_{\a_m}}) +
(T_{x_m} - T_{\nu_{\a_m}} - E_\w^{\nu_{\a_m}}
T_{x_m}),$
it is enough to prove
\begin{equation}
\frac{T_{\nu_{\a_m}} - E_\w T_{\nu_{\a_m}}}{\sqrt{v_{m,\w}}}
\limdw 0,\quad\text{and}\quad \frac{T_{x_m} - T_{\nu_{\a_m}} -
E_\w^{\nu_{\a_m}} T_{x_m}}{\sqrt{v_{m,\w}}} \limdw Z\sim N(0,1)
\label{irrelevantb}
\end{equation}
where we use the notation $Z_n \limdw Z$ to denote quenched
convergence in distribution, that is $\lim_{n\ra\infty} P_\w( Z_n
\leq z ) = P_\w( Z \leq z )$, $P-a.s$. For the first term in
\eqref{irrelevantb} note that for any $\e > 0$, we have from
Chebychev's inequality and $v_{m,\w} \geq
\tilde{d}_m^{2/s}$, that
\[
P_\w\left( \left| \frac{T_{\nu_{\a_m}} - E_\w
T_{\nu_{\a_m}}}{\sqrt{v_{m,\w}}} \right| \geq \e \right) \leq
\frac{Var_\w T_{\nu_{\a_m}}}{\e^2 v_{m,\w} } \leq \frac{Var_\w
T_{\nu_{\a_m}}}{\e^2 \tilde{d}_m^{2/s}}
\,.\]
Thus, the first claim in \eqref{irrelevantb} will be proved if we can
show that $Var_\w T_{\nu_{\a_m}} = o(\tilde{d}_m^{2/s})$. For this
we need the following lemma:
\begin{lem} \label{Varasymp}
Assume $s\leq 2$. Then for any $\d > 0$,
\[
P\left( Var_\w T_{\nu_n} \geq n^{2/s + \d} \right) = o( n^{-\d
s/4} )\,.
\]
\end{lem}
\begin{proof}
First, we claim that
\begin{equation}
    \label{eq-100407a}
    E_P (Var_\w T_1)^\gamma < \infty\; \text{for any}\;
\gamma < \frac{s}{2}\,.
\end{equation}
Indeed, from
\eqref{qvar}, we have that for any $\gamma < \frac{s}{2} \leq 1$
\begin{align*}
E_P (Var_\w T_1)^\gamma &\leq 4^\gamma E_P( W_0 + W_0^2 )^\gamma + 8^\gamma \sum_{i<0} E_P \left( \Pi_{i+1,0}^\gamma (W_i+W_i^{2})^\gamma  \right)\\
&= 4^\gamma E_P( W_0 + W_0^2 )^\gamma + 8^\gamma
\sum_{i=1}^{\infty} (E_P \rho_0^\gamma )^i E_P ( W_0 + W_0^2
)^\gamma,
\end{align*}
where we used that $P$ is i.i.d. in the last equality.
Since
$E_P \rho_0^\gamma < 1$ for
any $\gamma \in (0,s)$, we have that \eqref{eq-100407a} follows
as soon as $E_P(W_0+W_0^2)^\gamma < \infty$. However, from \eqref{PWtail} we
 get that $E_P(W_0+W_0^2)^\gamma < \infty$
  when $\gamma < \frac{s}{2}$.
  
As in Lemma \ref{mdevu} let $\bar \nu = E_P \nu$. Then,
\[
P\left( Var_\w T_{\nu_n} \geq n^{2/s + \d} \right) \leq P( Var_\w
T_{2\bar{\nu} n} \geq n^{2/s+\d} ) + P(\nu_n \geq 2\bar{\nu} n)\,.
\]
As in Lemma \ref{mdevu}, the second term is $\bigo\left(e^{-\d' n} \right)$ for some $\d'>0$. 
To handle the first term on the right side, we
note that for any $\gamma < \frac{s}{2} \leq 1$
\begin{equation}
    \label{eq-100407b}
P( Var_\w T_{2\bar{\nu} n} \geq n^{2/s+\d} ) \leq \frac{ E_P\left(
\sum_{k=1}^{2\bar{\nu} n} Var_\w (T_k-T_{k-1}) \right)^\gamma
}{n^{\gamma(2/s + \d)}} \leq \frac{2\bar{\nu}n E_P(Var_\w
T_1)^\gamma }{n^{\gamma(2/s+\d)}}.
\end{equation}
Then since $E_P (Var_\w T_1)^\gamma < \infty$ for any $\gamma <
\frac{s}{2}$, we can choose $\gamma$ arbitrarily close to
$\frac{s}{2}$ so that the last term on the right of \eqref{eq-100407b}
is $o(n^{-\d s
/4})$.
\end{proof}
As a result of Lemma \ref{Varasymp} and the Borel-Cantelli
lemma, we have that $Var_\w T_{\nu_{n_k}} = o(n_k^{2/s+\d})$ for
any $\d > 0$. Therefore, for any $\d\in(0,\frac{2}{s})$ we have $Var_\w T_{\nu_{\a_m}} = o( \a_m^{2/s
+\d}) = o(n_{k_m-1}^{2/s+\d}) = o(\tilde{d}_m^{2/s})$
(in the last equality we use that $d_k\sim n_k$ to grow much faster than exponentially in $k$).

For the next step in the proof, we show that reflections can be
added without changing the limiting distribution. Specifically, we show that it is enough to prove the following lemma, whose proof we postpone:
\begin{lem}\label{lem-latenight}
    With notation as in Theorem \ref{gaussianT}, we have
    \begin{equation}
\lim_{m\ra\infty} P_\w^{\nu_{\a_m}} \left( \frac{
\bar{T}^{(\tilde{d}_m)}_{x_m} - E_\w^{\nu_{\alpha_m}} \bar{T}^{(\tilde{d}_m)}_{x_m}
} { \sqrt{v_{m,\w}} } \leq y \right) = \Phi(y)\,. \label{refgaussian}
\end{equation}
\end{lem}
\noindent
Assuming Lemma \ref{lem-latenight}, we complete the
proof of Theorem \ref{gaussianT}.
It is enough to show that
\begin{equation}
\label{eq-maybelast}
\lim_{m\ra\infty} P_\w^{\nu_{\a_m}} (
\bar{T}^{(\tilde{d}_m)}_{x_{k_m}} \neq T_{x_m} ) = 0,
\quad\text{and}\quad \lim_{m\ra\infty} E_\w^{\nu_{\a_m}} (
T_{x_m}- \bar{T}^{(\tilde{d}_m)}_{x_{k_m}} ) = 0.
\end{equation}
Recall that the coupling introduced after \eqref{bdef} gives that $ T_{x_m} -  \bar{T}^{(\tilde{d}_m)}_{x_m}
 \geq 0$. Thus,
\[
 P_\w^{\nu_{\a_m}}( \bar{T}^{(\tilde{d}_m)}_{x_m}
\neq T_{x_m} ) = P_\w^{\nu_{\a_m}}\left( T_{x_m} -
\bar{T}^{(\tilde{d}_m)}_{x_m} \geq 1 \right) \leq E_\w^{\nu_{\a_m}} (
T_{x_m} - \bar{T}^{(\tilde{d}_m)}_{x_m} ).
\]
Then, since
$x_m \leq \nu_{\gamma_m}$ and $
\gamma_m=n_{k_m-1}+c_{k_m}\tilde{d}_m \leq n_{k_m+1}$ for all $m$
large enough, \eqref{eq-maybelast} will follow from
\begin{equation}
\lim_{k\ra\infty} E_\w^{\nu_{n_{k-1}}} \left( T_{\nu_{n_{k+1}}} -
\bar{T}^{(d_k)}_{\nu_{n_{k+1}}} \right) = 0, \quad
P-a.s.\label{ETETbar}
\end{equation}
To prove \eqref{ETETbar}, we argue as follows.
From Lemma \ref{ETdiff} we have that for any $\e>0$
\[
Q\left(\! E_\w^{\nu_{n_{k-1}}} \left(\! T_{\nu_{n_{k+1}}} -
\bar{T}^{(d_k)}_{\nu_{n_{k+1}}}\! \right) > \e \!\right) \leq n_{k+1}
Q\left(\! E_\w T_{\nu} - E_\w \bar{T}^{(d_k)}_{\nu} >
\frac{\e}{n_{k+1}}\! \right) = n_{k+1} \bigo\left(\! n_{k+1}^s e^{-\d'
b_{d_k}}\! \right)
.
\]
Since $n_k \sim d_k$, the last term on the
right is summable. Therefore, by the Borel-Cantelli lemma,
\begin{equation}
\lim_{k\ra\infty} E_\w^{\nu_{n_{k-1}}} \left( T_{\nu_{n_{k+1}}} -
\bar{T}^{(d_k)}_{\nu_{n_{k+1}}} \right) = 0, \quad  Q-a.s.
\label{QETETbar}
\end{equation}
This is almost the same as \eqref{ETETbar}, but with $Q$ instead
of $P$. To use this to prove \eqref{ETETbar} note that for $i>b_n$
using \eqref{reflectexpand} we can write
\[
E_\w^{\nu_{i-1}} T_{\nu_i} - E_\w^{\nu_{i-1}} \bar{T}^{(n)}_{\nu_i}
= A_{i,n}(\w) + B_{i,n}(\w) W_{-1}\,,
\]
where $A_{i,n}(\w)$ and $B_{i,n}(\w)$ are non-negative random variables depending only
on the environment to the right of 0. Thus, $E_\w^{\nu_{n_{k-1}}}
\left( T_{\nu_{n_{k+1}}} - \bar{T}^{(d_k)}_{\nu_{n_{k+1}}} \right)
= A_{d_k}(\w) + B_{d_k}(\w) W_{-1}$ where $A_{d_k}(\w)$ and $B_{d_k}(\w)$ are non-negative and only depend on
the environment to the right of zero (so $A_{d_k}$ and $B_{d_k}$ have the same
distribution under $P$ as under $Q$). Therefore \eqref{ETETbar}
follows from \eqref{QETETbar}, which finishes the proof of the theorem.
\end{proof}
\begin{proof}[\textit{
    Proof of Lemma \ref{lem-latenight}}.] \
Clearly, it suffices to show
the following claims:
\begin{equation}
\frac{\bar{T}^{(\tilde{d}_m)}_{x_m}-\bar{T}^{(\tilde{d}_m)}_{\nu_{\b_m}}
- E_\w^{\nu_{\b_m}}\bar{T}^{(\tilde{d}_m)}_{x_m} }{
\sqrt{v_{m,\w}} } \limdw 0, \label{irrelevant}
\end{equation}
and
\begin{equation}
\frac{\bar{T}^{(\tilde{d}_m)}_{\nu_{\b_m}}-\bar{T}^{(\tilde{d}_m)}_{\nu_{\a_m}}
- E_\w^{\nu_{\a_m}}\bar{T}^{(\tilde{d}_m)}_{\nu_{\b_m}} }{
\sqrt{v_{m,\w}} }  \limdw Z \sim N(0,1) \,.\label{relevant}
\end{equation}
To prove \eqref{irrelevant}, we note that
\begin{align*}
P_\w\left( \left| \frac{\bar{T}^{(\tilde{d}_m)}_{x_m} -
\bar{T}^{(\tilde{d}_m)}_{\nu_{\b_m}} - E_\w^{\nu_{\b_m}}
\bar{T}^{(\tilde{d}_m)}_{x_m}}{\sqrt{v_{m,\w}}}  \right| \geq \e
\right) & \leq \frac{ Var_\w
(\bar{T}^{(\tilde{d}_m)}_{x_m}-\bar{T}^{(\tilde{d}_m)}_{\nu_{\b_m}})
}{\e^2 v_{m,\w} } \leq \frac{\sum_{i=\b_m+1}^{\gamma_m}
\s_{i,\tilde{d}_m,\w} ^2 }{\e^2 \tilde{a}_m \tilde{d}_m^{2/s}},
\end{align*}
where the last inequality is because $x_m \leq \nu_{\gamma_m}$ and
$v_{m,\w} \geq  \tilde{a}_m \tilde{d}_m^{2/s}$. However, by
Corollary \ref{Vsdiff} and the Borel-Cantelli lemma,
\begin{align*}
\sum_{i=\b_m+1}^{\gamma_m} \s_{i,\tilde{d}_m,\w} ^2 =
\sum_{i=\b_m+1}^{\gamma_m} \mu_{i,\tilde{d}_m,\w} ^2 +
o\left((c_{k_m} \tilde{d}_m)^{2/s}\right)
\,.\end{align*}
The application of Corollary \ref{Vsdiff} uses the fact that for
$k$ large enough the reflections ensure that the events in
question do not involve the environment to the left of zero and
thus have the same probability under $P$ or $Q$. (This type of
argument will be used a few more times in the remainder of the
proof without mention.) By our choice of the subsequence
$n_{k_m}$ we have
\[
\sum_{i=\b_m+1}^{\gamma_m} \mu_{i,\tilde{d}_m,\w} ^2  \leq \left(
\sum_{i=\b_m+1}^{\gamma_m} \mu_{i,\tilde{d}_m,\w} \right)^2 \leq 4
\tilde{d}_m^{2/s}\,.
\]
Therefore, 
\[
\lim_{m\ra\infty} P_\w\left( \left|
\frac{\bar{T}^{(\tilde{d}_m)}_{x_m} -
\bar{T}^{(\tilde{d}_m)}_{\nu_{\b_m}} - E_\w^{\nu_{\b_m}}
\bar{T}^{(\tilde{d}_m)}_{x_m}}{\sqrt{v_{m,\w}}} \right| \geq \e
\right) \leq \lim_{m\ra\infty} \frac{ 4 \tilde{d}_m^{2/s} +
o\left((c_{k_m} \tilde{d}_m)^{2/s}\right) }{\e^2 \tilde{a}_m
\tilde{d}_m^{2/s}} = 0,\;P-a.s.
\]
where the last limit equals zero because $c_k = o(\log a_k)$.

It only remains to prove
\eqref{relevant}. Since re-writing we express
\[
\bar{T}^{(\tilde{d}_m)}_{\nu_{\b_m}}-\bar{T}^{(\tilde{d}_m)}_{\nu_{\a_m}}
- E_\w^{\nu_{\a_m}}\bar{T}^{(\tilde{d}_m)}_{\nu_{\b_m}} =
\sum_{i=\a_m+1}^{\b_m} \left( (\bar{T}^{(\tilde{d}_m)}_{\nu_i}-
\bar{T}^{(\tilde{d}_m)}_{\nu_{i-1}}) - \mu_{i,\tilde{d}_m,\w}
\right)
\]
as the sum of independent, zero-mean random variables (quenched),
we need only show the Lindberg-Feller condition. That is, we need
to show
\begin{equation}
\lim_{m\ra\infty} \frac{1}{v_{m,\w}} \sum_{i=\a_m+1}^{\b_m}
\s_{i,\tilde{d}_m,\w}^2 = 1, \quad P-a.s.\label{LF1}
\end{equation}
and for all $\e > 0$
\begin{equation}
\lim_{m\ra\infty} \frac{1}{v_{m,\w}} \sum_{i=\a_m+1}^{\b_m}
E_\w^{\nu_{i-1}} \left[ \left(
\bar{T}^{(\tilde{d}_m)}_{\nu_i}-\mu_{i,\tilde{d}_m,\w} \right)^2
\mathbf{1}_{ |
\bar{T}^{(\tilde{d}_m)}_{\nu_i}-\mu_{i,\tilde{d}_m,\w} | > \e
\sqrt{v_{m,\w}})}  \right] = 0, \quad P-a.s. \label{LF2}
\end{equation}
To prove \eqref{LF1} note that
\begin{align*}
\frac{1}{v_{m,\w}} \sum_{i=\a_m+1}^{\b_m} \s_{i,\tilde{d}_m,\w}^2
= 1 + \frac{\sum_{i=\a_m+1}^{\b_m} \left( \s_{i,\tilde{d}_m,\w}^2 -
\mu_{i,\tilde{d}_m,\w}^2 \right) }{v_{m,\w}}.
\end{align*}
However, another application of Corollary \ref{Vsdiff} and the Borel-Cantelli Lemma implies that
\[
\sum_{i=\a_m+1}^{\b_m} (\s_{i,\tilde{d}_m,\w}^2 -
\mu_{i,\tilde{d}_m,\w}^2) = o\left((\d_{k_m} \tilde{d}_m)^{2/s}\right). 
\]
Recalling that
 $v_{m,\w} \geq \tilde{a}_m \tilde{d}_m^{2/s}$, we have that \eqref{LF1} is
 proved.

To prove \eqref{LF2} we break the sum up into two parts depending
on whether $M_i$ is ``small" or ``large". Specifically, for $\e'\in(0,\frac{1}{3})$
we decompose the sum as
\begin{align}
&  \frac{1}{v_{m,\w}} \sum_{i=\a_m+1}^{\b_m} E_\w^{\nu_{i-1}} \left[ \left( \bar{T}^{(\tilde{d}_m)}_{\nu_i}-\mu_{i,\tilde{d}_m,\w} \right)^2 \mathbf{1}_{ | \bar{T}^{(\tilde{d}_m)}_{\nu_i}-\mu_{i,\tilde{d}_m,\w}  | > \e \sqrt{v_{m,\w}})}  \right] \mathbf{1}_{M_i \leq \tilde{d}_m^{(1-\e')/s}}\label{Msmall} \\
&\quad + \frac{1}{v_{m,\w}} \sum_{i=\a_m+1}^{\b_m}
E_\w^{\nu_{i-1}} \left[ \left(
\bar{T}^{(\tilde{d}_m)}_{\nu_i}-\mu_{i,\tilde{d}_m,\w} \right)^2
\mathbf{1}_{ |
\bar{T}^{(\tilde{d}_m)}_{\nu_i}-\mu_{i,\tilde{d}_m,\w} | > \e
\sqrt{v_{m,\w}}}  \right]\mathbf{1}_{M_i >
\tilde{d}_m^{(1-\e')/s}} . \label{Mlarge}
\end{align}
We get an upper bound for \eqref{Msmall} by first omitting the
indicator function inside the expectation,
and then expanding the sum to be up
to $n_{k_m} \geq \b_m$. Thus \eqref{Msmall} is bounded above by
\[
 \frac{1}{v_{m,\w}}
\sum_{i=\a_m+1}^{\b_m} \s_{i,\tilde{d}_m,\w}^2 \mathbf{1}_{M_i
\leq \tilde{d}_m^{(1-\e')/s}}  \leq \frac{1}{v_{m,\w}}
\sum_{i=n_{k_m-1}+1}^{n_{k_m}} \s_{i,\tilde{d}_m,\w}^2
\mathbf{1}_{M_i \leq \tilde{d}_m^{(1-\e')/s}} .
\]
However, since $d_k$ grows exponentially fast, the Borel-Cantelli lemma and
Lemma \ref{Vsmall} give that
\begin{equation}
\sum_{i=n_{k-1}+1}^{n_k} \s_{i,d_k,\w}^2 \mathbf{1}_{M_i \leq
d_k^{(1-\e')/s}} = o(d_k^{2/s}). \label{smV}
\end{equation}
Therefore, since our choice of the subsequence $n_{k_m}$ gives
that $v_{m,\w} \geq \tilde{d}_m^{2/s}$, we have
that \eqref{Msmall} tends to zero as $m\ra\infty$.

To get an upper bound for \eqref{Mlarge}, first note that our
choice of the subsequence $n_{k_m}$ gives that $\e \sqrt{v_{m,\w}}
\geq \e \sqrt{\tilde{a}_m} \mu_{i,\tilde{d}_m,\w}$ for any $i\in
(\a_m, \b_m]$. Thus, for $m$ large enough we can replace the
indicators inside the expectations in \eqref{Mlarge} by the
indicators of the events $\{ \bar{T}_{\nu_i}^{(\tilde{d}_m)}  >
(1+\e \sqrt{\tilde{a}_m}) \mu_{i,\tilde{d}_m,\w} \}$. Thus, for
$m$ large enough and $i\in(\a_m, \b_m]$, we have
\begin{align}
& E_\w^{\nu_{i-1}} \left[ \left(
\bar{T}^{(\tilde{d}_m)}_{\nu_i}-\mu_{i,\tilde{d}_m,\w} \right)^2
\mathbf{1}_{ |
\bar{T}^{(\tilde{d}_m)}_{\nu_i}-\mu_{i,\tilde{d}_m,\w} | >
\e \sqrt{v_{m,\w}}}  \right] \nonumber \\
&\qquad \leq E_\w^{\nu_{i-1}} \left[ \left(
\bar{T}^{(\tilde{d}_m)}_{\nu_i}-\mu_{i,\tilde{d}_m,\w} \right)^2
\mathbf{1}_{ \bar{T}^{(\tilde{d}_m)}_{\nu_i}> (1+\e
\sqrt{\tilde{a}_m})
\mu_{i,\tilde{d}_m,\w} }  \right]  \nonumber \\
&\qquad =  \e^2\tilde{a}_m \mu_{i,\tilde{d}_m,\w}^2 P_\w^{\nu_{i-1}}
\left( \bar{T}_{\nu_i}^{(\tilde{d}_m)} > (1+\e \sqrt{\tilde{a}_m}) \mu_{i,\tilde{d}_m,\w}
\right)  \label{truncmoment} \\
&\qquad\qquad + \int_{1+\e\sqrt{\tilde{a}_m}}^{\infty} P_\w^{\nu_{i-1}}
\left( \bar{T}_{\nu_i}^{(\tilde{d}_m)} > x \mu_{i,\tilde{d}_m,\w}
\right) 2(x-1)\mu_{i,\tilde{d}_m,\w}^2 dx \,. \nonumber
\end{align}
We want to use Lemma \ref{momentbound} get an upper bounds on the probabilities in the last line above. Lemma \ref{momentbound} and the Borel-Cantelli lemma give
that for $k$ large enough, $E_\w^{\nu_{i-1}}\left(
\bar{T}^{(d_k)}_{\nu_i} \right)^j \leq 2^j j! \mu_{i,d_k,\w}^j$,
for all $n_{k-1}<i\leq n_k$ such that $M_i > d_k^{(1-\e')/s}$.
Multiplying by $(4\mu_{i,d_k,\w})^{-j}$ and summing over $j$ gives
that $E_\w^{\nu_{i-1}} e^{ \bar{T}^{(d_k)}_{\nu_i} /(4
\mu_{i,d_k,\w}) } \leq 2$. Therefore, Chebychev's inequality gives that 
\[
P_\w^{\nu_{i-1}} \left( \bar{T}_{\nu_i}^{(d_k)} > x \mu_{i,d_k,\w}
\right) \leq e^{- x/4} E_\w^{\nu_{i-1}} e^{
\bar{T}^{(d_k)}_{\nu_i} / (4 \mu_{i,d_k,\w})} \leq 2 e^{-x/4}
\,.\]
Thus, for all $m$ large enough and for all $i$ with $\a_m<i\leq \b_m
\leq n_{k_m}$ and $M_i > \tilde{d}_m^{(1-\e')/s}$ we have from \eqref{truncmoment} that
\begin{align*}
& E_\w^{\nu_{i-1}} \left[ \left( \bar{T}^{(\tilde{d}_m)}_{\nu_i}-\mu_{i,\tilde{d}_m,\w} \right)^2
\mathbf{1}_{ |\bar{T}^{(\tilde{d}_m)}_{\nu_i}-\mu_{i,\tilde{d}_m,\w} | > \e \sqrt{v_{m,\w}}}  \right] \\
&\qquad \leq \e^2\tilde{a}_m \mu_{i,\tilde{d}_m,\w}^2 2 e^{-(1+\e\sqrt{\tilde{a}_m})/4} 
+ \int_{1+\e\sqrt{\tilde{a}_m}}^{\infty} 2e^{-x/4} 2(x-1)\mu_{i,\tilde{d}_m,\w}^2 dx \\
&\qquad = \left( 2\e^2\tilde{a}_m +  16(4+\e \sqrt{\tilde{a}_m}) \right) e^{-(1+\e\sqrt{\tilde{a}_m})/4}
\mu_{i,\tilde{d}_m,\w}^2 
\end{align*}
%
%
Recalling the definition of $v_{m,\w}=\sum_{i=\a_m+1}^{\b_m}
\mu_{i,\tilde{d}_m,\w}^2$, we have that as $m\ra\infty$,
\eqref{Mlarge} is bounded above by
\begin{align*}
&\lim_{m\ra\infty} \frac{1}{v_{m,\w}} \sum_{i=\a_m+1}^{\b_m}
\left( 2\e^2\tilde{a}_m +  16(4+\e \sqrt{\tilde{a}_m}) \right) e^{-(1+\e\sqrt{\tilde{a}_m})/4}
\mu_{i,\tilde{d}_m,\w}^2 \mathbf{1}_{M_i > \tilde{d}_m^{(1-\e')/s}} \\
&\qquad \leq \lim_{m\ra\infty} \left( 2\e^2\tilde{a}_m +  16(4+\e \sqrt{\tilde{a}_m}) \right)
e^{-(1+\e\sqrt{\tilde{a}_m})/4} = 0
\,.\end{align*}
This finishes the proof of \eqref{LF2} and thus of Lemma \ref{lem-latenight}.
\end{proof}
\begin{proof}[\textbf{Proof of Theorem
\ref{nonlocal}:}]$\left.\right.$\\
Note first that from Lemma \ref{mdevu} and the
Borel-Cantelli lemma, we have that for any $\e>0$, $E_\w
T_{\nu_{n_k}} = o(n_k^{(1+\e)/s})$, $P-a.s.$ This is equivalent to
\begin{equation}
    \label{eq-100407d}
\limsup_{k\ra\infty} \frac{\log E_\w T_{\nu_{n_k}}}{\log n_k} \leq
\frac{1}{s}, \quad P-a.s.
\end{equation}
We can also get bounds on the
probability of $E_\w T_{\nu_n}$ being small. Since $E_\w^{\nu_{i-1}} T_{\nu_i} \geq M_i$ we have
\[
P\left(E_\w T_{\nu_n} \leq n^{(1-\e)/s}\right) \leq P\left(
M_i \leq n^{(1-\e)/s}, \quad \forall i\leq
n \right) \leq \left( 1 - P\left(M_1 > n^{(1-\e)/s}
\right)\right)^n,
\]
and since $P(M_1 > n^{(1-\e)/s}) \sim C_5 n^{-1+\e}$, see
\eqref{Mtail},
we have
$P\left(E_\w T_{\nu_n} \leq n^{(1-\e)/s}\right) \leq
e^{-n^{\e/2}}$.  Thus, by the Borel-Cantelli lemma, for any $\e>0$
we have that $E_\w T_{\nu_{n_k}} \geq n_k^{(1-\e)/s}$ for all $k$
large enough, $P-a.s.$, or equivalently
\begin{equation}
    \label{eq-100407c}
\liminf_{k\ra\infty} \frac{\log E_\w T_{\nu_{n_k}}}{\log n_k} \geq
\frac{1}{s}, \quad P-a.s.
\end{equation}
Let $n_{k_m}$ be the subsequence specified in Theorem
\ref{gaussianT}, and define $t_m:= E_\w T_{n_{k_m}}$. Then,
by \eqref{eq-100407d} and \eqref{eq-100407c},
 $\lim_{m\ra\infty} \frac{\log
t_m}{\log n_{k_m}} = {1}/{s}$.\\
For any $t$ define $X_t^*:= \max \{ X_n : n\leq t \}$. Then, for
any $x\in (0,\infty)$ we have
\begin{align*}
P_\w\left( \frac{X_{t_m}^*}{n_{k_m}} < x \right) &= P\left(
X_{t_m}^* < x n_{k_m} \right)
= P_\w\left( T_{x n_{k_m}} > t_m \right)\\
&= P_\w\left( \frac{T_{x n_{k_m}} - E_\w T_{x
n_{k_m}}}{\sqrt{v_{m,\w}}}
> \frac{ E_\w T_{n_{k_m}}-E_\w T_{x n_{k_m}}}{\sqrt{v_{m,\w}}} \right)
\,.\end{align*}
Now, with notation as in Theorem \ref{gaussianT}, we have that for
all $m$ large enough $\nu_{\b_m} < x n_{k_m} < \nu_{\gamma_m}$
(note that this also uses the fact that $\nu_n/n \ra E_P \nu$,
$P-a.s.$). Thus $\frac{T_{x n_{k_m}} - E_\w T_{x
n_{k_m}}}{\sqrt{v_{m,\w}}} \limdw Z\sim N(0,1)$. Then, we will have
proved that $ \lim_{m\ra\infty} P_\w\left( \frac{X_{t_m}^*}{n_{k_m}} <
x \right) = \frac{1}{2}$ for any $x\in(0,\infty)$ if we can show
\begin{equation}
\lim_{m\ra\infty} \frac{ E_\w T_{n_{k_m}}-E_\w T_{x
n_{k_m}}}{\sqrt{v_{m,\w}}} = 0\,,\quad P-a.s.  \label{sd}
\end{equation}
For $m$ large enough we have $n_{k_m}, x n_{k_m} \in (\nu_{\b_m} , \nu_{\gamma_m})$. 
Thus, for $m$ large enough,
\begin{align*}
\left|\frac{ E_\w T_{x n_{k_m}}-E_\w T_{n_{k_m}}}{\sqrt{v_{m,\w}}}\right| \leq
\frac{E_\w^{\nu_{\b_m}} T_{\nu_{\gamma_m}} }{\sqrt{v_{m,\w}}} =
\frac{1}{\sqrt{v_{m,\w}}} \left( E_\w^{\nu_{\b_m}} \left(
T_{\nu_{\gamma_m}} - \bar{T}_{\nu_{\gamma_m}}^{(\tilde{d}_m)}
\right) + \sum_{i=\b_m+1}^{\gamma_m} \mu_{i,\tilde{d}_m,\w}
\right)\,.
\end{align*}
Since $\a_m \leq \b_m \leq \gamma_m \leq n_{k_m+1}$ for all
$m$ large enough, we can apply \eqref{ETETbar} to get
\[
\lim_{m\ra\infty} E_\w^{\nu_{\b_m}} \left( T_{\nu_{\gamma_m}} -
\bar{T}_{\nu_{\gamma_m}}^{(\tilde{d}_m)} \right) \leq
\lim_{m\ra\infty} E_\w^{\nu_{\a_m}} \left( T_{\nu_{n_{k_m+1}}} -
\bar{T}_{\nu_{n_{k_m+1}}}^{(\tilde{d}_m)} \right)=0.
\]
Also, from our choice of $n_{k_m}$ we have that
$\sum_{i=\b_m+1}^{\gamma_m} \mu_{i,\tilde{d}_m,\w} \leq 2
\tilde{d}_m^{1/s}$ and $v_{m,\w} \geq \tilde{a}_m
\tilde{d}_m^{2/s}$. Thus \eqref{sd} is proved. Therefore
\[
\lim_{m\ra\infty} P_\w\left(\frac{X_{t_m}^*}{n_{k_m}} \leq x
\right) = \frac{1}{2}, \quad \forall x\in(0,\infty),
\]
and obviously $\lim_{m\ra\infty}
P_\w\left(\frac{X_{t_m}^*}{n_{k_m}} < 0 \right) = 0$ since
$X_n$ is transient to the right $\P-a.s.$ due to Assumption
\ref{essentialasm}. Finally, note that
\[
\frac{X_t^* - X_t}{\log^2 t} =  \frac{X_t^* - \nu_{N_t}}{\log^2 t}
+ \frac{\nu_{N_t} - X_t}{\log^2 t} \leq \frac{\max_{i\leq t}
(\nu_i - \nu_{i-1}) }{\log^2 t} + \frac{\nu_{N_t} - X_t}{\log^2 t}
\,.\]
However, Lemma \ref{seperation} and an easy application of
Lemma \ref{nutail} and the Borel-Cantelli lemma gives that
\[
\lim_{t\ra\infty} \frac{X_t^* - X_t}{\log^2 t} = 0, \quad \P-a.s.
\]
This finishes the proof of the theorem.
\end{proof}
\end{subsection}
\end{section}


\begin{section}{Asymptotics of the tail of $E_\w T_\nu$} \label{tailofTnu}
Recall that
$
E_\w T_\nu = \nu + 2\sum_{j=0}^{\nu-1} W_j
= \nu + 2 \sum_{i\leq j, 0\leq j < \nu} \Pi_{i,j} ,
$
and for any $A>1$ define
\[
\s = \s_A = \inf\{ n \geq 1: \Pi_{0,n-1} \geq A \}\,.
\]
Note that $\s-1$ is a stopping time for the sequence $\Pi_{0,k}$. For any $A>1$, $\{\s > \nu \} = \{ M_1 < A \}$. Thus we have by \eqref{TbigMsmall} that for any $A>1$,
\begin{equation}
Q(E_\w T_\nu > x, \s > \nu) = Q(E_\w T_\nu > x, M_1 < A) =  o(x^{-s}). \label{long}
\end{equation}
Thus, we may focus on the tail estimates $Q(E_\w T_\nu > x, \s < \nu)$ in which case we can use the following expansion of $E_\w T_\nu$:
\begin{align}
E_\w T_\nu &= \nu + 2 \sum_{i<0\leq j<\s-1} \Pi_{i,j} + 2 \sum_{0 \leq i \leq j<\s-1} \Pi_{i,j}  + 2 \sum_{\s\leq i\leq j < \nu} \Pi_{i,j}+2 \sum_{i\leq \s-1 \leq j < \nu} \Pi_{i,j}\nonumber\\
&= \nu + 2 W_{-1}R_{0,\s-2} + 2 \sum_{j=0}^{\s-2} W_{0,j} + 2 \sum_{i=\s}^{\nu-1} R_{i,\nu-1} + 2 W_{\s-1}(1+R_{\s,\nu-1})\,.\label{expand}
\end{align}
We will show that the dominant term in \eqref{expand} is the last term: $2 W_{\s-1}(1+R_{\s,\nu-1})$. A few easy consequences of Lemmas \ref{nutail} and \ref{Wtail} are that the tails of the first three terms in the expansion \eqref{expand} are negligible. The following statements are true for any $\d>0$ and any $A>1$:
\begin{equation}
Q(\nu > \d x)=P(\nu>\d x) = o(x^{-s})\,,\label{first}
\end{equation}
\begin{align}
Q(2W_{-1}R_{0,\s-2}>\d x, \s < \nu) &\leq Q(W_{-1}> \sqrt{\d x}) + P(2R_{0,\s-2} > \sqrt{\d x},\s<\nu) \nonumber \\
&\leq Q(W_{-1}> \sqrt{\d x}) + P(2 \nu A > \sqrt{\d x}) \label{second}
= o(x^{-s}),
\end{align}
    \begin{equation}
        Q\left(2 \sum_{j=0}^{\s-2} W_{0,j} > \d x, \s<\nu\right)
    \leq P\left(2 \sum_{j=1}^{\s-1} j A > \d x, \s< \nu\right)
\leq P( \nu^2 A > \d x) \label{third}
= o(x^{-s}).
\end{equation}
In the first inequality in \eqref{third}, we used the fact that $\Pi_{i,j} \leq \Pi_{0,j}$ for any $0<i<\nu$ since $\Pi_{0,i-1} \geq 1$.

The fourth term in \eqref{expand} is not negligible, but we can make it arbitrarily small by taking $A$ large enough.
\begin{lem} \label{fourth}
For all $\d>0$, there exists an $A_0=A_0(\d)<\infty$ such that
\[
P\left( 2 \sum_{\s_A\leq i < \nu} R_{i,\nu-1} > \d x \right) < \d x^{-s}, \quad \forall A\geq A_0(\d)\,.
\]
\end{lem}
\begin{proof}
This proof is essentially a copy of the proof of Lemma 3 in \cite{kksStable}.
\begin{align*}
P\left( 2 \sum_{\s_A\leq i < \nu} R_{i,\nu-1} > \d x \right)
&\leq P\left( \sum_{\s_A \leq i < \nu} R_i > \frac{\d}{2} x \right) = P\left( \sum_{i=1}^\infty \mathbf{1}_{\s_A \leq i < \nu} R_i > \frac{\d}{2} x \frac{6}{\pi^2} \sum_{i=1}^\infty i^{-2} \right)\\
&\leq \sum_{i=1}^\infty P\left( \mathbf{1}_{\s_A \leq i < \nu} R_i >  x \frac{3\d}{\pi^2} i^{-2} \right)\,.
\end{align*}
However, since the event $\{\s_A \leq i < \nu\}$ depends only on $\rho_j$ for $j<i$, and $R_i$ depends only on $\rho_j$ for $j\geq i$, we have that
\[
P\left( 2 \sum_{\s_A\leq i < \nu} R_{i,\nu-1} > \d x \right) \leq \sum_{i=1}^\infty P\left( \s_A \leq i < \nu\right) P\left( R_i >  x \frac{3\d}{\pi^2} i^{-2} \right)\,.
\]
Now, from \eqref{Rtail} we have that there exists a $K_1 > 0$  such that $P(R_0 > x) \leq K_1 x^{-s}$ for all $x>0$. We then conclude that
\begin{align}
P\left( \sum_{\s_A\leq i < \nu} R_{i,\nu-1} > \d x \right)
&\leq  K_1 \left( \frac{3\d}{\pi^2}\right)^{-s}x^{-s} \sum_{i=1}^\infty P\left( \s_A \leq i < \nu\right) i^{2s}\nonumber\\
&= K_1 \left( \frac{3\d}{\pi^2}\right)^{-s}x^{-s} E_P\left[\sum_{i=1}^{\infty} \mathbf{1}_{\s_A\leq i < \nu} i^{2s}  \right]\nonumber\\
&\leq  K_1 \left( \frac{3\d}{\pi^2}\right)^{-s}x^{-s} E_P[ \nu^{2s+1} \mathbf{1}_{\s_A < \nu}]\,.\label{small}
\end{align}
Since $E_P \nu^{2s+1}<\infty$ and $\lim_{A\ra\infty} P(\s_A<\nu) = 0$,
we have that the right side of
\eqref{small} can be made less than $\d x^{-s}$ by choosing $A$ large enough.
\end{proof}
We need one more lemma before analyzing the dominant term in \eqref{expand}.
\begin{lem}\label{finitemoment}
$E_Q\left[ W_{\s_A-1}^s\mathbf{1}_{\s_A<\nu} \right]< \infty$ for any $A>1$.
\end{lem}
\begin{proof}
First, note that on the event $\{\s_A < \nu\}$ we have that $\Pi_{i,\s_A-1} \leq \Pi_{0,\s_A-1}$ for any $i\in[0,\s_A)$. Thus, 
\[
W_{\s_A-1} = W_{0,\s_A-1} + \Pi_{0,\s_A-1}W_{-1} \leq (\s_A + W_{-1})\Pi_{0,\s_A-1}.
\]
Also, note that $\Pi_{0,\s_A -1} \leq A \rho_{\s_A-1}$ by the definition of $\s_A$. Therefore
\[
E_Q\left[ W_{\s_A-1}^s\mathbf{1}_{\s_A<\nu} \right] \leq E_Q\left[ (\s_A + W_{-1})^s A^s\rho_{\s_A-1}^s \mathbf{1}_{\s_A<\nu} \right]
\]
Therefore, it is enough to prove that both $E_Q\left[ W_{-1}^s\rho_{\s_A-1}^s \mathbf{1}_{\s_A<\nu} \right]$ and $E_Q\left[ \s_A^s\rho_{\s_A-1}^s \mathbf{1}_{\s_A<\nu} \right]$ are finite (note that this is trivial if we assume that $\rho$ has bounded support). Since $W_{-1}$ is independent of $\rho_{\s_A-1}^s \mathbf{1}_{\s_A<\nu}$ we have that 
\[
E_Q\left[ W_{-1}^s\rho_{\s_A-1}^s \mathbf{1}_{\s_A<\nu} \right] = E_Q[W_{-1}^s]E_P[\rho_{\s_A-1}^s \mathbf{1}_{\s_A<\nu}],
\]
where we may take the second expectation over $P$ instead of $Q$ because the random variable only depends on the environment to the right of zero. By Lemma \ref{Wtail} we have that $E_Q[W_{-1}^s]<\infty$. Also, $E_P[\rho_{\s_A-1}^s \mathbf{1}_{\s_A<\nu}] \leq E_P[\s_A^s\rho_{\s_A-1}^s \mathbf{1}_{\s_A<\nu}]$, and so the Lemma will be proved once we prove the latter is finite.  However,
\begin{align*}
E_P\left[ \s_A^s \rho_{\s_A-1}^s \mathbf{1}_{\s_A<\nu} \right] = \sum_{k=1}^\infty E_P \left[ k^s \rho_{k-1}^s \mathbf{1}_{\s_A = k <\nu} \right] \leq \sum_{k=1}^\infty k^s E_P\left[\rho_{k-1}^s \mathbf{1}_{ k \leq \nu} \right],
\end{align*} 
and since the event $\{k\leq\nu\}$ depends only on $(\rho_0,\rho_1,\ldots \rho_{k-2})$ we have that $E_P\left[ \rho_{k-1}^s\mathbf{1}_{k\leq\nu}\right] = E_P\rho^s P(\nu\geq k)$ since $P$ is a product measure. Then since $E_P \rho^s = 1$ we have that 
\[
E_P\left[ \s_A^s \rho_{\s_A-1}^s \mathbf{1}_{\s_A<\nu} \right] \leq \sum_{k=1}^\infty k^s P(\nu\geq k). 
\]
This last sum is finite by Lemma \ref{nutail}.
\end{proof}
Finally, we turn to the asymptotics of the tail of $2W_{\s-1}(1+R_{\s,\nu-1})$, which is the dominant term in \eqref{expand}.
\begin{lem}\label{fifth}
For any $A>1$, there exists a constant $K_A\in (0,\infty)$ such that
\[
\lim_{x\ra\infty} x^s Q\left( W_{\s-1}(1+R_{\s,\nu-1})>x , \s < \nu \right) = K_A
\,.\]
\end{lem}
\begin{proof}
The strategy of the proof is as follows. First, note that on the event $\{\s<\nu\}$ we have $W_{\s-1}(1 + R_\s) = W_{\s-1}(1 + R_{\s,\nu-1}) + W_{\s-1}\Pi_{\s,\nu-1}R_\nu$. We will begin by analyzing the asymptotics of the tails of $W_{\s-1}(1 + R_\s)$ and $W_{\s-1}\Pi_{\s,\nu-1}R_\nu$. Next we will show that $W_{\s-1}(1 + R_{\s,\nu-1})$ and $W_{\s-1}\Pi_{\s,\nu-1}R_\nu$ are essentially independent in the sense that they cannot both be large. This will allow us to use the asymptotics of the tails of $W_{\s-1}(1 + R_\s)$ and $W_{\s-1}\Pi_{\s,\nu-1}R_\nu$ to compute the asymptotics of the tails of $W_{\s-1}(1 + R_{\s,\nu-1})$. \\
To analyze the asymptotics of the tail of $W_{\s-1}(1 + R_\s)$, we first recall from \eqref{Rtail} that there exists a $K>0$ such that $P(R_0 > x)\sim K x^{-s}$. Let $\mathcal{F}_{\s-1} = \s(\ldots,\w_{\s-2}, \w_{\s-1})$ be the $\s-$algebra generated by the environment to the left of $\s$. Then on the event $\{\s<\infty\}$,  $R_\s$ has the same distribution as $R_0$ and is independent of $\mathcal{F}_{\s-1}$. Thus,
\begin{align}
\lim_{x\ra\infty} x^s Q(W_{\s-1}(1+R_\s) > x, \s<\nu) &= \lim_{x\ra\infty} E_Q\left[ x^s Q\left(\left. 1+R_\s>\frac{x}{W_{\s-1}} , \s<\nu \right| \mathcal{F}_{\s-1} \right)  \right] \nonumber \\
&= K E_Q \left[ W_{\s-1}^s \mathbf{1}_{\s<\nu} \right].  \label{Rstail}
\end{align}
A similar calculation yields
\begin{align}
\lim_{x\ra\infty} x^s Q\left( W_{\s-1}\Pi_{\s,\nu-1} R_\nu > x, \s<\nu \right)
&= \lim_{x\ra\infty} E_Q \left[ x^s Q\left(\left. R_\nu > \frac{x}{W_{\s-1}\Pi_{\s,\nu-1}}, \s<\nu \right| \mathcal{F}_{\nu-1}  \right)  \right] \nonumber\\
&= E_Q \left[ W_{\s-1}^s \Pi_{\s,\nu-1}^s \mathbf{1}_{\s<\nu} \right] K. \label{prodtail}
\end{align}
Next, we wish to show that 
\begin{equation}
\lim_{x\ra\infty} x^s Q\left(W_{\s-1}(1+R_{\s,\nu-1}) > \e x, \, W_{\s-1}\Pi_{\s,\nu-1}R_\nu> \e x,\, \s<\nu\right) = 0\,. \label{both}
\end{equation} 
Since $\Pi_{\s,\nu-1} < \frac{1}{A}$ on the event $\{\s<\nu\}$ we have for any $\e>0$ that
\begin{align}
& x^s Q\left(W_{\s-1}(1+R_{\s,\nu-1}) > \e x,\, W_{\s-1}\Pi_{\s,\nu-1}R_\nu> \e x,\, \s<\nu \right) \nonumber
\\ &\qquad\leq  x^s Q\left(W_{\s-1}(1+R_{\s,\nu-1}) > \e x, W_{\s-1}R_\nu> A \e x, \s<\nu \right) \nonumber \\
&\qquad = x^s E_Q \left[ Q\left(1+R_{\s,\nu-1} > \frac{\e x}{W_{\s-1}} | \mathcal{F}_{\s-1}\right) Q\left( R_\nu> A \frac{\e x}{W_{\s-1}} |\mathcal{F}_{\s-1} \right) \mathbf{1}_{\s<\nu} \right] \nonumber \\
&\qquad \leq E_Q \left[ x^s Q\left(1+R_{\s} > \frac{\e x}{W_{\s-1}} | \mathcal{F}_{\s-1}\right) Q\left( R_\nu> A \frac{\e x}{W_{\s-1}} |\mathcal{F}_{\s-1} \right) \mathbf{1}_{\s<\nu} \right]\,, \label{bothbig}
\end{align}
where the equality on the third line is because $R_{\s,\nu-1}$ and $R_\nu$ are independent when $\s<\nu$ (note that $\{ \s < \nu \} \in \mathcal{F}_{\s-1}$), and the last inequality is because $R_{\s,\nu-1} \leq R_\s$. Now, conditioned on $\mathcal{F}_{\s-1}$, $R_\s$ and $R_\nu$ have the same distribution as $R_0$. Then, since by \eqref{Rtail} for any $\gamma\leq s$ there exists a $K_\gamma>0$ such that $P(1+R_0 > x)\leq  K_\gamma x^{-\gamma}$, we have that the integrand in \eqref{bothbig} is bounded above by $K_\gamma^2\e^{-2\gamma} W_{\s-1}^{2\gamma}\mathbf{1}_{\s<\nu}  x^{s-2\gamma},$ $Q-a.s.$ Choosing $\gamma=\frac{s}{2}$ gives that the integrand in \eqref{bothbig} is $Q-a.s.$ bounded above by $K^2_{\frac{s}{2}} \e^{-s} W_{\s-1}^s \mathbf{1}_{\s<\nu}$ which by Lemma \ref{finitemoment} has finite mean. However, if we choose $\gamma =s$ then we get that the integrand of \eqref{bothbig} tends to zero $Q-a.s.$ as $x\ra\infty$. Thus, by the dominated convergence theorem we have that \eqref{both} holds. \\
Now, since $R_\s = R_{\s,\nu-1} + \Pi_{\s,\nu-1}R_\nu$, we have that for any $\e>0$
\begin{align*}
Q(W_{\s-1}(1+R_\s)>(1+\e)x, \s<\nu) &\leq Q(W_{\s-1}(1+R_{\s,\nu-1})>\e x,\, W_{\s-1} \Pi_{\s,\nu-1}R_\nu > \e x,\, \s<\nu )\\
&\qquad + Q(W_{\s-1}(1+R_{\s,\nu-1}) > x,\s<\nu)  \\
&\qquad + Q(W_{\s-1}\Pi_{\s,\nu-1}R_\nu> x,\s<\nu)\,.
\end{align*}
Applying \eqref{Rstail},
\eqref{prodtail} and \eqref{both} we get that for any $\e>0$
\begin{equation}
\liminf_{x\ra\infty} x^s Q(W_{\s-1}(1+R_{\s,\nu-1}) > x, \s<\nu) \geq K E_Q[W_{\s-1}^s\mathbf{1}_{\s<\nu}](1+\e)^{-s} -  K E_Q[W_{\s-1}^s\Pi_{\s,\nu-1}^s\mathbf{1}_{\s<\nu}]\,.\label{lb}
\end{equation}
Similarly, for a bound in the other direction we have
\begin{align*}
Q(W_{\s-1}(1+R_\s)>x, \s<\nu)
&\geq Q(W_{\s-1}(1+R_{\s,\nu-1}) > x \text{, or } W_{\s-1}\Pi_{\s,\nu-1}R_\nu > x, \s<\nu) \\
&= Q(W_{\s-1}(1+R_{\s,\nu-1})>x, \s<\nu) \\
&\qquad + Q(W_{\s-1}\Pi_{\s,\nu-1}R_\nu>x, \s< \nu) \\
&\quad\quad - Q(W_{\s-1}(1+R_{\s,\nu-1})>x, W_{\s-1}
\Pi_{\s,\nu-1}R_\nu>x, \s < \nu)\,.
\end{align*}
Thus, again applying \eqref{Rstail},\eqref{prodtail} and \eqref{both} we get
\begin{equation}
\limsup_{x\ra\infty} x^s Q(W_{\s-1}(1+R_{\s,\nu-1}) > x, \s<\nu) \leq K E_Q[W_{\s-1}^s\mathbf{1}_{\s<\nu}] -  K E_Q[W_{\s-1}^s\Pi_{\s,\nu-1}^s\mathbf{1}_{\s<\nu}]\,.
\label{ub}
\end{equation}
Finally, applying \eqref{lb} and \eqref{ub} and letting $\e\ra 0$, we get that
\[
\lim_{x\ra\infty} x^s Q(W_{\s-1}(1+R_{\s,\nu-1}) > x, \s<\nu)  =  K E_Q[W_{\s-1}^s(1-\Pi_{\s,\nu-1}^s)\mathbf{1}_{\s<\nu}] =: K_A,
\]
and $K_A \in (0,\infty)$ by Lemma \ref{finitemoment} and the fact that $1-\Pi_{\s,\nu-1} \in  (1-\frac{1}{A}, 1)$.
\end{proof}
Finally, we are ready to analyze the tail of $E_\w T_\nu$ under the measure $Q$.
\begin{proof}[\textbf{Proof of Theorem \ref{Tnutail}}:]$\left.\right.$\\
Let $\d>0$, and choose $A\geq A_0(\d)$ as in Lemma
\ref{fourth}.
Then using \eqref{expand} we have
\begin{align*}
Q(E_\w T_\nu > x)
& = Q(E_\w T_\nu > x, \s>\nu)+  Q(E_\w T_\nu > x, \s<\nu)\\
& \leq Q(E_\w T_\nu > x, \s>\nu)+ Q(\nu> \d x) + Q(2W_{-1}R_{0,\s-2}> \d x, \s<\nu) \\
&\quad\quad + Q\left(2\sum_{j=0}^{\s-2} W_{0,j} > \d x, \s<\nu\right) + Q\left(2\sum_{\s\leq i < \nu} R_{i,\nu-1} > \d x\right)\\
&\quad\quad + Q(2W_{\s-1}(1+R_{\s,\nu-1})>(1-4\d)x, \s<\nu)\,.
\end{align*}
Thus combining
equations \eqref{long}, \eqref{first}, \eqref{second}, and \eqref{third}, and Lemmas \ref{fourth} and \ref{fifth}, we get that
\begin{equation}
\limsup_{x\ra\infty} x^s Q(E_\w T_\nu > x) \leq \d + 2^{s}K_A(1-4\d)^{-s}. \label{ub2}
\end{equation}
The lower bound is easier, since $Q(E_\w T_\nu > x) \geq Q(2W_{\s-1}(1+R_{\s,\nu-1})>x, \s<\nu)$. Thus
\begin{equation}
\liminf_{x\ra\infty} x^s Q(E_\w T_\nu > x) \geq 2^{s}K_A\,. \label{lb2}
\end{equation}
From \eqref{ub2} and \eqref{lb2} we get that $\overline{K}:=\limsup_{A\ra\infty} 2^s K_A < \infty$. Therefore, letting $\underline{K}:= \liminf_{A\ra\infty} 2^s K_A$ we have from \eqref{ub2} and \eqref{lb2} that 
\[
\overline{K} \leq  \liminf_{x\ra\infty} x^s Q(E_\w T_\nu > x)  \leq  \limsup_{x\ra\infty} x^s Q(E_\w T_\nu > x)   \leq \d+\underline{K}(1-4\d)^{-s}
\]
Then, letting $\d\ra 0$ completes the proof of the theorem with $K_\infty = \underline{K}=\overline{K}$. 
\end{proof}
\end{section}

\end{chapter}

\begin{chapter}{Quenched Limits: Ballistic Regime}\label{Thesis_AppendixBallistic}

This chapter consists of the article \emph{Quenched Limits for Transient, Ballistic, Sub-Gaussian One-Dimensional Random Walk in Random Environment}, by Jonathon Peterson, which was recently accepted for publication by the Annales de l'Institut Henri Poincar\'e - Probabilit\'es et Statistiques.
This article contains the full proofs of Theorem \ref{Thesis_qCLT}, Theorem \ref{Thesis_qEXP}, Proposition \ref{Thesis_generalprop}, and the first part of Theorem \ref{Thesis_qETVarStable} (sketches of these proofs were provided in Chapter \ref{1dlimitingdist}).  

In order to keep this chapter relatively self-contained, the above mentioned article has been left mostly unchanged. Therefore, much of the introductory material in Section \ref{sl1_Introduction} has already appeared in Chapters \ref{Thesis_Introduction} and \ref{1dlimitingdist}. 
Also, many of the results of this chapter build on the previous results of Chapter \ref{Thesis_AppendixZeroSpeed}.
The notation used in this chapter is consistent with the notation in the previous chapters.

\newpage

\begin{section}{Introduction, Notation, and Statement of Main Results}
Let $\Omega = [0,1]^\Z$ and let $\mathcal{F}$ be the Borel $\s-$algebra on $\Omega$. A random environment is an $\Omega$-valued random variable $\w = \{\w_i\}_{i\in\Z}$ with distribution $P$. We will assume that $P$ is an i.i.d. product measure on $\Omega$.
The \emph{quenched} law $P_\w^x$ for a random walk $X_n$ in the environment $\w$ is defined by
\[
P_\w^x( X_0 = x ) = 1 \quad \text{and} \quad
P_\w^x\left( X_{n+1} = j | X_n = i \right) =
\begin{cases}
\w_i &\quad \text{if } j=i+1, \\
1-\w_i &\quad \text{if } j=i-1.
\end{cases}
\]
$\Z^\N$ is the space for the paths of the random walk $\{X_n\}_{n\in\N}$
and $\mathcal{G}$ denotes the $\s-$algebra generated by the cylinder sets.
Note that for each $\w \in \Omega$, $P_\w$ is a probability measure
on $( \Z^\N, \mathcal{G} )$, and for each $G\in \mathcal{G}$,
$P_\w^x(G):(\Omega, \mathcal{F}) \ra [0,1]$ is a measurable
function of $\w$.  Expectations under the law $P_\w^x$ are denoted $E_\w^x$.
The \emph{annealed} law for the random walk in random
environment $X_n$ is defined by
\[
\P^x(F\times G) = \int_F P_\w^x(G)P(d\w),
\quad F\in \mathcal{F},  G\in \mathcal{G}\!.
\]
For ease of notation, we will use $P_\w$ and $\P$ in place
of $P_\w^0$ and $\P^0$ respectively. We will also use $\P^x$ to
refer to the marginal on the space of paths, i.e. $\P^x(G)=
\P^x(\Omega\times G) = E_P\left[ P^x_\w(G) \right]$ for
$G\in \mathcal{G}$. Expectations under the law $\P$ will be written $\E$.

A simple criterion for recurrence of a one-dimensional RWRE and a formula for the speed of
transience was given by Solomon in \cite{sRWRE}. For any integers
$i\leq j$, let
\begin{equation}
\rho_i := \frac{1-\w_i}{\w_i}  \quad \text{and}\quad
\Pi_{i,j} := \prod_{k=i}^j \rho_k\,. \label{ballistic_rhodef}
\end{equation}
Then, $X_n$ is transient to the right (resp., to the left)
if $E_P(\log \rho_0) < 0$ (resp. $E_P \log \rho_0 > 0$) and recurrent
if $E_P (\log \rho_0) = 0$. (Henceforth we will write $\rho$ instead of
$\rho_0$ in expectations involving only $\rho_0$.) In the case where
$E_P \log\rho < 0$ (transience to the right),
Solomon established the following law of large numbers
\begin{equation}
v_P:= \lim_{n\ra\infty} \frac{X_n}{n} =
\lim_{n\ra\infty} \frac{n}{T_n} = \frac{1}{\E T_1}, \quad \P-a.s. \label{ballistic_XTLLN}
\end{equation}
where $T_n:= \min\{k \geq 0:X_k=n\}$. 
For any integers $i<j$, let
\begin{equation}
W_{i,j} := \sum_{k=i}^j \Pi_{k,j}, \quad \text{and}
\quad W_j := \sum_{k\leq j} \Pi_{k,j}\,. \label{ballistic_Wdef}
\end{equation}
When $E_P \log \rho< 0$, it was shown in 
\cite{zRWRE}
that
\begin{equation}
E_\w^j T_{j+1} = 1+2W_j < \infty, \quad P-a.s., \label{ballistic_QET}
\end{equation}
and thus $v_P =
1/(1+2E_P W_0)$. Since $P$ is a product measure, $E_P W_0 =
\sum_{k=1}^\infty \left(E_P \rho\right)^k$. In particular, $v_P > 0$ if
$E_P \rho < 1$.

Kesten, Kozlov, and Spitzer \cite{kksStable} determined the
annealed limiting distribution of a RWRE with $E_P \log \rho < 0$, i.e.,
transient to the right. They derived the limiting distributions for the walk by first establishing a stable
limit law of index $s$ for $T_n$, where $s$ is defined by the equation
$
E_P\rho^s = 1$. 
In particular, they showed that when $s\in(1,2)$, there exists a  $b>0$ such that
\begin{equation}
\lim_{n\ra\infty} \P\left( \frac{T_n-\E T_n}{n^{1/s}} \leq x \right) =
L_{s,b}(x)\, \label{ballistic_annealedstableT}
\end{equation}
and
\begin{equation}
\lim_{n\ra\infty} \P\left( \frac{X_n - n v_P}{v_P^{1+1/s} n^{1/s}} \leq x
\right) = 1-L_{s,b}(-x), \label{ballistic_annealedstableX}
\end{equation}
where $L_{s,b}$ is the distribution function for a stable random
variable with characteristic function
\[
\hat{L}_{s,b}(t)= \exp\left\{ -b|t|^s \left(
1-i\frac{t}{|t|}\tan(\pi s/2)  \right) \right\}. 
\]
While the annealed limiting distributions for transient
one-dimensional RWRE have been known for quite a while, the corresponding
quenched limiting distributions have remained largely unstudied
until recently. 
In Chapter \ref{Thesis_AppendixQCLT} we proved that when $s>2$ a quenched CLT holds with a
random (depending on the environment)
centering. Goldsheid \cite{gQCLT} has provided an independent proof of this fact. 
Previously, in \cite{kmCLT} and \cite{zRWRE}
it had only been shown that the limiting statements for
the quenched CLT with random centering held in probability (rather than almost surely). In the case when $s<1$, it was shown in Chapter \ref{Thesis_AppendixZeroSpeed} that no quenched limiting distribution exists for the RWRE. In particular, it was shown that $P-a.s.$ there exist two different random sequences $t_k$ and $t_k'$ such that the behavior of the RWRE is either localized (concentrated in a interval of size $\log^2 t_k'$) or spread out (scaling of order $t_k^{s}$). 

In this chapter, we analyze the quenched limiting distributions of a
transient, one-dimensional RWRE in the case $s\in(1,2)$. We show that, as in the case when $s<1$, there is no quenched limiting distribution of the random walk. 
As was done when $s<1$, this is shown by first showing that there is no quenched limiting distribution for the hitting times $T_n$ of the random walk. However, in contrast to the case $s<1$, the existence of a positive speed for the random walk in the case $s\in(1,2)$ allows for a more direct transfer of limiting distributions from $T_n$ to $X_n$ (see Proposition \ref{ballistic_generalprop}). 

Throughout the Chapter, we will
make the following assumptions:
\begin{asm} \label{ballistic_essentialasm}
$P$ is a product measure on $\Omega$ such that
\begin{equation}
E_P \log\rho < 0 \quad\text{and}\quad E_P \rho^s = 1 \text{ for
some } s>0 . \label{ballistic_zerospeedregime}
\end{equation}
\end{asm}
\begin{asm}
The distribution of $\log \rho$ is non-lattice under
$P$ and $E_P ( \rho^s \log\rho ) < \infty$.  \label{ballistic_techasm}
\end{asm}
\noindent\textbf{Remarks:}\\
\textbf{1.} Assumption \ref{ballistic_essentialasm}
contains the essential assumptions for our results. The technical conditions contained in Assumption \ref{ballistic_techasm} were also invoked in \cite{kksStable} and in Chapter \ref{Thesis_AppendixZeroSpeed}. \\
\textbf{2.} Since $E_P \rho^\gamma$ is a convex function of
$\gamma$, the two statements in \eqref{ballistic_zerospeedregime} give that
$E_P \rho^\gamma < 1$ for all
$0<\gamma<s$ and $E_P \rho^\gamma > 1$ for all $\gamma > s$. In particular this implies that $v_P > 0 $ if and only if $ s > 1$. The main results of this Chapter are for $s\in(1,2)$, but many statements hold for a wider range of $s$. If no mention is made of bounds on $s$, then it is assumed that the statement holds for all $s>0$.\\
\textbf{3.} The cases $s\in\{1,2\}$ are not covered here or in Chapter \ref{Thesis_AppendixZeroSpeed}. It is not clear whether or not a quenched CLT holds in the case $s=2$, but we suspect that the results for $s=1$ will be similar to those of the cases $s\in(0,1)$ and $s\in(1,2)$,  i.e., quenched limiting distributions for the random walk do not exist. However, since $s=1$ is the bordering case between the zero-speed and positive-speed regimes, the analysis is likely to be more technical (as was also the case in \cite{kksStable}). 

Let $\Phi(x)$ and $\Psi(x)$ be the distribution functions for a Gaussian and exponential random variable respectively. That is, 
\[
\Phi(x):= \int_{-\infty}^x \frac{1}{\sqrt{2\pi}} e^{-t^2/2} dt \quad\text{and}\quad \Psi(x):= \begin{cases} 0 & x < 0, \\ 1-e^{-x} & x\geq 0. \end{cases}
\]
Our main results are the following:
\begin{thm}\label{ballistic_qCLT}
Let Assumptions \ref{ballistic_essentialasm} and \ref{ballistic_techasm} hold and let $s\in(1,2)$. Then, $P-a.s.$, there exists a random subsequence $n_{k_m}=n_{k_m}(\w)$ of $n_k=2^{2^k}$ and non-deterministic random variables $v_{k_m,\w}$, such that 
\[
\lim_{m\ra\infty} P_\w\left( \frac{ T_{n_{k_m}} - E_\w T_{n_{k_m}} }{ \sqrt{v_{k_m,\w}} } \leq x \right) = \Phi(x), \qquad \forall x\in\R,
\]
and 
\[
\lim_{m\ra\infty} P_\w\left( \frac{ X_{t_m} - n_{k_m} }{v_P \sqrt{v_{k_m,\w}} } \leq x \right) = \Phi(x), \qquad \forall x\in\R,
\]
where $t_m=t_m(\w):= \left\lfloor E_\w T_{n_{k_m}} \right\rfloor$.
\end{thm}
\begin{thm}\label{ballistic_qEXP}
Let Assumptions \ref{ballistic_essentialasm} and \ref{ballistic_techasm} hold and let $s\in(1,2)$. Then, $P-a.s.$, there exists a random subsequence $n_{k_m}=n_{k_m}(\w)$ of $n_k=2^{2^k}$ and non-deterministic random variables $v_{k_m,\w}$, such that 
\[
\lim_{m\ra\infty} P_\w\left( \frac{ T_{n_{k_m}} - E_\w T_{n_{k_m}} }{ \sqrt{v_{k_m,\w}} } \leq x \right) = \Psi(x+1), \qquad \forall x\in\R,
\]
and 
\[
\lim_{m\ra\infty} P_\w \left( \frac{X_{t_m} - n_{k_m}}{v_P \sqrt{v_{k_m,\w}} } \leq x \right) = 1-\Psi(-x+1), \qquad \forall x\in\R,
\] 
where $t_m=t_m(\w):= \left\lfloor E_\w T_{n_{k_m}} \right\rfloor$.
\end{thm}
\noindent\textbf{Remarks:}\\
\textbf{1.} Note that Theorems \ref{ballistic_qCLT} and \ref{ballistic_qEXP}
preclude the possiblity of quenched analogues of the annealed
statements \eqref{ballistic_annealedstableT} and \eqref{ballistic_annealedstableX}. \\
\textbf{2.} 
The choice of Gaussian and exponential distributions in Theorems \ref{ballistic_qCLT} and \ref{ballistic_qEXP} represent the two extremes of the quenched limiting distributions that can be found along random subsequences. In fact, it will be shown in Corollary \ref{ballistic_explimit} that $T_n$ is approximately the sum of a finite number of exponential random variables with random (depending on the environment) parameters. 
The exponential limits in Theorem \ref{ballistic_qEXP} are obtained when one of the exponential random variables has a much larger parameter than all the others. 
The Gaussian limits in Theorem \ref{ballistic_qCLT} are obtained when the exponential random variables with the largest parameters all have roughly the same size. 
We expect, in fact, that any distribution which is the sum of (or limit of sums of) exponential random variables can be obtained as a quenched limiting distribution of $T_n$ along a random subsequence. \\
\textbf{3.} The sequence $n_k=2^{2^k}$ in Theorems \ref{ballistic_qCLT} and \ref{ballistic_qEXP} is chosen only for convenience. In fact, for any sequence $n_k$ growing sufficiently fast, $P-a.s.$, there will be a random subsequence $n_{k_m}(\w)$ such that the conclusions of Theorems \ref{ballistic_qCLT} and \ref{ballistic_qEXP} hold. \\
\textbf{4.} The definition of $v_{k_m,\w}$ is given below in \eqref{ballistic_dkvkdef}. By an argument similar to the proof of Theorem \ref{ballistic_Varstable}, it can be shown that 
$\lim_{n\ra\infty} P\left( n_k^{-2/s} v_{k,\w} \leq x \right) = L_{\frac{s}{2},b}(x)$ for some $b>0$. Also, from \eqref{ballistic_XTLLN}, we have that $t_m \sim \E T_1 n_{k_m}$. Thus, the scaling in Theorems \ref{ballistic_qCLT} and \ref{ballistic_qEXP} is of the same order as the annealed scaling, but it cannot be replaced by a deterministic scaling.

Define the ``ladder locations'' $\nu_i$ of the environment by
\begin{align}
\nu_0 = 0, \quad\text{and}\quad \nu_i =
\begin{cases}
\inf\{n > \nu_{i-1}: \Pi_{\nu_{i-1},n-1} < 1\}, &\quad  i \geq 1,\\
\sup \{j < \nu_{i+1}: \Pi_{k,j-1}<1,\quad \forall k<j \}, &\quad  i \leq -1
\,.\end{cases}
\label{ballistic_nudef}
\end{align}
Throughout the remainder of the chapter we will let $\nu=\nu_1$. 
We will sometimes refer to sections of the environment between
$\nu_{i-1}$ and $\nu_i -1$ as ``blocks'' of the environment. Note
that the block between $\nu_{-1}$ and $\nu_0 -1$ is different from
all the other blocks between consecutive ladder locations (in particular, $\Pi_{\nu_{-1},\nu_0-1} \geq 1$ is possible), and that all the other blocks have the same distribution as the block from $0$ to $\nu-1$.  As in Chapter \ref{Thesis_AppendixZeroSpeed}, we define
the measure $Q$ on environments by $Q(\cdot)=P(\cdot\,|\mathcal{R})$, where
\[
\mathcal{R}:=\{ \w\in\Omega: \Pi_{-k,-1} < 1,\quad\forall k \geq 1\} = \left\{ \w \in \Omega: \sum_{i=-k}^{-1} \log \rho_i < 0, \quad \forall k \geq 1 \right\}.
\]
Note that $P(\mathcal{R}) > 0$ since $E_P \log \rho < 0$. 
$Q$ is defined so that the blocks of the environment between ladder locations are i.i.d.
under $Q$, all with the same distribution as the block
from $0$ to $\nu -1$ under $P$. In particular, $P$ and $Q$ agree on $\s( \w_i: i\geq 0)$. 

For any random variable $Z$, define the quenched variance $Var_\w Z := E_\w (Z-E_\w Z)^2$. 
In Theorem \ref{refstable}, it was proved that when $s\in(0,1)$, $n^{-1/s} E_\w T_{\nu_n}$ converges in distribution (under $Q$) to a stable distribution of index $s$. 
Correspondingly, when $s<2$ we will prove the following theorem:
\begin{thm} \label{ballistic_Varstable}
Let Assumptions \ref{ballistic_essentialasm} and \ref{ballistic_techasm} hold and let $s<2$. Then, there exists a $b>0$ such that 
\begin{equation}
\lim_{n\ra\infty} Q\left(\frac{Var_\w T_{\nu_n}}{n^{2/s}} \leq x \right) = \lim_{n\ra\infty} Q\left( \frac{1}{n^{2/s}} \sum_{i=1}^n \left( E_\w^{\nu_{i-1}} T_{\nu_i} \right)^2 \leq x \right) = L_{\frac{s}{2},b}(x) \, .  \label{ballistic_stableET2}
\end{equation}
\end{thm}
\noindent\textbf{Remarks:} \\
\textbf{1.} The constant $b$ in the above theorem may not be the same as in \eqref{ballistic_annealedstableT} and \eqref{ballistic_annealedstableX}. \\
\textbf{2.} Theorem \ref{ballistic_Varstable} can be used to show that $\lim_{n\ra\infty} P\left(\frac{Var_\w T_{n}}{n^{2/s}} \leq x \right) = L_{\frac{s}{2},b'}(x)$ for some $b'>0$, but we will not prove this since we do not use it for the other results in this chapter.

A major difficulty in analyzing $T_{\nu_n}$ is that the crossing time from $\nu_{i-1}$ to $\nu_i$ depends on the entire environment to the left of $\nu_i$. Thus, $Var_\w (T_{\nu_i} - T_{\nu_{i-1}})$ and $Var_\w (T_{\nu_j} - T_{\nu_{j-1}})$ are not independent even if $|i-j|$ is large. 
In order to make the crossing times of blocks that are far apart essentially
independent, we introduce some reflections to the RWRE. For
$n=1,2,\ldots$, define
\begin{equation}
b_n:= \lfloor \log^2(n) \rfloor. \label{ballistic_bdef}
\end{equation}
Let $\bar{X}_t^{(n)}$ be the random walk that is the same as $X_t$
with the added condition that, after reaching $\nu_k$, the
environment is modified by setting $\w_{\nu_{k-b_n}} = 1 $ (i.e.,
never allow the walk to backtrack more than $\log^2(n)$ blocks). 
We couple $\bar{X}_t^{(n)}$ with the random walk $X_t$ in such a way that $\bar{X}_t^{(n)} \geq X_t$ with equality holding until
the first time $t$ when the walk $\bar{X}_t^{(n)}$ reaches a modified environment location.
Denote by $\bar{T}_{x}^{(n)}$ the corresponding hitting
times for the walk $\bar{X}_t^{(n)}$. 
It was shown in Lemma \ref{bt} that $\lim_{n\ra\infty} P_\w( T_{\nu_n} \neq \bar{T}_{\nu_n}^{(n)} ) = 0$, $P-a.s.$, so that, in fact, with high probability the added reflections do not affect the walk at all before $T_{\nu_n}$. 
For ease of notation, let 
\[
\mu_{i,n,\w} := E_\w^{\nu_{i-1}} \bar{T}^{(n)}_{\nu_i}, \quad\text{and}\quad \s_{i,n,\w}^2 := Var_\w \left( \bar{T}^{(n)}_{\nu_i} - \bar{T}^{(n)}_{\nu_{i-1}} \right). 
\]

The structure of the chapter is as follows.
In Section \ref{ballistic_gp}, we prove the following general proposition that allows us to easily transfer quenched limit laws from subsequences of $T_n$ to $X_n$.
\begin{prop} \label{ballistic_generalprop}
Let Assumptions \ref{ballistic_essentialasm} and \ref{ballistic_techasm} hold and let $s\in(1,2)$. Also, let $n_k$ be a sequence of integers growing fast enough so that $\lim_{k\ra\infty} \frac{n_k}{n_{k-1}^{1+\d}} = \infty$ for some $\d>0$, and let
\begin{equation}
d_k:= n_k-n_{k-1}\quad\text{and}\quad v_{k,\w} := \sum_{i=n_{k-1}+1}^{n_k} \s_{i,d_k,\w}^2 = Var_\w \left( \bar{T}^{(d_{k})}_{\nu_{n_k}} - \bar{T}^{(d_{k})}_{\nu_{n_{k-1}}} \right) \, . \label{ballistic_dkvkdef}
\end{equation}
Assume that $F$ is a continuous distribution function for which, $P-a.s.$, there exists a subsequence $n_{k_m}= n_{k_m}(\w)$ such that, for $\a_m:= n_{k_m-1}$,
\[
\lim_{m\ra\infty} P_\w^{\nu_{\a_m}}\left( \frac{\bar{T}^{(d_{k_m})}_{x_m} - E_\w^{\nu_{\a_m}}\bar{T}^{(d_{k_m})}_{x_m} }{\sqrt{v_{k_m,\w}}} \leq y \right) = F(y), \quad \forall y\in \R,
\]
for any sequence $x_m \sim n_{k_m}$. Then, $P-a.s.$, for all $y\in \R$,
\begin{equation}
\lim_{m\ra\infty} P_\w\left( \frac{T_{x_m} - E_\w T_{x_m} }{\sqrt{v_{k_m,\w}}} \leq y \right) = F(y), \label{ballistic_Tlim}
\end{equation}
for any $x_m\sim n_{k_m}$, and 
\begin{equation}
\lim_{m\ra\infty} P_\w\left( \frac{X_{t_m} - n_{k_m} }{ v_P\sqrt{v_{k_m,\w}}} \leq y \right) = 1-F(-y),   \label{ballistic_Xlim}
\end{equation}
where $t_m:= \left\lfloor E_\w T_{n_{k_m}} \right\rfloor$. 
\end{prop}
Then in Sections \ref{ballistic_qGauss} and \ref{ballistic_exponential}, we use Theorem \ref{ballistic_Varstable} to find subsequences $n_{k_m}(\w)$ that allow us to apply Proposition \ref{ballistic_generalprop}. To find a subsequence that gives Gaussian behavior of $T_{n_{k_m}}$, we find a subsequence where none of the crossing times of the first $n_{k_m}$ blocks is too much larger than all the others and then use the Linberg-Feller condition for triangular arrays. In contrast, to find a subsequence that gives exponential behavior of $T_{n_{k_m}}$, we first prove that the crossing times of ``large'' blocks is approximately exponential in distribution. Then, we find a subsequence where the crossing time of one of the first $n_{k_m}$ blocks dominates the total crossing time of the first $n_{k_m}$ blocks. Finally, Section \ref{ballistic_qvs} contains the proof of Theorem \ref{ballistic_Varstable}, which is similar to the proof of Theorem \ref{refstable}.

Before continuing with the proofs of the main theorems, we recall some notation and results from Chapters \ref{1dlimitingdist} and  \ref{Thesis_AppendixZeroSpeed} that will be used throughout the Chapter.
First, recall that from Lemma \ref{nutail} there exist constants $C_1,C_2>0$ such that 
\begin{equation}
P(\nu > x) \leq C_1 e^{-C_2 x}, \quad \forall x\geq 0. \label{ballistic_nutail}
\end{equation}
Then, since $\nu_n = \sum_{i=1}^n \nu_i - \nu_{i-1}$ and the $\nu_i-\nu_{i-1}$ are i.i.d., the law of large numbers implies that
\begin{equation}
\lim_{n\ra\infty} \frac{\nu_n}{n} = E_P \nu =: \bar \nu < \infty, \qquad P-a.s. \label{ballistic_nuLLN}
\end{equation}
In Chapter \ref{Thesis_AppendixZeroSpeed} the following formulas for the quenched expectation and variance of $T_\nu$ were given:
\begin{equation}
E_\w T_\nu = \nu + 2 \sum_{j=0}^{\nu-1} W_j, \quad\text{and}\quad Var_\w T_\nu = 4\sum_{j=0}^{\nu-1}(W_{j}+W_{j}^2) + 8\sum_{j=0}^{\nu-1}\sum_{i< j} \Pi_{i+1,j}(W_{i}+W_{i}^2). \label{ballistic_qvformula}
\end{equation}
Note that since the added reflections only decrease the crossing times, obviously $T_\nu \geq \bar{T}^{(n)}_\nu$ and $E_\w T_\nu \geq E_\w \bar{T}^{(n)}_\nu$ for any $n$. Also, since \eqref{ballistic_qvformula} holds for any environment $\w$, the formula for $Var_\w \bar{T}^{(n)}_\nu$ is the same as in \eqref{ballistic_qvformula}, but with $\rho_{\nu_{-b_n}}$ replaced by 0. In particular, this shows that $Var_\w T_\nu \geq Var_\w \bar{T}^{(n)}_\nu$ for any $n$. As in Chapter \ref{Thesis_AppendixZeroSpeed}, for any integer $i$, let
\begin{equation}
M_i:=\max \{ \Pi_{\nu_{i-1}, j} : j\in[\nu_{i-1},\nu_i) \} \, . \label{ballistic_Mdef}
\end{equation} 
Then, \cite[Theorem 1]{iEV} implies that there exists a constant $C_3<\infty$ such that 
\begin{equation}
Q(M_i > x) = P(M_1 > x) \sim C_3 x^{-s}. 	\label{ballistic_Mtail}
\end{equation}
Note that $M_1 \leq \max_{0\leq j<\nu} W_j$. Therefore,
from the formulas for $E_\w T_\nu$ and $Var_\w T_\nu$ in \eqref{ballistic_qvformula}, it is easy to see that $E_\w T_\nu  \geq M_1$ and $Var_\w T_\nu  \geq M_1^2$. (The same is also true for $\bar{T}_\nu^{(n)}$.) 
Finally, recall the following results from Chapter \ref{Thesis_AppendixZeroSpeed}:
\begin{thm}[\textbf{Lemma \ref{reftail} \& Theorem \ref{qVartail}}] \label{ballistic_VETtail}
There exists a constant $K_\infty \in (0,\infty)$ such that 
\[
Q\left( Var_\w T_\nu > x \right) \sim Q\left( (E_\w T_\nu)^2 > x \right) \sim K_\infty x^{-s/2}, \qquad \text{as } x\ra\infty.
\]
Moreover, for any $\e>0$ and $x>0$, 
\begin{align*}
Q\left( Var_\w \bar{T}^{(n)}_\nu > x n^{2/s}, \: M_1 > n^{(1-\e)/s} \right) \sim Q\left( \left(E_\w \bar{T}^{(n)}_\nu\right)^2 > x n^{2/s}, \: M_1 > n^{(1-\e)/s} \right) \sim K_\infty x^{-s/2} \frac{1}{n}, 
\end{align*}
as $n\ra\infty$. 
%
\end{thm}

\end{section}

\begin{section}{Converting Time Limits to Space Limits}\label{ballistic_gp}
In this section, we develop a general method for transferring a quenched limit law for a subsequence of $T_n$ to a quenched limit law for a subsequence of $X_n$. We begin with some lemmas analyzing the a.s. asymptotic behavior of the quenched variance and mean of the hitting times.
\begin{lem} \label{ballistic_Vmdublb}
Assume $s\leq 2$. Then, for any $\d > 0$, 
\[
Q\left( Var_\w \bar{T}_{\nu_n}^{(n)} \notin \left(n^{2/s - \d}, n^{2/s+\d}\right) \right) \leq \frac{1}{P(\mathcal{R})} P\left( Var_\w \bar{T}_{\nu_n}^{(n)} \notin \left(n^{2/s - \d}, n^{2/s+\d}\right) \right) = o\left( n^{-\d s/4} \right)  \,.
\]
\end{lem}
\begin{proof}
The first inequality in the lemma is trivial since for any $A\in\mathcal{F}$, it follows from the definition of $Q$ that $Q(A)= \frac{P(A\cap \mathcal{R})}{P(\mathcal{R})} \leq \frac{P(A)}{P(\mathcal{R})}$. Next, note that when $s\leq 2$, Lemma \ref{Varasymp} implies
\begin{equation}
P\left( Var_\w \bar{T}_{\nu_n}^{(n)} \geq n^{2/s+\d} \right) \leq P\left( Var_\w T_{\nu_n} \geq n^{2/s + \d} \right) = o( n^{-\d
s/4} )\,. \label{ballistic_nrub}
\end{equation}
Also, since $Var_\w (\bar{T}_{\nu_{i}}^{(n)} - \bar{T}_{\nu_{i-1}}^{(n)}) \geq M_i^2$, 
\[
P\left( Var_\w \bar{T}_{\nu_n}^{(n)} \leq n^{2/s - \d} \right) \leq P\left( M_1^2 \leq n^{2/s-\d}\right)^n = \left( 1 - P\left( M_1 > n^{1/s-\d/2} \right) \right)^n = o\left( e^{-n^{\d s/4}} \right) \, ,
\] 
where the last equality is from \eqref{ballistic_Mtail}.
\end{proof}
\begin{cor}\label{ballistic_vkasym}
Assume $s\leq 2$. Then, for any $\d>0$,
\[
P\left( v_{k,\w} \notin \left( d_k^{2/s -\d}, d_k^{2/s + \d} \right) \right) = o\left( d_k^{-\d s /4} \right).
\]
Consequently, if $s<2$, then $\sqrt{v_{k,\w}} = o(d_k)$, $P-a.s.$ 
\end{cor}
\begin{proof}
Recall from \eqref{ballistic_dkvkdef} that by definition $v_{k,\w} = Var_\w \left( \bar{T}_{\nu_{n_k}}^{(d_k)} - \bar{T}_{\nu_{n_{k-1}}}^{(d_k)} \right)$. Also, note that the conditions on $n_k$ ensure that $n_k$ grows faster than exponentially and that $d_k\sim n_k$. Thus, for all $k$ large enough $v_{k,\w}$ only depends on the environment to the right of zero. Therefore for all $k$ large enough
\begin{align*}
P\left( v_{k,\w} \notin \left( d_k^{2/s -\d}, d_k^{2/s + \d} \right) \right) &= Q\left( Var_\w \left( \bar{T}_{\nu_{n_k}}^{(d_k)} - \bar{T}_{\nu_{n_{k-1}}}^{(d_k)} \right) \notin \left( d_k^{2/s -\d}, d_k^{2/s + \d} \right) \right) \\
&= Q\left( Var_\w  \bar{T}_{\nu_{d_k}}^{(d_k)} \notin \left( d_k^{2/s -\d}, d_k^{2/s + \d} \right) \right) = o\left( d_k^{-\d s /4} \right),
\end{align*}
where the last equality is from Lemma \ref{ballistic_Vmdublb}. Now, for the second claim in the corollary, first note that $2 > \frac{2}{s} + \frac{s-1}{s}$ since $s>1$. Therefore, for any $\e>0$ and for all $k$ large enough we have 
\[
P\left( v_{k,\w} > \e d_k^2 \right) \leq P\left( v_{k,\w} >  d_k^{2/s+(s-1)/s} \right) = o\left( d_k^{-(s-1)/4} \right).
\]
This last term is summable since $d_k$ grows faster than exponentially. Thus the Borel-Cantelli Lemma gives that $v_{k,\w} = o(d_k^2)$, $P-a.s.$
\end{proof}

\begin{cor} \label{ballistic_begVar}
Assume $s\leq 2$. Then 
\[
\lim_{k\ra\infty} \frac{Var_\w T_{\nu_{n_{k-1}}} }{ v_{k,\w}  } = 0, \quad P-a.s.
\]
\end{cor}
\begin{proof}
By the Borel-Cantelli Lemma it is enough to prove that for any $\e>0$ 
\[
\sum_{k=1}^{\infty} P\left( Var_\w T_{\nu_{n_{k-1}}} \geq \e v_{k,\w} \right) < \infty
\]
However, for any $\d>0$ we have
\begin{equation}
P\left( Var_\w T_{\nu_{n_{k-1}}} \geq \e v_{k,\w} \right) \leq P\left( Var_\w T_{\nu_{n_{k-1}}} \geq \e d_k^{2/s - \d} \right) + P\left( v_{k,\w} \leq d_k^{2/s-\d} \right). \label{ballistic_Varmd}
\end{equation}
By Corollary \ref{ballistic_vkasym} the last term in \eqref{ballistic_Varmd} is summable for any $\d>0$. To show that the second to last term in \eqref{ballistic_Varmd} is also summable first note that the conditions on the sequence $n_k$ give that there exists a $\d >0$ such that $\e d_k^{2/s - \d} \geq n_{k-1}^{2/s+\d}$ for all $k$ large enough. Thus, for some $\d>0$ and all $k$ large enough we have
\[
P\left( Var_\w T_{\nu_{n_{k-1}}} > \e d_k^{2/s + \d} \right) \leq P\left( Var_\w T_{\nu_{n_{k-1}}} > n_{k-1}^{2/s - \d} \right) = o(n_{k-1}^{-\d s/ 4}),
\]
where the last equality is from \eqref{ballistic_nrub}. 
\end{proof}
\begin{lem} \label{ballistic_ETaverage}
Assume $s\in(1,2)$. Then 
$\E T_1 <\infty$, and $P-a.s.$
\begin{equation}
\lim_{k\ra\infty} \frac{E_\w T_{n_k+\lceil x \sqrt{v_{k,\w}} \rceil } - E_\w T_{n_k}}{ \sqrt{v_{k,\w}}} 
= x \E T_1, \quad \forall x\in\R. \label{ballistic_ETavg}
\end{equation}
\end{lem}
\begin{proof}
Now, since $\frac{E_\w T_{n_k+\lceil x \sqrt{v_{k,\w}} \rceil } - E_\w T_{n_k}}{ \sqrt{v_{k,\w}}}$ is monotone in $x$ it is enough to prove that for arbitrary $x\in\Q$ the limiting statement in \eqref{ballistic_ETavg} holds. Obviously this is true when $x=0$ since both sides are zero. For the remainder of the proof we'll assume $x>0$. The proof for $x<0$ is essentially the same (recall that by Corollary \ref{ballistic_vkasym} $v_{k,\w}= o(d_k)=o(n_k)$ when $s<2$). Note that for $x\geq 0$ then we can re-write $E_\w T_{n_k+\lceil x \sqrt{v_{k,\w}} \rceil } - E_\w T_{n_k}= E_\w^{n_k} T_{n_k + \lceil x \sqrt{v_{k,\w}} \rceil }$. By the Borel-Cantelli Lemma it is enough to show that for any $\e>0$,
\begin{equation}
\sum_{k=1}^\infty P\left( \left| E_\w^{n_k} T_{n_k+\lceil x \sqrt{v_{k,\w}} \rceil }  - \lceil x \sqrt{v_{k,\w}} \rceil  \E T_1   \right| \geq \e  \sqrt{v_{k,\w}}  \right) < \infty \, . \label{ballistic_bc}
\end{equation}
However, for any $\d>0$ we have
\begin{align}
&P\left( \left| E_\w^{n_k} T_{n_k+\lceil x \sqrt{v_{k,\w}} \rceil } - \lceil x \sqrt{v_{k,\w}} \rceil  \E T_1   \right| \geq \e\sqrt{v_{k,\w}}  \right) \nonumber \\
&\:\: \leq P\left( \exists m \in \left[ \lceil x d_k^{1/s-\d} \rceil, \lceil x d_k^{1/s+\d} \rceil \right]: \left| E_\w^{n_k} T_{n_k+ m  } -  m \E T_1   \right| \geq \frac{\e m}{x} \right) + P\left( v_{k,\w} \notin \left[d_k^{2/s-2\d}, d_k^{2/s+2\d}\right] \right) \nonumber \\
&\:\: \leq P\left( \max_{m\leq \lceil x d_k^{1/s+\d} \rceil } \left| E_\w T_m - m\E T_1 \right| \geq \e  d_k^{1/s-\d}  \right) + o( d_k^{-\d s/2} ), \label{ballistic_rv}
\end{align}
where the last inequality is due to Corollary \ref{ballistic_vkasym} and the fact that $\{E_\w^{n_k} T_{n_k+m}\}_{m\in\Z}$ has the same distribution as $\{ E_\w T_m \}_{m\in \Z}$ since $P$ is a product measure. 
Thus, we only need to show that the first term in \eqref{ballistic_rv} is summable in $k$ for some $\d>0$. For this, we need the following lemma whose proof we defer.
\begin{lem} \label{ballistic_maxLD}
Assume $s\in(1,2]$. Then for any $0<\d' < \frac{s-1}{2s}$ we have that
\[
P\left( \max_{m\leq n} \left| E_\w T_m - m \E T_1 \right| \geq n^{1-\d'} \right) =o \left( n^{-(s-1)/2} \right)  
\]
\end{lem}
Assuming Lemma \ref{ballistic_maxLD}, fix $0 < \d' < \frac{s-1}{2s}$ and then choose $0<\d < \frac{\d'}{s(2-\d')}$. We choose $\d$ and $\d'$ this way to ensure that $(1/s+\d)(1-\d') < 1/s-\d$. Therefore, for all $k$ large enough, $ \e  d_k^{1/s-\d}  > \left\lceil x d_k^{1/s+\d} \right\rceil^{1-\d'}$. Thus for all $k$ large enough we have
\begin{align*}
P\left( \max_{m\leq \lceil x d_k^{1/s+\d}\rceil } \left| E_\w T_m - m\E T_1 \right| \geq \e  d_k^{1/s-\d}  \right) &\leq P\left( \max_{m\leq \lceil x d_k^{1/s+\d}\rceil} \left| E_\w T_m - m \E T_1\right| \geq \left\lceil x d_k^{1/s+\d} \right\rceil^{1-\d'} \right)\\
&=   o\left( d_k^{-(1/s+\d)(s-1)/2} \right), \qquad \text{as } k\ra\infty.
\end{align*}
Since $s>1$ this last term is summable in $k$. 
\end{proof}
\begin{proof}[Proof of Lemma \ref{ballistic_maxLD}:]
Before proceeding with the proof we need to introduce some notation for a slightly different type of reflection. Define $\tilde{X_t}^{(n)}$ to be the RWRE modified so that it cannot backtrack a distance of $b_n$ (the definition of $\bar{X_t}^{(n)}$ is similar except the walk was not allowed to backtrack $b_n$ blocks instead). That is, after the walk first reaches location $i$, we modify the environment by setting $\w_{i-b_n}=1$. Let $\tilde{T_x}^{(n)}$ be the corresponding hitting times of the walk $\tilde{X_t}^{(n)}$. Then
\begin{align}
P\left( \max_{m\leq n} \left| E_\w T_m - m \E T_1 \right| \geq n^{1-\d'} \right)
& \leq  P\left( E_\w T_n - E_\w \tilde{T}_n^{(n)}  \geq \frac{n^{1-\d'}}{3} \right)  
+  P\left( \E T_1 - \E \tilde{T}_1^{(n)}  \geq \frac{n^{-\d'}}{3} \right)\nonumber \\
&\quad +  P\left( \max_{m\leq n} \left| E_\w \tilde{T}_m^{(n)} - m \E \tilde{T}_1^{(n)} \right| \geq \frac{n^{1-\d'}}{3} \right) \nonumber \\
& \leq 3 n^{-1+\d'}  (\E T_n - \E \tilde{T}_n^{(n)})  
+ \mathbf{1}_{\E T_1 - \E \tilde{T}_1^{(n)} \geq n^{-\d'}/3} \nonumber \\
&\quad + P\left( \max_{m\leq n} \left| E_\w \tilde{T}_m^{(n)} - m \E \tilde{T}_1^{(n)} \right| \geq \frac{n^{1-\d'}}{3} \right) \label{ballistic_aref} 
\end{align}
Now, from \eqref{ballistic_QET} we get that $ E_\w T_1 - E_\w \tilde{T}_1^{(n)}  =  (1+2W_0)-(1+2W_{-b_n+1,0}) = 2 \Pi_{-b_n+1,0} W_{-b_n} $, and thus since $P$ is a product measure
\begin{equation}
\E T_n - \E \tilde{T}_n^{(n)} =  n E_P \left( E_\w T_1 - E_\w \tilde{T}_1^{(n)} \right) = \frac{ 2 n }{1-E_P \rho} (E_P \rho)^{b_n+1}. \label{ballistic_ett}
\end{equation}
Since $E_P \rho < 1$ and $b_n \sim \log^2 n$ the above decreases faster than any power of $n$. Thus by \eqref{ballistic_aref} we need only to show that $P\left( \max_{m\leq n} \left| E_\w \tilde{T}_m^{(n)} - m \E \tilde{T}_1^{(n)} \right| \geq \frac{n^{1-\d'}}{3} \right) = o(n^{-(s-1)/2})$. 
For ease of notation we define $\kappa_{i}^{(n)}:=E_\w^{i-1} \tilde{T}_i^{(n)} - \E \tilde{T}_1^{(n)}$. Thus, since $E_\w \tilde{T}_m^{(n)} - m \E \tilde{T}_1^{(n)} = \sum_{i=1}^m \kappa_{i}^{(n)} = \sum_{i=1}^{b_n} \sum_{j=0}^{\left\lfloor \frac{m-i}{b_n} \right\rfloor} \kappa_{jb_n+i}^{(n)} $, we have
\begin{align}
P\left( \max_{m\leq n} \left| E_\w \tilde{T}_m^{(n)} - m \E \tilde{T}_1^{(n)} \right| \geq \frac{n^{1-\d'}}{3} \right) 
&\leq P\left(  \max_{m\leq n} \sum_{i=1}^{b_n} \left| \sum_{j=0}^{\left\lfloor \frac{m-i}{b_n} \right\rfloor} \kappa_{jb_n+i}^{(n)} \right| \geq \frac{n^{1-\d'}}{3}  \right) \nonumber \\
&\leq \sum_{i=1}^{b_n} P\left( \max_{m\leq n}  \left| \sum_{j=0}^{\left\lfloor \frac{m-i}{b_n} \right\rfloor} \kappa_{jb_n+i}^{(n)} \right| \geq \frac{n^{1-\d'}}{3 b_n}  \right) \nonumber \\
&= \sum_{i=1}^{b_n} P\left( \max_{l\leq \left\lfloor \frac{n-i}{b_n} \right\rfloor}  \left| \sum_{j=0}^{l} \kappa_{jb_n+i}^{(n)} \right| \geq \frac{n^{1-\d'}}{3 b_n}  \right). \label{ballistic_depldb}
\end{align}
Due to the reflections of the random walk, $\kappa_i^{(n)}$ depends only on the environment between $i-b_n$ and $i-1$. Thus, for each $i$ $\{\kappa_{jb_n+i}^{(n)}\}_{j=0}^{\infty}$ is a sequence of $i.i.d.$ random variables with zero mean, and so $\{\sum_{j=0}^l \kappa_{jb_n+i}^{(n)}\}_{l\geq 0}$ is a martingale. Now, let $\gamma \in (1,s)$. Then, by the Doob-Kolmogorov inequality, for any integer $N$ we have
\[
P\left( \max_{l\leq N}  \left| \sum_{j=0}^{l} \kappa_{jb_n+i}^{(n)} \right| \geq \frac{n^{1-\d'}}{3 b_n}  \right) \leq 3^\gamma b_n^\gamma n^{-\gamma+\gamma\d'} E_P\left| \sum_{j=0}^N \kappa_{jb_n+i}^{(n)}  \right|^\gamma.
\]
Since $\{\kappa_{jb_n+i}^{(n)}\}_{j=0}^{\infty}$ is a sequence of independent, zero-mean random variables, the Marcinkiewicz-Zygmund inequality \cite[Theorem 2]{cPTIIM} implies that there exists a constant $B_\gamma<\infty$ depending only on $\gamma > 1$ such that 
\[
E_P\left| \sum_{j=0}^N \kappa_{jb_n+i}^{(n)}  \right|^\gamma \leq B_\gamma E_P \left| \sum_{j=0}^N \left(\kappa_{jb_n +i}^{(n)}\right)^2 \right|^{\gamma/2} \leq B_\gamma E_P \left( \sum_{j=0}^N \left| \kappa_{jb_n +i}^{(n)}\right|^\gamma \right) = B_\gamma (N+1) E_P|\kappa_1^{(n)}|^\gamma, 
\]
where the second inequality is because $\gamma < s \leq 2$ implies $\gamma/2 < 1$. 
Now, recall from \cite{kksStable} that $P(E_\w T_1 > x) \sim K x^{-s}$ for some $K>0$. Therefore, since $\gamma < s$ we have that $E_P |E_\w T_1|^\gamma < \infty$.  Thus, it's easy to see that $E_P|\kappa_1^{(n)}|^\gamma = E_P \left| E_\w \tilde{T}_1^{(n)} - \E \tilde{T}_1^{(n)} \right|^\gamma $ is uniformly bounded in $n$. So, there exists a constant $B_\gamma '$ depending on $\gamma \in (1,s)$ such that
\[
P\left( \max_{l\leq N}  \left| \sum_{j=0}^{l} \kappa_{jb_n+i}^{(n)} \right| \geq \frac{n^{1-\d'}}{3 b_n}  \right) \leq B'_\gamma b_n^\gamma n^{-\gamma+\gamma\d'} (N+1),
\]
and thus by \eqref{ballistic_depldb}
\[
P\left( \max_{m\leq n} \left| E_\w \tilde{T}_m^{(n)} - m \E \tilde{T}_1^{(n)} \right| \geq \frac{n^{1-\d'}}{3} \right) \leq B'_\gamma b_n^{\gamma+1} n^{-\gamma+\gamma\d'} \left( \frac{n}{b_n}+1 \right) = \bigo\left( b_n^\gamma n^{1-\gamma + \gamma \d'} \right).
\]
Since by assumption we have $\d'<\frac{s-1}{2s}$, we may choose $\gamma < s$ arbitrarily close to $s$ so that $b_n^\gamma n^{-\gamma +1+ \gamma \d'} = o\left(n^{-(s-1)/2}\right)$.
\end{proof}
\begin{proof}[\textbf{Proof of Proposition \ref{ballistic_generalprop}:}] \ \\
Recall the definition of $\a_m:= n_{k_m-1}$. To prove \eqref{ballistic_Tlim} it is enough to prove that $\forall \e>0$
\begin{equation}
\lim_{m\ra\infty} P_\w\left( \left| \frac{T_{\nu_{\a_m}} - E_\w T_{\nu_{\a_m}} }{\sqrt{v_{k_m,\w}}} \right| \geq \e \right) = 0,  \quad P-a.s. \label{ballistic_rmstart}
\end{equation}
and
\begin{equation}
\lim_{m\ra\infty} P_\w^{\nu_{\a_m}} \left( T_{x_m} \neq \bar{T}^{(d_{k_m})}_{x_m} \right) = 0, \quad\text{and}\quad \lim_{m\ra\infty} E_\w^{\nu_{\a_m}} \left( T_{x_m} - \bar{T}^{(d_{k_m})}_{x_m} \right) = 0, \quad P-a.s.  \label{ballistic_addreflections}
\end{equation}
To prove \eqref{ballistic_rmstart}, note that by Chebychev's inequality 
\[
P_\w\left( \left| \frac{T_{\nu_{\a_m}} - E_\w T_{\nu_{\a_m}} }{\sqrt{v_{k_m,\w}}} \right| \geq \e \right) \leq \frac{ Var_\w T_{\nu_{\a_m}} }{ \e^2 v_{k_m,\w} },
\]
which by Corollary \ref{ballistic_begVar} tends to zero $P-a.s.$ as $m\ra\infty$. Secondly, to prove \eqref{ballistic_addreflections}, note that since
\[
P_\w^{\nu_{\a_m}} \left( T_{x_m} \neq \bar{T}^{(d_{k_m})}_{x_m} \right) = P_\w^{\nu_{\a_m}} \left( T_{x_m} - \bar{T}^{(d_{k_m})}_{x_m} \geq 1 \right) \leq E_\w^{\nu_{\a_m}} \left( T_{x_m} - \bar{T}^{(d_{k_m})}_{x_m} \right),
\]
it is enough to prove only the second claim \eqref{ballistic_addreflections}. However, since $x_m \leq 2 n_{k_m}$ for all $m$ large enough, it is enough to prove
\begin{equation}
\lim_{k\ra\infty}  E_\w \left( T_{2n_k} - \bar{T}^{(d_k)}_{2n_k} \right) = 0, \quad P-a.s. \label{ballistic_rdz}
\end{equation}
To prove \eqref{ballistic_rdz}, note that for any $\e>0$ that 
\begin{equation}
P\left(  E_\w \left( T_{2n_k} - \bar{T}^{(d_k)}_{2n_k} \right)  \geq \e \right) \leq \frac{\E \left( T_{2n_k} - \bar{T}^{(d_k)}_{2n_k} \right)}{\e}  \leq \frac{\E \left( T_{2n_k} - \tilde{T}^{(d_k)}_{2n_k} \right)}{\e}  = \frac{2 n_k \E \left( T_1 - \tilde{T}^{(d_k)}_1 \right)}{\e} . \label{ballistic_rdz2}
\end{equation}
However, from \eqref{ballistic_ett} we have that $\E \left( T_1 - \tilde{T}^{(d_k)}_1 \right) = \frac{2}{1-E_P \rho} (E_P \rho)^{b_{d_k}}$ which decreases faster than any power of $n_k$ (since $E_P \rho < 1$ and $d_k\sim n_k$), and thus the last term in \eqref{ballistic_rdz2} is summable. Therefore, applying the Borel-Cantelli Lemma gives \eqref{ballistic_rdz} which completes the proof of \eqref{ballistic_Tlim}. 
Note, moreover, that the convergence in \eqref{ballistic_Tlim} must be uniform in $y$ since $F$ is continuous.

To prove \eqref{ballistic_Xlim}, let $X_t^*:=\max \left\{ X_n: n\leq t \right\}$ and let
\[
 x_m(y):= \left \lceil n_{k_m}+ y \, v_P \sqrt{v_{k_m,\w}} \right \rceil, \qquad y\in \R.
\]
Using this notation,
\begin{align}
P_\w \left( \frac{X_{t_m}^* - n_{k_m}}{v_P \sqrt{v_{k_m,\w}} } < y \right) &= P_\w \left( X_{t_m}^* < x_m(y) \right) \nonumber = P_\w \left( T_{ x_m(y) } > t_m \right) \nonumber \\
& = P_\w \left( \frac{ T_{ x_m(y) } - E_\w T_{ x_m(y) } }{ \sqrt{v_{k_m,\w}} } > \frac{t_m - E_\w T_{  x_m(y)  } }{\sqrt{v_{k_m,\w}}} \right) . \label{ballistic_timespace}
\end{align}
Now, recalling the definition of $t_m:= \left\lfloor E_\w X_{n_{k_m}} \right\rfloor$, by Lemma \ref{ballistic_ETaverage} we have
\[
\lim_{m\ra\infty} \frac{t_m - E_\w T_{ x_m(y) } }{\sqrt{v_{k_m,\w}}} =
\lim_{m\ra\infty} \frac{ \left\lfloor E_\w T_{n_{k_m}} \right\rfloor - E_\w T_{  n_{k_m} + y v_P \sqrt{v_{k_m,\w}}  } }{\sqrt{v_{k_m,\w}}} = -y, \quad \forall y\in \R \quad P-a.s.,
\]
where we used the fact that $v_P \E T_1 = 1$ due to \eqref{ballistic_XTLLN}. Also, by Corollary \ref{ballistic_vkasym} we have $P-a.s.$ that $\sqrt{v_{k,\w}} = o(d_k) = o(n_k)$ since $s<2$, and therefore $x_m(y) \sim n_{k_m}$.
Thus since the convergence in \eqref{ballistic_Tlim} is uniform in $y$, \eqref{ballistic_timespace} gives that 
\begin{equation}
\lim_{m\ra\infty} P_\w \left( \frac{X_{t_m}^* - n_{k_m}}{v_P \sqrt{v_{k_m,\w}} } < y \right) = 1-F(-y), \quad \forall y\in \R \quad P-a.s. \label{ballistic_Xtstarlim}
\end{equation}
Now, \eqref{ballistic_XTLLN} gives that $t_m \sim (\E T_1)n_{k_m}$, $P-a.s.$ Therefore, an easy argument involving Lemma \ref{seperation} and \eqref{ballistic_nutail} gives that $X_{t_m}^*-X_{t_m} = o( \log^2 t_m ) =o(\log^2 n_{k_m})$, $\P-a.s$. Also, Corollary \ref{ballistic_vkasym} and the Borel-Cantelli Lemma give $P-a.s.$ that $v_{k,\w} \geq d_k^{2/s-\d} \sim n_k^{2/s - \d}$ for any $\d>0$ and all $k$ large enough. Therefore, $\P-a.s.$ we have that $\lim_{m\ra\infty} \frac{ X_{t_m}^* - X_{t_m} }{\sqrt{v_{k_m,\w}}} = 0$. Combining this with \eqref{ballistic_Xtstarlim} completes the proof of \eqref{ballistic_Xlim}. 
\end{proof}
\noindent\textbf{Remark:} For the last conclusion of Proposition \ref{ballistic_generalprop} to hold it is crucial that $s>1$. The dual nature of $X_t^*$ and $T_n$ always allows the transfer of probabilities from time to space. However, if $s\leq 1$ then $\E T_1 = \infty$ and the averaging behavior of Lemma \ref{ballistic_ETaverage} does not occur. 
\end{section}

\begin{section}{Quenched CLT Along a Subsequence} \label{ballistic_qGauss}
For the remainder of the Chapter we will fix the sequence $n_k:= 2^{2^k}$ and let $d_k$ and $v_{k,\w}$ be defined accordingly as in \eqref{ballistic_dkvkdef}. Note that this choice of $n_k$ satisfies the conditions in Proposition \ref{ballistic_generalprop} for any $\d < 1$ since $n_k=n_{k-1}^2$. Our first goal in this section is to prove the following theorem, which when applied to Proposition \ref{ballistic_generalprop} proves Theorem \ref{ballistic_qCLT}.
\begin{thm}\label{ballistic_Tngaussian}
Assume $s<2$. Then for any $\eta\in(0,1)$, $P-a.s.$ there exists a subsequence $n_{k_m}=n_{k_m}(\w, \eta)$ of $n_k=2^{2^k}$ such that for $\a_m,\beta_m$ and $\gamma_m$ defined by
\begin{equation}
\a_m:= n_{k_m-1}, \quad \b_m:= n_{k_m-1} + \left\lfloor \eta d_{k_m} \right\rfloor, \quad\text{and}\quad \gamma_m:= n_{k_m} \label{ballistic_abgdef}
\end{equation}
and any sequence $x_m \in \left(\nu_{\beta_m} , \nu_{\gamma_m} \right]$  we have
\[
\lim_{m\ra\infty} P_\w^{\nu_{\a_m}} \left( \frac{\bar{T}^{(d_{k_m})}_{x_m} - E_\w \bar{T}^{(d_{k_m})}_{x_m} }{\sqrt{v_{k_m,\w}}} \leq x \right) = \Phi(x).  
\]
\end{thm}
The proof of Theorem \ref{ballistic_Tngaussian} is similar to the proof of Theorem \ref{gaussianT}. The key is to find a random subsequence where none of the variances $\s_{i,d_{k_m},\w}^2$ with $i\in (n_{k_m-1},n_{k_m}]$ is larger than a fraction of $v_{k_m,\w}$. 
To this end, let $\#(I)$ denote the cardinality of the set $I$, and for any $\eta\in(0,1)$ and any positive integer $a < n/2$
define the events
\[
\mathcal{S}_{\eta,n,a} := \bigcup_{  \substack{ I\subset [1,\eta n] \\ \#(I) = 2a }} \!\! \left( \bigcap_{i\in I} \left\{  \mu_{i,n,\w}^2 \in[n^{2/s},2n^{2/s}) \right\} \bigcap_{j\in[1,\eta n]\backslash I} \left\{ \mu_{j,n,\w}^2 < n^{2/s} \right\} \right) 
\,.
\]
%
and
\[
U_{\eta,n} := \left\{ \sum_{i\in(\eta n, n]} \s_{i,n,\w}^2 < 2 n^{2/s}   \right\} .
\]
On the event $\mathcal{S}_{\eta, n, a}$, $2a$ of the first $\eta
n$ crossings times from $\nu_{i-1}$ to $\nu_i$ have roughly the
same size variance and the rest are
all smaller. Define
\begin{equation}
a_k:=\lfloor\log\log k\rfloor \vee 1. \label{ballistic_akdef}
\end{equation} 
Then, we have the following Lemma:
\begin{lem}\label{ballistic_smbklemma}
Assume $s<2$. Then for any $\eta \in (0,1)$, we have $Q\left( \mathcal{S}_{\eta,d_k,a_k} \cap U_{\eta,d_k} \right)\geq \frac{1}{k}$ for all $k$ large enough.
\end{lem}
\begin{proof}
First we reduce the problem to getting a lower bound on $Q(\mathcal{S}_{\eta,d_k,a_k} )$. Define
\[
\tilde{U}_{\eta,n} := \left\{ \sum_{i\in(\eta n + b_n, n]} \s_{i,n,\w}^2 < n^{2/s} \right\}. 
\]
Note that $\mathcal{S}_{\eta,n,a}$ and $\tilde{U}_{\eta,n}$ are independent events since $\tilde{U}_{\eta,n}$ only depends on the environment to the right of the $\nu_{\lceil \eta n \rceil}$. Thus,
\begin{align*}
Q\left( \mathcal{S}_{\eta,n,a} \cap U_{\eta, n}\right) &\geq Q\left( \mathcal{S}_{\eta,n,a} \cap \tilde{U}_{\eta,n} \right) - Q\left( \sum_{i\in(\eta n, \eta n +b_n]} \s_{i,n,\w}^2 > n^{2/s} \right) \\
&\geq Q\left( \mathcal{S}_{\eta,n,a} \right) Q\left( \tilde{U}_{\eta,n} \right) - b_n Q\left( Var_\w \bar{T}^{(n)}_\nu > \frac{n^{2/s}}{b_n} \right).
\end{align*}
Now, Theorem \ref{ballistic_Varstable} gives that $Q\left( \tilde{U}_{\eta,n} \right) \geq 
Q\left( Var_\w T_{\nu_n} < n^{2/s} \right) = L_{\frac{s}{2},b}(1) + o(1)$, and Theorem \ref{ballistic_VETtail} gives that $b_n Q\left( Var_\w \bar{T}^{(n)}_\nu > \frac{n^{2/s}}{b_n} \right) \sim K_\infty b_n^{1+s} n^{-1}$. Thus,
\[
Q\left( \mathcal{S}_{\eta,d_k,a_k} \cap U_{\eta, d_k}\right) \geq Q(\mathcal{S}_{\eta,d_k,a_k})( L_{\frac{s}{2},b}(1) + o(1) ) - \bigo(b_{d_k}^{1+s} d_k^{-1}), \quad \text{as }k\ra\infty,
\]
and so to prove the lemma it is enough to show that $\lim_{k\ra\infty} k \, Q(\mathcal{S}_{\eta,d_k,a_k}) = \infty$.  
A lower bound for $Q( \mathcal{S}_{\eta,n,a})$ was derived in the argument preceeding Lemma \ref{smallblocklemma} in Chapter \ref{Thesis_AppendixZeroSpeed}.
A similar argument gives that for any $\e<\frac{1}{3}$ there exists a constant $C_\e>0$ such that
\begin{align}
Q\left( \mathcal{S}_{\eta,n,a} \right) & \geq  \frac{(\eta
C_\e)^{2a}}{(2a)!} \left( 1 - \frac{(2a-1)(1+4b_n)}{\eta n}
\right)^{2a} \left( Q\left( \sum_{i=1}^n \left( E_\w^{\nu_{i-1}} T_{\nu_i} \right)^2 < n^{2/s} \right) - a \, o(n^{-1+2\e}) \right)  \nonumber \\
&\qquad - \frac{(\eta n)^{2a}}{(2a)!} a \, o\left(
e^{-n^{\e/(6s)}} \right) \label{ballistic_calSlb}
\,,\end{align}
where asymptotics of the form $o(\cdot\,)$ in \eqref{ballistic_calSlb} are uniform in $\eta$ and $a$ as $n\ra\infty$. The proof of \eqref{ballistic_calSlb} is exactly the same as in Chapter \ref{Thesis_AppendixZeroSpeed} with the exception that 
$Q\left( \bigcap_{j\in[1,n]} \left\{ \mu_{j,n,\w}^2 < n^{2/s} \right\} \right)$ in \eqref{Glb} is bounded below by
$Q\left( \sum_{i=1}^n \left( E_\w^{\nu_{i-1}} T_{\nu_{i}} \right)^2 < n^{2/s} \right)$ instead of $Q\left( E_\w T_{\nu_n} < n^{1/s} \right)$. 
Then, replacing $n$ and $a$ in \eqref{ballistic_calSlb} by $d_k$ and
$a_k$ respectively, we have for $\e<\frac{1}{3}$ that 
\begin{align}
&Q\left( \mathcal{S}_{\eta,d_k,a_k} \right) \nonumber \\
&\quad \geq  \frac{(\eta
C_\e)^{2a_k}}{(2a_k)!} \left( 1 - \frac{(2a_k-1)(1+4b_{d_k})}{\eta
d_k} \right)^{2a_k} \left( Q\left( \sum_{i=1}^{d_k} \left( E_\w^{\nu_{i-1}} T_{\nu_i} \right)^2  < d_k^{2/s}
\right) - a_k
o(d_k^{-1+2\e}) \right)  \nonumber \\
&\qquad - \frac{(\eta d_k)^{2a_k}}{(2a_k)!} a_k o\left(
e^{-d_k^{\e/(6s)}} \right) \nonumber \\
& \quad=  \frac{(\eta C_\e)^{2a_k}}{(2a_k)!} \left( 1+ o(1)\right)
\left( L_{\frac{s}{2},b}(1) - o(1) \right) - o\left( \frac{1}{k} \right). \label{ballistic_sublb}
\end{align}
The last equality is a result of Theorem \ref{ballistic_Varstable} and the definitions of $a_k$
and $d_k$ in \eqref{ballistic_akdef} and \eqref{ballistic_dkvkdef}. 
Also, since $a_k\sim \log\log k$ we have that $\lim_{k\ra\infty} k \frac{C^{2a_k}}{(2a_k)!} = \infty $ for any constant $C>0$. Therefore, \eqref{ballistic_sublb} implies that $\lim_{k\ra\infty} k\, Q\left( \mathcal{S}_{\eta,d_k,a_k} \right) = \infty$. 
\end{proof}
\begin{cor}\label{ballistic_smallblocks2}
Assume $s<2$. Then for any $\eta\in(0,1)$,
$P$-a.s. there exists a random subsequence $n_{k_m}=n_{k_m}(\w,\eta)$
of $n_k=2^{2^k}$ such that for the sequences $\a_m,\b_m,$ and $\gamma_m$ defined as in \eqref{ballistic_abgdef} we have that for all $m$
\begin{equation}
\max_{ i\in(\a_m, \b_m] } \mu_{i,d_{k_m},\w}^2  \leq
2d_{k_m}^{2/s} \leq \frac{1}{a_{k_m}} \sum_{i=\a_m+1}^{\b_m}
\mu_{i,d_{k_m},\w}^2, \quad\text{and}\quad  \sum_{i=\b_m + 1}^{\gamma_m} \s_{i,d_{k_m},\w}^2 < 2 d_{k_m}^{2/s} .\label{ballistic_smbk}
\end{equation}
\end{cor}
\begin{proof}
Define the sequence of events
\[
\mathcal{S}_k' := \bigcup_{ \substack{ I\subset (n_{k-1}, n_{k-1}+\eta d_k] \\ \#(I) = 2a_k} } \!\! \left( \bigcap_{i\in I} \left\{  \mu_{i,d_k,\w}^2 \in[d_k^{2/s},2d_k^{2/s}) \right\} \bigcap_{j\in(n_{k-1},n_{k-1}+\eta d_k]\backslash I} \left\{ \mu_{j,d_k,\w}^2 < d_k^{2/s} \right\} \right) 
\,,
\]
and
\[
U_k' := \left\{ \sum_{i \in (n_{k-1}+ \eta d_k , n_k]} \s_{i,d_{k_m},\w}^2 < 2 d_{k_m}^{2/s} \right\}
\]
Note that due to the reflections of the random walk, the event
$\mathcal{S}_k' \cap U_k' $ depends on the environment between
ladder locations $n_{k-1}-b_{d_k}$ and $n_k$. Thus, since $n_{k-1} - b_{d_k} > n_{k-2}$ for all $k\geq 4$, we have that 
$\{\mathcal{S}_{2k}' \cap U_{2k}'\}_{k=2}^{\infty}$ is an independent sequence of events.
Similarly, for $k$ large enough $\mathcal{S}_k' \cap U_k'$ does not
depend on the environment to left of the origin. Thus
\[
P(\mathcal{S}_k' \cap U_k') =Q(\mathcal{S}_k' \cap U_k' )=Q\left( \mathcal{S}_{\eta,d_k,a_k} \cap U_{\eta,d_k} \right)
\]
for all $k$ large enough. Lemma \ref{ballistic_smbklemma} then gives
that $\sum_{k=1}^\infty P(\mathcal{S}_{2k}' \cap U_{2k}') =
\infty$, and the Borel-Cantelli Lemma then implies that infinitely
many of the events $\mathcal{S}_{2k}' \cap U_{2k}' $ occur $P-a.s$.
Therefore, $P-a.s.$ there exists a subsequence $k_m=k_m(\w,\eta)$ such that $\mathcal{S}_{k_m}' \cap U_{k_m}'$ occurs for each $m$. Finally, note that the event $\mathcal{S}_{k_m}' \cap U_{k_m}'$ implies \eqref{ballistic_smbk}. 
\end{proof}
\begin{proof}[\textbf{Proof of Theorem \ref{ballistic_Tngaussian}:}] \ \\
First, recall that Corollary \ref{Vsdiff} gives that there exists an $\eta'>0$ such that 
\begin{equation}
Q\left( \left| \sum_{i=1}^n \left( \s_{i,m,\w}^2-\mu_{i,m,\w}^2
\right) \right| \geq \d n^{2/s} \right) = o(n^{-\eta'}) \qquad \forall \d>0, \quad \forall m\in \N. \label{ballistic_c56}
\end{equation}
This can be applied along with the Borel-Cantelli Lemma to prove that
\begin{equation}
\sum_{i=n_{k-1}+1}^{n_{k-1} + \lfloor \eta d_k \rfloor } \!\!\! \left( \s_{i,d_k,\w}^2 - \mu_{i,d_k,\w}^2 \right) =  o\left(d_k^{2/s}\right), \quad P-a.s. \label{ballistic_vmucompare}
\end{equation}
Thus, $P-a.s.$ we may assume that \eqref{ballistic_vmucompare} holds and that there exists a subsequence $n_{k_m}=n_{k_m}(\w,\eta)$ such that condition \eqref{ballistic_smbk} in Corollary \ref{ballistic_smallblocks2} holds. Then, it is enough to prove that
\begin{equation}
\lim_{m\ra\infty} P_\w^{\nu_{\a_m}} \left( \frac{\bar{T}^{(d_{k_m})}_{\nu_{\b_m}} - E_\w^{\nu_{\a_m}} \bar{T}^{(d_{k_m})}_{\nu_{\b_m}}}{\sqrt{v_{k_m,\w}}} \leq y \right) = \Phi(y), \label{ballistic_reducegaussian}
\end{equation}
and
\begin{equation}
\lim_{m\ra\infty} P_\w^{\nu_{\b_m}} \left( \left| \frac{\bar{T}^{(d_{k_m})}_{x_m} - E_\w^{\nu_{\b_m}} \bar{T}^{(d_{k_m})}_{x_m}}{\sqrt{v_{k_m,\w}}} \right| \geq \e \right) = 0, \quad \forall \e>0. \label{ballistic_endsmall}
\end{equation}
To prove \eqref{ballistic_endsmall}, note that by Chebychev's inequality
\[
P_\w^{\nu_{\b_m}} \left( \left| \frac{\bar{T}^{(d_{k_m})}_{x_m} - E_\w^{\nu_{\b_m}} \bar{T}^{(d_{k_m})}_{x_m}}{\sqrt{v_{k_m,\w}}} \right| \geq \e \right) \leq \frac{Var_\w \left( \bar{T}^{(d_{k_m})}_{x_m} - \bar{T}^{(d_{k_m})}_{\b_m} \right) }{\e^2 v_{k_m,\w}} \leq \frac{\sum_{i=\b_m+1}^{\gamma_m} \s_{i,d_{k_m},\w}^2 }{\e^2 v_{k_m,\w}} 
\]
However, by \eqref{ballistic_vmucompare} and our choice of the subsequence $n_{k_m}$ we have that $\sum_{i=\b_m+1}^{\gamma_m} \s_{i,d_{k_m},\w}^2 < 2 d_{k_m}^{2/s}$, and $v_{k_m,\w} \geq \sum_{i=\a_m+1}^{\b_m} \s_{i,d_{k_m},\w}^2 = \sum_{i=\a_m+1}^{\b_m} \mu_{i,d_{k_m},\w}^2 + o\left( d_{k_m}^{2/s} \right)\geq a_{k_m} d_{k_m}^{2/s} + o\left( d_{k_m}^{2/s} \right)$. Thus
\begin{equation}
\lim_{m\ra\infty} \frac{\sum_{i=\b_m+1}^{\gamma_m} \s_{i,d_{k_m},\w}^2 }{v_{k_m,\w}} = 0, \label{ballistic_s2v}
\end{equation}
which proves \eqref{ballistic_endsmall}.
To prove \eqref{ballistic_reducegaussian}, it is enough to show that the Lindberg-Feller condition is satisfied. That is we need to show
\begin{equation}
\lim_{m\ra\infty} \frac{1}{v_{k_m,\w}} \sum_{i=\a_m+1}^{\b_m} \s_{i,d_{k_m},\w}^2 = 1, \label{ballistic_LF1}
\end{equation}
and 
\begin{equation}
\lim_{m\ra\infty} \frac{1}{v_{k_m,\w}} \sum_{i=\a_m+1}^{\b_m} E_\w^{\nu_{i-1}} \left[ \left( \bar{T}^{(d_{k_m})}_{\nu_i} - \mu_{i,d_{k_m},\w} \right)^2 \mathbf{1}_{ | \bar{T}^{(d_{k_m})}_{\nu_i}-\mu_{i,d_{k_m},\w}  | > \e \sqrt{v_{m,\w}}}  \right] = 0, \quad \forall \e>0. \label{ballistic_LF2}
\end{equation}
To show \eqref{ballistic_LF1} note that the definition of $v_{k_m,\w}$ and our choice of the subsequence $n_{k_m}$ give that
\[
\frac{1}{v_{k_m,\w}} \sum_{i=\a_m+1}^{\b_m} \s_{i,d_{k_m},\w}^2 = 1 - \frac{1}{v_{k_m,\w}} \sum_{i=\b_m+1}^{\gamma_m} \s_{i,d_{k_m,\w}}^2 = 1-o(1),
\]
where the last equality is from \eqref{ballistic_s2v}. To prove \eqref{ballistic_LF2}, first note that an application of Lemma \ref{Vsmall} gives that for any $\e'>0$
\[
\sum_{i=n_{k-1}+1}^{n_{k-1}+\lfloor \eta d_k \rfloor} \s_{i,d_k,\w}^2 \mathbf{1}_{M_i \leq d_k^{(1-\e')/s}} = o\left( d_k^{2/s} \right) ,\quad P-a.s. , 
\]
where $M_i$ is defined as in \eqref{ballistic_Mdef}. 
Then, since $v_{k_m,\w} \geq a_{k_m} d_{k_m}^{2/s} + o\left( d_{k_m}^{2/s} \right)$ we can reduce the sum in \eqref{ballistic_LF2} to blocks where $M_i > d_{k_m}^{(1-\e')/s}$. That is, it is enough to prove that for some $\e'>0$ and every $\e>0$
\begin{align}
\lim_{m\ra\infty} \frac{1}{v_{k_m,\w}} \sum_{i=\a_{m}+1}^{\b_m} E_\w^{\nu_{i-1}} \left[ \left( \bar{T}^{(d_{k_m})}_{\nu_i}-\mu_{i,d_{k_m},\w} \right)^2 \mathbf{1}_{ | \bar{T}^{(d_{k_m})}_{\nu_i}-\mu_{i,d_{k_m},\w}  | > \e \sqrt{v_{k_m,\w}}}  \right]\mathbf{1}_{M_i > d_{k_m}^{(1-\e')/s}} = 0 . \label{ballistic_Mlarge}
\end{align}
To get an upper bound for \eqref{ballistic_Mlarge}, first note that our
choice of the subsequence $n_{k_m}$ gives that for $m$ large enough $v_{k_m,\w} \geq \frac{1}{2} \sum_{i=\a_m+1}^{\b_m} \mu_{i,d_{k_m},\w}^2 \geq \frac{a_{k_m}}{2} \mu_{i,d_{k_m},\w}$ for any $i\in
(\a_m, \b_m]$. Thus, for $m$ large enough we can replace the
indicators inside the expectations in \eqref{ballistic_Mlarge} by the
indicators of the events $\left\{ \bar{T}_{\nu_i}^{(d_{k_m})}  >
(1+\e \sqrt{a_{k_m}/2}) \mu_{i,d_{k_m},\w} \right\}$. Thus, for
$m$ large enough and $i\in(\a_m, \b_m]$, we have
\begin{align*}
& E_\w^{\nu_{i-1}} \left[ \left(
\bar{T}^{(d_{k_m})}_{\nu_i}-\mu_{i,d_{k_m},\w} \right)^2
\mathbf{1}_{ |
\bar{T}^{(d_{k_m})}_{\nu_i}-\mu_{i,d_{k_m},\w} | >
\e \sqrt{v_{k_m,\w}}}  \right] \\
&\qquad \leq E_\w^{\nu_{i-1}} \left[ \left(
\bar{T}^{(d_{k_m})}_{\nu_i}-\mu_{i,d_{k_m},\w} \right)^2
\mathbf{1}_{ \bar{T}^{(d_{k_m})}_{\nu_i}> (1+\e
\sqrt{a_{k_m}/2})
\mu_{i,d_{k_m},\w} }  \right] \\
&\qquad = \int_{1+\e\sqrt{a_{k_m}/2}}^{\infty} P_\w^{\nu_{i-1}}
\left( \bar{T}_{\nu_i}^{(d_{k_m})} > x \mu_{i,d_{k_m},\w}
\right) 2(x-1)\mu_{i,d_{k_m},\w}^2 \, dx
\,.\end{align*}
We want to get an upper bound on the probabilities inside the integral. If $\e'<\frac{1}{3}$ we can use Lemma \ref{momentbound} to get that for $k$ large enough, $E_\w^{\nu_{i-1}}\left( \bar{T}^{(d_{k})}_{\nu_i} \right)^j \leq 2^j j! \mu_{i,d_k,\w}^j$ for all $n_{k-1}<i\leq n_k$ such that $M_i > d_k^{(1-\e')/s}$.
Multiplying by $(4\mu_{i,d_k,\w})^{-j}$ and summing over $j$ gives
that $E_\w^{\nu_{i-1}} e^{ \bar{T}^{(d_k)}_{\nu_i} /(4
\mu_{i,d_k,\w}) } \leq 2$. Therefore, Chebychev's inequality gives
\[
P_\w^{\nu_{i-1}} \left( \bar{T}_{\nu_i}^{(d_k)} > x \mu_{i,d_k,\w}
\right) \leq e^{- x/4} E_\w^{\nu_{i-1}} e^{
\bar{T}^{(d_k)}_{\nu_i} / (4 \mu_{i,d_k,\w})} \leq 2 e^{-x/4}
\,.\]
Thus, for all $m$ large enough we have for all $\a_m<i\leq \b_m
\leq n_{k_m}$ with $M_i > d_{k_m}^{(1-\e')/s}$ that
\begin{align*}
\int_{1+\e\sqrt{a_{k_m}/2}}^{\infty} P_\w^{\nu_{i-1}} \left(
\bar{T}_{\nu_i}^{(d_{k_m})} > x \mu_{i,d_{k_m},\w} \right)
2(x-1)\mu_{i,d_{k_m},\w}^2 dx 
&\leq \mu_{i,d_{k_m},\w}^2 \int_{1+\e\sqrt{a_{k_m}/2}}^{\infty} 4(x-1)e^{-x/4}
 dx\\
&= \mu_{i,d_{k_m},\w}^2 \: o\!\left( e^{-a_{k_m}^{1/4}} \right) \,.
\end{align*}
Therefore we have that as $m\ra\infty$,
\eqref{ballistic_Mlarge} is bounded above by
\begin{align}
\lim_{m\ra\infty} o\left( e^{-a_{k_m}^{1/4}} \right) \frac{1}{v_{k_m,\w}} \left( \sum_{i=\a_m+1}^{\b_m} \mu_{i,d_{k_m},\w}^2 \mathbf{1}_{M_i > d_{k_m}^{(1-\e')/s}} \right)
\,. \label{ballistic_finalest}
\end{align}
However, since 
\begin{align*}
\frac{1}{v_{k_m,\w}} \sum_{i=\a_m+1}^{\b_m} \mu_{i,d_{k_m},\w}^2 
&\leq \frac{1}{ \sum_{i=\a_m+1}^{\b_m} \s_{i,d_{k_m},\w}^2 } \left( \sum_{i=\a_m+1}^{\b_m} \s_{i,d_{k_m},\w}^2
 + o\left( d_{k_m}^{2/s} \right) \right) \\
&\leq 1 + \frac{o\left( d_{k_m}^{2/s} \right)}{2 a_{k_m} d_{k_m}^{2/s} + o\left( d_{k_m}^{2/s} \right)},
\end{align*}
we have that \eqref{ballistic_finalest} tends to zero as $m\ra\infty$. 
This finishes the proof of \eqref{ballistic_LF2} and thus of Theorem \ref{ballistic_Tngaussian}.
\end{proof}
\begin{proof}[\textbf{Proof of Theorem \ref{ballistic_qCLT}:}] \ \\
Choose $\eta \in (0,1)$ such that $\eta < \frac{1}{\bar\nu}$ where $\bar\nu = E_P \nu$, and then choose $n_{k_m}$ as in Theorem \ref{ballistic_Tngaussian}. Then for $\b_m$ and $\gamma_m$ defined as in \eqref{ballistic_abgdef}, we have that \eqref{ballistic_nuLLN} and the fact that $d_k\sim n_k$ give 
\[
\lim_{m\ra\infty} \frac{\nu_{\b_m}}{n_{k_m}} = \eta \bar{\nu} < 1 < \bar\nu = \lim_{m\ra\infty} \frac{\nu_{\gamma_m}}{n_{k_m}}. 
\]
Thus $x_m \sim n_{k_m} \Ra x_m \in [\nu_{\b_m}, \nu_{\gamma_m} ]$ for all $m$ large enough. Therefore, the conditions of Proposition \ref{ballistic_generalprop} are satisfied with $F(x)= \Phi(x)$. 
\end{proof}
\end{section}

\begin{section}{Quenched Exponential Limits} \label{ballistic_exponential}
\begin{subsection}{Analysis of $T_\nu$ when $M_1$ is Large} \label{ballistic_Laplace}
The goal of this subsection is to analyze the quenched distribution of $\bar{T}^{(n)}_\nu$ on ``large'' blocks (i.e. when $M_1>n^{(1-\e)/s}$). We want to show that conditioned on $M_1$ being large, $\bar{T}^{(n)}_\nu / E_\w \bar{T}^{(n)}_\nu$ is approximately exponentially distributed. We do this by showing that the quenched Laplace transform $E_\w \exp \left\{-\l \frac{\bar{T}^{(n)}_\nu}{E_\w \bar{T}^{(n)}_\nu}\right\}$ is approximately $\frac{1}{1+\l}$ on such blocks. 

As was done in \cite{eszStable}, we analyze the quenched Laplace transform of $\bar{T}^{(n)}_\nu$ by decomposing $\bar{T}^{(n)}_\nu$
into a series of excursions away from 0. An excursion is a ``failure'' if the random walk returns to zero before hitting $\nu$ (i.e. if $T_\nu > T_0^+:= \min\{ k > 0: X_k = 0 \}$), and a ``success'' if the random walk reaches $\nu$ before returning to zero (note that classifying an excursion as a failure/sucess is independent of any modifications to the environment left of zero since if the random walk ventures to the left at all, it must be in a failure excursion). Define $p_\w:=P_\w ( T_\nu<T^+_0)$, and let $N$ be a geometric random variable with parameter $p_\w$ (i.e. $P(N=k) = p_\w (1-p_\w)^k$ for $k\in \N$). Also, let $\{F_i\}_{i=1}^\infty$ be an i.i.d. sequence (also independent of $N$) with $F_1$ having the same distribution as $\bar{T}_\nu^{(n)}$ conditioned on $\left\{ \bar{T}^{(n)}_\nu > T_0^+ \right\}$, and let $S$ be a random variable with the same distribution as $T_\nu$ conditioned on $\left\{ T_\nu < T_0^+ \right\}$ and independent of everything else (note that for sucess excursions we can ignore added reflections to the left of zero). Thus, we have that 
\begin{equation}
\bar{T}^{(n)}_\nu \stackrel{Law}{=} S + \sum_{i=1}^N F_i \qquad\text{(quenched).} \label{ballistic_Tdec}
\end{equation}
In a slight abuse of notation we will still use $P_\w$ for the probabilities of $F_i, S,$ and $N$ to emphasize that their distributions are dependent on $\w$. 
The following results are easy to verify:
\begin{equation}
E_\w N = \frac{1-p_\w}{p_\w}  \quad\text{and}\quad E_\w \bar{T}^{(n)}_\nu = E_\w S + (E_\w N)( E_\w F_1), \label{ballistic_ETdec}
\end{equation}
\begin{align}
Var_\w \bar{T}^{(n)}_\nu &= (E_\w N)( Var_\w F_1) + (E_\w F)^2 (Var_\w N) + Var_\w S \nonumber \\
&= (E_\w N )(E_\w F^2) + (E_\w F)^2 (Var_\w N - E_\w N) + Var_\w S \nonumber \\
&= (E_\w N )(E_\w F^2) + (E_\w F)^2(E_\w N)^2 + Var_\w S, \label{ballistic_VTdec}
\end{align}
and 
\begin{align*}
E_\w e^{-\l \bar{T}^{(n)}_\nu} = E_\w e^{-\l S} E_\w\left[ \left(E_\w e^{-\l F_1}\right)^N \right] = E_\w e^{-\l S} \frac{p_\w}{1-(1-p_\w)\left(E_\w e^{-\l F_1}\right)}, \quad \forall \l\geq 0.
\end{align*}
Also, since $e^{-x} \geq 1-x$ for any $x\in\R$ we have for any $\l\geq 0$ that
\begin{align*}
E_\w e^{-\l \bar{T}^{(n)}_\nu} &\geq \left(1-\l E_\w S \right) \frac{p_\w}{1-(1-p_\w)\left(1- \l E_\w F_1\right)} 
=  \frac{1-\l E_\w S }{1 + \l(E_\w N)(E_\w F_1)}
\geq  \frac{1-\l E_\w S}{1 + \l E_\w \bar{T}^{(n)}_\nu},
\end{align*}
where the first equality and the last inequality are from the formulas for $E_\w N$ and $E_\w \bar{T}^{(n)}_\nu$ given in \eqref{ballistic_ETdec}. Similarly, since $e^{-x} \leq 1-x+\frac{x^2}{2}$ for all $x\geq 0$ we have that for any $\l \geq 0$ that
\begin{align*}
E_\w e^{-\l \bar{T}^{(n)}_\nu} &\leq \frac{p_\w}{1-(1-p_\w)\left(1- \l E_\w F_1 + \frac{\l^2}{2} E_\w F_1^2\right)} \\
&= \frac{1}{1+\l(E_\w N)(E_\w F_1) - \frac{\l^2}{2} (E_\w N)( E_\w F_1^2)}\\
&= \frac{1}{1+\l(E_\w N)(E_\w F_1) - \frac{\l^2}{2} (Var_\w \bar{T}^{(n)}_\nu - (E_\w N)^2(E_\w F_1)^2 - Var_\w S)}\\
&\leq \frac{1}{1+\l(E_\w \bar{T}^{(n)}_\nu - E_\w S) - \frac{\l^2}{2} (Var_\w \bar{T}^{(n)}_\nu - (E_\w \bar{T}^{(n)}_\nu - E_\w S)^2 )},
\end{align*}
where the first equality and last inequality are from \eqref{ballistic_ETdec} and the second equality is from \eqref{ballistic_VTdec}. Therefore, replacing $\l$ by $\l/(E_\w \bar{T}^{(n)}_\nu)$ we get 
\begin{equation}
 E_\w e^{-\l \frac{\bar{T}^{(n)}_\nu}{E_\w \bar{T}^{(n)}_\nu}} \geq \left(1-\l \frac{E_\w S}{E_\w \bar{T}^{(n)}_\nu} \right) \frac{1}{1 + \l } \, , \label{ballistic_mgflb}
\end{equation}
and
\begin{align}
E_\w e^{-\l \frac{\bar{T}^{(n)}_\nu}{E_\w \bar{T}^{(n)}_\nu}}
&\leq \frac{1}{1+\l - \l \frac{E_\w S}{E_\w \bar{T}^{(n)}_\nu } - \frac{\l^2}{2} \left(\frac{Var_\w \bar{T}^{(n)}_\nu}{(E_\w \bar{T}^{(n)}_\nu)^2 } -\frac{(E_\w \bar{T}^{(n)}_\nu - E_\w S)^2}{(E_\w \bar{T}^{(n)}_\nu)^2 } \right)} \nonumber \\
&\leq \frac{1}{1+\l - (\l+\l^2) \frac{E_\w S}{E_\w \bar{T}^{(n)}_\nu } - \frac{\l^2}{2} \left(\frac{Var_\w \bar{T}^{(n)}_\nu}{(E_\w \bar{T}^{(n)}_\nu)^2 } - 1 \right)} \, . \label{ballistic_mgfub}
\end{align}
Therefore, we have reduced the problem of showing $E_\w e^{-\l \frac{\bar{T}^{(n)}_\nu}{E_\w \bar{T}^{(n)}_\nu}} \approx \frac{1}{1+\l}$ when $M_1$ is large to showing that $\frac{E_\w S}{E_\w \bar{T}^{(n)}_\nu } \approx 0$ and $\frac{Var_\w \bar{T}^{(n)}_\nu}{(E_\w \bar{T}^{(n)}_\nu)^2 } \approx 1$ when $M_1$ is large. 
In order to analyze $E_\w S$, we define a modified environment which is essentially the environment the random walker ``sees'' once it is told that it reaches $\nu$ before returning to zero. A simple computation similar to the one in \cite[Remark 2 on pages 222-223]{zRWRE} gives that the random walk conditioned to reach $\nu$ before returning to zero is a homogeneous markov chain with transition probabilities given by $\bw_i := P_\w^i(X_1=i+1 | T_\nu < T^+_0 )$. 
Then the definition of $\bw_i$ gives that $\bw_0=\bw_1 =1$, and for $i \in [2,\nu)$ we have $\bw_i = \frac{\w_i P_\w^{i+1}( T_\nu < T_0 )}{P_\w^i ( T_\nu < T_0)}$. Using the hitting time formulas in \cite[(2.1.4)]{zRWRE} we have 
\begin{equation}
\bw_i = \frac{\w_i R_{0,i}}{R_{0,i-1}} \quad \forall i\in[2,\nu), \quad\text{where}\quad R_{0,i}:= \sum_{j=0}^i \Pi_{0,j}. \label{ballistic_Rdef}
\end{equation}
Let $\bar\rho_i:=\frac{1-\bw_i}{\bw_i}$ and define $\bar{\Pi}_{i,j},$ and $\bar{W}_{i,j}$ analogously to $\Pi_{i,j}$ and $W_{i,j}$ 
using $\bar\rho_i$ in place of $\rho_i$. Then the above formulas for $\bw_i$ give that $\bar\rho_0=\bar\rho_1=0$ and $\bar\rho_i = \rho_i \frac{R_{0,i-2}}{R_{0,i}} \quad\forall i\in[2,\nu)$. Thus,
\begin{equation}
\bar{\Pi}_{i,j} = \Pi_{i,j} \frac{R_{0,i-2}R_{0,i-1}}{R_{0,j-1}R_{0,j}},\quad\forall 2\leq i\leq j<\nu. \label{ballistic_rhobarformulas}
\end{equation}
Note that since $R_{0,i}\leq R_{0,j}$ for any $0\leq i\leq j$ we have from \eqref{ballistic_rhobarformulas} that 
\begin{equation}
\bar\Pi_{i,j}\leq \Pi_{i,j} \quad \text{for any } 0\leq i\leq j<\nu \label{ballistic_PibarlessPi}
\end{equation}
Now, since $E_\w S = E_{\bw} T_\nu$ we get from \eqref{ballistic_qvformula} that $E_\w S = \nu + 2\sum_{j=2}^{\nu-1} \bar{W}_{2,j} = \nu + 2\sum_{j=2}^{\nu-1} \sum_{i=2}^j \bar{\Pi}_{i,j}$.
Therefore, letting $\bar{M}_1:= \max \{ \bar{\Pi}_{i,j}: 0 \leq i\leq j<\nu \}$ we get the bound 
\begin{equation}
E_\w S \leq \nu + 2\nu^2 \bar{M}_1. \label{ballistic_ESbound}
\end{equation}
Thus, we need to get bounds on the tail of $\bar{M}_1$. To this end, recall the definition of $M_1$ in \eqref{ballistic_Mdef} and define $\tau := \max \{ k \in [1,\nu]: \Pi_{0,k-1} = M_1 \}$. Then, define
\begin{equation}
M^-:= \min \{ \Pi_{i,j}: 0<i\leq j < \tau \} \wedge 1, \quad\text{and}\quad M^+:= \max \{ \Pi_{i,j} : \tau < i \leq j < \nu  \}\vee 1 \, .\label{ballistic_Ddef}
\end{equation} 
\begin{lem} \label{ballistic_Mpmtail}
For any $\e,\d>0$ we have
\begin{equation}
P( M^- < n^{-\d} , M_1 > n^{(1-\e)/s} ) = o(n^{-1+\e-\d s + \e'}), \quad \forall \e' > 0, \label{ballistic_xdown}
\end{equation}
and
\begin{equation}
P( M^+ > n^{\d} , M_1 > n^{(1-\e)/s} ) = o(n^{-1+\e-\d s + \e'}), \quad \forall \e' > 0, \label{ballistic_xup}
\end{equation}
\end{lem}
\begin{proof}
Since $\Pi_{0,\tau-1} = M_1$ by definition we have
\begin{align}
P( M^- < n^{-\d} , M_1 > n^{(1-\e)/s} ) &\leq P\left( \exists 0 < i\leq j < \tau - 1: \Pi_{i,j} < n^{-\d}, \quad \Pi_{0,\tau-1} > n^{(1-\e)/s}   \right) \nonumber \\
&\leq P(\tau > b_n) + \sum_{ 0<i\leq j < k < b_n }P\left( \Pi_{i,j} < n^{-\d}, \quad \Pi_{0,k} > n^{(1-\e)/s}   \right) \nonumber \\
&\leq P(\nu > b_n) + \sum_{ 0<i\leq j < k < b_n }P\left( \Pi_{0,i-1}\Pi_{j+1,k} > n^{(1-\e)/s + \d}   \right) \label{ballistic_xfluc}.
\end{align}
Since \eqref{ballistic_nutail} gives that $P(\nu > b_n) \leq C_1 e^{-C_2 b_n}$ we need only handle the second term in \eqref{ballistic_xfluc} to prove \eqref{ballistic_xdown}.
However, Chebychev's inequality and the fact that $P$ is a product measure give that
\[
P\left( \Pi_{0,i-1}\Pi_{j+1,k} > n^{(1-\e)/s + \d}   \right) \leq n^{-1+\e-\d s} (E_P \rho^s)^{i+k-j} = n^{-1+\e-\d s}.
\]
Since the number of terms in the sum in \eqref{ballistic_xfluc} is at most $(b_n)^3 = o(n^{\e'})$ we have proved \eqref{ballistic_xdown}. The proof of \eqref{ballistic_xup} is similar:
\begin{align*}
P( M^+ > n^{\d} , M_1 > n^{(1-\e)/s} ) &\leq P\left( \exists \tau < i\leq j < \nu: \Pi_{i,j} > n^{\d}, \quad \Pi_{0,\tau-1} > n^{(1-\e)/s}   \right) \nonumber \\
&\leq P(\nu > b_n) + \sum_{ 0\leq k <i\leq j < b_n }P\left( \Pi_{0,k}\Pi_{i,j} > n^{(1-\e)/s + \d}   \right)\\
&\leq C_1 e^{-C_2 b_n} + (b_n)^3 n^{-1+\e-\d s} = o(n^{-1+\e-\d s + \e'})
\end{align*}
\end{proof}
\begin{cor} \label{ballistic_Sbig}
For any $\e,\d>0$ we have
\[
P\left( E_\w S \geq n^{5\d} , M_1 > n^{(1-\e)/s} \right) = o(n^{-1+\e-\d s + \e'}), \quad \forall \e'>0. 
\]
\end{cor}
\begin{proof}
Recall that \eqref{ballistic_ESbound} gives $E_\w S \leq \nu + 2 \nu^2 \bar{M}_1$.
We will use $M^-$ and $M^+$ to get bounds on $\bar{M}_1$. First, note that for any $i\in[0,\tau)$ we have
\[
R_{0,i} = \sum_{k=0}^i \Pi_{0,k} = \Pi_{0,i} + \sum_{k=0}^{i-1} \frac{\Pi_{0,i}}{\Pi_{k+1,i}} \leq \Pi_{0,i}\left(\frac{i+1}{M^-}\right).
\]
Note also that $R_{0,j} \geq \Pi_{0,j}$ holds for any $j\geq 0$. Thus, for any $2\leq i \leq  j \leq \tau$  we have
\begin{align*}
\bar{\Pi}_{i,j} &= \Pi_{i,j} \frac{R_{0,i-2}R_{0,i-1}}{R_{0,j-1}R_{0,j}} 
\leq \Pi_{i,j} \left(\frac{i}{M^-}\right)^2 	\frac{\Pi_{0,i-2}\Pi_{0,i-1}}{\Pi_{0,j-1}\Pi_{0,j}} 
=\left(\frac{i}{M^-}\right)^2 	\frac{1}{\Pi_{i-1,j-1}} \leq \frac{i^2}{(M^-)^3}.
\end{align*}
Also, from \eqref{ballistic_PibarlessPi} we have that $\bar{\Pi}_{i,j} \leq \Pi_{i,j} \leq M^+$ for $\tau < i\leq j<\nu$. 
Therefore we have that $\bar{M}_1 \leq \frac{\nu^2 M^+}{(M^-)^3}$ (note that here we used that $M^-\leq 1$ and $M^+ \geq 1$). Thus,
\[
P\left( E_\w S \geq n^{5\d} , M_1 > n^{(1-\e)/s} \right) \leq P\left( \nu + \frac{2 \nu^4 M^+}{(M^-)^3} > n^{5\d} , M_1 > n^{(1-\e)/s} \right).
\]
An easy argument using \eqref{ballistic_nutail} and Lemma \ref{ballistic_Mpmtail} finishes the proof. 
\end{proof}
\begin{lem} \label{ballistic_VarET2compare}
For any $\e,\d>0$ we have
\[
Q\left( \left| \frac{Var_\w \bar{T}^{(n)}_\nu}{(E_\w \bar{T}^{(n)}_\nu)^2} - 1 \right| \geq  n^{-\d} ,\quad  M_1 > n^{(1-\e)/s} \right) = o(n^{-2+2\e+\d s+\e'}), \quad \forall \e'>0
\]
\end{lem}
\begin{proof}
Recall that from \eqref{comparer} we have that there exist explicit non-negative random variables $D^+(\w)$ and $D^-(\w)$ such that
\[
\left( E_\w \bar{T}^{(n)}_\nu \right)^2 -D^+(\w) \leq Var_\w \bar{T}^{(n)}_\nu  \leq
\left( E_\w \bar{T}^{(n)}_\nu \right)^2 + 8R_{0,\nu-1}D^-(\w), 
\]
where $R_{0,\nu-1}$ is defined as in \eqref{ballistic_Rdef}.
Therefore, since $E_\w \bar{T}^{(n)}_\nu \geq M_1$, we have
\begin{align}
&Q\left( \left| \frac{Var_\w \bar{T}^{(n)}_\nu}{(E_\w \bar{T}^{(n)}_\nu)^2} - 1 \right| \geq  n^{-\d} ,  M_1 > n^{(1-\e)/s} \right) \nonumber \\
&\qquad \leq Q\left( 8 R_{0,\nu-1} D^-(\w) > n^{(2-2\e)/s - \d} \right) + Q\left( D^+(\w) >n^{(2-2\e)/s-\d} \right). \label{ballistic_Dpm}
\end{align}
However, Lemma \ref{abound} and Corollary \ref{bbound} give respectively that $Q(D^+(\w) > x) = o(x^{-s+\e''})$ and $Q\left( R_{0,\nu-1} D^-(\w) > x \right) = o(x^{-s+\e''})$ for any $\e''>0$. Therefore, both terms in \eqref{ballistic_Dpm} are of order
$o\left( n^{-2+2\e +\d s + \e''((2-2\e)/s - \d)} \right)$. The lemma then follows since $\e''>0$ is arbitrary. 
\end{proof}

For any $i$, define the scaled quenched Laplace transforms 
$\phi_{i,n}(\l) := E_\w^{\nu_{i-1}} \exp\left\{-\l \frac{\bar{T}^{(n)}_{\nu_i}}{\mu_{i,n,\w}} \right\}$.
\begin{lem} \label{ballistic_mgflem}
Let $\e<\frac{1}{8}$, and define $\e':= \frac{1-8\e}{5} > 0$. Then
\[
Q\left( \exists \l \geq 0: \phi_{1,n}(\l) \notin \left[ \frac{1-\l n^{-\e /s}}{1+\l} , \frac{1}{1+\l-\left(\l+\frac{3 \l^2}{2}\right)n^{-\e/s} } \right],\: M_1 > n^{(1-\e)/s} \right) = o\left( n^{-1-\e'} \right). 
\]
\end{lem}
\begin{proof}
Recall from \eqref{ballistic_mgflb} and \eqref{ballistic_mgflb} that 
\[
\left(1-\l \frac{E_\w S}{E_\w \bar{T}^{(n)}_\nu} \right) \frac{1}{1 + \l }  \leq \phi_{1,n}(\l) \leq \frac{1}{1+\l - (\l+\l^2) \frac{E_\w S}{E_\w \bar{T}^{(n)}_\nu } - \frac{\l^2}{2} \left(\frac{Var_\w \bar{T}^{(n)}_\nu}{(E_\w \bar{T}^{(n)}_\nu)^2 } - 1 \right)}
\]
for all $\l\geq 0$. Therefore
\begin{align*}
&Q\left( \exists \l \geq 0: \phi_{1,n}(\l) \notin \left[ \frac{1-\l n^{-\e /s}}{1+\l} , \frac{1}{1+\l-\left(\l+3 \l^2 /2\right)n^{-\e/s} } \right],\: M_1 > n^{(1-\e)/s} \right) \\
&\qquad \leq Q\left( \frac{E_\w S}{E_\w \bar{T}^{(n)}_\nu} \geq n^{-\e/s}, \quad M_1 \geq n^{(1-\e)/s}  \right) + Q\left( \frac{Var_\w \bar{T}^{(n)}_\nu }{(E_\w \bar{T}^{(n)}_\nu)^2} - 1 \geq n^{-\e/s} , \quad M_1 \geq n^{(1-\e)/s} \right)
\end{align*}
Now, since $E_\w \bar{T}^{(n)}_\nu \geq M_1$ we have
\[
Q\left( \frac{E_\w S}{E_\w \bar{T}^{(n)}_\nu} \geq n^{-\e/s}, \: M_1 \geq n^{(1-\e)/s}  \right)  \leq Q\left( E_\w S \geq n^{(1-2\e)/s}, \: M_1 \geq n^{(1-\e)/s}  \right) = o\left( n^{-(6-8\e)/5} \right),
\]
where the last equality is from Corollary \ref{ballistic_Sbig}. Also, by Lemma \ref{ballistic_VarET2compare} we have 
\[
Q\left( \frac{Var_\w \bar{T}^{(n)}_\nu }{(E_\w \bar{T}^{(n)}_\nu)^2} - 1 \geq n^{-\e/s} , \quad M_1 \geq n^{(1-\e)/s} \right) = o\left( n^{-2+4\e} \right).
\]
Then, since $-2+4\e < \frac{-6+8\e}{5}$ when $\e<\frac{1}{8}$ the lemma is proved. 
\end{proof} 
\begin{cor} \label{ballistic_explimit}
Let $\e<\frac{1}{8}$. Then $P-a.s.$, for any sequence $i_k=i_k(\w)$ such that $i_k\in (n_{k-1},n_k]$ and $M_{i_k} > d_k^{(1-\e)/s}$ we have
\begin{equation}
\lim_{k\ra\infty} \phi_{i_k,d_k}(\l) = \frac{1}{1+\l},\quad \forall \l\geq 0, \label{ballistic_mgfbounds}
\end{equation}
and thus
\begin{equation}
\lim_{k\ra\infty} P_\w^{\nu_{i_k-1}} \left( \bar{T}^{(d_k)}_{\nu_{i_k}} > x \mu_{i_k,d_k,\w} \right) =  \Psi(x), \quad \forall x\in \R. \label{ballistic_expl}
\end{equation}
\end{cor}
\begin{proof}
For $i\in(n_{k-1},n_k]$ and all $k$ large enough $\phi_{i,d_k}(\l)$ only depends on the environment to the right of zero, and thus has the same distribution under $P$ and $Q$. Therefore, Lemma \ref{ballistic_mgflem} gives that there exists an $\e'>0$ such that 
\begin{align*}
&P\left( \exists i\in (n_{k-1},n_k], \l\geq 0: \phi_{i,d_k}(\l) \notin \left[ \frac{1-\l d_k^{-\e /s}}{1+\l} , \frac{1}{1+\l-\left(\l+\frac{3 \l^2}{2}\right)d_k^{-\e/s} } \right],\: M_i > d_k^{(1-\e)/s} \right) \\
&\quad \leq d_k Q\left( \exists \l\geq 0: \phi_{1,d_k}(\l) \notin \left[ \frac{1-\l d_k^{-\e /s}}{1+\l} , \frac{1}{1+\l-\left(\l+\frac{3 \l^2}{2}\right)d_k^{-\e/s} } \right],\: M_1 > d_k^{(1-\e)/s} \right) \\
&\quad = o\left( d_k^{-\e'} \right).
\end{align*}
Since this last term is summable in $k$, the Borel-Cantelli Lemma gives that $P-a.s.$ there exists a $k_0=k_0(\w)$ such that for all $k\geq k_0$ we have
\[
i\in (n_{k-1},n_k] \text{ and } M_i\geq d_k^{(1-\e)/s} \Ra \phi_{i,d_k}(\l) \in \left[ \frac{1-\l d_k^{-\e /s}}{1+\l} , \frac{1}{1+\l-\left(\l+\frac{3 \l^2}{2}\right)d_k^{-\e/s} } \right] \quad \forall \l\geq 0,
\]
which proves \eqref{ballistic_mgfbounds}. Then, \eqref{ballistic_expl} follows immediately because $\frac{1}{1+\l}$ is the Laplace transform of an exponential disribution. 
\end{proof}
\end{subsection}

\begin{subsection}{Quenched Exponential Limits Along a Subsequence}
In the previous subsection we showed that the time to cross a single large block is approximately exponential. In this section we show that there are subsequences in the environment where the crossing time of a single block dominates the crossing times of all the other blocks. In this case the crossing time of all the blocks is approximately exponentially distributed. 
Recall the definition of $M_i$ in \eqref{ballistic_Mdef}. For any integer $n\geq 1$, and constants $C>1$ and $\eta>0$, define the event
\[
\mathcal{D}_{n,C,\eta}:= \left\{ \exists i\in\left[1,\eta n \right]: M_i^2 \geq C \sum_{j: i\neq j\leq n} \s_{j,n,\w}^2  \right\}
\]
\begin{lem} \label{ballistic_onebigblock}
Assume $s<2$. Then for any $C>1$ and $\eta>0$ we have $\liminf_{n\ra\infty} Q\left( \mathcal{D}_{n,C,\eta} \right) > 0$. 
\end{lem}
\begin{proof}
First, note that since $\s_{i,n,\w}^2 \geq M_i^2$ and $C>1$ we have 
\begin{equation}
Q\left(\mathcal{D}_{n,C,\eta}\right) = \sum_{i=1}^{\eta n} Q\left( M_i^2 \geq C \sum_{j: i\neq j\leq n} \s_{j,n,\w}^2 \right). \label{ballistic_disjoint}
\end{equation}
Thus, we want to get a lower bound on $Q\left( M_i^2 \geq C \sum_{j:i\neq j\leq n} \s_{j,n,\w}^2 \right)$ that is uniform in $i$. 
The following formula for the quenched variance of $\bar{T}^{(n)}_\nu$ can be deduced from \eqref{ballistic_qvformula} by setting $\rho_{\nu_{-b_n}}=0$:
\begin{align*}
Var_\w \bar{T}^{(n)}_\nu &= 4\sum_{j=0}^{\nu-1}(W_{\nu_{-b_n}+1,j}+W_{\nu_{-b_n}+1,j}^2) + 8\sum_{j=0}^{\nu-1}\sum_{i=\nu_{-b_n}+1}^j \Pi_{i+1,j}(W_{\nu_{-b_n}+1,i}+W_{\nu_{-b_n}+1,i}^2)\\
&\leq 4\sum_{j=0}^{\nu-1}(W_{\nu_{-b_n}+1,j}+W_{\nu_{-b_n}+1,j}^2) + 8\sum_{j=0}^{\nu-1}\sum_{i=\nu_{-b_n}+1}^j W_{\nu_{-b_n}+1,j}(1+W_{\nu_{-b_n}+1,i})\\
&\leq 4\sum_{j=0}^{\nu-1}(W_{\nu_{-b_n}+1,j}+W_{\nu_{-b_n}+1,j}^2) + 8\left( \sum_{j=0}^{\nu-1} W_{\nu_{-b_n}+1,j} \right)\left( \sum_{i=\nu_{-b_n}+1}^{\nu-1} (1+W_{\nu_{-b_n}+1,i}) \right)  ,
\end{align*}
where the first inequality is because $W_{\nu_{-b_n}+1,j} = W_{i+1,j}+\Pi_{i+1,j}W_{\nu_{-b_n}+1,i}$. 
Next, note that if $\nu_{k-1}\leq j < \nu_{k}$ for some $k>-b_n$, then 
\[
W_{\nu_{-b_n}+1,j} = \sum_{l=\nu_{-b_n}+1}^j \Pi_{l,j} = \sum_{l=\nu_{-b_n}+1}^{\nu_{k-1}-1} \Pi_{l,\nu_{k-1}-1}\Pi_{\nu_{k-1},j} + \sum_{l=\nu_{k-1}}^j \Pi_{l,j} \leq (\nu_k-\nu_{-b_n}) M_k,
\]
where the last inequality is because, under $Q$, $\Pi_{l,\nu_{k-1}-1} < 1$ for all $l<\nu_{k-1}$.
Therefore,
\begin{align*}
Var_\w \bar{T}^{(n)}_\nu &\leq 4 \nu_1 \left( (\nu_1-\nu_{-b_n})M_1 +  (\nu_1-\nu_{-b_n})^2M_1^2 \right) \\
&\qquad + 8 \left(\nu_1 (\nu_1-\nu_{-b_n})M_1 \right)
\left( (\nu_1-\nu_{-b_n}) + \sum_{i=-b_n+1}^{1} (\nu_k-\nu_{k-1})(\nu_k-\nu_{-b_n}) M_k \right)\\
&\leq \left( \nu_1-\nu_{-b_n} \right)^4 \left( 12 M_1 + 4 M_1^2 + 8 M_1 \sum_{k=-b_n+1}^1 M_k \right). 
\end{align*}
Similarly, we have that 
$\s_{j,n,\w}^2 \leq \left( \nu_j - \nu_{j-1-b_n} \right)^4 \left( 12 M_j + 4  M_j^2 + 8 M_j \sum_{k=j-b_n}^j M_k \right)$ $Q-a.s.$ for any $j$. 
Now, define the events
\begin{equation}
F_n:= \bigcap_{j\in(-b_n,n]} \left\{ \nu_j-\nu_{j-1} \leq b_n \right\}, \quad\text{and}\quad G_{i,n,\e}:= \bigcap_{j\in[i-b_n,i+b_n]\backslash\{i\}} \left\{ M_j \leq n^{(1-\e)/s} \right\} \label{ballistic_FGdef}
\end{equation}
Then, on the event $F_n\cap G_{i,n,\e} \cap \left\{ M_i \leq 2 n^{1/s} \right\}$ we have for $j\in (i,i+b_n]$ that
\begin{align*}
\s_{j,n,\w}^2 &\leq b_n^4(b_n+1)^4 \left( 12 n^{(1-\e)/s} + 4 n^{(2-2\e)/s} + 8 n^{(1-\e)/s}(b_n n^{(1-\e)/s} + 2 n^{1/s})\right)\\
&\leq  b_n^5(b_n+1)^4 \left( 12 n^{(1-\e)/s} + 12 n^{(2-2\e)/s} + 16 n^{(2-\e)/s} \right) \leq 80 b_n^9 n^{(2-\e)/s},
\end{align*}
where the last inequality holds for all $n$ large enough.
Therefore, for all $n$ large enough
\begin{align*}
&Q\left( M_i^2 \geq C \sum_{j:i\neq j\leq n} \!\!\!\! \s_{j,n,\w}^2 \right) \\
&\quad \geq Q\left(4 n^{2/s} \geq  M_i^2 \geq C \sum_{j:i\neq j\leq n} \s_{j,n,\w}^2 , \ F_n, \ G_{i,n,\e} \right)\\
&\quad \geq Q\left( 4 n^{2/s} \geq M_i^2 \geq C \left( \sum_{j\in[1,n]\backslash [i,i+b_n]} \!\!\!\!\!\! \s_{j,n,\w}^2 + 80 b_n^9 n^{(2-\e)/s} \right) , \ F_n, \ G_{i,n,\e} \right)\\
&\quad \geq Q\left( M_i \in[n^{1/s},2n^{1/s}] , \  \nu_i-\nu_{i-1} \leq b_n \right) \\
&\quad \qquad \times Q\left(  \sum_{j\in[1,n]\backslash [i,i+b_n]} \!\!\!\!\!\! \s_{j,n,\w}^2 + 80 b_n^9 n^{(2-\e)/s} \leq \frac{n^{2/s}}{C},  \ \tilde{F}_n, \ G_{i,n,\e} \right),
\end{align*}
where $\tilde{F}_n := F_n \backslash \{ \nu_i - \nu_{i-1} \leq b_n \}$. Note that in the last inequality we used that $\s_{j,n,\w}^2$ is independent of $M_i$ for $j\notin[i,i+b_n]$. Also, note that we can replace $\tilde{F}_n$ by $F_n$ in the last line above because it will only make the probability smaller. Then, since $\sum_{j\in[1,n]\backslash [i,i+b_n]} \s_{j,n,\w}^2 \leq Var_\w T_{\nu_n}$ we have
\begin{align}
&Q\left( M_i^2 \geq C \sum_{j:i\neq j\leq n} \s_{j,n,\w}^2 \right) \nonumber \\
&\qquad\geq Q\left( M_1 \in [ n^{1/s}, 2n^{1/s} ] , \  \nu \leq b_n \right)Q\left( Var_\w T_{\nu_n} \leq n^{2/s}C^{-1} -40 b_n^7 n^{(2-\e)/s},  \ F_n, \ G_{i,n,\e} \right) \nonumber \\
&\qquad\geq \left( Q(M_1 \in [ n^{1/s}, 2n^{1/s} ] ) - Q(\nu > b_n)   \right) \nonumber \\
&\qquad\qquad \times \left( Q\left(  Var_\w T_{\nu_n} \leq n^{2/s} ( C^{-1} - 40 b_n^7 n^{-\e/s} )  \right) - Q(F_n^c) - Q(G_{i,n,\w}^c)  \right) \nonumber \\
&\qquad\sim C_3 (1-2^{-s})\frac{1}{n} \, L_{\frac{s}{2}, b}\left(C^{-1}\right), \label{ballistic_obblb}
\end{align}
where the asymptotics in the last line are from \eqref{ballistic_nutail}, \eqref{ballistic_Mtail}, and Theorem \ref{ballistic_Varstable}, as well as the fact that $Q(F_n^c) + Q(G_{i,n,\w}^c) \leq (n+b_n)Q(\nu> b_n) + 2b_n Q(M_1 > n^{(1-\e)/s}) = \bigo\left(n e^{-C_2 b_n} \right) + o(n^{-1+2\e})$ due to \eqref{ballistic_nutail} and \eqref{ballistic_Mtail}. Combining \eqref{ballistic_disjoint} and \eqref{ballistic_obblb} finishes the proof.
\end{proof}
\begin{cor} \label{ballistic_subseq}
Assume $s<2$. Then for any $\eta\in(0,1)$, $P-a.s.$ there exists a subsequence $n_{k_m}= n_{k_m}(\w,\eta)$ of $n_k=2^{2^k}$ such that for $\a_m, \b_m,$ and $\gamma_m$ defined as in \eqref{ballistic_abgdef}
we have that 
\begin{equation}
\exists i_m = i_m(\w,\eta) \in (\a_m, \b_m]:\quad M_{i_m}^2 \geq m \!\!\!\! \sum_{j\in(\a_m, \gamma_m]\backslash\{i_m\}} \!\!\!\! \s_{j,d_{k_m},\w}^2 \, .\label{ballistic_sscond}
\end{equation}
\end{cor}
\begin{proof}
Define the events 
\[
\mathcal{D}_{k,C,\eta}':= \left\{ \exists i\in (n_{k-1}, n_{k-1}+\eta d_k]: M_i^2 \geq C \sum_{j\in(n_{k-1},n_k]\backslash\{i\}  } \s_{j,d_k,\w}^2  \right\}.
\]
Note that since $Q$ is invariant under shifts of the $\nu_i$, $Q(\mathcal{D}_{k,C,\eta}') = Q(\mathcal{D}_{d_k,C,\eta})$. Also, due to the reflections of the random walk the event $\mathcal{D}_{k,C,\eta}'$ only depends on the environment between $\nu_{n_{k-1}-b_{d_k}}$ and $\nu_{n_k}$. Thus, for $k$ large enough $\mathcal{D}_{k,C,\eta}'$ only depends on the environment to the right of zero and therefore $P( \mathcal{D}_{k,C,\eta}') =  Q(\mathcal{D}_{k,C,\eta}') = Q(\mathcal{D}_{d_{k},C,\eta})$. Therefore $\liminf_{k\ra\infty} P( \mathcal{D}_{k,C,\eta}') >0$. Also, since $n_{k-1}-b_{d_k} > n_{k-2}$ for all $k\geq 4$, we have that $\{ \mathcal{D}_{2k,C,\eta}' \}_{k=2}^\infty$ is an independent sequence of events. Thus, we get that for any $C>1$ and $\eta\in(0,1)$, infinitely many of the events $\mathcal{D}_{k,C,\eta}$ occur $P-a.s.$ Therefore, $P-a.s.$ there is a subsequence $k_m=k_m(\w)$ such that $\w\in \mathcal{D}_{k_m,m,\eta}$ for all $m$. In particular, for this subsequence $k_m$ we have that \eqref{ballistic_sscond} holds.
\end{proof}
\begin{thm}\label{ballistic_Tnexplimit}
Assume $s<2$. Then for any $\eta\in(0,1)$, $P-a.s.$ there exists a subsequence $n_{k_m}=n_{k_m}(\w, \eta)$ of $n_k=2^{2^k}$ such that for $\a_m,\beta_m$ and $\gamma_m$ defined as in \eqref{ballistic_abgdef} and any sequence $x_m \in \left(\nu_{\beta_m} , \nu_{\gamma_m} \right]$  we have
\[
\lim_{m\ra\infty} P_\w^{\nu_{\a_m}} \left( \frac{\bar{T}_{x_m}^{(d_{k_m})} - E_\w^{\nu_{\a_m}} \bar{T}_{x_m}^{(d_{k_m})} }{\sqrt{v_{k_m,\w}}} \leq x \right) = \Psi(x+1), \quad \forall x\in\R.
\]
\end{thm}
\begin{proof}
First, note that 
\[
 P\left( \max_{j\in (n_{k-1}, n_{k}]} M_j \leq d_k^{(1-\e)/s} \right) = \left( 1- P\left(M_1 > d_k^{(1-\e)/s}\right) \right)^{d_k} = o\left( e^{-d_k^{\e/2}} \right),
\]
where the last equality is due to \eqref{ballistic_Mtail}. Therefore, the Borel-Cantelli Lemma gives that, $P-a.s.$,
\begin{equation}
\max_{j\in (n_{k-1}, n_{k}]} M_j > d_k^{(1-\e)/s} \quad \text{ for all } k \text{ large enough.} \label{ballistic_maxMlarge}
\end{equation}
Therefore, $P-a.s.$ we may assume that \eqref{ballistic_maxMlarge} holds, the conclusion of Corollary \ref{ballistic_explimit} holds, and that there exist subsequences $n_{k_m}=n_{k_m}(\w,\eta)$ and $i_m=i_m(\w,\eta)$ as specified in Corollary \ref{ballistic_subseq}.
Then, by the choice of our subsequence $n_{k_m}$, only the crossing of the largest block (i.e. from $\nu_{i_m-1}$ to $\nu_{i_m}$) is relevant in the limiting distribution. Indeed,
\begin{align*}
&P_\w^{\nu_{\a_m}} \left( \left| \frac{ \left( \bar{T}_{\nu_{i_m-1}}^{(d_{k_m})} - E_\w ^{\nu_{\a_m}}\bar{T}_{\nu_{i_m-1}}^{(d_{k_m})} \right) + \left( \bar{T}_{x_m}^{(d_{k_m})} - \bar{T}_{\nu_{i_m}}^{(d_{k_m})}  - E_\w^{\nu_{i_m}} \bar{T}_{x_m}^{(d_{k_m})} \right) }{\sqrt{v_{k_m,\w}}} \right| \geq \e \right) \\
&\qquad \leq \frac{ Var_\w \left( \bar{T}_{x_m}^{(d_{k_m})} - \bar{T}_{\nu_{\a_m}}^{(d_{k_m})} \right) - \s_{i_m,d_{k_m},\w}^2 } { \e^2 v_{k_m,\w} }  \leq \frac{ \sum_{j\in (\a_m,\gamma_m]\backslash \{i_m\}} \s_{j,d_{k_m},\w}^2 }{ \e^2 M_{i_m}^2 } \leq \frac{1}{\e^2 m},
\end{align*}
where in the second to last inequality we used that $v_{k_m,\w} \geq \s_{i_m,d_{k_m},\w}^2 \geq M_{i_m}^2$, and the last inequality is due to our choice of the sequence $i_m$. 
Thus we have reduced the proof of the Theorem to showing that
\begin{equation}
\lim_{m\ra\infty} P_\w^{\nu_{i_m-1}} \left( \frac{\bar{T}_{\nu_{i_m}}^{(d_{k_m})} - \mu_{i_m,d_{k_m},\w} }{\sqrt{v_{k_m,\w}}} \leq x \right) = \Psi(x+1), \quad \forall x\in\R. 
\end{equation}
Now, since $i_m$ is chosen so that $M_{i_m} = \max_{j\in (n_{k_m-1}, n_{k_m}]} M_j$, we have that $M_{i_m} \geq d_{k_m}^{(1-\e)/s}$ for any $\e>0$ and all $m$ large enough. Then, the conclusion of Corollary \ref{ballistic_explimit} gives that
\[
\lim_{m\ra\infty} P_\w^{\nu_{i_m-1}} \left( \frac{ \bar{T}^{(d_{k_m})}_{\nu_{i_m}} }{ \mu_{i_m,d_{k_m},\w} } \leq x \right) = \Psi(x).
\]
Thus, the proof will be complete if we can show
\begin{equation}
\lim_{m\ra\infty} \frac{\mu_{i_m,d_{k_m},\w}}{\sqrt{v_{k_m,\w}}} = 1. \label{ballistic_vovermuim}
\end{equation}
However, by our choice of $n_{k_m}$ and $i_m$ we have 
\[
\s_{i_m,d_{k_m},\w}^2 \geq M_{i_m}^2 \geq  m \sum_{j\in (\a_m, \gamma_m] \backslash\{i_m\} } \s_{j,d_{k_m}\w}^2 = m\left( v_{k_m,\w} - \s_{i_m,d_{k_m},\w}^2 \right),
\]
which implies that
\begin{equation}
1\leq \frac{ v_{k_m,\w} }{ \s_{i_m,d_{k_m},\w}^2 } \leq \frac{m+1}{m} \underset{m\ra\infty}{\longrightarrow} 1. \label{ballistic_vovers}
\end{equation}
Also, we can use Lemma \ref{ballistic_VarET2compare} to show that for $k$ large enough and $\e>0$ 
\begin{align*}
&P\left( \exists i\in(n_{k-1},n_k]: \left| \frac{\s_{i,d_k,\w}^2}{\mu_{i,d_k,\w}^2} - 1 \right| \geq d_k^{-\e/s},\quad M_i \geq d_k^{(1-\e)/s} \right) \\
&\quad \leq d_k Q\left( \left| \frac{ Var_\w \bar{T}^{(d_k)}_\nu }{ \left( E_\w \bar{T}^{(d_k)}_\nu \right)^2 } - 1 \right| \geq d_k^{-\e/s},\quad M_1 \geq d_k^{(1-\e)/s} \right) = o\left(d_k^{-1+4\e}\right).
\end{align*}
Then, for $\e<\frac{1}{4}$ the Borel-Cantelli Lemma gives that $P-a.s.$ there exists a $k_0=k_0(\w)$ such that for $k\geq k_0$ and $i\in(n_{k-1}, n_k]$ with $M_i\geq d_k^{(1-\e)/s}$ we have $\left| \frac{\s_{i,d_k,\w}^2}{\mu_{i,d_k,\w}^2} - 1 \right| < d_k^{-\e/s}$. In particular, since $M_{i_m} \geq d_{k_m}^{(1-\e)/s}$ for all $m$ large enough, we have that
\begin{equation}
\lim_{m\ra\infty} \frac{ \s_{i_m,d_{k_m},\w}^2 }{ \mu_{i_m,d_{k_m},\w}^2 } = 1. \label{ballistic_soverm}
\end{equation}
Since \eqref{ballistic_vovers} and \eqref{ballistic_soverm} imply \eqref{ballistic_vovermuim}, the proof is complete. 
\end{proof}
\begin{proof}[\textbf{Proof of Theorem \ref{ballistic_qEXP}:}] \ \\
As in the proof of Theorem \ref{ballistic_qCLT} this follows from Proposition \ref{ballistic_generalprop}.
\end{proof}
\end{subsection}
\end{section}

\begin{section}{Stable Behavior of the Quenched Variance}\label{ballistic_qvs}
Recall from Theorem \ref{ballistic_VETtail} that $Q\left( Var_\w T_\nu > x \right) \sim K_\infty x^{-s/2}$. Since the sequence of random variables $\left\{ Var_\w (T_{\nu_i} - T_{\nu_{i-1}}) \right\}_{i\in \N}$ is stationary under $Q$ (and weakly dependent) it is somewhat natural to expect that $n^{-2/s} Var_\w T_{\nu_n}$ converges in distribution (under $Q$) to a stable law of index $\frac{s}{2}<1$.  

\begin{proof}[\textbf{Proof of Theorem \ref{ballistic_Varstable}:}] \ \\
Obviously it is enough to prove that the second equality in \eqref{ballistic_stableET2} holds and that
\begin{equation}
\lim_{n\ra\infty} Q\left( \left| Var_\w T_{\nu_n} - \sum_{i=1}^n (E_\w^{\nu_{i-1}} T_{\nu_i})^2 \right| > \d n^{2/s} \right) = 0, \quad \forall \d>0. \label{ballistic_VET2diff} 
\end{equation}
However, \eqref{ballistic_VET2diff} is the statement of Corollary \ref{Vsdiff} with $m=\infty$.
Thus it is enough to prove the second equality in \eqref{ballistic_stableET2}. To this end, first note that 
\begin{align}
\frac{1}{n^{2/s}}\sum_{i=1}^n \left( E_\w^{\nu_{i-1}} T_{\nu_i} \right)^2 &= 
\frac{1}{n^{2/s}} \sum_{i=1}^n \left( \left( E_\w^{\nu_{i-1}} T_{\nu_i} \right)^2 - \left( E_\w^{\nu_{i-1}} \bar{T}^{(n)}_{\nu_i} \right)^2 \right) \label{ballistic_switchET2} \\
&\qquad + \frac{1}{n^{2/s}} \sum_{i=1}^n \left( E_\w^{\nu_{i-1}} \bar{T}^{(n)}_{\nu_i} \right)^2 \mathbf{1}_{M_i \leq n^{(1-\e)/s}}  \label{ballistic_smallET2} \\
&\qquad + \frac{1}{n^{2/s}} \sum_{i=1}^n \left( E_\w^{\nu_{i-1}} \bar{T}^{(n)}_{\nu_i} \right)^2 \mathbf{1}_{M_i > n^{(1-\e)/s}}. \label{ballistic_bigET2}
\end{align}
Therefore, it is enough to show that \eqref{ballistic_switchET2} and \eqref{ballistic_smallET2} converge to $0$ in distribution (under $Q$) and that
\begin{equation}
\lim_{n\ra\infty} Q\left( \frac{1}{n^{2/s}} \sum_{i=1}^n \left( E_\w^{\nu_{i-1}} \bar{T}^{(n)}_{\nu_i} \right)^2 \mathbf{1}_{M_i > n^{(1-\e)/s}} \leq x \right) = L_{\frac{s}{2},b}(x) \label{ballistic_bigET2stable}
\end{equation}
for some $b>0$. To prove that \eqref{ballistic_switchET2} converges to $0$ in distribution, first note that factoring gives
\[
 \left( E_\w^{\nu_{i-1}} T_{\nu_i} \right)^2 - \left( E_\w^{\nu_{i-1}} \bar{T}^{(n)}_{\nu_i} \right)^2  \leq 2 E_\w^{\nu_{i-1}} T_{\nu_i} \left(  E_\w^{\nu_{i-1}} T_{\nu_i} - E_\w^{\nu_{i-1}} \bar{T}^{(n)}_{\nu_i} \right).
\]
Therefore, for any $\d>0$
\begin{align}
&Q\left( \sum_{i=1}^n \left( \left( E_\w^{\nu_{i-1}} T_{\nu_i} \right)^2 - \left( E_\w^{\nu_{i-1}} \bar{T}^{(n)}_{\nu_i} \right)^2 \right) > \d n^{2/s} \right) \nonumber \\
&\qquad \leq  Q\left( \sum_{i=1}^n 2 E_\w^{\nu_{i-1}} T_{\nu_i} \left(  E_\w^{\nu_{i-1}} T_{\nu_i} - E_\w^{\nu_{i-1}} \bar{T}^{(n)}_{\nu_i}  \right) > \d n^{2/s} \right) \nonumber \\
&\qquad \leq n Q\left(  E_\w T_{\nu} - E_\w \bar{T}^{(n)}_{\nu}  > 1 \right) + Q\left( 2 E_\w T_{\nu_n} > \d n^{2/s} \right). \label{ballistic_switchET2b}
\end{align}
Then, Lemma \ref{ETdiff} and Theorem \ref{refstable} give that both terms in \eqref{ballistic_switchET2b} tend to zero as $n\ra\infty$. The proof that \eqref{ballistic_smallET2} converges in distribution to $0$ is essentially a counting argument. Since the $M_i$ are all independent and from \eqref{ballistic_Mtail} we know the asymptotics of $Q(M_i > x)$, we can get good bounds on the number of $i\leq n$ with $M_i\in(n^\a, n^\b]$. Then, since by \eqref{TbigMsmall} we have $Q\left( E_\w^{\nu_{i-1}} \bar{T}^{(n)}_{\nu_i} \geq n^\b, M_i\leq n^\a\right) = o\left( e^{-n^{(\b-\a)/5}} \right)$ we can also get good bounds on the number of $i\leq n$ with $E_\w^{\nu_{i-1}} \bar{T}^{(n)}_{\nu_i} \in(n^\a, n^\b]$. The details of this argument are essentially the same as the proof of Lemma \ref{Vsmall} and will thus be ommitted. 
Finally, we will use \cite[Theorem 5.1(III)]{kGPD} to prove \eqref{ballistic_bigET2stable}. 
Now, Theorem \ref{ballistic_VETtail} gives that $Q\left( \left( E_\w T_{\nu} \right)^2 \mathbf{1}_{M_1 > n^{(1-\e)/s}} > x n^{2/s} \right) \sim K_\infty x^{-s/2} n^{-1}$, and Lemma \ref{alphamixing} gives bounds on the mixing of the array $\left\{ \left( E_\w^{\nu_{i-1}} T_{\nu_i} \right)^2 \mathbf{1}_{M_i > n^{(1-\e)/s}} \right\}_{i\in\Z, n\in\N}$.
This is enough to verify the first two conditions of \cite[Theorem 5.1(III)]{kGPD}. The final condition that needs to be verified is
\begin{equation}
\lim_{\d\ra 0}\limsup_{n\ra\infty} n E_Q \left[
n^{-2/s} (E_\w \bar{T}_{\nu}^{(n)})^2 \mathbf{1}_{M_1>n^{(1-\e)/s} }
\mathbf{1}_{n^{-1/s} E_\w \bar{T}_{\nu}^{(n)}  \leq \d} \right] = 0 \,.
\label{ballistic_truncexp}
\end{equation}
By Theorem \ref{ballistic_VETtail} we have that 
there exists a constant $C_4>0$ such that
for any $x > 0$,
\[
Q\left( E_\w \bar{T}_{\nu}^{(n)} > x
n^{1/s} , M_1
> n^{(1-\e)/s} \right) \leq Q\left( E_\w T_\nu > x
n^{1/s} \right) \leq C_4 x^{-s}\frac{1}{n}.
\]
Then using this we have
\begin{align*}
& n E_Q \left[ n^{-2/s} \left( E_\w \bar{T}_{\nu}^{(n)} \right)^2 \mathbf{1}_{M_1>n^{(1-\e)/s} }
\mathbf{1}_{n^{-1/s}E_\w \bar{T}_{\nu}^{(n)}  \leq \d} \right] \\
&\quad = n \int_0^{\d^2} Q\left( \left(E_\w \bar{T}_{\nu}^{(n)}\right)^2 > x n^{2/s} ,
M_1 > n^{(1-\e)/s} \right)  dx \\
&\quad \leq C_4 \int_0^{\d^2} x^{-s/2} dx = \frac{C_4 \d^{2-s}}{1-s/2}\,,
\end{align*}
where the last
integral is finite since $s<2$. \eqref{ballistic_truncexp} follows, and therefore by \cite[Theorem 5.1(III)]{kGPD} we have that \eqref{ballistic_bigET2stable} holds.
\end{proof}
\end{section}

\textbf{Acknowledgments.} I would like to thank Olivier Zindy for his helpful comments regarding the analysis of the quenched Laplace transform of $\bar{T}^{(n)}_\nu$ in Section \ref{ballistic_Laplace}. 

\end{chapter}

\begin{chapter}{Large Deviations for RWRE on $\Z^{d}$}\label{mdLDP}

We now turn to some results for multidimensional RWRE. 
Unless otherwise mentioned, we will consider only nearest neighbor RWRE with i.i.d. and uniformly elliptic environments.
Section \ref{mdprelims} is a review some of the basic results, notation, and open problems for multidimensional RWRE. Section \ref{RWRELDP} is a survey of the large deviation results that are known for multidimensional RWRE. 
The main results of the chapter are contained in Section \ref{LDPnewresults}, where we prove a new result on differentiability properties of the annealed rate function when the law on environments is ``non-nestling.''


\begin{section}{Preliminaries of Multi-dimensional RWRE}\label{mdprelims}
While RWRE on $\Z$ are quite well understood, much less is known about RWRE on $\Z^d$. 
In particular, even in the case of i.i.d., uniformly elliptic environments with $d\geq 3$, 
the 0-1 law for transience in a given direction is still an open problem. 
Let $S^{d-1}:= \{ \xi \in \R^d : \| \xi \| = 1 \}$, and for any $\ell \in S^{d-1}$ let $A_\ell := \{ \lim_{n\ra 0} X_n \cdot \ell = +\infty \}$ be the event of transience in the direction $\ell$. 
For uniformly elliptic environments satisfying certain strong mixing conditions (in particular, for uniformly elliptic, i.i.d. environments), it is known \cite{zRWRE} that $\P(A_\ell\cup A_{-\ell})\in \{0,1\}$. 
This prompts the following question:

\begin{op}[0-1 Law]\label{01law}
Is it true that $\P(A_\ell)\in\{0,1\}$?
\end{op}



If the answer to Question \ref{01law} is affirmative (or negative), then we say that the \emph{0-1 law holds} (or \emph{fails}).
For i.i.d., uniformly elliptic laws on environments, the 0-1 law holds when $d=2$ \cite{zm0-1Law}, but Question \ref{01law} is still an open problem when $d\geq 3$. Question \ref{01law} is also still an open problem when $d=2$ and the environment is i.i.d. but not uniformly elliptic. 
When $d=2$ there are examples of ergodic, elliptic laws on environments for which the 0-1 law fails \cite{zm0-1Law}, and when $d\geq 3$ there are examples of mixing, uniformly elliptic laws on environments for which the 0-1 law fails \cite{bzzTrees}. 

In general, it is not known if a law of large numbers exists (i.e., if $\lim_{n\ra\infty} \frac{X_n}{n}$ is constant, $\P-a.s.$).
However, for i.i.d., uniformly elliptic laws on environments, it is known that there are at most two limiting speeds of the random walk \cite{zRWRE,zLLN}. That is, there exists an $\ell \in S^{d-1}$ such that
\begin{equation}
\lim_{n\ra\infty} \frac{X_n}{n} = v_+ \mathbf{1}_{A_\ell}+ v_- \mathbf{1}_{A_{-\ell}},\quad \P-a.s., \label{twospeeds}
\end{equation}
where $v_+ = c_1 \ell$ and $v_- = -c_2 \ell$ for some $c_1,c_2 \geq 0$. (A recent result of Berger \cite{bVel} shows that when $d\geq 5$, $v_-$ and $v_+$ cannot both be non-zero.) 
Thus, if it can be shown that a 0-1 law holds, then \eqref{twospeeds} would imply a law of large numbers. 
If $\lim_{n\ra\infty} \frac{X_n}{n}$ is constant, $\P-a.s.$, then we will denote the limit by $v_P$. 

There are known conditions for laws on environments that imply a law of large numbers for the RWRE. 
Recall the following terminology originally introduced by Zerner \cite{zLDP}:
\begin{defn}
Let $d(\w):= E_\w X_1$ be the drift at the origin in the environment $\w$,
and let $\mathcal{K}:= conv \left( supp \left( d(\w) \right) \right)$ be the convex hull of the support, under $P$, of all possible drifts. Then, we say that the law on environments $P$ is \emph{nestling} if $0\in \mathcal{K}$ and \emph{non-nestling} if $0\notin \mathcal{K}$. We say that $P$ is \emph{non-nestling in direction} $\ell\in S^{d-1}$ if $\inf_{x\in \mathcal{K}} x \cdot \ell > 0$ (or equivalently, if $P( E_\w X_1 \cdot \ell > \e ) = 1$ for some $\e > 0$). 
\end{defn}
\noindent If $P$ is non-nestling, it is known that the 0-1 law holds and also that a law of large numbers holds with limiting velocity $v_P \neq 0$. In fact, this follows from the fact that non-nestling laws $P$ also satisfy what is known as Kalikow's condition \cite{kGRW}, which implies that the 0-1 law holds and that $v_P\neq 0$ (see \cite{szLLN}). Since Kalikow's condition holds for some (but not all) nestling laws $P$ (see \cite{kGRW} for examples), it is still not known for general i.i.d. nestling laws on environments if there can exist two limiting velocities $v_+ $ and $v_- $. 

A useful tool in much of the recent progress on multidimensional RWRE is what are referred to as regeneration times. Fix an $\ell \in S^{d-1}$ such that $ c \ell \in \Z^d$ for some $c>0$. Then, the regeneration times in direction $\ell$ are 
\[
\tau_1 := \inf \{ n > 0 : X_k \cdot \ell < X_n \cdot \ell \leq X_m \cdot \ell , \quad \forall k<n, \quad \forall m\geq n \},
\]
and
\[
\tau_i := \inf \{ n > \tau_{i-1} : X_k \cdot \ell < X_n \cdot \ell \leq X_m \cdot \ell , \quad \forall k<n, \quad \forall m\geq n \}, \quad\mbox{for } i > 1.
\]
\textbf{Remark:} The condition that $c \ell \in \Z^d$ is chosen to allow for a simpler definition of regeneration times that agrees with the one given by Sznitman and Zerner \cite{szLLN} (set $a = \frac{1}{c}$ in the definition of regeneration times in \cite{szLLN}). If $P$ is non-nestling in direction $\ell \in S^{d-1}$, then $P$ is also non-nestling for all $\ell'\in S^{d-1}$ in a neighborhood of $\ell$. Therefore, if $P$ is non-nestling, then we can always find an $\ell\in S^{d-1}$ such that $P$ is non-nestling in direction $\ell$, and $c \ell \in \Z^d$ for some $c>0$. 

The regeneration times $\tau_i$ are obviously not stopping times since they depend on the future of the random walk. The advantage of working with regeneration times is that they introduce an i.i.d. structure. Let $D:= \{ X_n \cdot \ell \geq 0, \; \forall n \geq 0 \}$, and when $\P(D)>0$, let $\bP$ be the annealed law of the RWRE conditioned on the event $D$ (i.e., $\bP(\,\cdot \,) :=\P( \cdot \, | D )$). Let expectations under the measure $\bP$ be denoted $\bE$. 
\begin{thm}[Sznitman and Zerner \cite{szLLN}]\label{Thesis_regindep}
 Assume $\P(A_\ell) = 1$. Then $\P(D)> 0$, and 
\[
 (X_{\tau_1}, \tau_1), (X_{\tau_2} -X_{\tau_1}, \tau_2 - \tau_1), \ldots ,(X_{\tau_{k+1}} -X_{\tau_k}, \tau_{k+1} - \tau_k), \ldots
\]
are independent random variables. Moreover, the above sequence is i.i.d. under $\bP$.  
\end{thm}
\noindent\textbf{Remarks:} \\
\textbf{1.} If $P$ is non-nestling in direction $\ell$, then $\P(A_\ell)=1$ and so Theorem \ref{Thesis_regindep} holds. \\
\textbf{2.} The assumption that $\P(A_\ell)=1$ in Theorem \ref{Thesis_regindep} is only needed to ensure that $\tau_1 < \infty$. In fact, what is shown in \cite{szLLN} is that $\P(A_\ell ) > 0$ implies that $\P(D)>0$ and that $(X_{\tau_1}, \tau_1), (X_{\tau_2} -X_{\tau_1}, \tau_2 - \tau_1), \ldots$ are i.i.d. under $\bP$. 

As mentioned above, for i.i.d., uniformly elliptic environments, $\P(A_\ell)=1$ also implies a law of large numbers. Thus, a consequence of Theorem \ref{Thesis_regindep} is the following formula for the limiting velocity $v_P$:
\begin{equation}
 v_P = \lim_{n\ra\infty} \frac{X_n}{n} = \frac{ \bE X_{\tau_1}}{\bE \tau_1}, \qquad \P-a.s. \label{vPregenformula}
\end{equation}

Recently, Sznitman introduced conditions, known as \emph{conditions $(T)$ and $(T')$ relative to a direction $\ell$}, that are more general than Kalikow's condition and which also imply a law of large numbers with non-zero limiting velocity and an annealed  central limit theorem \cite{sConditionT}.
Sznitman also described in \cite{sEffective} a criterion that can
be checked by considering the environment restricted to a box
$B_n=[-n,n]^d$, with the property that $(T')$ holds if and only if the condition holds for
some box $B_n$. Such a criterion is not known to exist for Kalikow's
condition.

%

\end{section}

\begin{section}{Large Deviations for RWRE on $\Z^d$}\label{RWRELDP}
In this section, we will review some of the known results for large deviations of RWRE on $\Z^d$. We recall the following terminology from \cite{dzLDTA}: A \emph{good rate function} is a lower semi-continuous $[0,\infty]$-valued function $h(x)$ with the property that $\{x: h(x) \leq a\}$ is compact for every $a<\infty$. A sequence $\R^d$-valued random variables $\xi_n$ is said to satisfy a \emph{large deviation principle} (LDP) with good rate function $I(x)$ if for any Borel $\Gamma \subset \R^d$,
\[
-\inf_{x\in\overset{\circ}{\Gamma} } I(x) \leq \liminf_{n\ra\infty} \frac{1}{n} \log P\left( \xi_n \in \Gamma \right) \leq \limsup_{n\ra\infty} \frac{1}{n} \log  P\left( \xi_n \in \Gamma \right) \leq -\inf_{x\in\overline{\Gamma}} I(x). 
\]
The random variables $\xi_n$ satisfy a \emph{weak large deviation principle} if the above inequalities hold for all bounded Borel $\Gamma \subset \R^d$.

\begin{subsection}{Large Deviations: $d=1$}\label{onedimLDP}
Comets, Gantert, and Zeitouni \cite{cgzLDP} give a rather complete treatment of quenched and annealed large deviations for one-dimensional RWRE. In \cite{cgzLDP}, quenched large deviations are first obtained for the hitting times $T_n$ and $T_{-n}$ using an argument similar to the proof of the G\"artner-Ellis Theorem (see Theorem 2.3.6 in \cite{dzLDTA}). The LDPs for the hitting times are then transferred to a quenched LDP for $X_n$. Finally, Varadhan's Lemma \cite[Theorem 4.3.1]{dzLDTA} is used to derive the annealed large deviations from the quenched large deviations. Even for random walks in i.i.d. environments, this method for deriving an annealed LDP require an understanding of the quenched LDP for random walks in ergodic environments.

One advantage to the approach used in \cite{cgzLDP} to derive an annealed LDP
is that the annealed rate function is given in terms of a variational formula involving the quenched rate function and the specific entropy of measures on environments. 
Another advantage is that qualitative behavior of both the quenched and annealed rate functions are derived. 
In particular, the rate function (quenched or annealed) in the negative direction is equal to the sum of the rate function in the positive direction and a linear function with slope $E_P \log \rho_0$. (This implies that for transient RWRE, the rate functions are not differentiable at the origin.) 
Other differentiability properties of the rate function, while not mentioned in \cite{cgzLDP}, are obtained without too much difficulty from the formulas given for the rate functions there. Also, if $P$ is nestling, then the quenched and annealed rate functions are zero on the interval $[0,v_P]$. 

The one-dimensional quenched LDP was first derived by Greven and den Hollander in \cite{gdhLDP} using homogenization techniques. 
Greven and den Hollander were also able to show the qualitative behavior of the quenched rate function mentioned above. 
Rosenbluth \cite{rThesis} has also recently derived the same formula for the one-dimensional quenched rate function as a special case of a multidimensional quenched LDP. 
Rosenbluth's approach also uses homogenization techniques, and he formulates the quenched rate function as the solution to a variational problem. 
In the one-dimensional case, Rosenbluth is able to solve this variational formula to obtain the simpler form of the quenched rate function which also appears in \cite{cgzLDP} and \cite{gdhLDP}. 
\end{subsection}

\begin{subsection}{Large Deviations: $d\geq 2$}
Although a law of large numbers is not known to hold for general i.i.d. environments, Varadhan \cite{vLDP} has given both a quenched and annealed LDP.
%
%
\begin{thm}[Varadhan \cite{vLDP}]\label{annealedLDP}
Let $X_n$ be a nearest neighbor RWRE on $\Z^d$, and let $P$ be a uniformly elliptic, i.i.d. measure on environments. Then, there exist convex (non-random) functions $h(v)$ and $H(v)$ such that $\frac{X_n}{n}$ satisfies both a quenched and an annealed large deviation principle with good rate functions $h(v)$ and $H(v)$, respectively. That is, for any Borel subset $\Gamma \subset \R^d$, 
\[
-\inf_{v\in\overset{\circ}{\Gamma} } h(v) \leq \frac{1}{n} \log \liminf_{n\ra\infty} P_\w\left( \frac{X_n}{n} \in \Gamma \right) \leq \limsup_{n\ra\infty} \frac{1}{n} \log P_\w\left( \frac{X_n}{n} \in \Gamma \right) \leq -\inf_{v\in\overline{\Gamma}} h(v), 
\]
for $P-$almost every environment $\w$, and 
\[
-\inf_{v\in\overset{\circ}{\Gamma} } H(v) \leq \frac{1}{n} \log \liminf_{n\ra\infty} \P \left( \frac{X_n}{n} \in \Gamma \right) \leq \frac{1}{n} \log \limsup_{n\ra\infty} \P\left( \frac{X_n}{n} \in \Gamma \right) \leq -\inf_{v\in\overline{\Gamma}} H(v).
\]
\end{thm}
\noindent\textbf{Remark:} In \cite{vLDP}, Varadhan actually proves a more general theorem. In particular, he shows that the conclusion of Theorem \ref{annealedLDP} holds for for RWRE with bounded jumps in i.i.d. environments with certain strong uniform ellipticity conditions. 

Varadhan's proof of the quenched LDP is based on a simple sub-additivity argument, but the argument does not give much information about the quenched rate function $h(v)$. In particular, the argument only shows that $h$ is convex. The proof of the annealed LDP in \cite{vLDP} is much more complicated. A RWRE $X_n$ is not a Markov chain (annealed) since it has ``long term memory.'' Therefore, Varadhan studies the path of the environment shifted to end at the origin 
\[
 W_n = (-X_n, -X_n+X_1 , \ldots, -X_n + X_{n-1}, 0).
\]
Since $W_n$ incorporates the history of the walk, it is a Markov chain on a very large state space $\mathrm{W}$. Varadhan then shows that a process level LDP holds for $W_n$. That is, a LDP holds for the measure valued process $\mathcal{R}_n := \frac{1}{n} \sum_{j=1}^n \d_{W_j}$ with good rate function $\mathcal{J}(\mu)$, which is infinite unless $\mu$ is a stationary measure on the space $\overline{\mathrm{W}}$, which is a specified compactification of $\mathrm{W}$. The annealed LDP for $X_n/n$ is then obtained by contraction, and the rate function $H(v)$ is given by 
\[
 H(v) = \inf_{\mu \in \mathcal{E}, \; m(\mu) = v} \mathcal{J}(\mu),
\]
where $\mathcal{E}$ is the set of ergodic measures on $\mathrm{W}$ and $m(\mu)$ is the average step size of the ergodic measure $\mu$. 

Since Varadhan's derivation of an annealed LDP requires an understanding of process level large deviations on the huge state space $\overline{\mathrm{W}}$, it is difficult to derive much qualitative information about the rate function $H(v)$ using Varadhan's formula for $H(v)$. Nevertheless, Varadhan was able to prove the following statement about the zero set of the quenched and annealed rate functions:
\begin{thm}[Varadhan \cite{vLDP}] \label{Thesis_zeroset}
The zero sets of the quenched and annealed rate functions in Theorem \ref{annealedLDP} are identical. That is, $h(v) = 0 \iff H(v) = 0$. Moreover, the zero set of the rate functions $Z=\{v: H(v) = 0\}$ has the following description:\\
\textbf{Non-nestling:} If $P$ is non-nestling, then the zero set is a single point $Z=\{ v_P \}$. \\
\textbf{Nestling:} If $P$ is nestling, then the zero set is a line segment containing the origin. If $\lim_{n\ra\infty} \frac{X_n}{n} = v_P, \; \P-a.s.$, then $Z=[0,v_P]$. Otherwise $Z=[v_-,v_+]$, where $v_-$ and $v_+$ are the two possible limiting velocities. 
\end{thm}

We end this section by briefly mentioning some of the other large deviation results for multidimensional RWRE. 
Rassoul-Agha \cite{rLDP} extended the approach of Varadhan to get an LDP for certain non-i.i.d. laws on environments and other non-Markov random walks on $\Z^d$. 
Zerner \cite{zLDP} also proved a quenched LDP using a sub-additivity argument. However, unlike Varadhan's subadditive argument, Zerner's method involved hitting times and was restricted to nestling laws on environments. Recent results of Yilmaz \cite{ySTLDP, yQLDP} take a different approach, using homogenization techniques to derive quenched LDP results. The techniques used in \cite{ySTLDP} are similar to those used in \cite{krvStochHomog} for diffusions in a random environment, and, in \cite{yQLDP}, it is shown that, for ``space-time'' RWRE, the quenched and annealed rate functions are identical in a neighborhood of the critical velocity $v_P$.

\end{subsection}
\end{section}

\begin{section}{Differentiability of the Annealed Rate Function}\label{LDPnewresults}
In this section, we will study the annealed rate function $H(v)$ from Theorem \ref{annealedLDP}. (Recall that we are only considering nearest neighbor RWRE in this chapter.) 
As mentioned in Subsection \ref{onedimLDP}, many differentiability properties of the annealed rate function are known when $d=1$. 
Until now, however, no differentiability properties of $H(v)$ were known when $d\geq 2$. 
Our main result is the following theorem:
\begin{thm}\label{Thesis_LDPdiff}
Let $X_n$ be a nearest neighbor RWRE on $\Z^d$, and let $P$ be a uniformly elliptic, i.i.d., and non-nestling measure on environments. Then, the annealed rate function $H(v)$ is analytic in a neighborhood of $v_P$. 
\end{thm}
The variational formula for $H(v)$ given in \cite{vLDP} is very hard to work with. 
Instead of approaching a LDP through the Markov chain $W_n$ on the huge state space $\overline{W}$, we take advantage of the i.i.d. structure present in regeneration times. From Cram\'er's Theorem, differentiability properties of the large deviation rate functions for sums of i.i.d. random variables can easily be obtained. We are then able to transfer these differentiability properties to a new function $\bar{J}$ defined in terms of large deviations for $(X_{\tau_k}, \tau_k)$, and then show that $\bar{J}(v)=H(v)$ in a neighborhood of $v_P$ when $P$ is non-nestling. 

We conclude the section by showing that when $d=1$ and $X_n \ra +\infty$, the equality $\bar{J}(v)=H(v)$ holds for all $v\geq 0$ (for both nestling and non-nestling laws $P$). 

\begin{subsection}{The Rate Function $\bar{J}$}
Since $P$ is non-nestling, for the remainder of this section, we fix a direction $\ell\in S^{d-1}$ such that $c \ell \in \Z^d$, for some $c>0$, and $P$ is non-nestling in direction $\ell$. 
Let $\tau_i$ be the regeneration times in direction $\ell$.
For $\l\in\R^{d}\times \R=\R^{d+1}$, let
\[
\bL(\l):= \log \bE e^{\l \cdot (X_{\tau_1}, \tau_1)}.
\]
By Theorem \ref{Thesis_regindep}, 
$(X_{\tau_1}, \tau_1), (X_{\tau_2} - X_{\tau_{1}} , \tau_2 - \tau_{1} ), \ldots$ 
is an $i.i.d.$ sequence under $\bP$. Therefore, Cram\'er's Theorem \cite[Theorem 6.1.3]{dzLDTA} implies that $\frac{1}{n}(X_{\tau_n}, \tau_n)$ satisfies a weak LDP under $\bP$ with convex, good rate function 
\[
 \bar{I}(x,t):= \inf_{\l\in \R^{d+1}} \l\cdot(x,t) - \bL(\l). 
\]
In particular, for any open, convex subset $G\subset \R^{d+1}$,
\begin{equation}
 \lim_{k\ra\infty} \frac{1}{k} \log \bP\left( \frac{1}{k}(X_{\tau_k}, \tau_k) \in  G \right) = - \inf_{(x,t)\in G} \bar{I}(x,t). \label{Thesis_XtauLDPlb}
\end{equation}
Let $H_\ell := \{ v\in \R^d : v\cdot \ell > 0 \}$. Then, for $v\in H_\ell$, let
\[
\bar{J}(v):= \inf_{0<s\leq 1} s \bar{I}\left(\frac{v}{s},\frac{1}{s}\right).
\]
Having defined the function $\bar{J}$, we now mention a few of its properties. 
\begin{lem}
$\bar{J}$ is a convex function on $H_\ell$, and $\bar{J}(v_P) = 0$. 
\end{lem}
\begin{proof}
We wish to show that $\bar{J}(t v_1 + (1-t) v_2) \leq t \bar{J}(v_1) + (1-t) \bar{J}(v_2)$ for any $v_1,v_2 \in H_\ell$ and $t\in[0,1]$. Obviously, we may assume that $\bar{J}(v_1),\bar{J}(v_2) < \infty$, since otherwise the inequality holds trivially. 
For $s\in(0,1]$ and $v\in H_\ell$, let
\[
f(v,s):= s \bar{I}\left(\frac{v}{s},\frac{1}{s}\right) = \sup_{\l \in \R^{d+1}} \l \cdot( v, 1) - s \bL(\l). 
\]
Since $f(\cdot, \cdot)$ is the supremum of a family of linear functions, $f(\cdot, \cdot)$ is a convex function on $H_\ell \times (0,1]$. 
For $\d>0$, the definition of $\bar{J}$ implies that there exist $s_1,s_2 \in (0,1]$ such that $f(v_1,s_1) < \bar{J}(v_1) + \d$ and $f(v_2,s_2) < \bar{J}(v_2) + \d$.
Therefore,
\begin{align*}
\bar{J}(t v_1 + (1-t)v_2) = \inf_{s\in(0,1]} f(tv_1+(1-t)v_2,s) &\leq f\left(t v_1 + (1-t)v_2,t s_1 + (1-t) s_2\right)\\
& \leq t f(v_1, s_1) + (1-t)f(v_2,s_2) \\
&< t\bar{J}(v_1) + (1-t)\bar{J}(v_2) + 2\d,
\end{align*}
where the second to last inequality is due to the convexity of $f(v,s)$.
Letting $\d\ra 0$ finishes the proof of the first part of the lemma.

For the second part of the lemma, note that \eqref{vPregenformula} implies that $v_P = \frac{\bE X_{\tau_1}}{\bE\tau_1}$. Then, the law of large numbers implies that 
\[
\lim_{k\ra\infty} \bP\left( \left\| \frac{1}{k}\left( X_{\tau_k}, \tau_k \right) - ( v_P \bE \tau_1, \bE \tau_1) \right\| < \d \right) = 1, \qquad \forall \d>0.
\]
Thus, \eqref{Thesis_XtauLDPlb} implies that 
\[
 \inf_{\|(x,t)-(v_P \bE \tau_1, \bE \tau_1)\| < \d} \bar{I}(x,t) = 0, \qquad \forall \d>0. 
\]
Since $\bar{I}$ is lower semi-continuous, this implies that $ \bar{I}(v_P \bE \tau_1, \bE \tau_1) = 0.$
Then, the definition of $\bar{J}$ and the fact that $\bar{I}$ is non-negative imply that $\bar{J}(v_P)=0$. 
\end{proof}

\begin{lem}\label{Thesis_analytic}
There exists an $\eta_0>0$ such that $\bar{J}(v)$ is analytic in $\{ v: \|v-v_P\| < \eta_0 \}$. 
\end{lem}
\begin{proof}
First, we claim that $\bL(\l)$ is finite for all $\l$ in a neighborhood of the origin. For, since $\|X_{\tau_1}\| \leq \tau_1$, we have $\bE e^{\l \cdot (X_{\tau_1}, \tau_1)} \leq \bE e^{\|\l\| \|(X_{\tau_1},\tau_1)\|} \leq \bE e^{\sqrt{2}\|\l\| \tau_1}$. However, it was shown in \cite[Theorem 2.1]{sSlowdown} that $\tau_1$ has exponential tails under $\P$ (and therefore also under $\bP$). Thus, $ \bE e^{\l \cdot (X_{\tau_1}, \tau_1)} < \infty$ if $\| \l \|$ is sufficiently small. 

Since $\bL(\l)$ is the logarithm of a non-degenerate moment generating function, $\bL$ is strictly convex and analytic in a neighborhood of the origin. 
Then, since $\bar{I}$ is the Fenchel-Legendre transform of $\bL$, $\bar{I}$ is strictly convex and analytic in a neighborhood of $(v_P \bE \tau_1, \bE\tau_1) = \nabla \bL(\mathbf{0})$ (see Lemma \ref{Appendix_analytic} in Appendix \ref{Aanalytic}).
Therefore, $f(v,s)= s \bar{I}(v/s, 1/s)$ is analytic for $(v,s)$ in a neighborhood of $(v_P, 1/\bE \tau_1)$. 
Thus, it is enough to show that there exists an analytic function $s(v)$ in a neighborhood of $v_P$ such that $\bar{J}(v) = f(v,s(v))$. To this end, first note that $\bar{J}(v_P) = f(v_P, 1/\bE \tau_1) = \inf_{s\in(0,1]} f(v_P, s) = 0$, and therefore, since $f$ is non-negative and analytic in a neighborhood of $(v_P, 1/\bE \tau_1)$, we have that $\frac{\del f}{\del s} (v_P, 1/\bE \tau_1 ) = 0$.
Also, since $f(v,s)$ is a convex function, $\frac{\del f}{\del s} (v,s_0) = 0$ implies that $\bar{J}(v) = f(v,s_0)$. Therefore, it is enough to find an analytic function $s(v)$ in a neighborhood of $v_P$ such that $\frac{\del f}{\del s} (v,s(v))= 0$. 
A version of the implicit function theorem \cite[Theorem 7.6]{fgHolomorphic} gives the existence of such a function $s(v)$ if we can show that
\begin{equation}
\frac{\del^2 f}{\del s^2} (v_P, 1/\bE \tau_1) \neq 0. \label{Thesis_2derivative}
\end{equation}

To see that \eqref{Thesis_2derivative} holds, note that the definition of $f(v,s)$ implies
\begin{equation}
\frac{\del^2 f}{\del s^2} (v,s) = \frac{1}{s^3} \left(v, 1 \right) \cdot D^2 \bar{I}\left( \frac{v}{s}, \frac{1}{s} \right) \cdot \binom{v}{1}, \label{Thesis_2derivI}
\end{equation}
where $D^2 \bar{I}$ is the matrix of second derivatives for $\bar{I}$. However, since $\bar{I}(x,y)$ is strictly convex in a neighborhood of $(v_P \bE \tau_1, \bE\tau_1)$, $D^2 \bar{I}(x,y)$ is strictly positive definite for $(x,y)$ in a neighborhood of $(v_P \bE \tau_1, \bE\tau_1)$. Thus, from \eqref{Thesis_2derivI} we see that $\frac{\del^2 f}{\del s^2} (v_P, 1/\bE \tau_1) > 0$ and so \eqref{Thesis_2derivative} holds. 
\end{proof}
\end{subsection}


\begin{subsection}{LDP Lower Bound}
We now prove the following large deviation lower bound:
\begin{prop}[Lower Bound] \label{Thesis_ldplb}
For any $v \in H_\ell$, 
\[
\lim_{\d\ra 0} \liminf_{n\ra\infty} \frac{1}{n} \log \P( \| X_n - nv \| < n \d ) \geq - \bar{J}(v).
\]
\end{prop}
\begin{proof}
Let $\| \xi \|_1$ denote the $L^1$ norm of the vector $\xi$. Then, it is enough to prove the statement of the proposition with $\| \cdot \|_1$ in place of $\| \cdot \|$. 
Also, since $\P( \| X_n - nv \|_1 < n \d ) \geq \P(D) \bP( \| X_n - nv \|_1 < n \d )$, it is enough to prove the statement of the proposition with $\bP$ in place of $\P$. That is, it is enough to show
\[
\lim_{\d\ra 0} \liminf_{n\ra\infty} \frac{1}{n} \log \bP( \| X_n - nv \|_1 < n \d ) \geq - \bar{J}(v).
\] 

Now, for any $\d>0$ and any integer $k$, since the walk is a nearest neighbor walk,
\[
\bP( \| X_n - nv \|_1 < 4 n\d ) \geq \bP\left( \| X_{\tau_k} - nv \|_1 < 2 n \d, \; | \tau_k - n | < 2n\d \right). 
\]
For any $t\geq 1$, let $k_n=k_n(t):= \lfloor n/t \rfloor$, so that $n - t < k_n t \leq n$ for all $n$. Thus, for any $\d>0$ and $t\geq 1$, and for all $n$ large enough (so that $n\d > t$), 
\begin{align*}
\bP( \| X_n - nv \|_1 < 4 n\d ) &\geq \bP\left( \| X_{\tau_{k_n}} - nv \|_1 < 2 n \d, \; | \tau_{k_n} - n | < 2n\d \right)
\\
&\geq \bP\left( \| X_{\tau_{k_n}} - k_n t v\|_1 <  k_n t \d, \; | \tau_{k_n} - k_n t | < k_n t \d \right).
\end{align*}
Therefore, for any $\d>0$ and $t\geq 1$,
\begin{align*}
&\liminf_{n\ra\infty} \frac{1}{n} \log \bP( \| X_n - nv \|_1 < 4 n\d )  \\
&\qquad \geq \liminf_{n\ra\infty} \frac{1}{n} \log \bP\left( \| X_{\tau_{k_n}} - k_n t v\|_1 <  k_n t \d, \; | \tau_{k_n} - k_n t | < k_n t \d \right) \\
&\qquad \geq \frac{1}{t} \liminf_{n\ra\infty} \frac{1}{k_n} \log \bP\left( \| X_{\tau_{k_n}} - k_n t v\|_1 <  k_n t \d, \; | \tau_{k_n} - k_n t | < k_n t \d \right) \\
&\qquad = \frac{1}{t} \liminf_{k\ra\infty} \frac{1}{k} \log \bP\left( \| X_{\tau_{k}} - k t v\|_1 <  k t \d, \; | \tau_{k} - k t | < k t \d \right)\\
&\qquad = -\frac{1}{t} \inf_{ \substack{  \| x - tv\|_1 < t\d \\  | y - t | < t \d } } \bar{I}(x,y),
\end{align*}
where the last equality is from \eqref{Thesis_XtauLDPlb}. 
Then, taking $\d\ra 0$ we get that for any $t\geq 1$, 
\[
\lim_{\d\ra 0} \liminf_{n\ra\infty} \frac{1}{n} \log \bP( \| X_n - nv \|_1 < 4 n\d ) \geq -\frac{1}{t} \bar{I}(vt,t). 
\]
Since the above is true for any $t$, the proof is completed by taking the supremum of the right hand side over all $t\geq 1$ and 
recalling the definition of $\bar{J}$.  
\end{proof}
\begin{cor}\label{Thesis_ldplbC}
$\bar{J}(v) \geq H(v)$ for all $v\in H_\ell$. 
\end{cor}
\begin{proof}
The annealed LDP in Theorem \ref{annealedLDP} implies that 
\[\lim_{\d\ra 0} \liminf_{n\ra 0} \frac{1}{n} \log \P( \| X_n - nv \| < n \d ) = -H(v).\]  
The proof then follows immediately from Proposition \ref{Thesis_ldplbC}.
\end{proof}
\noindent\textbf{Remark:} The proof of Proposition \ref{Thesis_ldplb} does not use that $P$ is non-nestling. In fact, it is enough to assume that $\P(A_\ell) > 0$ so that, by the remark after Theorem \ref{Thesis_regindep}, we can use the i.i.d. structure for regeneration times under $\bar{P}$. 
\end{subsection}

\begin{subsection}{LDP Upper Bound in a neighborhood of $v_P$}\label{Thesis_upperbound}
We now wish to prove a matching large deviation upper bound to Proposition \ref{Thesis_ldplb}. 
Ideally, we would like for the upper bound to be valid for all $v \in H_\ell$, and in the next subsection we prove that this is the case when $d=1$. 
For $d>1$ we are only able to prove a matching upper bound to Proposition \ref{Thesis_ldplb} in a neighborhood of $v_P$.
However, this is enough to be able to prove Theorem \ref{Thesis_LDPdiff}.

We first prove an upper bound involving regeneration times.
\begin{lem}\label{chcvlem}
For any $t,k\in \N$ and any $x\in \Z^d$,
\[
\bP( X_{\tau_k} = x, \; \tau_k = t ) \leq e^{-t \bar{J}\left( \frac{x}{t} \right) } . 
\]
\end{lem}
\begin{proof}
Chebychev's inequality implies that, for any $\l \in \R^{d+1}$,
\[
\bP\left( X_{\tau_k}=x, \tau_k=t \right) \leq e^{-\l\cdot(x,t)} \bE e^{\l \cdot(X_{\tau_k},\tau_k)} = e^{-k\left( \l\cdot(x/k, t/k) - \bL(\l) \right)},
\]
where in the last equality we used the i.i.d. structure of regeneration times from Theorem \ref{Thesis_regindep}. 
Thus, taking the infimum over all $\l \in \R^{d+1}$ and using the definition of $\bar{J}$ (with $s=\frac{k}{t}$),  
\[
\bP\left( X_{\tau_k}=x, \tau_k=t \right) \leq e^{-k \bar{I}\left( \frac{x}{k}, \frac{t}{k} \right) } = e^{-t \frac{k}{t} \bar{I}\left( \frac{x}{t} \frac{t}{k}, \frac{t}{k} \right) } \leq e^{-t \bar{J}\left( \frac{x}{t} \right)} .
\]
\end{proof}

We are now ready to give a matching upper bound to Proposition \ref{Thesis_ldplb} in a neighborhood of $v_P$. 
\begin{prop}[Upper Bound] \label{Thesis_ldpub}
There exists an $\eta > 0$ such that, for any $\|v-v_P\| < \eta$ and for all $\d$ sufficiently small, 
\[
\limsup_{n\ra\infty} \frac{1}{n} \log \P \left( \|X_n - n v \| < n \d \right) \leq - \inf_{\|x - v\| < \d}  \bar{J}(x).
\]
\end{prop}
\begin{proof}
Recall that, since $P$ is non-nestling, \cite[Theorem 2.1]{sSlowdown} implies that $\tau_1$ has exponential tails. Thus, there exist constants $C_1,C_2>0$ such that $\max \left\{ \P( \tau_1 > x ) , \bP( \tau_1 > x ) \right\} \leq C_1 e^{-C_2 x}$ for all $x>0$. 
By Lemma \ref{Thesis_analytic}, we know that there exists an $\eta_0$ such that $\bar{J}$ is analytic on $\{v: \|v-v_P\| < \eta_0 \}$. We now introduce the following functions which will be useful in the proof:
\[
\a(r) := \sup_{v: \|v-v_P\| \leq r  } \bar{J}(v) , \qquad \b(r) := \sup_{v: \|v-v_P\| \leq r } \left\| \nabla \bar{J}(v) \right\|, \qquad r< \eta_0 . 
\]
Since $\bar{J}$ is non-negative and analytic on $\{v: \| v - v_P \| < \eta_0 \}$, and since $\bar{J}(v_P)=0$, $\a(r)$ and $\b(r)$ are continuous on $[0,\eta_0)$ and  $\a(0)=\b(0) = 0$. 
Choose $\e>0$ such that $\frac{4\e}{1-2\e} < \eta_0$ and $\b\left( \frac{4\e}{1-2\e} \right) < \frac{C_2}{4}$. Then, choose $\eta > 0$ small enough so that $ \frac{ \eta + 4\e}{1-2\e} < \eta_0$, $\b\left( \frac{ \eta + 4\e}{1-2\e} \right) < \frac{C_2}{4}$ and $\a(\eta) < \frac{C_2}{2} \wedge \e C_2$.  

Having introduced the necessary notation, we now proceed with the proof. Let $\e,\eta>0$ be chosen as above, let $v$ be such that $\| v- v_P \| < \eta$, and let $\d < \eta - \| v- v_P \|$. Now,
\begin{align}
\P( \| X_n - n v \| < n \d ) &\leq \P(\exists k\leq n: \tau_k-\tau_{k-1} \geq \e n)  \nonumber \\
&\qquad + \P\left( \exists k: \tau_1 < \e n, \; \tau_k \in (n-\e n, n], \; \| X_n - n v \| < n \d, \; \tau_{k+1} > n \right). \label{Thesis_tausmall}
\end{align}
Then, since $\bar{J}(v) \leq \a( \eta ) < \e C_2 $, 
\[
\P(\exists k\leq n: \tau_k-\tau_{k-1} \geq \e n) \leq C_1 n e^{-C_2 \e n} \leq C_1 n e^{-n\bar{J}(v)}.
\]
Thus, we need only to bound the second term in \eqref{Thesis_tausmall}.  

Since the random walk is a nearest neighbor walk, $\| X_{\tau_k} - nv \| 
\leq \| X_n - nv \| + | n-\tau_k|$. Thus,
\begin{align}
&\P\left( \exists k: \tau_1 < \e n, \; \tau_k \in (n-\e n, n], \; \| X_n - n v \| < n \d, \; \tau_{k+1} > n \right) \nonumber \\
&\qquad \leq  \sum_{k\leq n} \sum_{u\in(0,\e)} \sum_{s\in [0,\e )} \P\left( \tau_1 = u n, \; \tau_k = (1-s)n, \; \|X_{\tau_k} - n v \| < n(\d + s), \; \tau_{k+1}-\tau_k > n s \right),  \nonumber
\end{align}
where the above sums are only over the finite number of possible $u,s$ and $x$ such that the probabilities are non-zero. However, 
\begin{align*}
&\P\left( \tau_1 = u n, \; \tau_k = (1-s)n, \; \|X_{\tau_k} - n v \| < n(\d + s), \; \tau_{k+1} > n  \right)\\
&\qquad \leq \P\left( \tau_1 = u n, \; \tau_k-\tau_1 = (1-s-u)n, \; \| X_{\tau_k} - X_{\tau_1} - nv \| \leq n (\d + s + u), \; \tau_{k+1}-\tau_k > n s \right) \\
&\qquad = \P(\tau_1 = u n) \bP\left( \tau_{k-1} = (1-s-u)n, \; \| X_{\tau_{k-1}} - nv \| \leq n(\d+s+u) \right) \bP(\tau_1 > ns),
\end{align*}
where the first inequality again uses the fact that the random walk is a nearest neighbor random walk, and the last equality uses the independence structure of regeneration times from Theorem \ref{Thesis_regindep}. Thus, since $\P(\tau_1 = u n ) \leq C_1 e^{-C_2 u n}$ and $\bP(\tau_1 > ns ) \leq C_1 e^{-C_2 s n}$,  
\begin{align}
&\P\left( \exists k: \tau_1 < \e n, \; \tau_k \in (n-\e n, n], \; \| X_n - n v \| < n \d, \; \tau_{k+1} > n \right) \nonumber \\
& \quad  \leq \sum_{k\leq n} \sum_{u\in(0,\e)} \sum_{s\in [0,\e )} C_1^2 e^{-C_2 (u+s) n} \bP\left( \tau_{k-1} = (1-s-u)n, \; \|X_{\tau_{k-1}} - n v \| < n(\d + s + u) \right). \label{Thesis_smallfirsttau}
\end{align}
Then, Lemma \ref{chcvlem} implies that \eqref{Thesis_smallfirsttau} is bounded above by
\begin{align}
&\sum_{k\leq n} \sum_{u\in(0,\e)} \sum_{s\in [0,\e )} \sum_{\| x - v \| < \d + u+s} e^{-n(1-s-u)\bar{J}\left( \frac{x}{1-s-u} \right)} C_1^2 e^{-C_2 (s+u) n } \nonumber \\
&\qquad \leq C_3 n^{d+3} \sup_{s\in [0,2\e)} \sup_{\|x - v\| < \d + s} e^{-n\left( (1-s)\bar{J}\left( \frac{x}{1-s} \right) + C_2 s \right)} \nonumber \\
&\qquad = C_3 n^{d+3} 
\exp \left\{ -n \left( \inf_{s\in[0,2\e)} \inf_{\|x - v\| < \d + s} (1-s)\bar{J}\left( \frac{x}{1-s} \right) + C_2 s \right) \right\}, \label{Thesis_2sups}
\end{align}
for some constant $C_3$ depending only on $\e, \eta$ and $C_1$. Therefore, to finish the proof of the proposition, it is enough to show that the infimum in \eqref{Thesis_2sups} is achieved when $s=0$. That is, it is enough to show the infimum is larger than $\inf_{\|x-v\| < \d} \bar{J}(x)$. 

To this end, note that 
\begin{align}
\inf_{s\in[0,2\e)} & \inf_{\|x - v\| < \d + s} (1-s)\bar{J}\left( \frac{x}{1-s} \right) + C_2 s \nonumber \\
& = \inf_{\|x-v\| < \d} \inf_{s\in[0,2\e)} \inf_{\|y - x\| < s} (1-s)\bar{J}\left( \frac{y}{1-s} \right) + C_2 s \nonumber \\
&\geq \inf_{\|x-v\| < \d} \inf_{s\in[0,2\e)} \inf_{\|y - x\| < s} \left[ \bar{J}(x) - (1-s)\left| \bar{J}\left( \frac{y}{1-s} \right) - \bar{J}(x)  \right|  + s( C_2 - \bar{J}(x) ) \right]. \label{Thesis_3infs}
\end{align}
Since $\d < \eta - \|v-v_P\|$, then $\|x- v\| < \d$ implies that $\| x - v_P \| < \eta$. Thus, $\|y - x \| < s$ implies that 
\[
\left\|\frac{y}{1-s} - v_P \right\| \leq  \left\|\frac{y-x}{1-s}\right\| + \left\| \frac{x}{1-s} - v_P \right\| \leq \frac{s}{1-s} + \frac{\|x-v_P\| + s}{1-s} \leq \frac{\eta+2s}{1-s} \leq \frac{\eta+4\e}{1-2\e}.
\] 
Therefore, $x, \frac{y}{1-s} \in \{v:\|v-v_P\| < \eta_0 \}$, since $\eta$ and $\e$ were chosen so that $\eta < \frac{\eta + 4\e}{1-2\e} < \eta_0$. Since $\bar{J}$ is analytic in $\{ v: \| v-v_P \| < \eta_0 \}$, the mean value theorem implies that
\[
 \bar{J}\left( \frac{y}{1-s} \right) - \bar{J}(x) = \nabla\bar{J}(\xi) \cdot \left( \frac{y}{1-s} - x \right) 
\]
for some $\xi$ on the segment between $x$ and $y/(1-s)$. Thus,
\[
\left| \bar{J}\left( \frac{y}{1-s} \right) - \bar{J}(x)  \right| \leq \sup_{\| \xi - v_P \| < \frac{\eta+4\e}{1-2\e}} \left\| \nabla\bar{J}(\xi) \right\| \left\| \frac{y}{1-s} - x \right\| \leq \b\left( \frac{ \eta+ 4\e}{1-2\e} \right) \frac{2s}{1-s} \leq \frac{ C_2 s}{2(1-s)},
\]
where the last inequality is due to our choice of $\eta$ and $\e$. Recalling \eqref{Thesis_3infs}, we obtain 
\begin{align*}
\inf_{s\in[0,2\e)} \inf_{\|x - v\| < \d + s} (1-s)\bar{J}\left( \frac{x}{1-s} \right) + C_2 s  &\geq \inf_{\|x-v\| < \d} \inf_{s\in[0,2\e)} \inf_{\|y - x\| < s} \bar{J}(x) - \frac{C_2}{2} s + s( C_2 - \bar{J}(x) )\\
&= \inf_{\|x-v\| < \d} \inf_{s\in[0,2\e)} \bar{J}(x) + s\left( \frac{C_2}{2} - \bar{J}(x) \right)\\
&= \inf_{\|x-v\| < \d} \bar{J}(x) ,
\end{align*}
where the last inequality is because $\| x - v_P \| < \eta$, and thus $\bar{J}(x) \leq \a(\eta) < \frac{C_2}{2}$ by our choice of $\eta$. This completes the proof of the proposition.
\end{proof}
\begin{cor}\label{Thesis_ldpubC}
There exists an $\eta > 0$ such that $\bar{J}(v) \leq H(v)$ for all $\|v-v_P\| < \eta$. 
\end{cor}
\begin{proof}
The proof is similar to the proof of Corollary \ref{Thesis_ldplbC}. Theorem \ref{annealedLDP} implies that
\[
\lim_{\d\ra 0} \limsup_{n\ra\infty} \frac{1}{n} \log \P \left( \|X_n - n v \| < n \d \right) = - H(v). 
\]
The corollary then follows immediately from Proposition \ref{Thesis_ldpub}. 
\end{proof}
\end{subsection}

The proof of Theorem \ref{Thesis_LDPdiff} is now almost immediate.
\begin{proof}[\textbf{Proof of Theorem \ref{Thesis_LDPdiff}}:]$\left.\right.$\\
Corollaries \ref{Thesis_ldplbC} and \ref{Thesis_ldpubC} imply that $H(v) = \bar{J}(v)$ in a neighborhood of $v_P$. 
Lemma \ref{Thesis_analytic} then implies that $H(v)$ is analytic in a neighborhood of $v_P$.
\end{proof}

\begin{subsection}{Equality of $\bar{J}$ and $H$ when $d=1$}
For $d>1$, we only know that $\bar{J}(v)=H(v)$ in a neighborhood of $v_P$ when $P$ is non-nestling. In this subsection, we show that when $d=1$ and $\ell = 1$ (that is, $E_P \log \rho < 0$), the equality $\bar{J}(v) = H(v)$ holds for all $v\geq 0$. (Note that in this subsection we no longer assume that $P$ is non-nestling, but we still require $P$ to be i.i.d. and uniformly elliptic.) 
A crucial step in our proof of this fact is the following lemma:
\begin{lem} \label{Thesis_J0}
Assume $E_P \log \rho < 0$ and let $\bar{J}(0):= \lim_{v\ra 0^+} \bar{J}(v)$. Then, $\bar{J}(0) = H(0)$. 
\end{lem}
The proof of Lemma \ref{Thesis_J0} is rather long and technical, and thus will be given in Appendix \ref{Jzero}. 
\begin{cor} \label{Thesis_worstcaseslowdown}
Assume $E_P \log \rho < 0$. Then, $\left \| P_\w( X_n \leq 0 ) \right\|_{\infty} \leq e^{-n \bar{J}(0)}. $
\end{cor}
\begin{proof}
If the environment is nestling, then $\bar{J}(0)= 0$ and the statement is trivial. On the other hand, if the environment is non-nestling, 
\cite[equation (79)]{cgzLDP} implies that $ P_\w( X_n \leq 0 ) \leq e^{-n H(0)}$, $P-a.s$. The corollary then follows immediately from Lemma \ref{Thesis_J0}.
\end{proof}
Using Corollary \ref{Thesis_worstcaseslowdown}, we obtain the following improvement of Proposition \ref{Thesis_ldpub}: 
\begin{prop}[Upper Bound ($d=1$)]\label{ldpubd1}
Assume that $E_P \log \rho < 0$. Then, for any $v>0$ and $\d < v$, 
\[
\lim_{n\ra\infty} \frac{1}{n} \log \P(|X_n - nv| < \d n) \leq - \inf_{|x-v|< \d} \bar{J}(x).
\]
\end{prop}
\begin{proof}
Let $\s_n:= \sup \{k\leq n: X_k = 0\}$ be the time of the last visit to zero before time $n$. Then, by decomposing the path of the random walk according to $\s_n$ and the first hitting time of $X_n$, 
\begin{align}
\P( |X_n & - nv| < \d n) = \sum_{|x-nv| < \d n} \P(X_n = x)\nonumber\\
&= \sum_{|x-nv|< \d n} \sum_{0\leq t<s\leq n} \P( \s_n = t, \; T_x = s, \; X_n = x ) \nonumber \\
&\leq \sum_{|x-nv|< \d n} \sum_{0\leq t<s\leq n} \left\| P_\w( X_t = 0 ) \right\|_\infty \P(T_x = s-t, \; T_{-1} > s-t) \left\| P_\w( X_{n-s} = 0 )\right\|_\infty \nonumber\\
&\leq \sum_{|x-nv|< \d n} \sum_{0\leq t<s\leq n} e^{-(n-s+t)\bar{J}(0)} \P(T_x = s-t, \; T_{-1} > s-t), \label{Thesis_pathdecomp}
\end{align}
where the last inequality is from Corollary \ref{Thesis_worstcaseslowdown}.
Next, note that
\begin{align}
\P(T_{x}=s-t, T_{-1} > s-t) = \frac{\P( T_{x}= s-t, \; T_{-1} > s-t) \P^x(T_{x-1} = \infty) }{ \P(T_{-1} = \infty) }. \label{addregen1}
\end{align}
Since $P_\w( T_{x}= s-t, \; T_{-1} > s-t)$ and $P_\w^x( T_{x-1} = \infty)$ depend on disjoint sections of the environment,  
\begin{align}
\P( T_{x}= s-t, \; T_{-1} > s-t) \P^x(T_{x-1} = \infty) &= E_P\left[ P_\w( T_{x}= s-t, \; T_{-1} > s-t) P_\w^x( T_{x-1} = \infty) \right] \nonumber \\
&= \P\left( T_{x}= s-t, \; T_{-1} > s-t, \; X_r \geq x \:\: \forall r\geq s-t \right) \nonumber \\
&= \P( \exists k: \tau_k = s-t, X_{\tau_k} = x , \; T_{-1} = \infty ). \label{addregen2}
\end{align}
Thus, \eqref{addregen1} and \eqref{addregen2} imply that
\begin{align}
\P(T_{x}=s-t, T_{-1} > s-t) = \bP( \exists k: \tau_k = s-t, X_{\tau_k} = x ) \leq n e^{-(s-t)\bar{J}\left( \frac{x}{s-t} \right)},  \label{Thesis_addregeneration}
\end{align}
where the last inequality is from Lemma \ref{chcvlem} and the fact that 
$\tau_k = s-t$ implies $k\leq s-t \leq n$. 
Combining \eqref{Thesis_pathdecomp} and \eqref{Thesis_addregeneration}, we obtain that 
\[
\P( |X_n - nv| < \d n) \leq n \sum_{|x-nv|< \d n} \sum_{0\leq t<s\leq n} e^{-(n-s+t)\bar{J}(0)} e^{-(s-t)\bar{J}\left( \frac{x}{s-t} \right)} \leq n \sum_{|x-nv|< \d n} \sum_{0\leq t<s\leq n} e^{-n \bar{J}\left(\frac{x}{n}\right)},
\]
where the last inequality is from the convexity of $\bar{J}$. Therefore,
\[
\P( |X_n - nv| < \d n) \leq 2\d n^4 \sup_{|x-v| < \d} e^{-n \bar{J}(x)}. 
\]
\end{proof}

\begin{cor}
 Assume $E_P \log \rho < 0$. Then, the annealed rate function $H(v)$ is identical to $\bar{J}(v)$ for all $v\geq 0$. 
\end{cor}
\begin{proof}
Corollary \ref{Thesis_ldplbC} and the remark that follows imply that $\bar{J}(v) \geq H(v)$ for all $v>0$, and Proposition \ref{ldpubd1} implies that $\bar{J}(v) \leq H(v)$ for all $v > 0$. Thus, $\bar{J}(v)=H(v)$ for all $v>0$. Lemma \ref{Thesis_J0} shows that equality holds for $v=0$ as well.  
\end{proof}

\end{subsection}
\end{section}
\end{chapter}

\appendix
\begin{chapter}{A Formula for the Quenched Variance of Hitting Times}\label{Thesis_qvarformula}
In this appendix, we will derive a formula for the quenched variance of $\tau_1$, where $\tau_1$ is the first time a random walk starting at $0$ reaches $1$. Recall that when $E_\w \tau_1 < \infty$, we use $Var_\w\tau_1 := E_\w\left( \tau_1 - (E_\w \tau_1) \right)^2$ to denote the quenched variance of $\tau_1$. Our goal is to prove the following formula for $Var_\w \tau_1$, which was stated in \eqref{Thesis_qvar}:
\begin{align}
Var_\w\tau_1  &= \S(\w)^2 - \S(\w) + 2 \sum_{n=1}^\infty \Pi_{-n+1,0}\S(\theta^{-n}\w)^2  \label{Appendix_qvar1} \\
&= 4(W_{0}+ W_{0}^2)+ 8 \sum_{i<0} \Pi_{i+1,0}(W_{i}+W_{i}^2), \label{Appendix_qvar2} 
\end{align}
where $\S(\w)$ and $W_i$ are defined in \eqref{Thesis_QET} and \eqref{Thesis_Wdef}, respectively. 
Since $Var_\w \tau_1 = E_\w \tau_1^2 - (E_\w \tau_1)^2$, and since \eqref{Thesis_QET} implies that $E_\w \tau_1= \S(\w)$, \eqref{Appendix_qvar1} is equivalent to 
\begin{equation}
E_\w \tau_1^2 =  2\S(\w)^2 - \S(\w) + 2 \sum_{n=1}^\infty \Pi_{-n+1,0} \S(\theta^{-n}\w)^2.  \label{Appendix_qvar3} 
\end{equation}
\eqref{Appendix_qvar2} then follows from \eqref{Appendix_qvar1} by noting that $\S(\theta^{-n}\w)=1+2 W_{-n}$.

To prove \eqref{Appendix_qvar3}, we first truncate $\tau_1$ to guarantee finiteness of expectations. Let $M>0$.  Then, the decomposition of $\tau_1$ in \eqref{Thesis_taudecomposition} implies that 
\[
\tau_1\wedge M \leq 1 + \mathbf{1}_{X_1=-1}(\tau_0'\wedge M + \tau_1'\wedge M),
\]
where $\tau_0'$ is the time it takes to reach $0$ after first hitting $-1$, and $\tau_1'$ is the time it takes to go from $0$ to $1$ after first hitting $-1$. 
Squaring both sides of the above equation and taking quenched expectations, it follows from the strong Markov property that
\begin{align*}
E_\w( \tau_1\wedge M )^2 &\leq 1 + 2E_\w\left( \mathbf{1}_{X_1=-1}(\tau_0'\wedge M+\tau_1'\wedge M)\right)  \\
&\qquad + E_\w\left(\mathbf{1}_{X_1=-1}((\tau_0'\wedge M)^2+2(\tau_0'\wedge M)(\tau_1'\wedge M) + (\tau_1'\wedge M)^2)\right)\\
&= 1 + 2(1-\w_0)\left( E_{\theta^{-1}\w}(\tau_1\wedge M) + E_\w (\tau_1\wedge M)\right) \\
&\quad\quad + (1-\w_0)\left( E_{\theta^{-1}\w}(\tau_1\wedge M)^2  +2E_{\theta^{-1}\w}(\tau_1\wedge M)E_\w(\tau_1\wedge M)+ E_\w(\tau_1\wedge M)^2 \right).
\end{align*}
Solving for $E_\w(\tau_1\wedge M)^2$ gives
\begin{align*}
E_\w( \tau_1\wedge M )^2 
&\leq \frac{1}{\w_0} + 2\rho_0\left(E_{\theta^{-1}\w}(\tau_1\wedge M) + E_\w(\tau_1\wedge M) \right) \\
&\qquad + \rho_0 \left( E_{\theta^{-1}\w}(\tau_1\wedge M)^2+ 2 E_{\theta^{-1}\w}(\tau_1\wedge M)E_\w(\tau_1\wedge M)  \right) \\
&\leq  \frac{1}{\w_0} + 2\rho_0\left( \S(\theta^{-1}\w) + \S(\w) + \S(\theta^{-1}\w)\S(\w)\right) + \rho_0 E_{\theta^{-1}\w}(\tau_1\wedge M)^2 \\
&= 2 \S(\w)^2 - \frac{1}{\w_0} + \rho_0E_{\theta^{-1}\w}(\tau_1\wedge M)^2,
\end{align*}
where the second inequality holds because $E_\w \tau_1 = \S(\w)$, and the last equality is due to the fact that $\rho_0\S(\theta^{-1}\w) = \S(\w)-\frac{1}{\w_0}$.
By iterating the above inequality, we get that
\begin{align}
E_\w( \tau_1\wedge M )^2 
&\leq 2\left( \S(\w)^2 + \rho_0\S(\theta^{-1}\w)^2 + \cdots + \Pi_{-n+1,0}\S(\theta^{-n}\w)^2 \right) \label{Thesis_top}\\
&\quad\quad - \left( \frac{1}{\w_0} + \frac{1}{\w_{-1}}\rho_0 + \cdots + \frac{1}{\w_{-n}} \Pi_{-n+1,0} \right) \label{Thesis_middle}\\
& \quad\quad + \Pi_{-n,0}E_{\theta^{-n-1}\w}(\tau_1\wedge M)^2.  \label{Thesis_bottom}
\end{align}
As $n\ra\infty$, \eqref{Thesis_middle} tends to $\S(\w)$, which is finite by assumption. Also, since $\S(\w)$ is finite, $\Pi_{-n,0} \ra 0$ as $n\ra\infty$, which implies that \eqref{Thesis_bottom} tends to zero as $n\ra\infty$.
Therefore,  
\[
 E_\w( \tau_1\wedge M )^2 \leq 2\S(\w)^2 + 2 \sum_{n=1}^\infty \Pi_{-n+1,0} \S(\theta^{-n}\w)^2  - \S(\w).
\]
monotone convergence then implies that
\begin{equation}
 E_\w \tau_1^2 \leq 2\S(\w)^2 + 2 \sum_{n=1}^\infty \Pi_{-n+1,0} \S(\theta^{-n}\w)^2  - \S(\w). \label{2ndmomentub}
\end{equation}
If $E_\w \tau_1 = \S(\w) < \infty$ but $E_\w \tau_1^2 = \infty$, then obviously the above can be replace by equality. On the other hand, if $E_\w \tau_1^2 < \infty$ we can repeat this argument without truncating by $M$. That is,  
\[
 E_\w \tau_1^2 =  2 \S(\w)^2 - \frac{1}{\w_0} + \rho_0E_{\theta^{-1}\w}\tau_1^2,
\] 
which, after iterating, implies that
\begin{align*}
E_\w \tau_1^2 
&= 2\left( \S(\w)^2 +  \cdots + \Pi_{-n+1,0}\S(\theta^{-n}\w)^2 \right) - \left( \frac{1}{\w_0} +  \cdots + \frac{1}{\w_{-n}} \Pi_{-n+1,0} \right) + \Pi_{-n,0}E_{\theta^{-n-1}\w}\tau_1^2.  
\end{align*}
Omitting the last term and letting $n\ra\infty$, we obtain
\begin{equation}
 E_\w \tau_1^2 \geq 2\S(\w)^2 + 2 \sum_{n=1}^\infty \Pi_{-n+1,0} \S(\theta^{-n}\w)^2  - \S(\w),  \label{2ndmomentlb}
\end{equation}
whenever $E_\w \tau_1^2 < \infty$. Thus, \eqref{Appendix_qvar3} is implied by \eqref{2ndmomentub} and \eqref{2ndmomentlb}. 
\end{chapter}


\begin{chapter}{Analyticity of Fenchel-Legendre Transforms}\label{Aanalytic}
Let $F:\R^d \ra \R$ be a convex function. Then, the Fenchel-Legendre transform $F^*$ of $F$ is defined by 
\begin{equation}
 F^*(x) = \sup_{\l \in\R^d} \l \cdot x - F(\l). \label{FLdef}
\end{equation}
\begin{lemA}\label{Appendix_analytic}
Let $F$ be strictly convex and analytic on an open subset $U\subset \R^d$. Then, $F^*$ is strictly convex and analytic in 
$U' := \{ y\in \R^d : y= \nabla F(\l) \text{ for some } \l \in U \}$. 
\end{lemA}
\begin{proof}
Since $F$ is strictly convex on $U$, $\nabla F$ is one-to-one on $U$. Therefore, for any $x\in U'$, there exists a unique $\l(x) \in U$ such that $\nabla F(\l(x)) = x$. (That is, $x\mapsto \l(x)$ is the inverse function of $\nabla F$ restricted to $U$.) This implies, 
since $\l \mapsto \l \cdot x - F(\l)$ is a concave function in $\l$, that the supremum in \eqref{FLdef} is achieved with $\l=\l(x)$ when $x\in U'$. 
That is, 
\begin{equation}
 F^*(x) = \l(x) \cdot x - F\left((\l(x)\right), \qquad \forall x\in U'. \label{maxachieved}
\end{equation}

Since $F$ is analytic on $U$, then $\nabla F$ is also analytic on $U$. Then, a version of the inverse function theorem \cite[Theorem 7.5]{fgHolomorphic} implies that $\l(\cdot)$ is analytic on $U'$ if 
\begin{equation}
 \det \left(D^2 F(x)\right) \neq 0 , \qquad \forall x \in U, \label{Jacobianneq0}
\end{equation}
where $D^2 F$ is the matrix of second derivatives of $F$. However, since $F$ is strictly convex on $U$, $D^2 F(x)$ is strictly positive definite for all $x\in U$. Thus, \eqref{Jacobianneq0} holds and so $x\mapsto \l(x)$ is analytic on $U'$. Recalling \eqref{maxachieved}, we then obtain that $F^*$ is also analytic on $U'$. 

An application of the chain rule to \eqref{maxachieved} implies that 
\[
 \nabla F^*(x) = \l(x) \quad\text{and}\quad D^2 F^* (x) = D \l(x) = \left( D^2 F  (\l(x)) \right)^{-1}, \qquad \forall x\in U'.
\]
Since $D^2 F$ is strictly positive definite on $U$, the above implies that $D^2 F^*(x)$ is strictly positive definite for all $x\in U'$. Thus $F^*$ is strictly convex on $U'$. 
\end{proof}

\end{chapter}


\begin{chapter}{Proof of Lemma \ref{Thesis_J0}}\label{Jzero}

Recall that for Lemma \ref{Thesis_J0} we are assuming that $P$ is uniformly elliptic and i.i.d., and that $E_P \log \rho < 0$. 
To prove Lemma \ref{Thesis_J0}, we first need the following lemma:
\begin{lemA} \label{Thesis_PbarLDPT}
Assume that $E_P \log \rho_0 < 0$. Then,
\[
\lim_{M\ra\infty} \liminf_{n\ra\infty} \frac{1}{M n} \log \bP\left( T_n \in [ M n, (M+1)n] \right) \geq -H(0).  
\]
\end{lemA}
\begin{proof}
First, note that
\begin{align*}
\bP\left( T_n \in [M n, (M+1)n] \right) &= \frac{1}{\P(T_{-1}= \infty)} \P\left( T_n \in [M n, (M+1)n],\; T_{-1} = \infty \right) \\
&\geq \frac{1}{\P(T_{-1}= \infty)} E_P \left[ P_\w \left( T_n \in [M n, (M+1)n], \; T_{-1} > T_n \right) P_{\theta^n \w} (T_{-1} = \infty) \right]\\
&= \P \left( T_n \in [M n, (M+1)n], \; T_{-1} > T_n  \right),
\end{align*}
where in the last equality we used that the environment is i.i.d.  
Therefore, it is enough to prove
\begin{equation}
\lim_{M\ra\infty} \liminf_{n\ra\infty} \frac{1}{M n} \log \P \left( T_n \in [M n, (M+1)n], \; T_{-1} > T_n  \right) \geq -H(0).  \label{Trapping}
\end{equation}

The idea of the proof of \eqref{Trapping} is to construct an environment which is most likely to make both $T_{-1}$ and $T_n$ large. 
Let $\w_{\min}:= \inf \{ x > 0: P(\w_0 \leq x) > 0 \}$.
%
The proof of \eqref{Trapping} is divided into three cases: $\w_{\min}< \frac{1}{2}$, $\w_{\min}>\frac{1}{2}$, and $\w_{\min}=\frac{1}{2}$. 

\noindent\textbf{Case I: $\w_{\min} < \frac{1}{2}.$} 

$E_P \log \rho_0 < 0$ and $ \w_{\min} < \frac{1}{2}$ imply that $P$ is nestling. Therefore, Theorem \ref{Thesis_zeroset} gives that $H(0)= 0$. 
Now, the event $\{ T_n \in [M n, (M+1)n], \; T_{-1} > T_n \}$ is implied by not reaching $-1$ or $n$ by time $M n$ and then taking $n$ consecutive steps in the positive direction. 
Thus, 
\[
\P \left( T_n \in [M n, (M+1)n], \; T_{-1} > T_n  \right) \geq \w_{\min}^n \P \left( T_n \wedge T_{-1} > M n  \right).
\]
Since $P$ is uniformly elliptic, $\w_{\min} > 0$ and therefore $\lim_{M\ra\infty}\lim_{n\ra\infty} \frac{1}{Mn} \log \w_{\min}^{n} = 0$. Thus, to prove \eqref{Trapping}, it is enough to show that
\begin{equation}
 \lim_{M\ra\infty} \liminf_{n\ra\infty} \frac{1}{M n} \log \P \left( T_n \wedge T_{-1} > M n \right) = 0.  \label{nestTrapping}
\end{equation}

We now define a collection of environments on which the event $\{T_n \wedge T_{-1} > M n \}$ is likely.
Let 
\[
 \mathcal{T}_n := \left\{ \w_{x} \geq \frac{1}{2}, \quad \forall x \in [0, \lfloor n/2 \rfloor) \right\} \cap \left\{\w_x \leq \frac{1}{2}, \quad \forall x \in (\lfloor n/2\rfloor , n ) \right\}.
\]
Now, we can force the event $ \{T_n \wedge T_{-1} > M n \} $ to happen by forcing $M n$ visits to $\lfloor n/2 \rfloor$ before first reaching $n$ or $-1$. That is, letting $T_x^+ := \inf \{ k > 0: X_k = x \}$ be the first return time to $x$, 
\begin{align}
 \P \left( T_n \wedge T_{-1} > M n  \right) 
& \geq E_P\left[ P_\w\left( T_{\lfloor n/2\rfloor} < T_{-1} \right) P_\w^{\lfloor n/2 \rfloor} \left( T_{\lfloor n/2 \rfloor}^+ < (T_{-1} \wedge T_n)  \right)^{Mn} \mathbf{1}_{\mathcal{T}_n} \right] \nonumber \\
&\geq P(\mathcal{T}_n) \inf_{\w \in \mathcal{T}_n} P_\w\left( T_{\lfloor n/2\rfloor} < T_{-1} \right) P_\w^{\lfloor n/2 \rfloor} \left( T_{\lfloor n/2 \rfloor}^+ < (T_{-1} \wedge T_n)  \right)^{Mn}. \label{forceTrap}
\end{align}
Since $E_P \log \rho < 0$ and $\w_{\min} < \frac{1}{2}$, we have that $P(\w_0 \geq \frac{1}{2} ) , P(\w_0 \leq \frac{1}{2} ) > 0$. Therefore,  
\begin{align}
 \lim_{M\ra\infty} \liminf_{n\ra\infty} \frac{1}{M n} \log P(\mathcal{T}_n ) &= \lim_{M\ra\infty} \liminf_{n\ra\infty} \frac{1}{M n} \log \left( P\left(\w_0 \geq \frac{1}{2} \right)^{\lfloor n/2 \rfloor} P\left(\w_0 \leq \frac{1}{2} \right)^{n-\lfloor n/2 \rfloor -1} \right) \nonumber \\
& = \lim_{M\ra\infty} \frac{1}{2 M} \log \left(  P\left(\w_0 \geq \frac{1}{2} \right) P\left(\w_0 \leq \frac{1}{2} \right) \right) = 0. \label{PTrap} 
\end{align}
A coupling argument with a simple random walk implies that
\[
 P_\w\left( T_{\lfloor n/2\rfloor} < T_{-1} \right) \geq \frac{1}{1+\lfloor n/2 \rfloor}, \quad\mbox{ and }\quad P_\w^{\lfloor n/2 \rfloor} \left( T_{\lfloor n/2 \rfloor}^+ < (T_{-1} \wedge T_n)  \right) 
\geq 1-\frac{2}{n}, \qquad \forall \w \in \mathcal{T}_n.
\]
Therefore,
\begin{align}
&\lim_{M\ra\infty} \liminf_{n\ra\infty} \frac{1}{M n} \log \left( \inf_{\w \in \mathcal{T}_n} P_\w\left( T_{\lfloor n/2\rfloor} < T_{-1} \right) P_\w^{\lfloor n/2 \rfloor} \left( T_{\lfloor n/2 \rfloor}^+ < (T_{-1} \wedge T_n)  \right)^{Mn} \right) \nonumber \\
&\qquad \geq  \lim_{M\ra\infty} \liminf_{n\ra\infty} \frac{1}{M n} \log \left( \frac{1}{1+\lfloor n/2 \rfloor} \left(1-\frac{2}{n}\right)^{Mn}  \right) = 0. \label{Pmanyreturns}
\end{align}
Applying \eqref{PTrap} and \eqref{Pmanyreturns} to \eqref{forceTrap}, we obtain \eqref{nestTrapping}.

\noindent\textbf{Case II: $\w_{\min}  > \frac{1}{2}.$}

%
We will prove \eqref{Trapping} in the case where $P(\w_0 = \w_{\min}) > 0$.
(The case where $P(\w_0 = \w_{\min}) = 0$ is then handled by approximation.) Let $\bar{\w}_{\min}$ be the environment with $\w_x=\w_{\min}$ for all $x$. Then,
\begin{align*}
 & \P \left( T_n \in [M n, (M+1)n],\; T_n < T_{-1} \right) \\
 &\qquad \geq P( \w_0 = \w_{\min} )^{n} P_{\bar{\w}_{\min}}\left( T_n \in [M n, (M+1)n],\; T_n < T_{-1} \right)\\
 &\qquad = P( \w_0 = \w_{\min} )^{n} P_{\bar{\w}_{\min}}\left( T_n < T_{-1} \right) P_{\bar{\w}_{\min}}\left( T_n \in [M n, (M+1)n] \bigl | \; T_n < T_{-1} \right).
\end{align*}
Letting $\rho_{\max}:= \frac{1-\w_{\min}}{\w_{\min}} < 1$, 
we have that
$P_{\bar{\w}_{\min}}\left( T_n < T_{-1} \right) \geq P_{\bar{\w}_{\min}}\left( T_{-1} = \infty \right) = 1- \rho_{\max} > 0$. 
Since 
$\lim_{M_\ra\infty} \liminf_{n\ra\infty} \frac{1}{M n} \log  P( \w_0 = \w_{\min} )^{n} = 0$
, to complete the proof of the lemma it is enough to prove that
\begin{equation}
\lim_{M\ra\infty} \liminf_{n\ra\infty} \frac{1}{M n} \log P_{\bar{\w}_{\min}}\left( T_n \in [M n, (M+1)n]\, | \, T_n < T_{-1} \right) \geq -H(0). \label{conditionalLDP}
\end{equation}

Let $\bar{\l} := -\frac{1}{2} \log\left( 4 \w_{\min} (1-\w_{\min}) \right)$, and recall from \cite[proof of Lemma 4]{cgzLDP} that 
\[
\phi(\l) := E_{\bar{\w}_{\min}} e^{\l T_1} = 
\begin{cases} 
	\frac{1 - \sqrt{1-e^{2(\l-\bar{\l})}}}{2(1-\w_{\min}) e^{\l}} & \quad\mbox{ if } \l \leq \bar{\l},\\
	\infty & \quad\mbox{ if } \l > \bar{\l}. 
\end{cases}
\]
We claim that 
\begin{equation}
\lim_{n\ra\infty} \frac{1}{n} \log E_{\bar{\w}_{\min}} \left[  e^{\l T_n} | T_n < T_{-1} \right] = \log \phi(\l), \qquad \forall \l<\infty. \label{Thesis_logMGFlimit}
\end{equation}
To see this, first note that 
\begin{align}
E_{\bar{\w}_{\min}} \left[ e^{\l T_n} | T_n < T_{-1} \right] &= \frac{1}{P_{\bar{\w}_{\min}}( T_n < T_{-1} )} E_{\bar{\w}_{\min}} \left[ e^{\l T_n} \mathbf{1}_{ T_n < T_{-1}} \right] \nonumber\\
&= \left( \frac{1- \rho_{\max}}{1-\rho_{\max}^n} \right) E_{\bar{\w}_{\min}} \left[ e^{\l T_n} \mathbf{1}_{ T_n < T_{-1}} \right]. \label{Thesis_indicator}
\end{align}
Since $\rho_{\max} < 1$, to prove \eqref{Thesis_logMGFlimit} it is enough to show that $\lim_{n\ra\infty} \frac{1}{n} \log E_{\bar{\w}_{\min}} \left[ e^{\l T_n} \mathbf{1}_{ T_n < T_{-1}} \right] = \phi(\l)$. 
For $\l\leq \bar{\l}$, let $\psi_{n,\l}(x):= E_{\bar{\w}_{\min}}^x \left[ e^{\l T_n} \mathbf{1}_{ T_n < T_{-1}} \right]$ for $-1\leq x \leq n$. Then,
$\psi_{n, \l}(-1) = 0$, $\psi_{n, \l}(n) = 1$, and 
\begin{equation*}
\psi_{n, \l}(x) = \w_{\min} e^{\l} \psi_{n, \l}(x+1) + (1-\w_{\min}) e^{\l} \psi_{n, \l}(x-1), \qquad \forall -1 < x < n. \label{psirecursion}
\end{equation*}
This system of equations can be solved explicitly. In particular,
\begin{equation}
\psi_{n, \l}(0) = E_{\bar{\w}_{\min}} \left[ e^{\l T_n} \mathbf{1}_{ T_n < T_{-1}} \right] 
= \phi(\l)^n \left( \sum_{k=0}^n \rho_{\max}^k \phi(\l)^{2k} \right)^{-1} . \label{Thesis_mgfformula}
\end{equation}
Since $\phi(\l) \leq \phi(\bar{\l}) = \rho_{\max}^{-1/2}$ for all $\l \leq \bar{\l}$, we have that $\sum_{k=0}^n \rho_{\max}^k \phi(\l)^{2k} \leq n+1$. 
Thus, for $\l\leq \bar{\l}$, the limit in \eqref{Thesis_logMGFlimit} follows from \eqref{Thesis_indicator} and \eqref{Thesis_mgfformula}. 
For $\l > \bar{\l}$, the limit in \eqref{Thesis_logMGFlimit} then follows from the fact that $\log E_{\bar{\w}_{\min}} \left[ e^{\l T_n} \mathbf{1}_{ T_n < T_{-1}} \right]$ is a convex function of $\l$ and $\lim_{\l \ra\bar{\l}-} \phi'(\l) = +\infty$. 

It is easily checked that 
for any $t\in(1,\infty)$, there exists a unique $\l<\bar{\l}$ such that $\left( \log \phi(\l) \right)' = t$. 
Then, since \eqref{Thesis_logMGFlimit} holds, the G\"artner-Ellis Theorem \cite[Theorem 2.3.6]{dzLDTA} implies that 
\begin{equation*}
 \liminf_{n\ra\infty} \frac{1}{n} \log P_{\bar{\w}_{\min}}\left( T_n \in [A n, B n]\, | \, T_n < T_{-1} \right) \geq - \inf_{t\in(A,B)} r(t),  \qquad \forall 1 < A < B < \infty,  
\end{equation*}
where $r(t)=\sup_{\l} \l t - \log \phi(\l) =  t \bar{\l} + \frac{t}{2} \log(1-t^{-2}) + \frac{1}{2} \log\left( \rho_{\max} \frac{t+1}{t-1} \right)$ is the Fenchel-Legendre transform of $\log \phi(\l)$.
Since $r(t)$ is increasing for $t>(2\w_{\min} -1)^{-1}$,  $\inf_{t\in(M,M+1)} r(t) = r(M)$ for all $M$ large enough. Therefore,
\[
\lim_{M\ra\infty} \liminf_{n\ra\infty} \frac{1}{M n} \log P_{\bar{\w}_{\min}}\left( T_n \in [M n, (M+1)n]\, | \, T_n < T_{-1} \right) \geq - \lim_{M\ra\infty} \frac{1}{M} r(M) = -\bar{\l}. 
\]
Finally, it was shown in \cite[Lemma 4 and proof of Theorem 1]{cgzLDP} that $H(0)=\bar{\l}$. This completes the proof of \eqref{conditionalLDP}, and thus also the proof of the lemma,  when $\w_{\min} > \frac{1}{2}$.

\noindent\textbf{Case III: $\w_{\min} = \frac{1}{2}.$}

The proof when $\w_{\min} = \frac{1}{2}$ is essentially the same as in the case $\w_{\min} > \frac{1}{2}$.
In particular, it is enough to show \eqref{conditionalLDP}.
The same argument as above shows that 
\[
 \lim_{n\ra\infty} \frac{1}{n} \log E_{\bar{\w}_{\min}} \left[  e^{\l T_n} | T_n < T_{-1} \right] = \log \phi(\l), \qquad \forall \l<\infty,
\]
where $ \phi(\l):= E_{\bar{\w}_{\min}} e^{\l T_1} = \frac{1-\sqrt{1-e^{2\l}}}{e^{\l}}$. Since $0$ is not in the interior of $\{ \l \in \R : \phi(\l) < \infty \}$,
we cannot directly apply the G\"artner-Ellis Theorem as was done above.
However, it is still true that for any $t\in(1,\infty)$, there exists a unique $\l < 0$ such that $\left( \log \phi(\l) \right)' = t$. Thus, the standard exponential change in measure argument which gives the lower bound in the G\"artner-Ellis Theorem is still valid for bounded subsets of $(1,\infty)$. Therefore, 
\[
 \liminf_{n\ra\infty} \frac{1}{n} \log P_{\bar{\w}_{\min}}\left( T_n \in [A n, B n]\, | \, T_n < T_{-1} \right) \geq - \inf_{t\in(A,B)} r(t),  \qquad \forall 1 < A < B < \infty, 
\]
where $r(t) = \frac{t}{2} \log( 1- t^{-2} ) - \frac{1}{2} \log \left( \frac{t-1}{t+1} \right)$. Since $r(t)$ is decreasing with $\lim_{t\ra\infty} r(t) = 0$, 
\[
 \lim_{M\ra\infty} \liminf_{n\ra\infty} \frac{1}{M n} \log P_{\bar{\w}_{\min}}\left( T_n \in [M n, (M+1)n]\, | \, T_n < T_{-1} \right) \geq - \lim_{M\ra\infty} \frac{1}{M} r(M+1) = 0.
\]
Note that $H(0)= 0$ since $\w_{\min}=0$ implies that $P$ is nestling. Thus, \eqref{conditionalLDP} holds when $\w_{\min}=\frac{1}{2}$ as well. 
\end{proof}

\begin{proof}[\textbf{Proof of Lemma \ref{Thesis_J0}:}]$\left.\right.$\\
From Corollary \ref{Thesis_ldplbC} and the remark that follows, we know that 
\[
\bar{J}(0)=\lim_{v\ra 0^+} \bar{J}(v) \geq \lim_{v\ra 0^+} H(v) = H(0).
\] 
Thus, we only need to show that $\bar{J}(0) \leq  H(0)$. From Lemma \ref{Thesis_PbarLDPT} (replacing $M$ by $\frac{1}{\e}$ and $n$ by $\lfloor \e n \rfloor$), 
\begin{equation*}
\lim_{\e\ra 0+} \liminf_{n\ra\infty} \frac{1}{n} \log \bP\left( T_{\lfloor \e n \rfloor } \in [n, n+ \e n] \right) \geq -H(0).  \label{Thesis_Ten}
\end{equation*}
Therefore, it is enough to show 
\begin{equation}
 \lim_{\e\ra 0^+} \limsup_{n\ra\infty} \frac{1}{n} \log \bP\left( T_{\lfloor \e n \rfloor } \in [n, n+ \e n] \right) \leq -\bar{J}(0). \label{Thesis_Ten2}
\end{equation}

For an upper bound on $ \bP( T_{\lfloor \e n \rfloor } \in [n, n+ \e n] ) $, note that
\begin{align}
\bP\left( T_{\lfloor \e n \rfloor } \in [n, n+ \e n] \right) 
& = \frac{\P\left(  T_{\lfloor \e n \rfloor } \in [n, n+ \e n] , \; T_{-1}=\infty \right) }{\P(T_{-1} = \infty)} \nonumber \\
&\leq \frac{ \P\left(  T_{\lfloor \e n \rfloor } \in [n, n+ \e n] , \;  T_{\lfloor \e n \rfloor } < T_{-1} \right) }{\P(T_{-1} = \infty)} \nonumber \\
&= \frac{ \P\left(  T_{\lfloor \e n \rfloor } \in [n, n+ \e n] , \; T_{\lfloor \e n \rfloor } < T_{-1} \right) \P^{\lfloor \e n \rfloor} (T_{\lfloor \e n \rfloor -1} = \infty) }{\P(T_{-1} = \infty)^2}. \label{forceregen1} 
\end{align}
Since $P_\w\left( T_{\lfloor \e n \rfloor } \in [n, n+ \e n] , \; T_{-1} < T_{\lfloor \e n \rfloor } \right)$ and $P_\w^{\lfloor \e n \rfloor} \left(T_{\lfloor \e n \rfloor -1} = \infty \right)$ depend on disjoint sections of the environment, 
\begin{align*}
&\P\left(  T_{\lfloor \e n \rfloor } \in [n, n+ \e n] , \; T_{-1} < T_{\lfloor \e n \rfloor } \right) \P^{\lfloor \e n \rfloor} \left(T_{\lfloor \e n \rfloor -1} = \infty \right) \\
&\qquad = E_P \left[ P_\w \left( T_{\lfloor \e n \rfloor } \in [n, n+ \e n] , \; T_{-1} < T_{\lfloor \e n \rfloor } \right) P_\w^{\lfloor \e n \rfloor} \left(T_{\lfloor \e n \rfloor -1} = \infty \right) \right] \\
&\qquad = \P \left( T_{\lfloor \e n \rfloor } \in [n, n+ \e n] , \; T_{-1} < T_{\lfloor \e n \rfloor }, \; X_k \geq \lfloor \e n \rfloor \quad \forall k \geq T_{\lfloor \e n \rfloor } \right) \\
&\qquad =  \P\left(  \exists k  : \tau_k \in [n,n+\e n], \; X_{\tau_k} = \lfloor \e n \rfloor, \; T_{-1} = \infty \right).
\end{align*}
Therefore, \eqref{forceregen1} implies that
\begin{align}
\bP\left( T_{\lfloor \e n \rfloor } \in [n, n+ \e n] \right) 
&=\frac{ \P\left(  \exists k  : \tau_k \in [n,n+\e n], \; X_{\tau_k} = \lfloor \e n \rfloor, \; T_{-1} = \infty \right) }{\P(T_{-1} = \infty)^2} \nonumber \\
&= \frac{\bP\left(  \exists k : \tau_k \in [n,n+\e n], \; X_{\tau_k} = \lfloor \e n \rfloor \right)}{\P(T_{-1} = \infty)}. \label{forceregen}
\end{align}
But then, 
\begin{align}
\bP\left(\exists k: \tau_k \in [n, n+\e n] , X_{\tau_k} = \lfloor \e n \rfloor \right) &\leq \sum_{k\leq \e n} \: \sum_{t \in [1,1+\e]} \bP( \tau_k = t n , X_{\tau_k} = \lfloor \e n \rfloor ) \nonumber \\
&\leq (\e n)^2 \sup_{t\in[1,1+\e]} e^{-n t \bar{J}\left( \frac{\lfloor \e n \rfloor }{n t} \right ) }, \label{Thesis_regnearn}
\end{align}
where the last inequality is due to Lemma \ref{chcvlem}.
Thus, \eqref{forceregen} and \eqref{Thesis_regnearn} imply that 
\[
 \lim_{\e\ra 0^+} \limsup_{n\ra\infty} \frac{1}{n} \log \bP\left( T_{\lfloor \e n \rfloor } \in [n, n+ \e n] \right) \leq - \lim_{\e\ra 0^+} \limsup_{n\ra\infty}  \inf_{t\in[1,1+\e]} t \bar{J}\left( \frac{\lfloor \e n \rfloor}{n t} \right) = -\bar{J}(0),
\]
where the last equality is due to the fact that $\bar{J}(0) = \lim_{v\ra 0^+} \bar{J}(v)$ by definition. This finishes the proof of \eqref{Thesis_Ten2} and thus also the proof of the lemma.
\end{proof}

\end{chapter}


\providecommand{\bysame}{\leavevmode\hbox to3em{\hrulefill}\thinspace}
\providecommand{\MR}{\relax\ifhmode\unskip\space\fi MR }
\providecommand{\MRhref}[2]{%
  \href{http://www.ams.org/mathscinet-getitem?mr=#1}{#2}
}
\providecommand{\href}[2]{#2}

\end{document}